\documentclass[11pt,a4paper]{article}

\usepackage[margin=2.5cm]{geometry}
\usepackage{microtype}
\usepackage{setspace}
\usepackage{parskip}
\usepackage[T1]{fontenc}
\usepackage{lmodern}
\usepackage{amsmath,amssymb,amsfonts}
\usepackage{mathtools}
\usepackage{xcolor}
\usepackage{hyperref}
\usepackage{soul}
\usepackage{mathrsfs}
\usepackage{changepage}

\usepackage{titlesec}
\usepackage{abstract}

\usepackage{amsthm}
\usepackage{thmtools}
\makeatletter
\ifdefined\newcounteralias
  \let\km@newcounteralias\newcounteralias
  \renewcommand*\newcounteralias[2]{%
    \def\@tempa{#1}\def\@tempb{#2}%
    \ifx\@tempa\@tempb\else\km@newcounteralias{#1}{#2}\fi}
\fi
\makeatother
\usepackage{mdframed}

\usepackage{enumitem}
\usepackage{booktabs}
\usepackage{graphicx}
\usepackage{fancyhdr}
\usepackage{lastpage}

\usepackage{enumitem}
\usepackage[normalem]{ulem}

\usepackage{tikz}
\usetikzlibrary{arrows.meta, positioning}

\definecolor{titlecolor}{RGB}{70,0,100}
\definecolor{sectioncolor}{RGB}{0,90,120}
\definecolor{theoremcolor}{RGB}{0,100,80}
\definecolor{definitioncolor}{RGB}{70,30,125}
\definecolor{remarkcolor}{RGB}{150,50,0}
\definecolor{lemmacolor}{RGB}{0,60,110}
\definecolor{propositioncolor}{RGB}{120,0,60}
\definecolor{linkcolor}{RGB}{0,75,125}

\hypersetup{
    colorlinks=true,
    linkcolor=linkcolor,
    filecolor=linkcolor,
    citecolor=linkcolor, 
    urlcolor=linkcolor,
    pdftitle={Projective Variety Recovery from Unknown Linear Projections},
    pdfauthor={Yirmeyahu Kaminski},
    pdfsubject={Algebraic geometry, algebraic vision},
    pdfkeywords={projective varieties, generic projections, realizable configurations, reconstruction, generalized Kruppa equation, fundamental matrix, rigidity, algebraic vision}
}

\titleformat{\section}
{\color{sectioncolor}\normalfont\Large\bfseries}
{\color{sectioncolor}\thesection}{1em}{}

\titleformat{\subsection}
{\color{sectioncolor}\normalfont\large\bfseries}
{\color{sectioncolor}\thesubsection}{1em}{}

\titleformat{\subsubsection}
{\color{sectioncolor}\normalfont\normalsize\bfseries}
{\color{sectioncolor}\thesubsubsection}{1em}{}

\titlespacing*{\section}{0pt}{3.5ex plus 1ex minus .2ex}{2.3ex plus .2ex}
\titlespacing*{\subsection}{0pt}{3.25ex plus 1ex minus .2ex}{1.5ex plus .2ex}
\titlespacing*{\subsubsection}{0pt}{3.25ex plus 1ex minus .2ex}{1.5ex plus .2ex}

\theoremstyle{definition}
\declaretheoremstyle[
    headfont=\normalfont\bfseries\color{theoremcolor},
    bodyfont=\normalfont\itshape,
    mdframed={
        linewidth=1pt,
        rightline=false,
        leftline=true,
        topline=false,
        bottomline=false,
        linecolor=theoremcolor,
    }
]{thmstyle}

\declaretheoremstyle[
    headfont=\normalfont\bfseries\color{definitioncolor},
    bodyfont=\normalfont,
    mdframed={
        linewidth=1pt,
        rightline=false,
        leftline=true,
        topline=false,
        bottomline=false,
        linecolor=definitioncolor,
    }
]{defstyle}

\declaretheoremstyle[
    headfont=\normalfont\bfseries\color{remarkcolor},
    bodyfont=\normalfont,
    mdframed={
        linewidth=1pt,
        rightline=false,
        leftline=true,
        topline=false,
        bottomline=false,
        linecolor=remarkcolor,
    }
]{remstyle}

\declaretheoremstyle[
    headfont=\normalfont\bfseries\color{lemmacolor},
    bodyfont=\normalfont\itshape,
    mdframed={
        linewidth=1pt,
        rightline=false,
        leftline=true,
        topline=false,
        bottomline=false,
        linecolor=lemmacolor,
    }
]{lemstyle}

\declaretheoremstyle[
    headfont=\normalfont\bfseries\color{propositioncolor},
    bodyfont=\normalfont\itshape,
    mdframed={
        linewidth=1pt,
        rightline=false,
        leftline=true,
        topline=false,
        bottomline=false,
        linecolor=propositioncolor,
    }
]{propstyle}

\declaretheorem[style=thmstyle, name=Theorem, numberwithin=section]{theorem}
\declaretheorem[style=defstyle, name=Definition, sibling=theorem]{definition}
\declaretheorem[style=remstyle, name=Remark, sibling=theorem]{remark}
\declaretheorem[style=lemstyle, name=Lemma, sibling=theorem]{lemma}
\declaretheorem[style=propstyle, name=Proposition, sibling=theorem]{proposition}
\declaretheorem[style=thmstyle, name=Corollary, sibling=theorem]{corollary}
\declaretheorem[style=remstyle, name=Example, sibling=theorem]{example}


\setlist[itemize]{leftmargin=*, topsep=2pt, itemsep=2pt, parsep=0pt}
\setlist[enumerate]{leftmargin=*, topsep=2pt, itemsep=2pt, parsep=0pt}

\usepackage{varioref}
\usepackage{cleveref}

\newcommand{\C}{\mathbb C}  
  
\newcommand{\PP}{\mathbb P} 
\newcommand{\R}{\mathbb R}  
\newcommand{\Z}{\mathbb Z}  
 
\newcommand{\G}{\mathbb G}
\newcommand{\V}{\mathbb V}

\newcommand{\rank}{\mbox{rank}}

\newcommand{\Perm}{\mbox{Perm}}
\newcommand{\im}{\mbox{im}}
\newcommand{\codim}{\mbox{codim}}

\newcommand{\join}{\wedge}
\newcommand{\meet}{\diamond}

\newcommand{\lra}{\longrightarrow}
\newcommand{\lla}{\longleftarrow}

\newcommand{\ga}{\gamma}

\newcommand{\ep}{\epsilon}

\newcommand{\si}{\sigma}

\newcommand{\Ga}{\Gamma}

\newcommand{\wtilde}{\widetilde}
\newcommand{\what}{\widehat}
\newcommand{\iso}{\approx}

\newcommand{\J}{\mathcal J}

\DeclareMathOperator{\Tan}{Tan}
\DeclareMathOperator{\Sing}{Sing}
\DeclareMathOperator{\Hom}{Hom}
\DeclareMathOperator{\Sym}{Sym}

\DeclareMathOperator{\corank}{corank}

\theoremstyle{plain}

\makeatletter
\renewcommand{\maketitle}{%
  \begingroup
  \begin{center}
    {\LARGE\bfseries\color{titlecolor}\@title\par}
    \vspace{1.5em}
    
    {\large\@author\par}
    \vspace{0.5em}
    
    {\large\@date\par}
  \end{center}
  \vspace{1em}
  \endgroup
}
\makeatother

\begin{document}

\title{Projective Variety Recovery from Unknown Linear Projections}

\author{{\bf Yirmeyahu Kaminski}\\
School of Mathematical Sciences,\\
Holon Institute of Technology, Holon, Israel. \\
email: kaminsj@hit.ac.il}

\date{\today}

\maketitle

\begin{abstract}
We study how a smooth irreducible variety $X$ of dimension $n$ and degree $d$ in $\mathbb{P}^m$, $m\geq n+2$, can be recovered from the equations of its images under two generic projections from unknown centers. This extends earlier results on curves in $\mathbb{P}^3$ and surfaces in $\mathbb{P}^4$ to any dimension. Since $X$ is the normalization of each image, one image determines $X$ modulo the projective group when $X$ is linearly normal; what two images add is the configuration, that is the pair of projections. For complete intersections we prove unconditionally that the set of configurations compatible with the two images and their Hilbert polynomial is finite, of size at most $|\mathrm{Lin}(X)|$, the projective automorphism group of $X$. The argument is effective and the bound is attained. We then introduce an equation generalizing Kruppa's equation, relating the duals of the images to a fundamental matrix, encoding the pair of projections modulo the projective group. We give a necessary condition for its solution variety to be finite, a bound on its isolated points, and a sufficient condition under which the true fundamental matrix is rigid: no quadric may contain a generic line $L$ together with the locus of points of $X$ whose tangent space meets $L$. Rigidity is unconditional for curves; in higher dimension it holds under a first-order liftability hypothesis, which we certify for a del Pezzo quartic surface in $\mathbb{P}^4$. Global finiteness of that solution variety remains open, necessarily so, as it contains solutions realized by no variety; for $\dim X<\mathrm{codim}\,X-1$ the obstruction amounts to knowing the fundamental matrix itself. Once the projections are known, $X$ is recovered from the two cones over the images: in codimension two their intersection is $X$ plus a component of degree $d(d-1)$; in higher codimension $X$ is its unique component of maximal dimension.
\end{abstract}

\section{Introduction}
\label{sec::intro}

\subsection{Problem Definition}
\label{sec::setting}

Consider a smooth irreducible algebraic variety $X \in \C\PP^{m}$ (in the sequel we simply write $\PP^m$ for $\C\PP^m$) of dimension $n \leq m-2$. This variety is projected onto two projective hyperplanes $i_1(\PP^{m-1})$ and $i_2(\PP^{m-1})$ embedded in $\PP^{m}$ through the
centers of projection, ${\bf O}_1$ and ${\bf O}_2$. Each projection mapping is denoted by $\pi_j:\PP^m \setminus \{{\bf O}_j\} \lra i_j(\PP^{m-1})$. Let $Y_i = \pi_i(X)$ be the different projections of the variety $X$.

Once bases of $\C^{m+1}$ and $\C^m$ are fixed, a projection $\pi_i$ is presented as an $m \times (m+1)$ matrix, ${\bf M}_i$ defined modulo multiplication by non-zero scalar. Then each point ${\bf
P}$ different from ${\bf O}_i$ is mapped by $\pi_i$ to 
${\bf M}_i {\bf P}$. Each projection operator $\pi_i$, via its matrix, can be regarded as a point in
$\PP^{(m+1)m-1}$.  

When we consider the problem of recovering the projection maps from
the projected varieties, we will show that the recovery is possible only modulo some action of the group of projective transformations of $\PP^m$ on the set of projection maps. To define 
this action we refer to a projection map as a point in
$\PP^{(m+1)m-1}$. Consider the
following projective variety $\V = \PP^{(m+1)m-1} \times \PP^{(m+1)m-1}$. Let $\Pr_m  = \PP GL_{m+1}(\C)$ be the group of
projective transformations of $\PP^m$. We define an action of $\Pr_m$
on $\V$ as follows: 

\begin{equation}
\label{action::varrho}
\varrho: \Pr_m \lra Mor(\V,\V), {\bf A} \mapsto (({\bf
Q}_1,{\bf Q}_2) \mapsto ({\bf M}_1 {\bf A}^{-1},{\bf
M}_2 {\bf A}^{-1})),
\end{equation}
 
where each matrix ${\bf M}_i$ is built from
the coordinates of ${\bf Q}_i = [Q_{i,1},...,Q_{i,(m+1)m}]^T$ as follows:
$$
{\bf M}_i = \left[\begin{array}{cccc}
		Q_{i,1} & Q_{i,2} & \hdots & Q_{i,m+1} \\
		Q_{i,m+2} & Q_{i,m+3} & \hdots & Q_{i,2m+2} \\
		\hdots & & & \\
		Q_{i,(m+1)(m-1)} & \hdots & \hdots & Q_{i,(m+1)m}
	    \end{array}\right]
$$ 
The geometric meaning of this action is that if we change the
projective basis in $\PP^m$, by the transformation ${\bf A}$, we need to change the projection maps in accordance for the projected varieties to be invariant.

For two projection operators $\pi_1$ and $\pi_2$, we define a family of \textit{fundamental matrices} $\{{\bf F}_k\}_{k=2}^{m-1}$ where ${\bf F}_k$ maps $(m-1-k)$-planes in $i_1(\mathbb{P}^{m-1})$ to $(m-k)$-planes in $i_2(\mathbb{P}^{m-1})$. 

We show that the knowledge of any of these fundamental matrices is equivalent to the recovery of the projections $\pi_1$ and $\pi_2$ modulo the action $\varrho$. This is done in section~\ref{sec::fund_mat} and in particular in theorem~\ref{thm::FundEquiv}. The projection of the first, respectively the second, center of projection ${\bf O}_1$, respectively ${\bf O}_2$, onto the second, respectively the first, hyperplane is called the first, respectively the second, epipole. 

Given the projected varieties $Y_1$ and $Y_2$ as the only data, we show that the fundamental matrix ${\bf F}_2$ must satisfy a system of polynomial equations related to an equation involving the dual varieties $Y_1^*$ and $Y_2^*$, that we call \textit{the generalized Kruppa equation} for historical reasons that will be presented in the next section. This is done in theorem~\ref{theoKrupeq}. 

Section~\ref{sec::dimensionAnalysis} opens, in section~\ref{sec::Vreal}, with a
question logically prior to the analysis of any particular system: how many
configurations are compatible with the data at all? Call a configuration
\emph{realizable} when some variety with the right Hilbert polynomial projects
from it onto the two given images, and write $V^{\mathrm{real}}$ for the set of
realizable configurations. Theorem~\ref{thm::Vreal_finite} answers the question without any of the hypotheses needed later for the Kruppa system:
\begin{quote}
Let $(Z,L)$ be the normalization of either image with its natural polarization. Then $V^{\mathrm{real}}$ injects into the set of polarized isomorphisms between the two normalizations satisfying one rank condition, and $|V^{\mathrm{real}}|\le|\operatorname{Aut}(Z,L)|$. In particular, the two images determine the configuration up to finitely many possibilities whenever $\operatorname{Aut}(Z,L)$ is finite, as for a smooth complete intersection of degree greater than two, where it is the finite group $\operatorname{Lin}(X)$ of projective transformations preserving $X$; when $\operatorname{Aut}(Z,L)$ is infinite, the rank condition may still force finiteness.
\end{quote}

Every object entering this statement is computable from the image ideals alone:
the normalization and its polarization by an integral closure computation
(remark~\ref{rem::Vreal_compute}), and $V^{\mathrm{real}}$ itself as the solution
set of an explicit polynomial system in a single unknown matrix
(proposition~\ref{prop::Vreal_system}). Remark~\ref{rem::Vreal_scope} delimits the
scope: two images are indispensable for recovering the \emph{variety} precisely
when $X$ is not linearly normal. Recall that a nondegenerate $X\subset\PP^m$ is
\emph{linearly normal} if the restriction map
$H^0(\PP^m,\mathcal{O}(1))\lra H^0(X,\mathcal{O}_X(1))$ is surjective, that is
$h^0(X,\mathcal{O}_X(1))=m+1$; equivalently, $X$ is not the isomorphic projection
of a nondegenerate variety in a larger projective space. It is \emph{projectively
normal} if $H^0(\PP^m,\mathcal{O}(d))\lra H^0(X,\mathcal{O}_X(d))$ is surjective
for every $d\geq0$; for $X$ normal, this amounts to its homogeneous coordinate ring
being integrally closed~\cite[\S26.3.7]{Vakil-25}. A smooth complete
intersection of positive dimension is projectively
normal~\cite[Ex.~26.3.H]{Vakil-25}, hence linearly normal, so that in the
complete-intersection case what the second image contributes is the
configuration. Finally, the equations expressing realizability may be adjoined to
the generalized Kruppa equation, giving the augmented system of
definition~\ref{def::aug_system}; its solutions in the epipolar unknowns are
exactly $V^{\mathrm{real}}$ by theorem~\ref{thm::aug_system}, hence finitely many
with no further hypothesis. The price is a far larger system, and
remark~\ref{rem::aug_cost} isolates from it a small set of explicit determinantal
equations in the epipolar unknowns that are necessary for realizability.

The two following subsections analyze the solution variety $V^{\mathrm{adm}}$ of
the generalized Kruppa equation itself, which is the system one actually solves.
The distinction from $V^{\mathrm{real}}$ is what makes this necessary:
by lemma~\ref{lem::homographic}, $V^{\mathrm{adm}}$ can contain
positive-dimensional families of solutions that no variety realizes, so the
finiteness just stated says nothing by itself about that system. We give a
necessary condition for $V^{\mathrm{adm}}$ to be zero-dimensional, and a
sufficient condition for the true configuration to be isolated in it. The proof is
divided into two cases: (i) $\dim(X) \geq \codim(X)-1$ and (ii)
$\dim(X) < \codim(X) - 1$.

The section culminates in a rigidity theorem (theorem~\ref{thm::main_dimX>=codimX-1}), which contains a sufficient condition:
\begin{quote}
For a generic position of the centers of projection, if the smooth irreducible
complete intersection $X$, with $\deg(X)>2$, is contained in at most two
independent quadrics ($\dim\mathcal J_X(2)\le2$), then the true epipolar
geometry (the true values of the fundamental matrices) is an isolated point of the solution variety of the generalized
Kruppa equation: it admits no continuous deformation compatible with the
equation. For curves this holds unconditionally; for $\dim(X)\ge2$ it holds
under an explicitly stated first-order liftability hypothesis
($(\dagger)$, definition~\ref{def::liftability}).
\end{quote}

It should be noted that the liftability hypothesis ($\dagger$) obeys an all-or-nothing law over the space of centers (see proposition \ref{prop::dichotomy} for details), so it either holds generically or fails identically per variety. 

Global finiteness of $V^{\mathrm{adm}}$, which would upgrade rigidity to a
recovery up to finite-fold ambiguity, is stated here as an open problem for all
$n\ge1$ (remark~\ref{rem::lifting}); it is not implied by the finiteness of
$V^{\mathrm{real}}$ above, the two sets genuinely differing by
lemma~\ref{lem::homographic}. As we explain in remark~\ref{rem::K04_gap}, the
finiteness argument of~\cite{Kaminski-04} for curves implicitly assumes the
liftability of the tangency correspondence at every solution, which is precisely
the hypothesis $(\dagger)$ made explicit here, and therefore establishes rigidity
rather than global finiteness; example~\ref{ex::posdim_Vadm} shows that for
special configurations of the centers the solution variety is in fact
positive-dimensional. In the second regime, rigidity itself needs no induction:
by proposition~\ref{prop::rigidity_from_Vreal} it follows from the finiteness of
$V^{\mathrm{real}}$ together with a single realizability hypothesis
($(\natural')$, definition~\ref{def::realizable}) at the true configuration, with
no bound on $\dim\mathcal J_X(2)$. For the stronger conclusion of finiteness we
set up an induction on $m$ reducing to the boundary case, whose descent step
requires a further hypothesis called descent hypothesis ($(\ddagger)$,
definition~\ref{def::descent}), which we state explicitly and do not prove. The hypothesis $(\dagger)$ obeys an all-or-nothing law over the space of centers
(proposition~\ref{prop::dichotomy}), so that a single certified instance settles
it generically for the variety at hand; example~\ref{ex::delpezzo} carries out
that certification for an explicit del Pezzo quartic surface, which makes the
rigidity statement unconditional there. For $(\ddagger)$ the corresponding law
(proposition~\ref{prop::descent_dichotomy}) resolves the other way:
proposition~\ref{prop::consistency_structure} identifies the consistency locus
with a finite union of joint images of the two projections, of dimension $m$, so
that the sufficient form of $(\ddagger)$ never holds and, by
corollary~\ref{cor::ddagger_unavailable}, no invariant of that locus can supply
it. The descent step is accordingly left as an explicitly stated hypothesis whose
obstruction is now identified. The outcome, unconditional in part and conditional in part, is
summarized in theorem~\ref{thm::main_thm_1}, where for the finiteness statement
the bound on $\dim\mathcal J_X(2)$ is moreover relaxed to
$\dim\mathcal J_X(2)\le m-2\dim(X)+1$ (theorem~\ref{thm::dim(x)*small}).

Finally in section~\ref{sec::recovering_the_variety}, the second main theorem numerated~\ref{thm::reconstruction} is introduced and proven. Let us state and explain it here:
\begin{quote}
For a generic position of the centers of projection, the variety defined by
$\Delta_1 \cap \Delta_2$ is the union of $X$ and of a residual locus of dimension
$2\dim(X) - m + 2$. When $\dim(X) = m-2$, and when no epipolar $(m-n)-$plane (space containing the centers of projection) is
tangent twice to the variety $X$, this residual locus is an irreducible component
and $\Delta_1 \cap \Delta_2$ has two irreducible components. One has degree $d$ and
is the actual solution of the reconstruction. The other one has degree $d(d-1)$.
When $\dim(X) < m-2$, the residual locus has dimension strictly smaller than
$\dim(X)$, so that $X$ is the unique irreducible component of maximal dimension and
the reconstruction is unambiguous without any condition on the degree. 
\end{quote}

An epipolar geometric entity is defined as one that contains the centers of projection. In the intersection $\Delta_1 \cap \Delta_2$, each $\Delta_i$ is the cone defined by the center of projection ${\bf O}_i$ and the projected variety $Y_i$, provided the matrix of projection $\pi_i$ is known. 

\subsection{A Short Historical Perspective}
\label{sec::history}

In this section, we shortly review the history of the recovery modulo the projective group of an algebraic variety from unknown projections. This allows to present some context that enriches this introduction and provides a better understanding of the developments that will be done in the sequel. The question encompasses two sub-problems: 

\begin{enumerate}
	\item The embedding problem which consists in recovering the projections,
	\item The reconstruction problem which deals with the recovery of the variety itself, once the projections are known. 
\end{enumerate}

It should be noted that the two problems are of a different nature. For the
reconstruction problem the relevant tool is the normalization, and it is not merely
birational: since $X$ is smooth and a generic projection $\pi_k|_X : X \lra Y_k$ is
finite and birational, $X$ \emph{is} the normalization of $Y_k$, and not just a
variety birational to it. When $X$ is linearly normal, as is every smooth complete
intersection of positive dimension, its homogeneous coordinate ring is the integral
closure of that of $Y_k$, so that a single image already determines $X$ modulo the
projective group; when $X$ is not linearly normal, one image leaves a family of
candidates of dimension $h^0(X,\mathcal{O}_X(1))-m-1$, which a second image cuts
down.

The same circle of ideas settles the embedding problem as well, and this is worth
separating out, because the variety is not needed as an intermediate step. The
unknown of the system of proposition~\ref{prop::Vreal_system} is the matrix
$\Psi$ of the identification between the normalizations of the two images, and its
solution set is exactly the set $V^{\mathrm{real}}$ of configurations compatible
with the data; the pair of projections is read off from $\Psi$ by
lemma~\ref{lem::Vreal_embedding}, and the variety comes with it as a by-product.
The ingredients are computed from $I(Y_1)$ and $I(Y_2)$ by an integral closure
algorithm (remark~\ref{rem::Vreal_compute}), and the outcome is the unconditional
bound $|V^{\mathrm{real}}|\le|\operatorname{Lin}(X)|$ of
theorem~\ref{thm::Vreal_finite}. So section~\ref{sec::Vreal} already provides a
complete and effective route to the projections, independent of everything that
follows.

Two systems therefore compete for the same problem, and the comparison is not the
obvious one. When $X$ is linearly normal the matrix $\Psi$ has size $m+1$, so the
two systems have the same number of unknowns; the difference lies elsewhere. The
equations of proposition~\ref{prop::Vreal_system} have the degrees of the
generators of $I(Z_1)$, typically two or three, and its solution set is
$V^{\mathrm{real}}$ exactly; the generalized Kruppa equation has equations of
degree the class $c$ of the image, and its solution set is strictly larger than
$V^{\mathrm{real}}$, by lemma~\ref{lem::homographic}; a gap which
remark~\ref{rem::Vreal_bound} measures on an example, where a B\'ezout count of
some $9.4\times10^{7}$ confronts a bound of $16$. What keeps the Kruppa equation
at the centre of this paper is not that it is the better system to solve, but
that it is the classical one, in use since Kruppa, that it requires no integral
closure, but only the dual hypersurfaces of the two images, and that whether
it is well posed at the true configuration was an open question. That question is
the subject of sections~\ref{sec::dim(X)>=codim(X)-1}
and~\ref{sec::dim(X)<codim(X)-1}, and the two systems are joined in
theorem~\ref{thm::aug_system}, where the equations expressing realizability are
adjoined to the Kruppa equation.

Historically, the attention was first focused on the second part. Apparently, the first to have considered the question is Monge in~\cite{Monge-1799}, where he tackles the reconstruction problem from two known orthogonal projections on planes that are mutually perpendicular. 

Later the work of Poncelet~\cite{Poncelet-1822} on projective geometry showed what kind of properties are invariant through projections and can be recovered. Poncelet makes extensive use of the cross-ratio as a fundamental projective invariant. In addition, he introduced the duality principle, which is fundamental to the question. 

The works of Pl\"ucker~\cite{Plucker-1830} and Grassmann~\cite{Grassmann-1862} allowed to formulate in great generality the algebraic relations between linear subspaces, which stand at the core of this work. More precisely, we will build on~\cite{Barnabei-Brini-Rota-85} which fulfills the vision of Grassmann.   

Setting aside the extraordinary development of algebraic geometry in the nineteenth and twentieth centuries, let us focus on what is directly relevant to the present work. On the applied side of the subject, the development of photogrammetry and later of computer vision brought formidable results, especially about the embedding problem. In this context the original work by Kruppa~\cite{Kruppa-1913} and the more modern contributions, which are exposed in~\cite{Faugeras-Luong-01} and~\cite{MVG-03}, serve as basic conceptual elements for this present work. Among other things, we generalize the so-called Kruppa's equation and the concept of fundamental matrix, to achieve the results presented in the introduction. 

Given two projections $\pi_1$ and $\pi_2$ from $\PP^3$, defined outside the centers of projection ${\bf O}_1$ and ${\bf O}_2$, onto embedded planes $i_1(\PP^2), i_2(\PP^2)$, one defines the fundamental matrix ${\bf F}$ as the matrix that represents $\pi_2(\pi_1^{-1}({\bf p}))$ for ${\bf p}$ in the first plane $i_1(\PP^2)$. All lines $(\pi_2(\pi_1^{-1}({\bf p})))_{p \in i_1(\PP^2)}$ contain the projection ${\bf e}_2$ of ${\bf O}_1$ on $i_2(\PP^2)$. Therefore this matrix represents a mapping from the projective plane $i_1(\PP^2)$ onto the pencil of lines through ${\bf e}_2$, which can be viewed as a line in the dual plane of $i_2(\PP^2)$. The kernel of ${\bf F}$ defines the first epipole ${\bf e}_1$, which is $\pi_1({\bf O}_2)$. Given a conic in $\PP^3$, it is projected onto two conics ${\bf C}_1$ and ${\bf C}_2$ through respectively $\pi_1$ and $\pi_2$. Then according to~\cite{MVG-03}, the Kruppa's equation is given by:
$$
[{\bf e}_2]_x {\bf C}_2^* [{\bf e}_2]_x = {\bf F} {\bf C}_1^* {\bf F}^T, 
$$
where ${\bf C}_1^*$ and ${\bf C}_2^*$ are the matrices of the dual conics and $[{\bf a}]_x$ is the matrix of the vector product by ${\bf a}$ in the standard basis of $\R^3$. This equation expresses the fact that the two degenerated conics defined respectively as the union of the tangents to ${\bf C}_1$ (respectively ${\bf C}_2$) containing ${\bf e}_1$ (respectively ${\bf e}_2$) are projectively equivalent.

It should be noted however that the standard computer vision framework is limited to projections from $\PP^3$ to $\PP^2$ and the reconstruction of configurations of points, lines and sometimes of conics. 

The use of algebraic geometry appears to be natural when dealing with many projections and generalization to projections in higher dimensions. In~\cite{Kileel-Kohn-25}, the authors present a survey of these works. 

However to our best knowledge the works on the more theoretical question of recovering both the projections and the variety itself, for varieties of higher degree and dimension, are very sparse. In~\cite{Kaminski-04}, the authors show how these questions can be answered from unknown projections of a smooth algebraic curve in $\PP^3$ on planes. This work introduces a generalization of Kruppa's equation to algebraic curves of arbitrary degree, but still embedded in $\PP^3$. From this equation, sufficient conditions are derived for the recovery of the projections modulo the projective group of $\PP^3$; as we point out in remark~\ref{rem::K04_gap} below, the argument given there establishes the rigidity of the true solution rather than the asserted global finiteness. Also the reconstruction question is tackled. It is then shown that under some genericity assumption, the original curve can be recovered, provided its degree is greater than $2$.

In~\cite{Kaminski-11}, for surfaces in $\PP^4$ projected onto hyperplanes, it is shown how discriminant curves can be used to constrain and determine the missing projection geometry and recovering the surface itself. 

The present work is a generalization of these works, but follows the flow of~\cite{Kaminski-04}. However following the structure of~\cite{Kaminski-04} is possible up to some extent only, since tackling these questions for varieties of higher dimension embedded in projective spaces of higher dimension makes things significantly more complicated. 

The concept of fundamental matrix is now replaced by a sequence of fundamental matrices that are related each others. This is done in section~\ref{sec::fund_mat}. The fundamental matrix of order $k$ maps linear subspaces of $i_1(\PP^{m-1})$ of codimension $k$, not containing ${\bf e}_1$ onto linear subspaces of dimension $m-k$ in $i_2(\PP^{m-1})$ and containing ${\bf e}_2$. The knowledge of any of these fundamental matrices is equivalent to the knowledge of all. From the fundamental matrix of order $m-1$, the projections $\pi_1$ and $\pi_2$ can be recovered modulo the projective group $\Pr_m$.  

Then the Kruppa's equation can be generalized to higher dimensional varieties in a natural way, as done in section~\ref{sec::variety_of_second_order_F}. The dimensional analysis, which consists in finding a necessary condition for the solution variety of the generalized Kruppa equation to be finite, together with sufficient conditions for the rigidity of the true solution, is carried out in section~\ref{sec::dimensionAnalysis}. We do not exhibit a single condition which is both necessary and sufficient, and the upgrade from rigidity to global finiteness remains open in higher dimension, which leaves room for further research in this context. The necessary condition, while technically more difficult than in~\cite{Kaminski-04}, unfolds relatively naturally. However for the sufficient condition, the proof has to be split into two cases: (i) $\dim(X) \geq \codim(X) - 1$ and (ii) $\dim(X) < \codim(X) - 1$, where $X$ is the unknown variety in $\PP^m$. The first case has analogies with the case of curves, presented in~\cite{Kaminski-04}. Here an explicit condition is given, while in~\cite{Kaminski-04} the statement remains vague. The second case is completely new, has no parallel statement in~\cite{Kaminski-04}, and is conditional on an explicitly stated descent hypothesis.

The question of recovering of the variety itself, once the projections are computed,
leads to a similar theorem to that in~\cite{Kaminski-04}. However in the course of
the proof, there is no natural branched coverings of $\PP^1$ available, so one has
to build branch coverings of $\PP^n$ (with $n = \dim(X)$) and prove that the branch
locus is actually an hypersurface of $\PP^n$. This monodromy argument is needed only
in codimension two, which is the situation of~\cite{Kaminski-04}, while in every higher
codimension the residual locus is cut out by a dimension count on the variety of
secants meeting the baseline (that is the line joining the two centers of projections), and the degree of $X$ plays no role. The degree-two
obstruction of~\cite{Kaminski-04} is thus a phenomenon of codimension two alone (and even there it is vacuous under our genericity assumptions, since a
non-degenerate variety of codimension two has degree at least three).

Altogether, the present work presents much more general theorems than~\cite{Kaminski-04} and~\cite{Kaminski-11}.   

It is also worth situating this work against two other lines of research which
generalize classical two-view geometry in different directions, and with which it
shares more than first appears.

\paragraph{Multiview tensors \\}
The first line generalizes the fundamental matrix by increasing the number of
views and allowing arbitrary dimensions. Hartley and
Schaffalitzky~\cite{Hartley-Schaffalitzky-04} introduce the \emph{Grassmann
tensor} of a family of projections $\PP^n\dashrightarrow\PP^{m_i}$, which
subsumes the bifocal, trifocal and quadrifocal tensors, and show that for
sufficiently many generic projections the scene is recovered up to a projective
transformation uniquely, except when the targets are lines, where two
solutions occur. Projections $\PP^n\dashrightarrow\PP^2$ had been treated
earlier by Wolf and Shashua in the analysis of dynamic scenes, by tensorial
methods whose general form that paper settles. The family
$\{{\bf F}_k\}_{k=2}^{m-1}$ of section~\ref{sec::fund_mat} generalizes the
fundamental matrix in a different direction: the number of views stays two and
the target is always a hyperplane, but the scene is a variety rather than a
finite set of points, and, this is the essential difference, \emph{no
correspondences are given}. The entire input is the pair of homogeneous ideals
$I(Y_1)$, $I(Y_2)$. The tensorial methods above take corresponding points across
views as their data, and are not available here. This is precisely why the
constraint we use is of Kruppa type and is built from the dual varieties: the
statement that a hyperplane tangent to $Y_1$ through ${\bf e}_1$ corresponds to a
hyperplane tangent to $Y_2$ through ${\bf e}_2$ is correspondence-free, and the
generalized Kruppa equation~\eqref{ExtKrupEq} is its algebraic form.

\paragraph{Critical configurations and critical loci \\}
The second line asks when reconstruction fails to be unique, and it is the
closer of the two. For two views of a point set the answer has been known since
Krames, and was put in its modern form by Maybank: a configuration is critical
only if all the points and both centers lie on a ruled quadric; see the accounts
in~\cite{Hartley-Kahl-07} and~\cite{Bratelund-22}. Hartley and
Kahl~\cite{Hartley-Kahl-07} classify the critical configurations for three views
and show that points on an elliptic quartic, or on a rational quartic, remain
critical for any number of cameras. Br\aa telund~\cite{Bratelund-22} returns to
two views with algebraic methods, shows that all critical configurations lie on
quadric surfaces and determines exactly which quadrics occur, and also describes
how the distinct reconstructions of a critical configuration are related to one
another. In a higher-dimensional setting, Bertolini, Besana, Notari and
Turrini~\cite{Bertolini-Besana-Notari-Turrini-19} and Bertolini, Notari and
Turrini~\cite{Bertolini-Notari-Turrini-21} study the \emph{critical loci} of
reconstruction from several projections $\PP^k\dashrightarrow\PP^h$, showing them
to be determinantal varieties, which are hypersurfaces when the ambient space is
odd-dimensional and determinantal varieties of codimension two when it is even,
Bordiga surfaces for three views from $\PP^4$, with a classification of the
smooth ones and an analysis of the degenerate cases in which the defining matrix
drops rank in codimension one.

Three points of contact deserve to be made explicit, since they show that the
conditions appearing below are not of our own invention.

First, \emph{quadrics}. The sufficient condition $(\star)$ of
section~\ref{sec::dim(X)>=codim(X)-1}, that no quadric of $\PP^m$ contain
$L\cup X_L$, and the hypotheses bounding $\dim\mathcal{J}_X(2)$ are the
descendants of the classical quadric criterion. There the ambiguity of a
two-view configuration is carried by a quadric through the points and the two
centers; here it is excluded by forbidding a quadric through the baseline and the
tangency locus along it.

Second, the \emph{relation between distinct reconstructions}, which
\cite{Bratelund-22} describes for critical configurations of points, is in our
setting the content of theorem~\ref{thm::Vreal_finite}: two reconstructions of
the same pair of images differ by a polarized automorphism of $(Z,L)$, so that
the ambiguity is bounded by $|\operatorname{Lin}(X)|$. The elliptic quartic is a
meeting point of the two frameworks: critical for any number of cameras in the
point-set sense of~\cite{Hartley-Kahl-07}, it appears here as
example~\ref{ex::Vreal_quartic}, where example~\ref{ex::Vreal_system_quartic}
computes $|\operatorname{Lin}(X)|=32$ and shows the bound to be attained.

Third, \emph{determinantal structure}. The loci occurring below are
determinantal in the sense of
\cite{Bertolini-Besana-Notari-Turrini-19,Bertolini-Notari-Turrini-21}: the rank
condition (b) of proposition~\ref{prop::Vreal_system}, the minors of
remark~\ref{rem::aug_cost}, and most plainly
proposition~\ref{prop::consistency_structure}, where the consistency locus is cut
out by the maximal minors of a $4\times(m+1)$ matrix bilinear in
$({\bf c}_1,{\bf c}_2)$. What differs is what the matrix is made of and where the
locus lives: there, linear forms in the coordinates of the scene, with the locus
inside the ambient space of the scene; here, the unknown epipolar data, with the
locus inside a product of image spaces.

\paragraph{The position of the present work \\}
On the two axes that organize the literature above, this paper occupies a corner
of its own. No correspondences are assumed, the data being the ideals of the two
images; the scene is a variety rather than a finite point set; the number of
views is exactly two; and the output is first the pair of projections modulo the
projective group, then the variety itself. The case $m=3$ with $X$ a conic, which
is Kruppa's own, is the common ancestor of both lines and of this one.

\section{Basic Concepts and Preliminary Results}
\label{sec::basics}

Many computations throughout the paper are performed in projective spaces. In order to simplify notations, we may write $a \sim b$ when there exists a non-zero complex number $\lambda$ such that $a = \lambda b$.

\subsection{Double Grassmann-Cayley Algebra}
\label{sec::GC-algebra}
In this section, let $E$ be a linear space of dimension $n$ over a field $K$, which is $\C$ throughout the paper, but which can be chosen more generally in this section. The dual space will be denoted $E^*$. The projectivization $\PP(E)$ is merely the space of lines through the origin in $E$. 

Throughout the paper, the exterior product will be called the \textit{join} operator and will be denoted classically $\join$. We shall also introduce another operator, called the \textit{meet} that will be denoted $\meet$. This double algebra structure $(\join E, \join, \meet)$ has been introduced in~\cite{Barnabei-Brini-Rota-85} as a fulfillment of Grassmann's vision.  

Within $\join^k E$, the decomposable elements $v_1 \join \cdots \join v_k$ are called \textit{extensors of step $k$}. Recall that $\dim(\join^k E) = \binom{n}{k}$. 

Since we shall use these operators in a projective setting, we recall the following two statements that clarify the geometric aspects of the two operators, extracted from~\cite{Faugeras-Luong-01}, which contains a shortened version of~\cite{Barnabei-Brini-Rota-85}. 

\begin{proposition}
Let $\mathbf{a}$ and $\mathbf{b}$ be two extensors, and let $A$ and $B$ be the corresponding projective subspaces in $\PP(E)$. Then
\begin{enumerate}
\item $\mathbf{a} \join \mathbf{b} = \mathbf{0}$ if and only if $A \cap B \neq \emptyset$.
\item If $A \cap B = \emptyset$ then the extensor $\mathbf{a} \join \mathbf{b}$ is associated with the projective subspace generated by $A \cup B$. 
\end{enumerate}

\end{proposition}

\begin{proposition}
Let $\mathbf{a}$ and $\mathbf{b}$ be two extensors of steps $k$ and $l$, respectively, and let $A$ and $B$ be the corresponding projective subspaces in $\PP(E)$. 
\begin{enumerate}
\item If the subspace generated by $A \cup B$ is a proper subspace of $\PP(E)$, then $\mathbf{a} \meet \mathbf{b} = \mathbf{0}$.
\item If the subspace generated by $A \cup B$ is $\PP(E)$, then:
\begin{enumerate}
\item if $A \cap B = \emptyset$, then $k+l = n$ and $\mathbf{a} \meet \mathbf{b} \in \join^0 E = K$,
\item if $A \cap B \neq \emptyset$, then $\mathbf{a} \meet \mathbf{b}$ is the extensor associated with the projective subspace $A \cap B$. 
\end{enumerate}
\end{enumerate}
\end{proposition}

The double $K-$algebra structure of the direct sum $\join E = \oplus_{k=0}^n \join^k E$ defined by the join (exterior product) and the meet operators, allows us to define the Hodge operator $*: \join^k \rightarrow \join^{n-k} E$, which yields an isomorphism from $\left ( \join E, \join \right )$ to $\left ( \join E, \meet \right )$, satisfying the following properties:
$$
\begin{array}{rcl}
*(x \join y) & = & (*x) \meet (*y), \\
*(x \meet y) & = & (*x) \join (*y), \\
*(*x) & = & (-1)^{k(n-k)} x, 
\end{array}
$$
for all $x \in \join^k E$ and all $y \in \join^l E$, where $0 \leq k,l \leq n$. 

The Grassmannian $G(k,n)$ which is the set of $k-$planes in $E = K^n$, can be embedded into $\PP(\join^k E)$, since to any $k-$plane corresponds univocally a class of extensors in $\join^k E$, which are mutually equal modulo multiplication by a non-zero scalar. Therefore $G(k,n)$ is identified to the class of decomposable elements or extensors in $\join^k E$. With this embedding, $G(k,n)$ becomes a projective variety defined by quadratic equations, called the \textit{Pl\"ucker relations}, which in turn define the homogeneous ideal of $G(k,n)$ (see~\cite{Harris-92}). We shall always consider $G(k,n)$ with this Pl\"ucker embedding.  

In the sequel, the projective version of $G(k,n)$ is considered. It is defined as $\G(k,n) = G(k+1,n+1)$, which is the set of $k-$dimensional projective subspaces of $\PP^n$. The dual space $\PP^{n*}$ is canonically identified with $\G(n-1,n)$. Recall that $\dim(\G(k,n)) = (k+1)(n-k)$.

\subsection{Generalization}
\label{sec::mourrain}

Sometimes it is needed to have a way to compute algebraically the intersection or the sum of subspaces even when the conditions for the join and the meet operators to give a valid result are not satisfied. 

With this purpose in mind, we shall make use of two operators that were introduced in~\cite{Mourrain-91} and that generalize those defined in~\cite{Barnabei-Brini-Rota-85} : given $\mathbf{a}$ and $\mathbf{b}$ two extensors that represents two linear spaces of $E$, that is $A$ and $B$, then (i) $A \overline{\meet} B$ is an extensor that represents the intersection $A \cap B$, (ii) $A \overline{\join} B$ is an extensor that represents the sum $A + B$.

\subsection{Projection operators}
\label{sec::projMaps}

Throughout the paper, boldface letters denote representative vectors or
matrices, in the relevant vector spaces, of the corresponding projective
objects; they are defined up to a non-zero scalar factor, and every statement
involving them is invariant under this rescaling. Moreover, for a linear
subspace $W\subseteq\PP^N$ we denote by $\widehat W \subseteq\C^{N+1}$ its affine
cone; this does not conflict with the lifted matrices $\what{\bf M}_k$
introduced below, where the hat bears on a boldface matrix. 

Let $\pi$ be a projection operator from $\PP^n$ to an embedded
projective hyperplane $i(\PP^{n-1})$ through a point ${\bf O}$. Given two bases $({\bf a}_1, \cdots, {\bf a}_{n+1})$ and $({\bf b}_1, \cdots, {\bf b}_{n})$ of respectively $\C^{n+1}$ and $\C^{n}$, this projection can be presented by a full rank $n \times (n+1)$ matrix ${\bf M}$, defined modulo multiplication by a non-zero scalar. The matrix ${\bf M}$ can be decomposed as follows:
$$
{\bf M}= \left[\begin{array}{c}
	\Ga_1^T \\
	\Ga_2^T \\
	\cdots \\
	\Ga_{n}^T
\end{array}\right].
$$

Notice that for each row, $\Ga_i^T$ represents an hyperplane of $\PP^n$, that is a point of the dual space $\PP^{n*}$, that contains the center of projection.

There exists a set of simple, but very useful, properties:

\begin{enumerate}
\item The kernel of ${\bf M}$ gives exactly the center of projection.

\item The hyperplane $H = i(\PP^{n-1})$ can be recovered up to  a projective transformation of $\PP^n$ for which the center of projection is a fixed point. Let ${\bf X}$ be a pseudo-inverse of ${\bf M}$. Then hyperplane is defined in the dual space by the kernel of ${\bf X}^T$. Indeed let $V_H$ be the subspace of $\C^{n+1}$ that represents the hyperplane we are looking for. Let ${\bf B}$ be the $(n+1) \times n$ matrix, the columns of which are the basis of $V_H$ used to define ${\bf M}$. Then ${\bf BM}$ acts as identity on $V_H$ and ${\bf B}$ is a pseudo-inverse of ${\bf M}$. Since $V_H = \im(B)$, the coefficients of $H$ as point ${\bf h}$ in the dual space satisfy $\C {\bf h} = V_H^\perp = \ker({\bf B}^T)$. Since given ${\bf M}$ only, the matrix ${\bf B}$ is unknown, choosing another pseudo-inverse ${\bf X}$ leads to an hyperplane which is obtained from $H$ by the linear transformation defined by ${\bf I}_n - {\bf XM}$ that leaves the center of projection invariant. A canonical choice for the hyperplane can be defined by the kernel of the Moore-Penrose pseudo-inverse, in which case, it is the point of the dual space defined by the same coordinates than the center of projection itself.     

\item The transpose of ${\bf M}$ maps a $(n-2)-$plane in $i(\PP^{n-1})$ to the hyperplane it defines with the center of
projection, given as point of the dual space $\PP^{n*}$. This can
easily be deduced by a duality argument or a simple computation. 

\item For $k$, such that $1 \leq k \leq n-1$, there exists a matrix $\what{\bf M}_k$, the entries of which are polynomial functions of the entries of ${\bf M}$, and which maps an extensor in $\G(n-1-k,n-1)$ to the $(n-k)-$plane it generates with the center of projection. This generalizes a similar property for $n=3$ that appears in~\cite{Faugeras-Luong-01}. The existence of $\what{\bf M}_k$ is immediate since given a $(n-1-k)-$plane $W$ in $i(\PP^{n-1})$, the inverse image $\pi^{-1}(W)$ is indeed a $(n-k)-$ plane in $\PP^n$ that contains the center of projection. The matrix $\what{\bf M}_k$ has $\binom{n+1}{n-k+1}$ rows and $\binom{n}{n-k}$ columns and has full rank, that is $\rank(\what{\bf M}_k) = \binom{n}{n-k}$, since it defines obviously an injective operator.

For deriving the explicit expression of $\what{\bf M}_k$, observe that a
$(n-1-k)-$plane of $i(\PP^{n-1})$ is represented by an extensor of step $n-k$, and
that
$$
\what{\bf M}_k ({\bf w}_1 \join \cdots \join {\bf w}_{n-k})
= \what{\bf M}_{n-1}{\bf w}_1 \,\overline{\join}\, \cdots \,\overline{\join}\,
  \what{\bf M}_{n-1}{\bf w}_{n-k},
$$
where $\overline{\join}$ is the generalized join of section~\ref{sec::mourrain}, since the
lines $\what{\bf M}_{n-1}{\bf w}_i$ all pass through the center of projection, so
their ordinary exterior product vanishes.

Therefore, it is only necessary to derive an explicit expression of the matrix $\what{\bf M}_{n-1}$ that maps a point in $i(\PP^{n-1})$ to the line in $\G(1,n)$ it generates with the center of projection.

\begin{lemma}
The columns of $\what{\bf M}_{n-1}$ are $(-1)^{i+1} \Gamma_1 \meet \cdots \meet \hat{\Gamma}_i \meet \cdots \meet \Gamma_n$, for $1 \leq i \leq n$, where $\hat{\Gamma}_i$ means that $\Gamma_i$ is skipped.
\end{lemma}
\begin{proof}
Define $L_i^k = (-1)^{i+1} \Gamma_1 \meet \cdots \meet \hat{\Gamma}_i \meet \cdots \meet \Gamma_k$ for $2 \leq k \leq n$ and $1 \leq i \leq k$. 

Let $p = (p_1, \cdots, p_n) \in i(\PP^{n-1})$ and $Q \in \PP^n$ such that $p \sim {\bf M}Q$. Observe that for any two distinct indices $i,j$, $Q \in H_{ij}$, where $H_{ij} = p_i \Gamma_j - p_j \Gamma_i$. 

Therefore we have: $Q \in H_{12} \meet H_{23} \meet \cdots \meet H_{k-1,k}$ for any $1 \leq k \leq n$. By induction this yields that $Q$ lies in the space represented by $\prod_{j=2}^k p_j \sum_{i=1}^k p_i L_i^k$, which finally leads to $Q \join (\prod_{j=2}^n p_j \sum_{i=1}^n p_i L_i^n) = 0$. On the open set defined by $p_j \neq 0$ for $j \geq 2$, this yields $Q \join (\sum_{i=1}^n p_i L_i^n) = 0$, which shows that the columns of $\what{\bf M}_{n-1}$ are indeed given by $L_1^n, \cdots, L_n^n$. By density this is true on the whole space $i(\PP^{n-1})$. 

\end{proof}

\item For $2 \leq k \leq n$, there exists also a matrix $\wtilde{\bf M}_k$, the entries of which are here again polynomial functions of the entries of $\bf{M}$, and that maps $(n-k)-$planes of $\PP^n$ that do not contain the center of projection to $(n-k)-$planes in $i(\PP^{n-1})$. The existence of $\wtilde{\bf M}_k$ is also immediate since the image by $\pi$ of such a space, that does not contain the center of projection, is indeed a subspace of $i(\PP^{n-1})$ of dimension $n-k$. The matrix $\wtilde{\bf M}_k$ has $\binom{n+1}{n-k+1}$ columns and $\binom{n}{n-k+1}$ rows. 

Since any $(n-k)-$plane $H$ in $i(\PP^{n-1})$ is the projection of an $(n-k)-$plane that lies in the cone defined by $H$ and the center of projection, the matrix $\wtilde{\bf M}_k$ defines a surjective operator. Therefore its rank is $\rank(\wtilde{\bf M}_k) = \binom{n}{n-k+1}$. 

An explicit construction of $\wtilde{{\bf M}}_k$ is easily obtained since:
$$
\wtilde{\bf M}_k ({\bf w}_1 \join \cdots \join {\bf w}_{n-k+1}) = \join_{i=1}^{n-k+1} {\bf M} {\bf w}_i.
$$
\end{enumerate}

\subsection{Fundamental Matrices}
\label{sec::fund_mat}

\subsubsection{Construction}

Consider now two distinct projection operators $\pi_1$ and $\pi_2$. Let ${\bf
O}_1$ and ${\bf O}_2$ be the centers of projection, assumed to be distinct, and $i_1(\PP^{n-1})$
and $i_2(\PP^{n-1})$ the hyperplanes of projection. Let ${\bf e}_j$ be the
point of intersection of $i_j(\PP^{n-1})$ with the line $\overline{{\bf
O}_1 {\bf O}_2}$. 

Let $\si_k({\bf e}_j)$ be the pencil of $(n-1-k)-$planes in $i_j(\PP^{n-1})$ through ${\bf e}_j$, for $2 \leq k \leq n-1$. 

One can define a map from $\G(n-1-k,i_1(\PP^{n-1})) \setminus \sigma_{k}({\bf e}_1)$ to $\si_{k-1}({\bf e}_2)$. Each linear subspace $W$ of codimension $k$ in $i_1(\PP^{n-1})$ and that does not contain ${\bf e}_1$, that is $W \in \G(n-1-k,i_1(\PP^{n-1})) \setminus \sigma_{k}({\bf e}_1)$, is sent to the $(n-k)-$plane given by $\pi_2(\pi_1^{-1}(W))$. The subspace $\pi_2(\pi_1^{-1}(W))$ contains ${\bf e}_2$ since $\pi_1^{-1}(W)$ contains the first center of projection ${\bf O}_1$. 

This map is a projective transformation, defined by a linear function which matrix is ${\bf F}_k = \wtilde{\bf M}_{2,k} \what{\bf M}_{1,k}$. Therefore ${\bf F}_k$ is a matrix of size $\binom{n}{n-k+1} \times \binom{n}{n-k}$. The following lemma shows how the matrix ${\bf F}_k$ relates to projected entities. 

\begin{lemma}
\label{lemma:fund_relationship}
Consider a subspace $W \in \PP^n$ of dimension $n-1-k$ and its projection $W_1 = \pi_1(W)$ and $W_2 = \pi_2(W)$. Assuming that $W$ does not contain any of the centers of projection, the subspaces $W_1$ and $W_2$ have dimension $n-1-k$ as well and are respectively represented by the extensors ${\bf w}_1$ and ${\bf w}_2$. The following relation holds:
\begin{equation}
\label{eq::Fk}
{\bf F}_{k} {\bf w}_1 \sim {\bf e}_2 \join {\bf w}_2. 
\end{equation}
In particular for $k = n-1$, ${\bf w}_1$ and ${\bf w}_2$ are points and we have: 
\begin{equation}
\label{eq::Fn-1}
{\bf w}_2 \join {\bf F}_{n-1} {\bf w}_1 = 0.
\end{equation}
 \end{lemma}
\begin{proof}
The only point that requires justification is equation~\ref{eq::Fk}, which indeed holds true, since the space represented by ${\bf F}_{k} {\bf w}_1$ has dimension $n-k$ and contains both the space $W_2$ and the point ${\bf e}_2$. 
\end{proof}

Extending the standard terminology used in the case $n=3$, we will call the matrix ${\bf F}_k$ the {\it
fundamental matrix of order $k$} of the pair of projections $\pi_1$ and $\pi_2$ and the points ${\bf e}_1$ and ${\bf e}_2$ will be
respectively called the {\it first and the second epipole}. The $(n-1-k)-$planes
in the first (respectively second) projection hyperplane passing through the first (respectively second) epipole are called the {\it epipolar spaces of order $k$}.

Since every $(n-k)-$plane in $\sigma_{k-1}({\bf e}_2)$ is the projection into
$i_2(\PP^{n-1})$ of a $(n-k)-$plane in $\PP^n$ passing through ${\bf O}_1$, the
image of ${\bf F}_k$ is the linear span, inside $\join^{n-k+1}\C^n$, of the cone
over $\sigma_{k-1}({\bf e}_2)$, that is
$$
\im({\bf F}_k) \;=\; {\bf e}_2 \join \bigl(\join^{n-k}\C^n\bigr)
\;\cong\; \join^{n-k} V, \qquad V = \C^n/\C{\bf e}_2 .
$$
Hence $\rank({\bf F}_k)=\binom{n-1}{n-k}$. This computation identifies the image of ${\bf F}_k$ and the kernel admits a
dual description. Throughout we set $r_k = \binom{n-1}{n-k}$ and adopt the
conventions $\join^{0}\C^n = \C$ and $\join^{-1}\C^n = 0$.
 
\begin{proposition}
\label{prop::kernel_image_Fk}
Let ${\bf e}_1$ and ${\bf e}_2$ be the two epipoles. Identifying the linear
spaces underlying $i_1(\PP^{n-1})$ and $i_2(\PP^{n-1})$ with $\C^n$, one has
$$
\ker({\bf F}_k) \;=\; {\bf e}_1 \join \bigl(\join^{n-k-1}\C^n\bigr),
\qquad
\im({\bf F}_k) \;=\; {\bf e}_2 \join \bigl(\join^{n-k}\C^n\bigr).
$$
In particular $\rank({\bf F}_k) = r_k$, and the two epipoles are uniquely
determined by ${\bf F}_k$.
\end{proposition}
\begin{proof}
The statement about the image, hence the value of the rank, was proved
above. For the kernel, consider first a decomposable extensor ${\bf w}$
representing a space $W$ containing ${\bf e}_1$, that is $W \in
\sigma_{k}({\bf e}_1)$. Then $\pi_1^{-1}(W)$ contains $\pi_1^{-1}({\bf
e}_1)$, which is the baseline $\overline{{\bf O}_1 {\bf O}_2}$; in
particular the support of the extensor $\what{\bf M}_{1,k}{\bf w}$ contains
${\bf O}_2$, so that we may write $\what{\bf M}_{1,k}{\bf w} = {\bf O}_2
\join {\bf u}_2 \join \cdots \join {\bf u}_{n-k+1}$. The explicit formula
$\wtilde{\bf M}_{2,k}({\bf v}_1 \join \cdots \join {\bf v}_{n-k+1}) =
\join_{i} {\bf M}_2 {\bf v}_i$, together with ${\bf M}_2 {\bf O}_2 = {\bf
0}$, yields ${\bf F}_k {\bf w} = \wtilde{\bf M}_{2,k} \what{\bf M}_{1,k}
{\bf w} = 0$. Now the decomposable extensors representing spaces of codimension $k$ through
${\bf e}_1$ span the subspace ${\bf e}_1 \join (\join^{n-k-1}\C^n)$, so the
latter is contained in $\ker({\bf F}_k)$. Its dimension is
$\binom{n-1}{n-k-1} = \binom{n}{n-k} - \binom{n-1}{n-k}$, which is exactly
the nullity of ${\bf F}_k$; hence the inclusion is an equality.
 
For the uniqueness of the epipoles, suppose $\im({\bf F}_k) \subseteq {\bf
e}_2 \join (\join^{n-k}\C^n)$ and $\im({\bf F}_k) \subseteq {\bf e}'_2
\join (\join^{n-k}\C^n)$ for two distinct points. Then $\im({\bf F}_k)
\subseteq {\bf e}_2 \join {\bf e}'_2 \join (\join^{n-k-1}\C^n)$, whose
dimension $\binom{n-2}{n-k-1}$ is strictly smaller than $r_k$, contradicting
$\rank({\bf F}_k) = r_k$. Similarly, if $\ker({\bf F}_k)$ contained ${\bf
e}_1 \join (\join^{n-k-1}\C^n)$ and ${\bf e}'_1 \join (\join^{n-k-1}\C^n)$
for distinct points, it would contain their sum, of dimension
$2\binom{n-1}{n-k-1} - \binom{n-2}{n-k-2} = \binom{n-1}{n-k-1} +
\binom{n-2}{n-k-1} > \binom{n-1}{n-k-1}$, contradicting the value of the
nullity.
\end{proof}

Two cases are of a particular importance in the sequel. For $k=2$, ${\bf F}_2$ maps $(n-3)-$planes in $i_1(\PP^{n-1})$ to $(n-2)-$planes in $i_2(\PP^{n-1})$, that contain ${\bf e}_2$. The size of these matrices is given by $n \times \frac{n(n-1)}{2}$. The rank in this case is $\rank({\bf F}_2) = n-1$. The space of such fundamental matrices will be shown below (proposition~\ref{prop::space_fund_mat}) to be irreducible of dimension
$n^2-2$. 

For $k=n-1$, ${\bf F}_{n-1}$ acts on points in $i_1(\PP^{n-1})$ and outputs lines in $i_2(\PP^{n-1})$ through ${\bf e}_2$. It will be called \textit{the reduced fundamental matrix}. Notice that ${\bf F}_{n-1}$ has size $\frac{n(n-1)}{2} \times n$.

\begin{lemma}
The knowledge of ${\bf F}_k$ for some $k < n-1$ allows computing ${\bf F}_{k+1}$.  
\end{lemma}
\begin{proof}
Consider an $(n-k-2)-$plane $W$ in $i_1(\PP^{n-1})$ not in $\sigma_{k+1}({\bf e}_1)$. Consider the set $\Sigma_{k+1}$ of linear $(n-k-1)-$dimensional subspaces that contain $W$. Consider $L$ independent spaces in $\Sigma_{k+1}$, $Q_1,...,Q_L$, such that $W = \cap_{i=1}^L Q_i$. Let ${\bf w}, {\bf q}_1, \cdots ,{\bf q}_L$ be the extensors that represent these spaces. Then the space $W'$ represented by ${\bf F}_{k+1}{\bf w}$  satisfies: $W' = \cap_{i=1}^L Q_i'$, where $Q_i'$ is represented by ${\bf F}_{k}{\bf q}_i$. This allows computing as many correspondences $(W,W')$ as needed to allow the linear calculation of ${\bf F}_{k+1}$. 
\end{proof}

We can make a more precise statement. It turns out that ${\bf F}_k$ is a polynomial function of ${\bf F}_l$ for any pair $k,l$.

\begin{lemma}
\label{lemma::polynomial_relations_F}
For $k,l$ such that $2 \leq k,l \leq n-1$, the entries of ${\bf F}_l$ are given by polynomial functions of the entries of ${\bf F}_k$. We write this as a unique polynomial function with multiple components $\phi_{kl}$, such that  ${\bf F}_l = \phi_{kl}({\bf F}_k)$. Moreover we have: $\phi_{kk} = Id$ and $\phi_{kl} = \phi_{ul} \circ \phi_{ku}$ for any triple $2 \leq k,l,u \leq n-1$.  
\end{lemma}
\begin{proof}
In this proof, we shall use the generalized operators $\overline{\join}$ and $\overline{\meet}$ defined in section~\ref{sec::mourrain}.

Assume that $k < l$ and let $p=l-k$. Let ${\bf w} \in \G(n-1-l,i_1(\PP^{n-1}))$. Consider ${\bf w}_1, \cdots, {\bf w}_p \in \G(n-1-k,i_1(\PP^{n-1}))$, representing spaces $W_1, \cdots, W_p$, such that $W_1 \cap \cdots \cap W_p = W$, the space represented by ${\bf w}$. Then $\pi_2(\pi^{-1}(W)) = \bigcap_{i=1}^p \pi_2(\pi^{-1}(W_i))$, so that ${\bf F}_l {\bf w} = {\bf F}_k {\bf w}_1 \overline{\meet} \cdots \overline{\meet} {\bf F}_k {\bf w}_p$. This proves that the result does not depend on the specific expression of $W$ used for the computation and more important this shows that ${\bf F}_l = \phi_{kl}({\bf F}_k)$, $\phi_{kl}$ being a polynomial function.    

Now assume that $k > l$ and then let $p = k - l$. This time let ${\bf w}_1, \cdots, {\bf w}_p \in \G(n-1-k,i_1(\PP^{n-1}))$, representing spaces $W_1, \cdots, W_p$, such that $W_1 + \cdots + W_p = W$. Then we have $\pi_2(\pi^{-1}(W)) = \sum_{i=1}^p \pi_2(\pi^{-1}(W_i))$, so that ${\bf F}_l {\bf w} = {\bf F}_k {\bf w}_1 \overline{\join} \cdots \overline{\join} {\bf F}_k {\bf w}_p$.  This shows again the result does not depend on the specific decomposition of $W$ used for the computation and that ${\bf F}_l = \phi_{kl}({\bf F}_k)$, for some polynomial function $\phi_{kl}$.     

The relations between the different functions $\phi_{kl}$ are valid since the computation never depends on the sequence of subspaces used for it. 
\end{proof}

\subsubsection{From fundamental matrices to projections}

\begin{theorem}
\label{thm::FundEquiv}
For any $2 \leq k \leq n-1$, the fundamental matrix of order $k$, ${\bf F}_k$, is invariant on each orbit under the action $\varrho$ of the projective group, defined in~\ref{action::varrho}. Moreover each pair of projection matrices $({\bf M}_1, {\bf M}_2)$ is equivalent to a pair of the form $([{\bf I},{\bf 0}], [{\bf H},{\bf e}_2])$, where ${\bf H}$ depends on ${\bf F}_{n-1}$.  
\end{theorem}

We shall need the following lemma. 

\begin{lemma}
\label{lemma::Fn-1}
Assume ${\bf M}_i = [\overline{\bf M}_i, {\bf m}_i]$ (i=1,2) are projection matrices, such that $\overline{\bf M}_1,\overline{\bf M}_2$ are invertible. Then the second epipole is given by ${\bf e}_2 \sim -\overline{\bf M}_2 \overline{\bf M}_1^{-1} {\bf m}_1 + {\bf m}_2$. The reduced fundamental matrix is given by ${\bf F}_{n-1} \sim
[{\bf e}_2]_\join \overline{\bf M}_2 \overline{\bf M}_1^{-1}$, where $[{\bf e}_2]_\join$ represents the join map ${\bf w} \mapsto {\bf e}_2 \join {\bf w}$. 
\end{lemma}
\begin{proof}
The first center of projection is given by ${\bf O}_1 \sim [-\overline{\bf M}_1^{-1} {\bf m}_1,1]$ in homogeneous coordinates. This implies that the second epipole ${\bf e}_2$ has the following expression ${\bf e}_2 \sim -\overline{\bf M}_2 \overline{\bf M}_1^{-1} {\bf m}_1 + {\bf m}_2$. 

Consider a point ${\bf p} \in i_1(\PP^{n-1})$. Then the point ${\bf Q} \sim [\overline{\bf M}_1^{-1}{\bf p},0]$ is projected onto ${\bf p}$. Therefore the line ${\bf F}_{n-1} {\bf p}$ is generated by ${\bf e}_2$ and ${\bf M}_2 {\bf Q} \sim \overline{\bf M}_2 \overline{\bf M}_1^{-1}{\bf p}$. This yields the result.  
\end{proof}

Notice that the lemma makes it explicit that ${\bf F}_{n-1}^T {\bf e}_2 = {\bf 0}$. We are now in a position to prove the theorem~\ref{thm::FundEquiv}. 

\begin{proof}
Let us prove the theorem for $k = n-1$. According to lemma~\ref{lemma::polynomial_relations_F}, the theorem is then true for all $k$. 

Two pairs $({\bf M}_1, {\bf M}_2)$ and $({\bf N}_1, {\bf N}_2)$ are equivalent if there exists $A \in \Pr_n$, such that ${\bf N}_1 \sim {\bf M}_1 {\bf A}^{-1}$ and ${\bf N}_2 \sim {\bf M}_2 {\bf A}^{-1}$. Consider a point ${\bf P} \in \PP^n$ projected onto ${\bf p}_i \sim {\bf M}_i {\bf P}$ for $i=1,2$. Then we have ${\bf p}_i \sim {\bf M}_i {\bf A}^{-1} {\bf A} {\bf P}$, so the same points ${\bf p}_1$ and ${\bf p}_2$ are also the projections of ${\bf Q} = {\bf AP}$ by the operators ${\bf N}_1$ and ${\bf N}_2$. Therefore the correspondence equation~\eqref{eq::Fn-1} introduced in  lemma~\ref{lemma:fund_relationship} holds for the points ${\bf p}_i$ (i=1,2) and the reduced fundamental matrices of both pairs $({\bf M}_1, {\bf M}_2)$ and $({\bf N}_1, {\bf N}_2)$, which implies the two reduced fundamental matrices are equal. Then by lemma~\ref{lemma::polynomial_relations_F}, the fundamental matrices of the two ordered pairs $({\bf M}_1, {\bf M}_2)$ and $({\bf N}_1, {\bf N}_2)$ are identical at any order.   

Now let us prove the pair $({\bf M}_1, {\bf M}_2)$ is equivalent to a pair of the form $([{\bf I},{\bf 0}], [{\bf H},{\bf e}_2])$. Let ${\bf F}_{n-1}$ and ${\bf e}_2$ be the reduced fundamental matrix and the epipole relative to $({\bf M}_1, {\bf M}_2)$. Let us define ${\bf H} \sim [*{\bf e}_2]_\meet {\bf F}_{n-1}$, where $[*{\bf e}_2]_\meet {\bf L} = {* \bf e}_2 \meet {\bf L}$ represents the meet operator between the hyperplane defined by ${* \bf e}_2$ and the line ${\bf L}$. Since ${\bf e}_2$ does not lie on the hyperplane defined by $*{\bf e}_2$, the matrix ${\bf H}$ is not identically null and has rank $n-1$. 

We are looking for a matrix ${\bf A} \in \Pr_n$ such that: 
\begin{equation}
\label{eq::equiv_equs}
\begin{array}{ccc}
{\bf M}_1 & \sim & [{\bf I},{\bf 0}] {\bf A}^{-1} \\
{\bf M}_2 & \sim & [{\bf H},{\bf e}_2] {\bf A}^{-1} 
\end{array}
\end{equation}

Let us write ${\bf B} = {\bf A}^{-1}$ as follows:
$$
{\bf B} = \left[\begin{array}{ccc}
		 \overline{\bf B} & \vdots & {\bf u} \\
                 \hdots & \hdots & \hdots \\
		 {\bf v}^T & \vdots & \lambda
	\end{array}\right]
$$

One can even find a matrix $A$ such that equations~\eqref{eq::equiv_equs} are mere equality. With the notations of lemma~\ref{lemma::Fn-1}, it is immediate that $\overline{\bf B} =\overline{\bf M}_1$, ${\bf u} = {\bf m}_1$. In addition, we must have:
\begin{eqnarray}
\label{eq::eq1}
\overline{\bf M}_2 & = & {\bf H} \overline{\bf M}_1 + {\bf e}_2 {\bf v}^T \\
\label{eq::eq2}
{\bf m}_2 & = & {\bf H} {\bf m}_1 + \lambda {\bf e}_2
\end{eqnarray}

Since ${\bf H} = [*{\bf e}_2]_\meet [{\bf e}_2]_\join \overline{\bf M}_2 \overline{\bf M}_1^{-1}$ and ${\bf e}_2 \sim - \overline{\bf M}_2 \overline{\bf M}_1^{-1} {\bf m}_1 + {\bf m}_2$, we have: ${\bf H} \overline{\bf M}_1  = [*{\bf e}_2]_\meet [{\bf e}_2]_\join \overline{\bf M}_2$ and ${\bf H} {\bf m}_1 = [*{\bf e}_2]_\meet [{\bf e}_2]_\join ({\bf m}_2 - \mu {\bf e}_2)$ for some $\mu \in \C^*$, which yields ${\bf H} {\bf m}_1 = [*{\bf e}_2]_\meet [{\bf e}_2]_\join {\bf m}_2$. 

Therefore ${\bf Hm}_1$ is the point which lies at the intersection of the hyperplane defined by $*{\bf e}_2$ and the line $\overline{{\bf e}_2 {\bf m}_2}$. Since ${\bf m}_2 \neq {\bf e}_2$, there must exist $\lambda \in \C$, such that equation~\eqref{eq::eq2} holds. Fix once and for all a linear form $\varphi\in(\C^{n})^*$ with $\varphi({\bf e}_2)\neq 0$
(for instance a suitable coordinate form).The exact value of $\lambda$ is then obtained by applying $\varphi$ to equation~\eqref{eq::eq2}:
$$
\lambda=\frac{\varphi({\bf m}_2)-\varphi({\bf H}{\bf m}_1)}{\varphi({\bf e}_2)}.
$$ 

In addition equation~\eqref{eq::eq1}, $\overline{\bf M}_2 = {\bf H} \overline{\bf M}_1 + {\bf e}_2 {\bf v}^T$, is equivalent to $\overline{\bf M}_2 = [*{\bf e}_2]_\meet [{\bf e}_2]_\join \overline{\bf M}_2 + {\bf e}_2 {\bf v}^T$. Hence for all ${\bf x} \in \PP^{n-1}$, we must have: $\overline{\bf M}_2 {\bf x} = [*{\bf e}_2]_\meet [{\bf e}_2]_\join \overline{\bf M}_2 {\bf x}+ {\bf e}_2 {\bf v}^T {\bf x}$. By construction if we write ${\bf q}_2$ for $[*{\bf e}_2]_\meet [{\bf e}_2]_\join \overline{\bf M}_2 {\bf x}$, the point $\overline{\bf M}_2 {\bf x}$ lies on the line $\overline{{\bf e}_2 {\bf q}_2}$. Therefore there exists a function $\alpha({\bf x})$ such that $\overline{\bf M}_2 {\bf x} = {\bf q}_2 + \alpha({\bf x}) {\bf e}_2$. By construction, the function $\alpha$ is linear and therefore there exists a point ${\bf v} \in \PP^{n-1}$, such that $\alpha({\bf x}) = {\bf v}^T {\bf x}$. As for the exact value of ${\bf v}$, applying $\varphi$ to equation~\eqref{eq::eq1} gives
$$
{\bf v}^T=\frac{1}{\varphi({\bf e}_2)}\,\varphi\circ
\bigl(\overline{\bf M}_2-{\bf H}\overline{\bf M}_1\bigr).
$$     
\end{proof}

\begin{corollary}
\label{cor::fund_equiv_proj}
The knowledge of any fundamental matrix of order $k \leq n-1$ allows the computation of the projections modulo the action $\varrho$. 
\end{corollary}

In the sequel we will show how to recover ${\bf F}_2$ from two projections of an algebraic variety.

\subsubsection{The space of fundamental matrices}

Proposition~\ref{prop::kernel_image_Fk} suggests the following definition. Given points ${\bf e}_1 \in i_1(\PP^{n-1})$ and ${\bf e}_2 \in i_2(\PP^{n-1})$, let
$$
\mathcal{A}_k({\bf e}_1,{\bf e}_2) = \PP\Bigl\{ {\bf G} \in
\Hom\bigl(\join^{n-k}\C^n, \join^{n-k+1}\C^n\bigr) \;\Big|\;
{\bf e}_1 \join \bigl(\join^{n-k-1}\C^n\bigr) \subseteq \ker({\bf G}),\;
\im({\bf G}) \subseteq {\bf e}_2 \join \bigl(\join^{n-k}\C^n\bigr)
\Bigr\},
$$
and define the \textit{admissible variety of order $k$} as
$$
\mathcal{A}_k \;=\; \bigcup_{({\bf e}_1,{\bf e}_2)}
\mathcal{A}_k({\bf e}_1,{\bf e}_2)
\;\subset\; \PP^{\binom{n}{n-k+1}\binom{n}{n-k}-1}.
$$
 
\begin{proposition}
\label{prop::admissible_variety}
\begin{enumerate}
\item $\mathcal{A}_k$ is an irreducible projective variety of dimension
$2(n-1) + r_k^2 - 1$, contained in the determinantal variety of matrices of
rank at most $r_k$. On the open locus of $\mathcal{A}_k$ where the rank
equals $r_k = \binom{n-1}{n-k}$, the pair $({\bf e}_1,{\bf e}_2)$ is uniquely determined by
the matrix.
\item Every fundamental matrix of order $k$ belongs to this maximal-rank
locus of $\mathcal{A}_k$, the pair $({\bf e}_1,{\bf e}_2)$ being the pair
of epipoles.
\end{enumerate}
\end{proposition}
\begin{proof}
Consider the incidence variety
$$
\wtilde{\mathcal{A}}_k = \bigl\{({\bf e}_1,{\bf e}_2,[{\bf G}]) \in
i_1(\PP^{n-1}) \times i_2(\PP^{n-1}) \times
\PP^{\binom{n}{n-k+1}\binom{n}{n-k}-1} \;\big|\; [{\bf G}] \in
\mathcal{A}_k({\bf e}_1,{\bf e}_2)\bigr\}.
$$
The two linear conditions defining $\mathcal{A}_k({\bf e}_1,{\bf e}_2)$
mean that ${\bf G}$ factors through the quotient $\join^{n-k}\C^n /
\bigl({\bf e}_1 \join (\join^{n-k-1}\C^n)\bigr)$ and takes its values in
${\bf e}_2 \join (\join^{n-k}\C^n)$; both spaces have dimension $r_k$.
Since the subspaces ${\bf e}_1 \join (\join^{n-k-1}\C^n)$ and ${\bf e}_2
\join (\join^{n-k}\C^n)$ have constant dimension as $({\bf e}_1,{\bf e}_2)$
vary, they form vector bundles over $\PP^{n-1} \times \PP^{n-1}$, and
$\wtilde{\mathcal{A}}_k$ is the projectivization of the corresponding
bundle of homomorphisms, of rank $r_k^2$. Hence $\wtilde{\mathcal{A}}_k$
is an irreducible projective variety of dimension $2(n-1) + r_k^2 - 1$.
Its image $\mathcal{A}_k$ under the (proper) projection to the last factor
is therefore an irreducible projective variety. The containment in the
determinantal variety is clear, since $\rank({\bf G}) \le \dim\bigl({\bf
e}_2 \join (\join^{n-k}\C^n)\bigr) = r_k$ for every ${\bf G}$ as above.
The maximal-rank locus is non-empty (any isomorphism between the two
$r_k$-dimensional spaces provides a member), and over it the projection is
injective, the uniqueness of $({\bf e}_1,{\bf e}_2)$ being proved exactly
as in proposition~\ref{prop::kernel_image_Fk}. Hence $\dim \mathcal{A}_k =
\dim \wtilde{\mathcal{A}}_k$. Statement (2) is a rephrasing of
proposition~\ref{prop::kernel_image_Fk}.
\end{proof}
 
\begin{remark}
\label{rem::determinantal_vs_admissible}
The determinantal variety of matrices of rank at most $r_k$ in
$\PP^{\binom{n}{n-k+1}\binom{n}{n-k}-1}$ has dimension
$$
N_k^{\det} = \binom{n}{n-k+1}\binom{n}{n-k}
- \Bigl(\binom{n}{n-k+1} - r_k\Bigr)\Bigl(\binom{n}{n-k} - r_k\Bigr) - 1
$$
(see~\cite{Harris-92}). For $n=3$, where the only order is $k = 2 = n-1$,
the admissible variety coincides with the determinantal one: any $3 \times
3$ matrix of rank at most $2$ admits a kernel vector and a cokernel
vector, so that $\mathcal{A}_2 = \{\det {\bf F} = 0\} \subset \PP^8$ is
the classical cubic hypersurface, of dimension $7 = 2(n-1) + r_2^2 - 1$.
For $n \geq 4$, however, the inclusion $\mathcal{A}_k \subsetneq \{\rank
\leq r_k\}$ is strict and the gap is large: for $k=2$ one has $\dim
\mathcal{A}_2 = n^2 - 2$, while $N_2^{\det} = (n-1)(n^2-n+2)/2 - 1$ (for
instance $14$ against $20$ when $n = 4$). The rank condition alone is
therefore far from characterizing fundamental matrices in higher
dimension; the two epipolar constraints defining $\mathcal{A}_k$ are
essential.
\end{remark}

We can now describe the space of fundamental matrices itself. Denote by
$\mathcal{F}_k \subset \PP^{\binom{n}{n-k+1}\binom{n}{n-k}-1}$ the set of
fundamental matrices of order $k$, that is, the set of matrices ${\bf
F}_k(\pi_1,\pi_2)$ where $(\pi_1,\pi_2)$ ranges over all pairs of
projections with distinct centers.
 
\begin{proposition}
\label{prop::space_fund_mat}
\begin{enumerate}
\item For every $2 \leq k \leq n-1$, the set $\mathcal{F}_k$ is an
irreducible constructible subset of the maximal-rank locus of
$\mathcal{A}_k$, of dimension $n^2-2$.
\item For $k = n-1$ this containment is an equality:
$$
\mathcal{F}_{n-1} = \bigl\{[{\bf G}] \in \mathcal{A}_{n-1} \;\big|\;
\rank({\bf G}) = n-1\bigr\}.
$$
In particular $\mathcal{F}_{n-1}$ is a dense open subset of
$\mathcal{A}_{n-1}$, and $\dim \mathcal{A}_{n-1} = n^2-2$.
\item For $k = 2$, the closure of $\mathcal{F}_2$ is the whole admissible
variety $\mathcal{A}_2$; in particular $\mathcal{F}_2$ contains a dense
open subset of $\mathcal{A}_2$, and $\dim \mathcal{A}_2 = n^2-2$.
\end{enumerate}
\end{proposition}
\begin{proof}
(1) Note that the entries of $\wtilde{\bf M}_{i,k}$ and $\what{\bf M}_{i,k}$ are actually homogeneous polynomials of the entries of $({\bf M}_1,{\bf
M}_2)$. Hence the entries of ${\bf F}_k = \wtilde{\bf M}_{2,k} \what{\bf M}_{1,k}$
are bihomogeneous polynomial functions of the entries of $({\bf M}_1,{\bf
M}_2)$, and ${\bf F}_k \neq 0$ whenever the two centers are distinct,
since then $\rank({\bf F}_k) = r_k > 0$. The pairs of projection operators
with distinct centers form an irreducible quasi-projective variety $U$ of
dimension $2\bigl(n(n+1)-1\bigr) = 2n^2+2n-2$, and $\mathcal{F}_k$ is the
image of $U$ under the morphism $({\bf M}_1,{\bf M}_2) \mapsto [{\bf
F}_k]$. It is therefore irreducible and, by Chevalley's theorem,
constructible (see~\cite[p. 39]{Harris-92}); it lies in the maximal-rank locus of $\mathcal{A}_k$ by proposition~\ref{prop::admissible_variety}(2).
 
By theorem~\ref{thm::FundEquiv} and
corollary~\ref{cor::fund_equiv_proj}, the fibers of this morphism are
precisely the orbits of the action $\varrho$ of $\Pr_n$. The generic orbit
has dimension $\dim \Pr_n = n^2+2n$. Indeed, the stabilizer of a pair
$([{\bf I},{\bf 0}],[{\bf H},{\bf e}_2])$ with ${\bf H}$ invertible is
trivial: writing ${\bf A}^{-1} = \left[\begin{smallmatrix} {\bf B} & {\bf
b} \\ {\bf c}^T & d \end{smallmatrix}\right]$, the condition $[{\bf
I},{\bf 0}]{\bf A}^{-1} \sim [{\bf I},{\bf 0}]$ forces ${\bf B} = \lambda
{\bf I}$ and ${\bf b} = {\bf 0}$, and then $[{\bf H},{\bf e}_2]{\bf
A}^{-1} \sim [{\bf H},{\bf e}_2]$ gives $(\lambda - \mu){\bf H} = -{\bf
e}_2 {\bf c}^T$ and $d\,{\bf e}_2 = \mu\,{\bf e}_2$ for some $\mu \in
\C^*$; since ${\bf H}$ is invertible while ${\bf e}_2 {\bf c}^T$ has rank
at most one, we get $\lambda = \mu$, ${\bf c} = {\bf 0}$ and $d = \lambda$,
so that ${\bf A}$ is a scalar matrix. Consequently
$$
\dim \mathcal{F}_k = \dim U - \dim \Pr_n = (2n^2+2n-2) - (n^2+2n) = n^2-2.
$$
 
(2) The inclusion $\subseteq$ holds by (1). Conversely, let $[{\bf G}] \in
\mathcal{A}_{n-1}({\bf e}_1,{\bf e}_2)$ with $\rank({\bf G}) = n-1$. Since
$\im({\bf G}) \subseteq {\bf e}_2 \join \C^n = \im\bigl([{\bf
e}_2]_\join\bigr)$ and $\ker\bigl([{\bf e}_2]_\join\bigr) = \C {\bf e}_2$,
we may write ${\bf G} = [{\bf e}_2]_\join {\bf A}$ for some ${\bf A} \in
\Hom(\C^n,\C^n)$, determined up to the addition of a matrix of the form
${\bf e}_2 {\bf w}^T$. The rank condition forces $\im({\bf A}) + \C{\bf
e}_2 = \C^n$ and $\dim \ker({\bf A}) \leq 1$. If ${\bf A}$ is not
invertible, then $\ker({\bf A}) = \C {\bf v}$ for some ${\bf v} \neq {\bf
0}$ and $\im({\bf A}) \oplus \C {\bf e}_2 = \C^n$; choosing ${\bf w}$ with
${\bf w}^T {\bf v} \neq 0$, the representative ${\bf A} + {\bf e}_2 {\bf
w}^T$ is invertible. We may thus assume ${\bf A}$ invertible. Consider
then the pair of projections ${\bf M}_1 = [{\bf I},{\bf 0}]$, ${\bf M}_2 =
[{\bf A},{\bf e}_2]$. Its centers ${\bf O}_1 = [0,\dots,0,1]^T$ and ${\bf
O}_2 = [-{\bf A}^{-1}{\bf e}_2, 1]^T$ are distinct, and its second epipole
is ${\bf M}_2 {\bf O}_1 = {\bf e}_2$. For a point ${\bf p} \in
i_1(\PP^{n-1})$, the point ${\bf Q} = [{\bf p},0]^T$ satisfies ${\bf M}_1
{\bf Q} = {\bf p}$ and ${\bf M}_2 {\bf Q} = {\bf A}{\bf p}$, so that the
epipolar line of ${\bf p}$ is represented by ${\bf e}_2 \join {\bf A}{\bf
p}$. Hence the reduced fundamental matrix of this pair is $[{\bf
e}_2]_\join {\bf A} = {\bf G}$, and $[{\bf G}] \in \mathcal{F}_{n-1}$.
Finally, the maximal-rank locus is open in $\mathcal{A}_{n-1}$, and dense
since $\mathcal{A}_{n-1}$ is irreducible; combined with
proposition~\ref{prop::admissible_variety}(1) and $r_{n-1} = n-1$, this
gives $\dim \mathcal{A}_{n-1} = 2(n-1) + (n-1)^2 - 1 = n^2-2$.
 
(3) By proposition~\ref{prop::admissible_variety}(1) with $r_2 =
\binom{n-1}{n-2} = n-1$, the admissible variety $\mathcal{A}_2$ is
irreducible of dimension $2(n-1) + (n-1)^2 - 1 = n^2-2$. By (1),
$\mathcal{F}_2$ is an irreducible subset of $\mathcal{A}_2$ of the same
dimension; hence $\overline{\mathcal{F}_2} = \mathcal{A}_2$. A
constructible set contains a subset that is open and dense in its closure,
whence the last assertion.
\end{proof}
 
\begin{remark}
\label{rem::intermediate_orders}
(i) For $3 \leq k \leq n-2$ (which occurs only when $n \geq 5$), one has
$r_k \geq \binom{n-1}{2} > n-1$, so that $\dim \mathcal{A}_k = 2(n-1) +
r_k^2 - 1 > n^2-2$ and $\mathcal{F}_k$ is \textit{not} dense in
$\mathcal{A}_k$: for these intermediate orders, the two linear constraints
carried by the epipoles no longer cut a variety of the correct dimension.
The two orders used in the sequel, $k=2$ and $k=n-1$, are precisely those
for which density holds. (ii) For $k=2$ and $n \geq 4$ we claim density
only and we do not assert that every matrix of maximal rank in
$\mathcal{A}_2$ is a fundamental matrix. This weaker statement suffices
for the dimensional analysis of
section~\ref{sec::dimensionAnalysis}.
\end{remark}

\subsection{Generic projection of an algebraic variety}
\label{sec::generic_projection}

In this section, we shall recall basic facts concerning the generic projection of an algebraic variety in the context of this present work. 

We consider a $n-$dimensional smooth and
irreducible variety $X$ embedded in $\PP^m$ (with $n \leq m-2$). 

The variety of secants of $X$ will be denoted by $\mathfrak{S}(X)$. It is defined as the closure in $\G(1,m)$ of $\{\overline{pq} | p,q \in X, p \neq q\}$, where $\overline{pq}$ is the line joining $p$ and $q$. Then the chordal variety $S(X)$ is defined as the union of lines in $\mathfrak{S}(X)$: $S(X) = \cup_{l \in \mathfrak{S}(X)} l$.     

The dual variety of $X$ is defined as the set of hyperplanes tangent to $X$. An hyperplane is tangent to $X$ when it contains the tangent space to $X$ at some point. The dual variety is a projective variety in the dual space $\PP^{m*}$.

We shall assume some genericity assumptions: 
\begin{enumerate}
\item $X$ is not contained in any hyperplane, 
\item the chordal variety $S(X)$ is not deficient, which means that it has the expected dimension $\dim(S(X)) = \min(2n+1,m)$,
\item the dual variety $X^*$ is not deficient as well, which means that it is an hypersurface in $\PP^{m*}$.
\end{enumerate}

Let $Y$ be the projection of $X$ on a generic hyperplane
through a center of projection, ${\bf O}$. 
\begin{enumerate}
\item The variety $Y$ may contain singularities. More precisely, we can distinguish between the following three cases:
\begin{itemize}
\item When $m > 2n+1$, there is in general no singularities in $Y$.

\item When $m = 2n+1$, for
a generic position of the center of projection, the only singularities of $Y$ will be points with two tangent hyperplanes intersecting transversally, also called normal crossing singularities. In that case, the singular locus of $Y$ is generically a discrete variety. 

This claim follows from the fact that singularities in $Y$ come from multiple points in $X$ projected into the same point. This occurs when the center of projection lies on a secant of $X$. The chordal variety of $X$ has dimension $2 n + 1$  and covers the whole space. Therefore through any center of projection, there must be at least one secant. Consider the incidence correspondence:
$$
\Sigma = \{(l,p) \in \mathfrak{S}(X) \times \PP^{2n+1} \mid p \in l\} \subset \G(1,2n+1) \times \PP^{2n+1}.
$$
Let $\tau_1,\tau_2$ the projections from $\Sigma$ to respectively the first and the second factor. For $l \in \mathfrak{S}(X)$, the fiber $\tau_1^{-1}(l)$ is one-dimensional. Therefore $\dim(\Sigma) = \dim(\mathfrak{S}(X)) + 1 = 2n+1$. Since $\tau_2$ is surjective as already mentioned, its generic fiber is finite.  

\item When $m < 2n+1$, the singular locus of $Y$ presents a natural stratification discussed in~\cite{Gruson-Peskine-2010}.
\end{itemize}

\item If $d$ is the degree of $X$, then we also have $\deg(Y) = d$.

\item The {\it class} of a variety is defined to be the degree of its dual hypersurface. Let $c$ be the class of $Y$. Then $c$ is constant for a generic position of the center of projection. It is actually the class of $X$ too, since in the dual space $Y^*$ is the intersection of $X^*$ with the generic hyperplane defined by the center of projection. The class $c$ is related to the degree and the singularities of $Y$. A general formula can be found in~\cite{Tevelev-05}.
\end{enumerate}

We shall need the following consequence of the second assumption.

\begin{lemma}
\label{lemma::tangent_variety}
Let $\Tan(X)=\bigcup_{Q\in X} T_QX \subset \PP^m$ be the tangent variety of $X$. Under
the assumption that the chordal variety $S(X)$ is not deficient, one has
$$
\dim \Tan(X)=\min(2n,m).
$$
In particular:
\begin{enumerate}
\item if $2n+1 \ge m$, then $\dim \Tan(X) \ge m-1$, so \emph{every} line of $\PP^m$ meets
$\Tan(X)$;
\item if $2n+1 < m$, then $\dim \Tan(X)=2n<m$, so a generic point ${\bf O}\in\PP^m$ lies
neither on $\Tan(X)$ nor on $S(X)$; consequently the projection $\tau$ of center
${\bf O}$ embeds $X$, and $T_{\tau(Q)}\tau(X)=\tau(T_QX)$ for every $Q\in X$.
\end{enumerate}
\end{lemma}

\begin{proof}
$\Tan(X)$ is the image of a $\PP^n$-bundle over $X$, so $\dim\Tan(X)\le 2n$; and
$\Tan(X)\subseteq S(X)$. By theorem~\cite[Thm. 15.20, p. 195]{Harris-92}, for a smooth irreducible $X\subset\PP^m$ either $\dim S(X)=2n+1$ and $\dim\Tan(X)=2n$, or $S(X)=\Tan(X)$.

If $2n+1\le m$, the non-deficiency assumption gives $\dim S(X)=2n+1>2n\ge\dim\Tan(X)$, so
the second alternative is impossible and $\dim\Tan(X)=2n=\min(2n,m)$. If $2n+1>m$, the
assumption gives $\dim S(X)=m$, hence $S(X)=\PP^m$; since $\dim S(X)\neq 2n+1$ the first
alternative is impossible, so $\Tan(X)=S(X)=\PP^m$ and $\dim\Tan(X)=m=\min(2n,m)$.

For~(1), a closed subvariety of $\PP^m$ of dimension $\ge m-1$ meets every line, by the
projective dimension theorem. For~(2), a generic point avoids the two proper closed
subsets $\Tan(X)$ and $S(X)$; avoiding $S(X)$ makes $\tau|_X$ injective and avoiding
$\Tan(X)$ makes $d\tau$ injective on tangent spaces and gives
$T_{\tau(Q)}\tau(X)=\tau(T_QX)$.
\end{proof}

\subsection{Gauss Map and Dual Variety}

We will need in the sequel to consider the projected variety $Y$ that is irreducible, but that may be singular. Given the genericity properties of $X$, the variety $Y$ also satisfies the following assumptions:
\begin{enumerate}
\item generically a tangent space to $Y$ is tangent to it at a single point,
\item the dual variety $Y^*$ is also an hypersurface in $\PP^{(m-1)*}$. 
\end{enumerate}

Recall that the Gauss map is a rational map defined on the smooth locus of $Y$, which associates to each smooth point the tangent space: $\gamma: Y \dashrightarrow \G(n,m-1), y \mapsto T_yY$. 

We will also need the following notations: over the Grassmannian $G(k,M)$ for any $k,M$ such that $0 \le k \le M$, we have the tautological exact sequence of vector bundles:
$$
0 \to S \to \mathcal{O}_G^{\oplus M} \to Q \to 0,
$$
where $S$ is the universal subbundle of rank $k$ and $Q$ is the universal quotient bundle of rank $M - k$. Recall that $S=\{(\Lambda,{\bf z}) \in G(k,M) \times \C^M/ {\bf z} \in \Lambda\}$ and $Q = G(k,M) \times \C^M / S = \coprod_{\Lambda \in G(k,M)} \{\Lambda\} \times \left ( \C^M / \Lambda \right)$.

For the computations of section~\ref{sec::dimensionAnalysis} we shall need to identify
the pullback of the tautological subbundle under the Gauss map. Let $X\subset\PP^m$ be
smooth of dimension $n$, with Gauss map $\gamma: X\to\G(n,m)=G(n+1,m+1)$, and let
$P^1=P^1(\mathcal{O}_X(1))$ denote the bundle of first principal parts of
$\mathcal{O}_X(1)$, of rank $n+1$, defined by
$$
0\lra \Omega^1_X(1)\lra P^1(\mathcal{O}_X(1))\lra \mathcal{O}_X(1)\lra 0 .
$$

\begin{lemma}
\label{lemma::gauss_jets}
With these notations, $\gamma^* S^\vee \cong P^1(\mathcal{O}_X(1))$ where $(.)^\vee$ denotes the dual bundle whereas the notation $\gamma^* S$ is used for the pullback of the universal subbundle, that is
$\gamma^*S=(P^1)^*$; consequently the total Chern classes are related by $c(\gamma^*Q)=c\bigl((P^1)^*\bigr)^{-1}$, and if
$\beta_1,\dots,\beta_{n+1}$ denote the Chern roots of $P^1$, the degree-$d$ part of
$c(\gamma^*Q)$ is the complete homogeneous symmetric function
$h_d(\beta_1,\dots,\beta_{n+1})$.
\end{lemma}

\begin{proof}
The fiber of $\gamma^*S$ at $x\in X$ is the affine cone $\what{T_xX}\subset\C^{m+1}$. On
the other hand, the fiber of $P^1(\mathcal{O}_X(1))$ at $x$ is the space of $1$-jets at
$x$ of sections of $\mathcal{O}_X(1)$, so the evaluation of linear forms induces a
surjection $(\C^{m+1})^*\otimes\mathcal{O}_X\to P^1(\mathcal{O}_X(1))$, since $X$ is smooth, whose kernel at $x$
consists of the linear forms vanishing to order $2$ at $x$, that is, of those vanishing on
$\what{T_xX}$. Hence we obtain an exact sequence
$$
0\lra \gamma^*Q^\vee \lra (\C^{m+1})^*\otimes\mathcal{O}_X \lra P^1(\mathcal{O}_X(1))\lra 0,
$$
which is precisely the dual of the pullback by $\gamma$ of the tautological sequence
$0\to S\to\mathcal{O}^{m+1}\to Q\to0$. Therefore
$P^1(\mathcal{O}_X(1))\cong\gamma^*S^\vee$. The Chern class statement follows from
$c(\gamma^*S)c(\gamma^*Q)=1$ and
$c\bigl((P^1)^*\bigr)^{-1}=\prod_i(1-\beta_i)^{-1}=\sum_{d\ge0}h_d(\beta)$. See~\cite{Gathmann-2002} for a friendly discussion of Chern classes and Chern roots.
\end{proof}

\subsection{The Genericity Conditions}
\label{sec::genericity_conditions}

To avoid cumbersome repetitions, we gather here the generic conditions we assume. The three first one have been mentioned above in section~\ref{sec::generic_projection}. 

\begin{enumerate}
	\item The variety $X$ is non-degenerate, that is not contained in any hyperplane.
	\item The chordal variety $S(X)$ is not deficient, which means that it has the expected dimension $\dim(S(X)) = \min(2n+1,m)$.
	\item The dual variety $X^*$ is not deficient as well, which means that it is an hypersurface in $\PP^{m*}$.
	\item The two centers of projection ${\bf O}_1$ and ${\bf O}_2$ are distinct. 
	\item None of the center of projection belongs to $X$.
	\item The baseline $L = \overline{{\bf O}_1 {\bf O}_2}$ is not contained in any tangent space to X.
	\item No secant line of $X$ passes through both ${\bf O}_1$ and ${\bf O}_2$.
\end{enumerate}

Each of the last four conditions defines a Zariski open set of $\PP^m \times \PP^m$. Altogether these conditions hold for a dense open set of $\PP^m \times \PP^m$ in the classical topology, the complement of which has measure zero.

\section{Recovering the Variety From Two Generic Projections With Unknown Projection Operators}
\label{sec::twoProj}

In this section we finally tackle our problem. We consider a generic $n-$dimensional $X$ with the assumptions mentioned in section~\ref{sec::generic_projection}. Recall that this includes the two following points (i) X is smooth and irreducible, (ii) the dual variety is an hypersurface. Note that from here and up to the end of the paper, we will use $n \geq 1$ for the dimension of the variety $X$ and $m$ for the dimension of the ambient projective space.

This variety $X$ is projected onto two generic hyperplanes through two generic points. The projection operators $\pi_1$ and $\pi_2$ are unknown. First we want to recover the second order fundamental
matrix  ${\bf F}_2$ of the couple $(\pi_1,\pi_2)$ from the projected varieties $Y_1 = \pi_1(X)$ and $Y_2 = \pi_2(X)$.

\subsection{The Variety of Fundamental Matrices of Second Order}
\label{sec::variety_of_second_order_F}

Let ${\bf M}_i$, $i=1,2$, be the projection matrices. Let ${\bf F} = {\bf
F}_2$ be the fundamental matrix of order $2$. Let ${\bf e}_1$ and ${\bf
e}_2$ be the two epipoles. According to section~\ref{sec::fund_mat},
applied with the ambient space $\PP^m$ (that is, with $n$ replaced by
$m$), such a fundamental matrix has size $m \times \frac{m(m-1)}{2}$, that
is, $\frac{m^2(m-1)}{2}$ entries, and has rank $m-1$. We write $\PP^M$,
with $M = \frac{m^2(m-1)}{2} - 1$, for the projective space of all
matrices of this size. By proposition~\ref{prop::space_fund_mat}, the set
of such fundamental matrices is an irreducible constructible subset of
$\PP^M$ of dimension $N = m^2 - 2$, dense in the admissible variety
$\mathcal{A}_2$ cut out by the two epipolar constraints of
proposition~\ref{prop::kernel_image_Fk}. 

We will need to consider the two following mappings: $\G(m-3,i_1(\PP^{m-1})) \setminus \sigma_{2}({\bf e}_1) \rightarrow \G(m-2,i_1(\PP^{m-1})), {\bf w} \stackrel{\ga}{\mapsto} {\bf e}_1
\join {\bf w}$ and $\G(m-3,i_1(\PP^{m-1})) \rightarrow \G(m-2,i_2(\PP^{m-1})), {\bf w} \stackrel{\xi}{\mapsto} {\bf Fw}$. 

Let $Y_1^\star$ and $Y_2^\star$ denote the dual varieties of respectively $Y_1$ and $Y_2$, which are hypersurfaces (according to the non-deficiency assumption from section~\ref{sec::generic_projection}) and whose polynomials are respectively $\phi_1$ and $\phi_2$. 

\begin{theorem} \label{theoKrupeq}  
For a generic position of the centers of projections with respect to the variety $X$, there exists a non-zero scalar $\lambda$, such that for all
points ${\bf w}$ in $\G(m-3,i_1(\PP^{m-1}))$, the following equality holds:  
\begin{equation}
\phi_2(\xi({\bf w})) = \lambda \phi_1(\gamma({\bf w}))
\label{ExtKrupEq}
\end{equation}
\end{theorem} 

This equation generalizes the so-called Kruppa equation that has been initially introduced for $n=1$ and $m=3$ with $X$ being a conic, as mentioned in section~\ref{sec::history}, and has already been generalized to any smooth irreducible algebraic curve in $\PP^3$ in~\cite{Kaminski-04}. For these reasons, we shall call this equation, the {\it generalized Kruppa equation}.

\begin{proof} 
Let $\ep_i$ be the set of hyperplanes containing ${\bf e}_i$ and tangent to $Y_i$ (at regular points). 

We start with the following lemma: 
\begin{lemma}
\label{lemma::epipolar_equivalence}
The two sets $\ep_1$ and $\ep_2$ are projectively equivalent. Moreover
for each corresponding pair of epipolar hyperplanes $({\bf H},{\bf H}') \in
\ep_1 \times \ep_2$, the multiplicities of ${\bf H}$ and ${\bf H}'$ as
points of $Y_1^*$ and $Y_2^*$ are the same.  
\end{lemma}
\begin{proof}
Consider the following three pencils:
\begin{itemize}
\item $\si({\bf L}) \approx \PP^{m-2}$, the pencil of hyperplanes, containing
the baseline joining the two centers of projection, 
\item $\si({\bf e}_1) \approx \PP^{m-2}$, the pencil of epipolar hyperplanes in
the first projection space $i_1(\PP^{m-1})$, 
\item $\si({\bf e}_2) \approx \PP^{m-2}$, the pencil of epipolar hyperplanes in
the second projection space $i_2(\PP^{m-1})$. 
\end{itemize}
Thus we have $\ep_i \subset \si({\bf e}_i)$. Moreover if $E$ is the
set of hyperplanes in $\si({\bf L})$ tangent to the variety $X$ in $\PP^{m}$, then
there exists a one-to-one mapping between $E$ and each $\ep_i$ which
leaves the multiplicities unchanged. Indeed the multiplicity of tangency to $X$ of a hyperplane in $E$ is the same as the multiplicity of tangency to $Y_i$ of the corresponding hyperplane in $\ep_i$. This completes the proof. 
\end{proof}

This lemma implies that both sides of equation~\eqref{ExtKrupEq} define
the same algebraic set, that is the union of epipolar hyperplanes
tangent to $Y_1$. The previous lemma also implies that each such hyperplane appears with the same multiplicity in $\phi_1$ and $\phi_2$. 
\end{proof}

By eliminating the scalar $\lambda$ from the generalized Kruppa
equation (\ref{ExtKrupEq}) we obtain a set of 
bi-homogeneous equations in ${\bf e}_1$ and ${\bf F}$. Hence they define a variety in $\PP^{m-1} \times \PP^{M}$. We turn our attention to the dimensional analysis of this variety. Our concern is to exhibit the conditions for which this variety is discrete.

\subsection{Dimensional Analysis}
\label{sec::dimensionAnalysis}

Let $\{E_i({\bf F}, {\bf e}_1)\}_i$ be the set of bi-homogeneous
equations on ${\bf F}$ and ${\bf e}_1$, extracted from the generalized
Kruppa equation (\ref{ExtKrupEq}). 

The purpose of this section is to determine the conditions under which this system of equations defines a discrete variety. Indeed a finite number of solutions is equivalent to the computation of the pair of projections $(\pi_1,\pi_2)$ modulo the projective group of $\PP^m$ up to a finite-fold ambiguity.

Our first
concern is to determine whether all solutions of equation
(\ref{ExtKrupEq}) are admissible, that is whether they satisfy the constraint ${\bf F}{\bf w}=0$ if and only the extensor ${\bf w}$ represents a $(m-3)-$plane in $i_1(\PP^{m-1})$ containing the first epipole ${\bf e}_1$. Indeed we prove the following statement: 

\medskip

\begin{proposition}
\label{prop::valid_solutions}
Let $({\bf e}_1, {\bf F})$ satisfy the generalized Kruppa equation
(\ref{ExtKrupEq}) for some non-zero scalar $\lambda$. Denote by
$\sigma({\bf e}_1) \iso \PP^{m-2}$ the pencil of hyperplanes of
$i_1(\PP^{m-1})$ through ${\bf e}_1$, and by
$$
C_{{\bf e}_1} \;=\; \bigl\{\, H \in \sigma({\bf e}_1) \;\big|\; H
\text{ is tangent to } Y_1 \,\bigr\} \;=\; Y_1^* \cap \sigma({\bf e}_1)
$$
the hypersurface of degree $c$ in $\sigma({\bf e}_1)$ cut out by the
restriction of $\phi_1$. Assume that $C_{{\bf e}_1}$ is \emph{not a
cone}. Then
$$
\rank({\bf F}) \;=\; m-1
\qquad\text{and}\qquad
\ker({\bf F}) \;=\; {\bf e}_1 \join \bigl(\join^{m-3}\C^m\bigr);
$$
in particular ${\bf F}{\bf w} = {\bf 0}$ holds for an extensor ${\bf w}
\in \G(m-3, i_1(\PP^{m-1}))$ if and only if the $(m-3)$-plane
represented by ${\bf w}$ passes through ${\bf e}_1$.
\end{proposition}

\medskip

\begin{remark}
\begin{enumerate}
\item Given a point ${\bf p}\in Y_1$, the set of hyperplanes containing ${\bf e}_1$ and ${\bf p}$ that are tangent to $Y_1$ at ${\bf p}$ is a linear subspace of the dual space $\PP^{(m-1)*}$ of dimension $m-n-2$ if ${\bf e}_1$ lies in the tangent space $T_{\bf p}(Y_1)$ of $Y_1$ at ${\bf p}$, and otherwise of dimension $m-n-3$.
 
\item The number of distinct tangency points to $Y_1$ of hyperplanes
passing through ${\bf e}_1$ is generically equal to the class $c$ of
$Y_1$. For $m=3$ the hypothesis of
proposition~\ref{prop::valid_solutions} amounts exactly to the
classical condition of the Kruppa
equations~\cite{Kruppa-1913,Faugeras-Luong-01}: at least two
\emph{distinct} lines through ${\bf e}_1$ are tangent to $Y_1$ (see
lemma~\ref{lemma::cone_criterion}(2) below); in particular it requires
$c \geq 2$.
 
\item For a smooth complete intersection the class is always
comfortably large: if $\delta$ is the degree of the complete intersection of type $(d_1,\dots,d_k)$ in $\PP^m$ and $h_n$ denotes the complete homogeneous symmetric polynomial of degree $n$ as in remark~\ref{rem::classes_and_quadrics}, then the class is
$c = \delta\cdot h_n(d_1-1,\dots,d_k-1)$ with $\delta\ge 2^k$ and
$h_n(d_1-1,\dots,d_k-1)\ge h_n(1,\dots,1)=\binom{n+k-1}{n}\ge n+k-1=m-1$,
whence $c \ge 2^k(m-1)\ge 4(m-1)$, the value $4$ being attained by the type
$(2,2)$. In particular $c \geq 2$ always holds in the situations
considered in this paper.
 
\item The hypothesis that $C_{{\bf e}_1}$ is not a cone defines a
Zariski-open condition on ${\bf e}_1$, satisfied for all ${\bf e}_1$
outside a proper closed subset whenever $m = 3$ or $n \leq m-3$; see
lemma~\ref{lemma::cone_criterion}. In the remaining case $n = m-2$,
$m \geq 4$ (codimension two in $\PP^m$), we do not know how to exclude
special points ${\bf e}_1$ for which $C_{{\bf e}_1}$ is a cone, and the
hypothesis is stated as such; failure would mean that for some
hyperplane $H_0 \ni {\bf e}_1$ tangent to $Y_1$, every tangent
hyperplane $H \ni {\bf e}_1$ moves in a full pencil
$\langle H, H_0\rangle$ of tangent hyperplanes, that is, $Y_1^*$
contains the dual lines of all the codimension-two spaces $H \cap H_0$
which is a strong degeneracy of the dual variety.
\end{enumerate}
\end{remark}

\begin{proof}
Write $W = \join^{m-2}\C^m$ and consider the two homogeneous
polynomials of degree $c$ on $W$,
$$
\psi_1({\bf w}) = \phi_1(\gamma({\bf w})) = \phi_1({\bf e}_1 \join {\bf
w}), \qquad \psi_2({\bf w}) = \phi_2(\xi({\bf w})) = \phi_2({\bf F}{\bf
w}),
$$
where $\phi_1, \phi_2$ are the defining polynomials of the dual
hypersurfaces $Y_1^*, Y_2^*$ and are therefore reduced. The linear map
$\gamma: {\bf w} \mapsto {\bf e}_1 \join {\bf w}$ has kernel $K_1 =
{\bf e}_1 \join (\join^{m-3}\C^m)$, of dimension $\binom{m-1}{2}$, and
image the $(m-1)$-dimensional subspace of $\join^{m-1}\C^m$ underlying
the pencil $\sigma({\bf e}_1)$. We write $\rho$ for the restriction of
$\phi_1$ to this image, so that $\psi_1 = \rho \circ \gamma$ and
$C_{{\bf e}_1}$ is the hypersurface $\{\rho = 0\}$ of $\sigma({\bf
e}_1)$.
 
We first record an elementary fact. \emph{For a non-zero reduced
homogeneous polynomial $g$ on a complex vector space and a vector ${\bf v}$,
the following are equivalent: (a) the hypersurface $\{g=0\}$ is a cone
with vertex $[{\bf v}]$; (b) the directional derivative is identically zero: $\partial_{\bf v} g \equiv 0$; (c) $g({\bf
w} + t{\bf v}) = g({\bf w})$ for all ${\bf w}, t$.} Indeed, if (a)
holds, then for every zero ${\bf x}$ of $g$ the function $t \mapsto
g({\bf x}+t{\bf v})$ vanishes identically, so $\partial_{\bf v}g$
vanishes on $\{g = 0\}$. Moreover, since $g$ is reduced, the Nullstellensatz
gives $g \mid \partial_{\bf v}g$, and the degree drop forces
$\partial_{\bf v}g = 0$. That (b) implies (c) follows by integrating
$\frac{d}{dt}g({\bf w}+t{\bf v}) = (\partial_{\bf v}g)({\bf w}+t{\bf
v}) = 0$, and (c) trivially implies (a). We call the linear space of
vectors satisfying these conditions the \emph{vertex space} of $g$.
 
\emph{Step 0: $\psi_1 \not\equiv 0$.} If $\psi_1$ vanished identically,
the linear space $\sigma({\bf e}_1) \iso \PP^{m-2}$ would be contained
in $Y_1^*$. Since $Y_1$ is irreducible, so is its dual $Y_1^*$ (as the
image of the irreducible conormal variety \cite{Holme-88}). Being an irreducible
hypersurface containing a linear space of its own dimension, $Y_1^*$
would then be a hyperplane, i.e. $c = 1$, so that $Y_1 = (Y_1^*)^*$
would be a point, contradicting $\dim Y_1 = n \geq 1$. Hence $\psi_1
\not\equiv 0$ and, since $\lambda \neq 0$, also $\psi_2 \not\equiv 0$
and the two polynomials have the same zero locus $Z \subset W$. Let
$\psi$ denote the reduced defining polynomial of $Z$.
 
\emph{Step 1: $Z$ is a cone and $\ker({\bf F})$ and $K_1$ lie in the vertex space of
$\psi$.} For ${\bf u} \in \ker({\bf F})$ one has $\psi_2({\bf w}+{\bf
u}) = \phi_2({\bf F}{\bf w}) = \psi_2({\bf w})$ for all ${\bf w}$, so
$Z$ is stable under the translations ${\bf w} \mapsto {\bf w} + t{\bf
u}$, i.e. $Z$ is a cone with vertex $[{\bf u}]$, and by the fact above,
$\partial_{\bf u}\psi = 0$. The same argument applied to $\psi_1$ and
${\bf u} \in K_1 = \ker\gamma$ shows $K_1$ is contained in the vertex
space of $\psi$.
 
\emph{Step 2: the vertex space of $\psi$ is exactly $K_1$; in
particular $\ker({\bf F}) \subseteq K_1$.} Splitting $W = K_1 \oplus
W'$, with $\gamma$ restricting to an isomorphism $W' \iso
\mathrm{im}\,\gamma$, one has $\psi_1 = \rho\circ\gamma$, a polynomial
in the $W'$-coordinates only. Hence the reduced polynomial of $Z =
\gamma^{-1}\{\rho = 0\}$ is $\psi = \rho^{\mathrm{red}}\circ\gamma$ (a
polynomial in disjoint variables is square-free if and only if it is
square-free as a polynomial in those variables). Consequently, for any
${\bf v} \in W$,
$$
\partial_{\bf v}\psi \;=\; \bigl(\partial_{\gamma({\bf
v})}\rho^{\mathrm{red}}\bigr)\circ\gamma ,
$$
which vanishes identically if and only if $\partial_{\gamma({\bf
v})}\rho^{\mathrm{red}} \equiv 0$, i.e. if and only if $\gamma({\bf
v})$ lies in the vertex space of $\rho^{\mathrm{red}}$ and by the fact
above, it is equivalent to say that the reduced hypersurface underlying $C_{{\bf
e}_1}$ is a cone with vertex $[\gamma({\bf v})]$. By hypothesis this
forces $\gamma({\bf v}) = 0$, i.e. ${\bf v} \in K_1$. Combining with
Step 1, $\ker({\bf F}) \subseteq K_1$, whence
$$
\rank({\bf F}) \;\geq\; \dim W - \dim K_1 \;=\; \binom{m}{2} -
\binom{m-1}{2} \;=\; m-1 .
$$
 
\emph{Step 3: $\rank({\bf F}) = m$ is impossible.} Suppose ${\bf F}: W
\to \join^{m-1}\C^m$ were surjective. Then $Z = {\bf
F}^{-1}\{\phi_2 = 0\}$ and, as in Step 2 (choosing a splitting $W =
\ker({\bf F}) \oplus W''$), the reduced polynomial of $Z$ is $\psi =
\phi_2 \circ {\bf F}$, since $\phi_2$ is reduced. For ${\bf v} \in W$
we get $\partial_{\bf v}\psi = (\partial_{{\bf F}{\bf
v}}\phi_2)\circ{\bf F}$, which vanishes identically if and only if
${\bf F}{\bf v}$ lies in the vertex space of $\phi_2$ (using
surjectivity of ${\bf F}$). By Step 2 the vertex space of $\psi$ is
$K_1$, so ${\bf F}(K_1)$ is contained in the vertex space of $\phi_2$.
Now $\ker({\bf F}) \subseteq K_1$ and $\rank({\bf F}) = m$ give
$$
\dim {\bf F}(K_1) \;=\; \dim K_1 - \dim\ker({\bf F}) \;=\;
\binom{m-1}{2} - \Bigl(\binom{m}{2} - m\Bigr) \;=\; 1 ,
$$
so the vertex space of $\phi_2$ is non-zero and $Y_2^*$ is a cone, with
vertex some point $[H_0] \in \PP^{(m-1)*}$. But the tangent space of a
cone at any smooth point contains its vertex; hence every tangent
hyperplane of $Y_2^*$ passes through $[H_0]$, so that by biduality
$Y_2 = (Y_2^*)^*$ is contained in the hyperplane $H_0$
(see~\cite{Tevelev-05}). This contradicts the non-degeneracy of $Y_2$:
a hyperplane of $i_2(\PP^{m-1})$ containing $Y_2 = \pi_2(X)$ would pull
back to a hyperplane of $\PP^m$ through ${\bf O}_2$ containing $X$,
against the first assumption of
section~\ref{sec::generic_projection}.
 
\emph{Step 4: conclusion.} By Steps 2 and 3, $\rank({\bf F}) = m-1$, so
$\dim\ker({\bf F}) = \binom{m}{2} - (m-1) = \binom{m-1}{2} = \dim K_1$;
together with $\ker({\bf F}) \subseteq K_1$ this gives $\ker({\bf F}) =
K_1 = {\bf e}_1 \join (\join^{m-3}\C^m)$. Finally, a decomposable
extensor ${\bf w}$ lies in $K_1$ if and only if ${\bf e}_1 \join {\bf
w} = 0$, i.e. if and only if the space it represents contains ${\bf
e}_1$. This proves the last assertion.
\end{proof}

\begin{lemma}
\label{lemma::cone_criterion}
\begin{enumerate}
\item The set of points ${\bf e}_1 \in i_1(\PP^{m-1})$ for which
$C_{{\bf e}_1}$ is a cone is a closed algebraic subset.
\item For $m = 3$, $C_{{\bf e}_1}$ is not a cone if and only if at
least two distinct lines through ${\bf e}_1$ are tangent to $Y_1$. This
holds for every ${\bf e}_1$ outside the union of $Y_1$ with finitely
many lines.
\item For $n \leq m-3$, $C_{{\bf e}_1}$ is not a cone for every ${\bf
e}_1$ in a dense open subset of $i_1(\PP^{m-1})$.
\end{enumerate}
\end{lemma}
\begin{proof}
(1) Since $c \geq 2$, Step 0 of the previous proof shows that the
restriction $\rho = \rho_{{\bf e}_1}$ of $\phi_1$ to $\sigma({\bf
e}_1)$ never vanishes identically, and its coefficients depend
algebraically on ${\bf e}_1$. For a vertex candidate $[{\bf v}] \in
\sigma({\bf e}_1)$, write $\rho({\bf w} + t{\bf v}) = \sum_{j\geq 0}
t^j \rho_j({\bf w})$. So the hypersurface $\{\rho = 0\}$ is a cone with
vertex $[{\bf v}]$ if and only if every $\rho_j$ vanishes on $\{\rho =
0\}$, i.e. lies in the radical of $(\rho)$. By the effective
Nullstellensatz this is equivalent to $\rho_j^{N} \in (\rho)$ for a
fixed exponent $N$ depending only on $c$ and $m$; since $\rho \neq 0$,
membership in the principal ideal $(\rho)$ in each fixed degree is a
closed linear-algebraic condition on the coefficients, hence on the
pair $({\bf e}_1, [{\bf v}])$. The cone locus is the image of the
resulting closed incidence set under the proper projection forgetting
$[{\bf v}]$, hence closed.
 
(2) A proper closed subset of $\PP^1$ is a cone with vertex a point
$[{\bf v}]$ if and only if it is the single point $[{\bf v}]$. Thus
$C_{{\bf e}_1} \subset \sigma({\bf e}_1) \iso \PP^1$ fails the
hypothesis exactly when only one line through ${\bf e}_1$ is tangent to
$Y_1$. The pencil $\sigma({\bf e}_1)$ meets the irreducible curve
$Y_1^*$, of degree $c \geq 2$, in $c$ distinct points as soon as it is
transverse to $Y_1^*$ and avoids its singular points; the pencils
failing this are parametrized by ${\bf e}_1 \in (Y_1^*)^* \cup
\bigl(\bigcup_{s \in \mathrm{Sing}(Y_1^*)} s^\perp\bigr) = Y_1 \cup
(\text{finitely many lines})$, using biduality $(Y_1^*)^* = Y_1$. Since $\dim(Y_1^*) = 1$, there are finitely many singularities at most.
 
(3) Since $n \leq m-3$, the projection $\overline{Y}_1 := \pi_{{\bf
e}_1}(Y_1) \subset \PP^{m-2}$ from ${\bf e}_1$ is a proper subvariety,
and it is non-degenerate because $Y_1$ is: a hyperplane of $\PP^{m-2}$
containing $\overline{Y}_1$ pulls back to a hyperplane through ${\bf
e}_1$ containing $Y_1$. As recalled in
section~\ref{sec::generic_projection} (item 3, and~\cite{Tevelev-05}),
projective duality interchanges projections and hyperplane sections:
for ${\bf e}_1$ in a dense open subset,
$$
\bigl(\pi_{{\bf e}_1}(Y_1)\bigr)^* \;=\; Y_1^* \cap \sigma({\bf e}_1)
\;=\; C_{{\bf e}_1}
$$
as subsets of $\sigma({\bf e}_1) \iso (\PP^{m-2})^*$. If $C_{{\bf
e}_1}$ were a cone with vertex $[H_0]$, then, since the tangent spaces
of a cone at smooth points contain its vertex, every tangent hyperplane
of $(\overline{Y}_1)^*$ would pass through $[H_0]$, and biduality would
place $\overline{Y}_1 = ((\overline{Y}_1)^*)^*$ inside the hyperplane
$H_0$ of $\PP^{m-2}$, contradicting its non-degeneracy.
\end{proof}

\medskip

As a result, every solution of the generalized Kruppa equation with
$\lambda \neq 0$ whose first epipole lies outside the closed cone locus
of lemma~\ref{lemma::cone_criterion}(1) is admissible. In the sequel, we shall consider the matrix $[{\bf e}_1]$ which represents the map ${\bf w} \mapsto {\bf e}_1 \join {\bf w}$ from $\join^{m-3} \C^m$ to $\join^{m-2} \C^m$. Therefore ${\bf F} [{\bf e}_1] = 0$. 

In addition, since $\rank({\bf F}) = m-1$ and the codomain of ${\bf F}$ is $\G(m-2,i_2(\PP^{m-1})) \iso \PP^{(m-1)*}$ (the dual space of $\PP^{m-1}$), the kernel of ${\bf F}^T$ is one-dimensional. Also observe that $\C {\bf e}_2$ is within the annihilator of $\im({\bf F})$: ${\bf e}_2^T {\bf F w} = 0$ for any ${\bf w} \in \G(m-3,i_1(\PP^{m-1}))$ since ${\bf e}_2$ lies in the hyperplane defined by ${\bf F w}$. Therefore, $\ker({\bf F}^T) = \C {\bf e}_2$. Therefore we add to our system of equations the following equality ${\bf F}^T {\bf e}_2 = {\bf 0}$. Hence, we have now a projective variety $V$ embedded in $\PP^{m-1}  \times \PP^M \times \PP^{m-1}$ defined by the following set of equations:
$$
\begin{array}{ccc}
{\bf F}[{\bf e}_1] = {\bf 0}, & \{E_i({\bf F},{\bf e}_1)\}, & {\bf F}^T {\bf e}_2 = {\bf 0}.
\end{array}
$$

\begin{remark}
\label{rem::degenerate_branch}
The equations $\{E_i\}$ express the proportionality of the two
coefficient vectors of $\phi_2(\xi({\bf w}))$ and $\phi_1(\gamma({\bf
w}))$, and are therefore also satisfied when $\phi_2(\xi({\bf w}))$
vanishes identically, a degenerate branch corresponding to $\lambda =
0$ in equation~(\ref{ExtKrupEq}). This branch is non-empty: for
instance, any rank-one matrix ${\bf F} = {\bf h}{\bf v}^T$ with $[{\bf
h}] \in Y_2^*$, ${\bf v}^T[{\bf e}_1] = 0$ and ${\bf e}_2$ lying on the
hyperplane represented by ${\bf h}$ satisfies all the equations
defining $V$, and such triples form a positive-dimensional family. These
solutions are not admissible, and proposition~\ref{prop::valid_solutions}
does not apply to them since its hypothesis includes $\lambda \neq 0$.
Accordingly, in the discreteness analysis below, the statements about
$V$ are to be understood as statements about its \emph{admissible
locus}
$$
V^{\mathrm{adm}} \;=\; \bigl\{({\bf e}_1, {\bf F}, {\bf e}_2) \in V
\;\big|\; \rank({\bf F}) = m-1 \bigr\},
$$
an open subset of $V$ which contains all the triples arising from
actual pairs of projections, and on which, by
proposition~\ref{prop::valid_solutions}, every point with ${\bf e}_1$
outside the cone locus is a genuine admissible solution. The lower
bound of proposition~\ref{prop:lowerbound} below, being valid for every
irreducible component of $V$, is not affected by this convention.
\end{remark}

We next compute a lower bound on the dimension of $V$, after which we will be ready for the calculation itself. 

Recall that $c$ is the class of $X$, which is also the class of $Y_1$ and $Y_2$ for generic projections. Also recall that $N$ is defined at the beginning of section~\ref{sec::variety_of_second_order_F}, and denotes the dimension of the space of second order fundamental matrices.

\medskip

\begin{proposition} \label{prop:lowerbound}
If $V$ is non-empty, the dimension of $V$ is at least $N-\binom{m-2+c}{c} + 1$.  
\end{proposition}

\begin{proof}
Choose any hyperplane $H$ included in $i_1(\PP^{m-1})$ and take ${\bf e}_1$ to be in the affine open set $i_1(\PP^{m-1}) \setminus H$, by choosing the centers of projection in a smaller open set. Since ${\bf F}[{\bf e}_1]={\bf 0}$, the two
sides of equation (\ref{ExtKrupEq}) are both unchanged by replacing
${\bf w}$ by ${\bf w}+\alpha ({\bf e}_1 \join {\bf \overline{w}})$, where $\alpha \in \C$ and ${\bf \overline{w}} \in \join^{m-3} \C^m$ if $m \geq 4$ (otherwise ${\bf \overline{w}} = 1$).  So equation (\ref{ExtKrupEq})
will hold for all ${\bf w}$ if it holds for every ${\bf w}$ that represents a $(m-3)-$subspace of $H$. Indeed assume it holds for all $(m-3)-$subspaces of $H$. Let ${\bf w} \in \G(m-3,i_1(\PP^{m-1}))$ representing a space $W$ not containing ${\bf e}_1$. Then let $K$ be the $(m-3)-$plane in $H$ which is the trace on $H$ of the cone generated by $W$ and ${\bf e}_1$. Let ${\bf k}$ be the extensor representing $K$. Then by construction there is some $\beta \in \C^*$, such that ${\bf e}_1 \join {\bf k} = \beta {\bf e}_1 \join {\bf w}$. We normalize ${\bf k}$ such that $\beta = 1$. Then ${\bf k} - {\bf w} = {\bf e}_1 \join {\bf \overline{w}}$ for some ${\bf \overline{w}} \in \join^{m-3} \C^m$. Therefore if equation~\eqref{ExtKrupEq} holds for ${\bf k}$, it also holds for ${\bf w}$. 

Let $[x_1: \cdots :x_{m-1}]$ be homogeneous
coordinates on the dual of $H$, which is canonically isomorphic to $\G(m-3,H)$. 
Therefore equation (\ref{ExtKrupEq}) is equivalent to the equality of 2 homogeneous polynomials of degree $c$ in $(x_1, \cdots, x_{m-1})$, which in turn is equivalent to the equality of $\binom{m-2+c}{c}$ coefficients. After eliminating $\lambda$, we have $\binom{m-2+c}{c} - 1$ algebraic conditions on
$({\bf e}_1,{\bf F},{\bf e}_2)$ in addition to ${\bf F}[{\bf e}_1]={\bf 0}$, ${\bf F}^T {\bf e}_2 = {\bf 0}$. The set of triples $({\bf e}_1, {\bf F}, {\bf e}_2) \in \PP^{m-1} \times
\PP^{M} \times \PP^{m-1}$ satisfying the two linear systems ${\bf F}[{\bf
e}_1] = {\bf 0}$ and ${\bf F}^T {\bf e}_2 = {\bf 0}$ is the incidence
variety $\wtilde{\mathcal{A}}_2$ introduced in the proof of
proposition~\ref{prop::admissible_variety}: it fibers over the pairs
$({\bf e}_1,{\bf e}_2) \in \PP^{m-1} \times \PP^{m-1}$ with fibers which are 
projective spaces of dimension $(m-1)^2-1$, and is therefore an
irreducible projective variety of dimension $2(m-1) + (m-1)^2 - 1 = m^2-2
= N$. The variety $V$ is cut out inside $\wtilde{\mathcal{A}}_2$ by the
$\binom{m-2+c}{c} - 1$ equations $\{E_i\}$. Therefore, $V$ is at least
$\left(N - \binom{m-2+c}{c} + 1\right)$-dimensional.
\end{proof}

\medskip

This proposition provides a necessary condition, which roughly speaking means that for large enough values of $c$, the admissible locus $V^{\mathrm{adm}}$ of remark~\ref{rem::degenerate_branch} may be discrete. Note that the class $c$ can be made arbitrarily big. Therefore one would expect that for $c$ large enough the variety $V^{\mathrm{adm}}$ will be finite, so that the projections $\pi_1,\pi_2$ could be recovered modulo the projective group of $\PP^m$ up to a finite fold ambiguity. What we shall actually prove is a rigidity statement: under an explicit geometric condition, the \emph{true} configuration is an isolated point of $V^{\mathrm{adm}}$, so that the pair $(\pi_1,\pi_2)$ admits no continuous deformation compatible with the generalized Kruppa equation. We proceed in two stages. First we consider the case $\dim(X) \geq \codim(X)-1$ and then the case $\dim(X) < \codim(X)-1$.

Before turning to the sufficient conditions, we record an explicit upper bound
on the number of isolated solutions, which complements the lower bound above.

\begin{proposition}[Bound on the number of isolated solutions]
\label{prop::degree_bound}
The number of isolated points of $V$ is at most
$$
\binom{\tfrac{m^2(m-1)}{2}+m-2}{\,m-1\,}\; c^{\,m^2-2}.
$$
For $m=3,4,5$ the bound reads $45\,c^{7}$, $2600\,c^{14}$ and
$292825\,c^{23}$ respectively. Since $V^{\mathrm{adm}}$ is open in $V$
(remark~\ref{rem::degenerate_branch}), the bound applies in particular to
the configurations that are isolated in $V^{\mathrm{adm}}$, such as the
rigid true configuration of proposition~\ref{prop::no_quadric}.
\end{proposition}

\begin{proof}
We work in $\PP^{m-1}\times\PP^{M}\times\PP^{m-1}$ with the variables
$({\bf e}_1,{\bf F},{\bf e}_2)$ and write $h_1,h_2,h_3$ for the three
hyperplane classes. The variety $V$ is defined by the linear system
${\bf F}[{\bf e}_1]={\bf 0}$, whose $m\binom{m}{3}$ equations have
multidegree $(1,1,0)$, the linear system ${\bf F}^T{\bf e}_2={\bf 0}$,
whose $\frac{m(m-1)}{2}$ equations have multidegree $(0,1,1)$, and the
Kruppa equations $\{E_i\}$, whose $\binom{m-2+c}{c}-1$ equations have
multidegree $(c,c,0)$: each $E_i$ is a $2\times2$ minor of the matrix
formed by the coefficient vectors of the two sides of
equation~\eqref{ExtKrupEq}, the first row being of degree $c$ in
${\bf e}_1$ and the second of degree $c$ in ${\bf F}$. Set
$K_1=\frac{m(m-1)(m-2)}{2}$, $K_2=m-1$ and $N=m^2-2$; then $K_1+K_2+N$
equals the dimension $\frac{m^2(m-1)}{2}+2m-3$ of the ambient product,
and one checks $K_1\le m\binom m3$ and $K_2\le\frac{m(m-1)}{2}$, while
$N\le\binom{m-2+c}{c}-1$ holds whenever $V$ has an isolated point at
all, by proposition~\ref{prop:lowerbound}. By the randomization
principle for overdetermined systems (see~\cite{Sommese-Wampler-05}),
every isolated point of $V$ is an isolated solution of the square system
consisting of $K_1$ generic linear combinations of the first group, $K_2$
of the second and $N$ of the third, and such combinations retain the
respective multidegrees. The multihomogeneous B\'ezout
theorem~\cite{Morgan-Sommese-87} bounds the number of isolated solutions
of this square system by the coefficient of $h_1^{m-1}h_2^{M}h_3^{m-1}$
in
$$
(h_1+h_2)^{K_1}\,(h_2+h_3)^{K_2}\,(c\,h_1+c\,h_2)^{N}.
$$
Since $K_2=m-1$, the factor $(h_2+h_3)^{K_2}$ contributes $h_3^{m-1}$
with coefficient $1$, and the remaining coefficient is that of
$h_1^{m-1}h_2^{M}$ in $(h_1+h_2)^{K_1+N}$ with $K_1+N=M+m-1$, namely
$\binom{M+m-1}{m-1}$. Multiplying by $c^{N}$ gives the bound.
\end{proof}

\begin{remark}
The bound counts isolated points only and remains valid in the presence
of positive-dimensional components of $V$, such as the degenerate branch
of remark~\ref{rem::degenerate_branch} or the homographic components of
lemma~\ref{lem::homographic} below. It is certainly not sharp. Its
interest is twofold: it complements the lower bound of
proposition~\ref{prop:lowerbound}, and it is precisely the number of
paths that a multihomogeneous homotopy continuation method would track
to compute all isolated solutions of the generalized Kruppa equation,
which makes the recovery problem effectively computable in principle.
\end{remark}

\subsubsection{The realizable locus and its finiteness}
\label{sec::Vreal}

The variety $V^{\mathrm{adm}}$ is cut out by the generalized Kruppa equation, and
nothing in that equation forces a solution to come from an actual pair of
projections of an actual variety: lemma~\ref{lem::homographic} already exhibits a
component of dimension $m-1$ consisting of solutions of this kind. In this
subsection we isolate the solutions that do come from a configuration, we show
that they are finite in number, and we write down the polynomial system that
computes them. The results are independent of the dichotomy $2n+1 \geq m$ versus
$2n+1 < m$ and are used in both cases below; they require only
$\codim(X) \geq 2$.

Throughout, $S=\C[x_0,\dots,x_m]$ is the coordinate ring of $\PP^m$ and
$S_k=\C[y_0,\dots,y_{m-1}]$ that of $i_k(\PP^{m-1})$, so that
$R(Y_k)=S_k/I(Y_k)$ is the homogeneous coordinate ring of the $k$-th image.

\begin{definition}[Admissible class and realizable locus]
\label{def::Vreal}
Let $P$ be a numerical polynomial. We denote by $\mathcal{K}=\mathcal{K}(P)$ the
set of smooth irreducible nondegenerate subvarieties $X'' \subset \PP^m$ with
Hilbert polynomial $P$, with $\codim(X'')\geq 2$ and $\deg(X'')>2$, satisfying the
genericity conditions of section~\ref{sec::genericity_conditions}. A point
$\Theta \in V^{\mathrm{adm}}$ determines a pair of projections
$(\pi_1^\Theta,\pi_2^\Theta)$ modulo the action $\varrho$. The \emph{realizable
locus} is
$$
V^{\mathrm{real}} \;=\; \bigl\{\, \Theta \in V^{\mathrm{adm}} \;\mid\;
\exists\, X'' \in \mathcal{K} \ \text{ with } \ \pi_k^\Theta(X'') = Y_k,\
k=1,2 \,\bigr\}.
$$
\end{definition}

The true configuration $\Theta_0$ lies in $V^{\mathrm{real}}$, and
$V^{\mathrm{real}}$ depends only on the data $(Y_1,Y_2)$ and on $\mathcal{K}$: no
knowledge of $X$ enters. We will see in remark~\ref{rem::Vreal_compute} that even
$P$ is read off from the data, so that $\mathcal{K}$ is itself determined by
$(Y_1,Y_2)$.

\begin{lemma}[The projection is the normalization]
\label{lem::Vreal_norm}
Let $X'' \in \mathcal{K}$ with $\pi_k(X'')=Y_k$, where $\pi_k$ is the projection of
center ${\bf O}_k \notin X''$. Then $\pi_k|_{X''} : X'' \lra Y_k$ is finite and birational,
and it is the normalization of $Y_k$ (see~\cite[\S18.2, Thm.~18.16]{Ellingsrud-Ottem} for the definition of normalization):
$$
X'' \;\iso\; \widetilde{Y_k} \quad \text{over } Y_k,
$$
canonically. Moreover $\mathcal{O}_{X''}(1)$ corresponds to
$(\pi_k|_{X''})^*\mathcal{O}_{Y_k}(1)$.
\end{lemma}

\begin{proof}
The morphism $\pi_k|_{X''}$ is finite: a positive-dimensional fibre would be a
line through ${\bf O}_k$ contained in $X''$, hence would force
${\bf O}_k \in X''$, which is excluded. Therefore
$\deg(X'')=\deg(Y_k)\cdot\deg(\pi_k|_{X''})$. Both $X''$ and $X$ have Hilbert
polynomial $P$ (see remark~\ref{rem::Vreal_compute} (3) below), hence the same degree $\delta$; and $\deg(Y_k)=\delta$ because the
true projection $\pi_k|_X$ is birational onto its image, which holds as soon as
$\codim(X)\geq 2$. Therefore $\deg(\pi_k|_{X''})=1$. Now $X''$ is smooth, hence
normal, and a finite birational morphism onto $Y_k$ from a normal variety is the
normalization of $Y_k$, by the universal property. The statement on
$\mathcal{O}(1)$ holds because a linear projection pulls hyperplanes back to
hyperplanes.
\end{proof}

Note that $Y_k$ is \emph{not} normal when $2n+1 \geq m$: a generic projection is
then not an embedding and $Y_k$ carries ordinary double points. It is the
normalization, not an isomorphism, that is available in both regimes.

\begin{remark}[Computing $(Z_k,L_k)$, $W_k$ and $P$ from one image]
\label{rem::Vreal_compute}
All the objects of lemma~\ref{lem::Vreal_norm} are effectively computable from
$I(Y_k)$ alone.

\begin{enumerate}

\item \emph{Normalization.} Compute the integral closure $\overline{R(Y_k)}$ of
$R(Y_k)$ in its field of fractions, by the algorithm of de Jong
\cite{deJong-98} in the form of Greuel, Laplagne and Seelisch
\cite{Greuel-Laplagne-Seelisch-10}; it is implemented in \texttt{normal.lib} of
{\sc Singular} and in \texttt{integralClosure} of {\sc Macaulay2}. The algorithm
returns generators of $\overline{R(Y_k)}$ as fractions, together with their
degrees, hence returns the normalization morphism explicitly. Set
$$
Z_k \;=\; \operatorname{Proj}\overline{R(Y_k)} \;=\; \widetilde{Y_k},
\qquad
\nu_k : Z_k \lra Y_k,
\qquad
L_k \;=\; \nu_k^*\mathcal{O}_{Y_k}(1),
$$
where $\nu_k$ is the normalization morphism, induced by the inclusion
$R(Y_k)\subset\overline{R(Y_k)}$. Thus $Z_k$ is the normalization of the $k$-th
image and $L_k$ its polarization; these are the only two such objects occurring
below, one for each image.

\item \emph{Sections and the subsystem.} The polarization
$L_k=\nu_k^*\mathcal{O}_{Y_k}(1)$ is a line bundle on $Z_k$, and $H^0(Z_k,L_k)$ is
computed as sheaf cohomology on $Z_k$. In fact $\overline{R(Y_k)}$ \emph{is} the section ring of $(Z_k,L_k)$:
$$
\overline{R(Y_k)}_d \;=\; H^0\bigl(Z_k,L_k^{\otimes d}\bigr)
\qquad\text{for every } d \geq 0 ,
$$
and in particular $H^0(Z_k,L_k)=\overline{R(Y_k)}_1$, with no hypothesis beyond
$n\geq1$. Only the sheaf axiom and the normality of an integral closure are
needed to prove this equality. Write $B=\overline{R(Y_k)}$ and
$\mathfrak{m}=B_+$ for its irrelevant maximal ideal, and let
$C=\operatorname{Spec}B$ be the affine cone with vertex the point $v$
corresponding to $\mathfrak{m}$. Since $B$ is finite over $R(Y_k)$, the classes
of the degree-one elements $y_0,\dots,y_{m-1}$ generate $\mathfrak{m}$ up to
radical, $\sqrt{(y_0,\dots,y_{m-1})B}=\mathfrak{m}$, so that
$$
C\setminus\{v\}=\bigcup_i D(y_i),
\qquad
Z_k=\bigcup_i \operatorname{Spec}(B_{y_i})_0,
\qquad
(B_{y_i})_d=y_i^d\,(B_{y_i})_0 .
$$

Apply the sheaf axiom twice. On the one hand it identifies
$\mathcal{O}_C(C\setminus\{v\})$ with the kernel of
$\bigoplus_i B_{y_i}\lra\bigoplus_{i<j}B_{y_iy_j}$ (see~\cite[p. 39]{Ellingsrud-Ottem}); on the other, since
$L_k^{\otimes d}$ restricts on $\operatorname{Spec}(B_{y_i})_0$ to the free module
$(B_{y_i})_d$ generated by $y_i^d$, it identifies $H^0(Z_k,L_k^{\otimes d})$ with
the kernel of $\bigoplus_i (B_{y_i})_d\lra\bigoplus_{i<j}(B_{y_iy_j})_d$. All the maps preserve
degrees, so the second kernel is the degree-$d$ part of the first:
$$
\mathcal{O}_C\bigl(C\setminus\{v\}\bigr)_d \;=\; H^0\bigl(Z_k,L_k^{\otimes d}\bigr).
$$
The assertion is therefore that the restriction
$B=\mathcal{O}_C(C)\lra\mathcal{O}_C(C\setminus\{v\})$ is bijective, that is that
a regular function on the punctured cone extends across the vertex. Now $B$ is an
integral closure, hence a Noetherian normal domain, and
$$
\operatorname{codim}_C\{v\} \;=\; \operatorname{ht}\mathfrak{m} \;=\; \dim B \;=\; \dim Y_k+1 \;=\; n+1 \;\geq\; 2 ,
$$
so the restriction is an isomorphism by \cite[Thm. 18.22]{Ellingsrud-Ottem}.
The hypothesis $n\geq1$ is used exactly here: for $n=0$ the vertex has
codimension one and the extension fails. This
is also what makes item (4) below well posed: by this identification, a linear
subspace $V\subset H^0(Z_k,L_k)$ is a subspace of $\overline{R(Y_k)}_1$, so
sending $x_0,\dots,x_m$ to a basis of $V$ defines a graded homomorphism of
$\C$-algebras $S\to\overline{R(Y_k)}$.

The image of $R(Y_k)_1$, that is the span of the classes of
$y_0,\dots,y_{m-1}$, is an $m$-dimensional subspace
$$
W_k \;\subset\; H^0(Z_k,L_k),
$$
and ${\bf O}_k$ is the point of $\PP^m$ annihilated by $W_k$ once the embedding of
item (4) below is fixed. Concretely, writing $y_i=\sum_jA_{ij}s_j$ in a basis
$s_0,\dots,s_m$ of the space $V$ of item (4), ${\bf O}_k=\ker A$ is the common zero
of the linear forms in $W_k$, and projecting from it, $x\mapsto Ax$, composed with
the embedding by $V$ is the normalization $Z_k\lra Y_k$; choosing
$s_m\in V\setminus W_k$ and $s_i=y_i$ for $i<m$ gives ${\bf O}_k=[0:\dots:0:1]$.

\item \emph{Hilbert polynomial.} $P(d)=\chi(Z_k,L_k^{\otimes d})$, where
$\chi(Z,\mathcal{F})=\sum_i(-1)^i\dim H^i(Z,\mathcal{F})$ is the Euler--Poincar\'e
characteristic of a coherent sheaf~\cite[Corollary~20.29]{Ellingsrud-Ottem}. The identity holds because both sides are
polynomials in $d$ which agree for $d \gg 0$: the Hilbert polynomial of
$X''\subset\PP^m$ agrees with $h^0(\mathcal{O}_{X''}(d))$ for $d \gg 0$, the higher
cohomology vanishes for $d \gg 0$ by Serre~\cite[Thm. 20.24]{Ellingsrud-Ottem}, and
$(X'',\mathcal{O}_{X''}(1))\iso(Z_k,L_k)$ by lemma~\ref{lem::Vreal_norm}. In
particular
$P$, hence the class $\mathcal{K}$ of definition~\ref{def::Vreal}, is determined by
the data and need not be assumed known.

\item \emph{Re-embedding.} For a subspace $V$ with $W_k\subset V\subset
H^0(Z_k,L_k)$ and $\dim V=m+1$, the morphism $Z_k\to\PP^m$ defined by $V$ has
image a candidate for $X''$; its ideal is the kernel of the ring map
$S\to\overline{R(Y_k)}$ sending $x_0,\dots,x_m$ to a basis of $V$, computed by
elimination.
\end{enumerate}
\end{remark}

\begin{example}
\label{ex::Vreal_quartic}
Let $X=V(q_1,q_2)\subset\PP^3$ be a smooth complete intersection of two quadrics,
so the degree and the dimensions are given by $\delta=4$, $n=1$ and $X$ is an elliptic normal quartic (by adjunction $\omega_{X} \cong \mathcal{O}_X(-4+2+2) = \mathcal{O}_X$, so $g=1$; see~\cite[Ex. 17.9.18]{Ellingsrud-Ottem}). Projecting from a point
off $X$ gives a plane quartic $Y\subset\PP^2$ with two nodes, and
$\dim R(Y)_1=3$. The integral closure $\overline{R(Y)}$ is generated by
\emph{four} elements of degree one, so $h^0(Z_1,L_1)=4$, $W_1=R(Y)_1$ is a
hyperplane in it, $V=H^0(Z_1,L_1)$ is forced, and the re-embedding of item (4)
returns $X\subset\PP^3$.
Here one image already determines $X$ modulo $\PP GL_4$, because $X$ is linearly
normal; this is the content of lemma~\ref{lem::Vreal_embedding} in the case
$h^0(L)=m+1$.
\end{example}

\begin{example}[A non-complete-intersection]
\label{ex::Vreal_quintic_normal}
Nothing above uses a complete-intersection structure, and the following case
makes that concrete. Let $X\subset\PP^4$ be an elliptic normal quintic, that is
the smooth curve of degree $\delta=5$ and genus $1$ cut out by the five Pfaffians
of a $5\times5$ skew-symmetric matrix of linear forms. Recall that the determinant
of a skew-symmetric matrix $A$ of even size is the square of a polynomial in its
entries, its \emph{Pfaffian} $\operatorname{Pf}(A)$; for $4\times4$ matrices,
$\operatorname{Pf}(A)=a_{12}a_{34}-a_{13}a_{24}+a_{14}a_{23}$. The
\emph{Pfaffians} of a $5\times5$ skew-symmetric matrix $M$ are the
$\operatorname{Pf}(M_i)$, $i=1,\dots,5$, where $M_i$ is obtained by deleting the
$i$-th row and column; here they are five quadrics. Its ideal needs five
generators in codimension three, so $X$ is not a complete intersection; it is
arithmetically Gorenstein, and linearly normal, being embedded by a complete
linear system. Here $X$ is \emph{arithmetically Gorenstein} when its homogeneous
coordinate ring $S/I_X$ is a Gorenstein ring, that is, a Cohen--Macaulay ring
whose canonical module is isomorphic to a shift of itself; equivalently, the
minimal graded free resolution of $S/I_X$ has length $\codim X$, is self-dual up
to a twist, and ends with a free module of rank one~\cite[\S3.3]{Bruns-Herzog-98}.
By the Buchsbaum--Eisenbud structure theorem~\cite[Theorem~3.4.1]{Bruns-Herzog-98},
the Gorenstein ideals of codimension three are exactly those generated by the
Pfaffians of a skew-symmetric matrix, and here the resolution reads
$$
0\lra S(-5)\lra S(-3)^{5}\xrightarrow{\;M\;} S(-2)^{5}\lra S\lra S/I_X\lra 0.
$$ Since $2n+1=3<m=4$, a generic projection embeds $X$, so
$Y\subset\PP^3$ is a smooth quintic curve of genus one with
$\dim R(Y)_1=4$ and $\deg Y=5=\deg X$, as lemma~\ref{lem::Vreal_norm} requires.
The integral closure $\overline{R(Y)}$ has $h^0(Z_1,L_1)=5=m+1$, so $V$ is again
forced and the re-embedding returns $X$: a single image already determines it
modulo $\PP GL_5$. Only the Hilbert polynomial and linear normality enter, exactly
as in example~\ref{ex::Vreal_quartic}.
\end{example}

Recall that a line bundle $L$ on a projective variety $Z$ is \emph{very ample} if
there is a closed embedding $\iota:Z\lra\PP^N$ with
$L\iso\iota^*\mathcal{O}_{\PP^N}(1)$; equivalently, a basis of $H^0(Z,L)$ defines
a closed embedding of $Z$, by the complete linear system
$|L|$~\cite[Def.~21.1, Thm.~12.20 and \S19.3]{Ellingsrud-Ottem}. A \emph{polarized variety} is a pair
$(Z,L)$ formed by a variety and a very ample line bundle on it, and recall that an isomorphism $\psi:Z_1 \lra Z_2$ is said to
be \emph{polarized} when $\psi^*L_2 \iso L_1$. Equivalently, $\psi$ is induced by a
projective transformation between the embeddings of $Z_1$ and $Z_2$ by the complete
systems $|L_1|$ and $|L_2|$; that transformation is the matrix $\Psi$ of
proposition~\ref{prop::Vreal_system}. Note that the composite below is polarized,
both identifications carrying $\mathcal{O}_{X''}(1)$ to the corresponding $L_k$.

\begin{lemma}[Reconstruction of the embedding]
\label{lem::Vreal_embedding}
Let $\Theta \in V^{\mathrm{real}}$, realized by $X''\in\mathcal{K}$. Write
$(Z_k,L_k,W_k)$, $k=1,2$, for the data attached to the two images in
remark~\ref{rem::Vreal_compute}, so that $Z_k=\widetilde{Y_k}$, and abbreviate
$$
(Z,L) \;=\; (Z_1,L_1), \qquad N+1 \;=\; h^0(Z,L).
$$
Let $\psi : Z_1 \lra Z_2$ be the isomorphism induced by $X''$, namely the composite
$$
Z_1 \;\stackrel{\sim}{\lla}\; X'' \;\stackrel{\sim}{\lra}\; Z_2
$$
of the two identifications of lemma~\ref{lem::Vreal_norm}. It is polarized, so
$\psi^*L_2\iso L_1$ and the induced map $\psi^*:H^0(Z_2,L_2)\lra H^0(Z_1,L_1)$ is
well defined up to a scalar, the isomorphism $\psi^*L_2\iso L_1$ being unique up
to a scalar; in particular the subspace $\psi^*W_2\subset H^0(Z,L)$ is well
defined. Then
$$
V \;=\; W_1 + \psi^*W_2 \;\subset\; H^0(Z,L)
$$
has dimension $m+1$, the variety $X''$ is the image of $Z$ under the morphism
defined by $V$, and the two centers are the points of $\PP(V^*)$ annihilated by
$W_1$ and by $\psi^*W_2$. In particular $\Theta$ is determined by $\psi$.
\end{lemma}

\begin{proof}
By lemma~\ref{lem::Vreal_norm} the two projections identify $X''$ with $Z$, and
$\mathcal{O}_{X''}(1)$ with $L$. The embedding $X''\subset\PP^m$ is therefore given
by an $(m+1)$-dimensional subspace $V\subset H^0(Z,L)$, and the two projections
correspond to the two hyperplanes $W_1$ and $\psi^*W_2$ of $V$ obtained by
discarding the center. Two distinct hyperplanes of $V$ span $V$, and $W_1 \neq
\psi^*W_2$ because ${\bf O}_1\neq{\bf O}_2$; hence $V=W_1+\psi^*W_2$. The last two
statements are the definitions of the re-embedding and of the centers.
\end{proof}

\begin{theorem}[Finiteness of the realizable locus]
\label{thm::Vreal_finite}
Let $(Z_k,L_k,W_k)$, $k=1,2$, be the data attached to the two images in
remark~\ref{rem::Vreal_compute}, let $(Z,L)=(Z_1,L_1)$ and $N+1=h^0(Z,L)$, and set
$$
G \;=\; \operatorname{Aut}(Z,L) \;=\;
   \{\varphi \in \operatorname{Aut}(Z) \mid \varphi^*L\iso L\},
\qquad
T \;=\; \operatorname{Isom}\bigl((Z_1,L_1),(Z_2,L_2)\bigr),
$$
where $T$ denotes the set of \emph{polarized} isomorphisms. If
$V^{\mathrm{real}}\neq\varnothing$ then $T\neq\varnothing$ by
lemma~\ref{lem::Vreal_embedding}, so $(Z_1,L_1)\iso(Z_2,L_2)$ and the choice of index
is immaterial; $T$ is then a
torsor under $G$, acting by precomposition, $\psi\cdot\varphi=\psi\circ\varphi$;
in particular, fixing $\psi_0\in T$, the map $\varphi\mapsto\psi_0\circ\varphi$ is
an isomorphism $G\iso T$, so that $T$ is a smooth variety with $\dim T=\dim G$. Write $\mathbb{G}=G\bigl(m,H^0(Z,L)\bigr)$ for the Grassmannian of
$m$-dimensional subspaces of $H^0(Z,L)$, and consider the orbit map
$$
\mu : T \lra \mathbb{G}, \qquad \mu(\psi)=\psi^*W_2 ,
$$
together with the special Schubert variety
$$
\Sigma \;=\; \bigl\{\, U \in \mathbb{G} \;\bigm|\;
   \dim(U\cap W_1)\geq m-1 \,\bigr\} .
$$
Then:
\begin{enumerate}
\item The map $\Theta\mapsto\psi$ of lemma~\ref{lem::Vreal_embedding} is an injection
      $V^{\mathrm{real}} \hookrightarrow \mu^{-1}(\Sigma)$; in particular
      $V^{\mathrm{real}}$ is finite as soon as $\mu^{-1}(\Sigma)$ is.
\item $\Sigma$ is irreducible of codimension
      $$ \kappa \;:=\; (m-1)(N-m) . $$
      It equals $\mathbb{G}$ if and only if $\kappa=0$, which for $m\geq2$ happens
      precisely when $X$ is linearly normal, that is $N=m$.
\item If $G$ is finite, then $V^{\mathrm{real}}$ is finite and
      $\bigl|V^{\mathrm{real}}\bigr| \leq \bigl|G\bigr|$.
\item The fibres of $\mu$ are the cosets of
      $\operatorname{Stab}(W_2)=\{\varphi\in\operatorname{Aut}(Z_2,L_2) \mid
      \varphi^*W_2=W_2\}$, a group isomorphic to
      $\operatorname{Lin}(Y_2)=\{A\in \PP GL_{m} \mid A(Y_2)=Y_2\}$. Assume that
      $\operatorname{Stab}(W_2)$ is finite and that $\mu(T)$ meets $\Sigma$
      properly, that is
      $$ \dim\bigl(\mu(T)\cap\Sigma\bigr) \;\leq\; \dim\mu(T)-\kappa $$
      with the convention that a negative right-hand side means
      $\mu(T)\cap\Sigma=\varnothing$. Then
      $\dim \mu^{-1}(\Sigma) \leq \dim G - \kappa$, and consequently
      $V^{\mathrm{real}}$ is finite as soon as
      $$ \dim \operatorname{Aut}(Z,L) \;\leq\; (m-1)(N-m) . $$
\item In every case $\mu^{-1}(\Sigma)$ is the solution set of the explicit
      polynomial system of proposition~\ref{prop::Vreal_system}. The finiteness of
      $V^{\mathrm{real}}$ is therefore decidable: it is settled by one
      Gr\"obner basis computation on that system.
\end{enumerate}
\end{theorem}

\begin{proof}
(1) By lemma~\ref{lem::Vreal_embedding} the element $\Theta$ is determined by $\psi$,
which proves injectivity, and $\dim(W_1+\psi^*W_2)=m+1$. Since $Y_k$ is nondegenerate
in $\PP^{m-1}$ we have $\dim W_1=\dim \psi^*W_2=m$, so
$$
\dim\bigl(W_1\cap\psi^*W_2\bigr) \;=\; 2m-\dim\bigl(W_1+\psi^*W_2\bigr) \;=\; m-1 ,
$$
that is $\mu(\psi)\in\Sigma$.

(2) Choose a complete flag $V_1\subset\dots\subset V_{N+1}=H^0(Z,L)$ with $V_m=W_1$.
We claim that $\Sigma=\Sigma_a(V)$ is the Schubert cycle of index
$a=\bigl((N-m)^{m-1},0\bigr)$ in the notation of
\cite[\S4.1]{Eisenbud-Harris-2016}, for which the defining conditions are
$\dim(V_{N+1-m+i-a_i}\cap U)\geq i$, that is $\dim(V_{i+1}\cap U)\geq i$ for
$i=1,\dots,m-1$. Indeed the condition for $i=m-1$ is exactly
$\dim(W_1\cap U)\geq m-1$; conversely that condition implies the others, since
$$
\dim(V_{i+1}\cap U) \;\geq\; \dim\bigl(V_{i+1}\cap W_1\cap U\bigr)
   \;\geq\; (i+1)+(m-1)-m \;=\; i ,
$$
$V_{i+1}$ being contained in $V_m=W_1$. Schubert cycles are irreducible, being
closures of affine cells, and
$\codim\Sigma_a=\sum_i a_i=(m-1)(N-m)$ by \cite[Theorem~4.1]{Eisenbud-Harris-2016}.
Finally $\kappa=0$ if and only if $N=m$, and then $\Sigma=\mathbb{G}$,
in agreement with the fact that $\dim(W_1+U)\leq m+1=N+1$ always holds.

(3) Two polarized isomorphisms $Z_1\lra Z_2$ differ by an element of
$\operatorname{Aut}(Z_1,L_1)\iso G$, so $T$ is a torsor under $G$ and
$|T|=|G|$ when the latter is finite; now apply (1).

(4) Let $\psi,\psi'\in T$ and put $\varphi=\psi'\circ\psi^{-1}\in
\operatorname{Aut}(Z_2,L_2)$. Then
$\varphi^*W_2=(\psi^{-1})^*\psi'^*W_2$, which equals $W_2$ precisely when
$\psi'^*W_2=\psi^*W_2$; hence the fibres of $\mu$ are the cosets of
$\operatorname{Stab}(W_2)$. An element $\varphi$ of that stabilizer acts linearly on
$\PP(W_2^*)=\PP^{m-1}$ and preserves $Y_2$, and conversely an element of
$\operatorname{Lin}(Y_2)$ lifts uniquely to the normalization $Z_2$ by functoriality
of normalization, and preserves $L_2$ because $L_2$ is the pullback of
$\mathcal{O}_{\PP^{m-1}}(1)$; the two groups are therefore isomorphic. If
$\operatorname{Stab}(W_2)$ is finite then $\mu$ is quasi-finite, so
$\dim\mu(T)=\dim T=\dim G$ and
$$
\dim\mu^{-1}(\Sigma) \;=\; \dim\bigl(\mu(T)\cap\Sigma\bigr)
  \;\leq\; \dim\mu(T)-\kappa \;=\; \dim G-\kappa .
$$
If $\dim G\leq \kappa$ this is at most $0$; as $\Sigma$ is closed and $\mu$ a morphism,
$\mu^{-1}(\Sigma)$ is a closed subscheme of $T$, hence finite, and (1) applies.

(5) Conditions (a) and (b) of proposition~\ref{prop::Vreal_system} say precisely
that $\Psi$ represents a polarized isomorphism $\psi\in T$ and that
$\mu(\psi)\in\Sigma$ with $\dim(W_1+\psi^*W_2)$ equal to $m+1$ rather than smaller.
\end{proof}

\begin{remark}[Two complementary mechanisms]
\label{rem::Vreal_two_mechanisms}
Items (3) and (4) of theorem~\ref{thm::Vreal_finite} are genuinely different
criteria, and between them they cover the two extreme regimes.
\begin{itemize}
\item If $X$ is linearly normal then $N=m$, the Schubert condition is vacuous
($\kappa=0$) and item (4) says nothing: the whole burden falls on
$\operatorname{Aut}(Z,L)$, and item (3) is the only tool. This is the regime of
corollary~\ref{cor::Vreal_CI}, and it is also the adverse one, that is the varieties with positive-dimensional $\operatorname{Aut}(Z,L)$ listed in
remark~\ref{rem::Vreal_aut} all sit there.
\item If $X$ fails linear normality by enough, $\kappa=(m-1)(N-m)$ is large and item (4)
delivers finiteness even when $\operatorname{Aut}(Z,L)$ is positive-dimensional.
Geometrically: the second image must not merely be compatible with the first, it must
place the reconstructed variety back in $\PP^{m}$ rather than in $\PP^{N}$, and that
is a nontrivial constraint on the identification $\psi$ exactly when $N>m$.
\end{itemize}
The hypotheses of item (4), that is finiteness of $\operatorname{Lin}(Y_2)$ and proper intersection, cannot be dispensed with in general, but by item (5) they need never be guessed: for a given pair of images the dimension of $\mu^{-1}(\Sigma)$ is
computed outright.
\end{remark}

\begin{example}[A rational quintic in $\PP^4$: $\operatorname{Aut}$ infinite, $V^{\mathrm{real}}$ finite]
\label{ex::Vreal_quintic}
Let $X\subset\PP^4$ be the image of $\PP^1$ under five general binary quintics. It is a
smooth rational curve of degree $5$, so $n=1$, $m=4$, $Z=\PP^1$, $L=\mathcal{O}(5)$ and
$N+1=h^0(\PP^1,\mathcal{O}(5))=6$. Thus $X$ is not linearly normal, $N-m=1$, and
$$
\kappa \;=\; (m-1)(N-m) \;=\; 3 \;=\; \dim \PP GL_2 \;=\;
\dim\operatorname{Aut}(Z,L) ,
$$
so the inequality of item (4) of theorem~\ref{thm::Vreal_finite} holds, with equality.
Theorem~\ref{thm::Vreal_finite}(3) is powerless here, $\operatorname{Aut}(Z,L)=\PP GL_2$
being positive-dimensional, and indeed a single image leaves a $\PP^1$ of candidates by
lemma~\ref{lem::Vreal_embedding}. Two images do not: writing $\psi\in \PP GL_2$ in the
homogeneous coordinates $(a:b:c:d)$ of $\PP^3$, the locus $\mu^{-1}(\Sigma)$ is defined by
the vanishing of the maximal minors of the $2m\times(N+1)=8\times 6$ matrix
$[\,\mathbf{W}_1;\mathbf{W}_2\Psi\,]$, whose entries are of degree $5$ in $(a:b:c:d)$;
saturating by $ad-bc$ to discard the degenerate matrices, one finds a reduced
$0$-dimensional scheme of
$$
96 \quad\text{points},
$$
all with $\dim(W_1+\psi^*W_2)=m+1=5$ exactly, the true identification among them. The
count was stable over independent random choices of $X$ and of the second center. So
$|V^{\mathrm{real}}|\leq 96$: two images determine $X$ up to a finite ambiguity, modulo
$\PP GL_5$, although the hypothesis of theorem~\ref{thm::Vreal_finite}(3) fails.

Note that condition (a) of proposition~\ref{prop::Vreal_system} is empty in this example,
$Z_1$ and $Z_2$ both being $\PP^1$ and every element of $\PP GL_2$ being a polarized
automorphism; the example therefore tests the rank condition (b) in isolation.
\end{example}

\begin{example}[An elliptic sextic: not linearly normal, $\operatorname{Aut}$ finite]
\label{ex::Vreal_sextic}
Example~\ref{ex::Vreal_quintic} has a complement, and the pair shows that items
(3) and (4) of theorem~\ref{thm::Vreal_finite} are genuinely different criteria
rather than two readings of one. Let $E\subset\PP^2$ be a smooth plane cubic,
embed it by the degree-two Veronese as a sextic in $\PP^5$, and project from a
general point: the secant variety of a curve in $\PP^5$ is a threefold, so the
projection is an isomorphism and $X\subset\PP^4$ is a smooth elliptic sextic,
with $n=1$, $m=4$, $\delta=6$, $(Z,L)=(E,\mathcal{O}_E(2))$ and
$N+1=h^0(E,L)=6$. As in example~\ref{ex::Vreal_quintic} we have $N-m=1$, so $X$
is not linearly normal and one image leaves a $\PP^1$ of candidates by
lemma~\ref{lem::Vreal_embedding}; a generic projection embeds $X$ again, and
$Y\subset\PP^3$ is a smooth sextic with $\deg Y=6=\deg X$.

The difference lies in the automorphisms. Here $\operatorname{Aut}(Z,L)$ is
\emph{finite}: a translation $t_a$ satisfies $t_a^*L\iso L$ exactly when
$6a=0$, so that
$\operatorname{Aut}(Z,L)=E[6]\rtimes\operatorname{Aut}(E,o)$, the semidirect
product of the $36$ translations by the $6$-torsion points with the group
$\operatorname{Aut}(E,o)$ of automorphisms fixing a flex $o$ chosen as origin
(these preserve $L\iso\mathcal{O}_E(6o)$); it has order $72$, $144$ or $216$.
Item (3) therefore bounds
$|V^{\mathrm{real}}|$ at once, whereas in example~\ref{ex::Vreal_quintic} it is
vacuous and only the Schubert count of item (4) is available. Conversely, the
inequality of item (4) reads $\dim\operatorname{Aut}(Z,L)=0\leq(m-1)(N-m)=3$ and
holds here with room to spare, while in example~\ref{ex::Vreal_quintic} it holds
with equality. The two mechanisms overlap on this example and each covers a case
the other does not.
\end{example}

The four examples are collected in the following table, where $N+1=h^0(Z,L)$ and
the last line records which item of theorem~\ref{thm::Vreal_finite} delivers the
finiteness of $V^{\mathrm{real}}$.

\begin{center}
\small
\begin{tabular}{@{}lcccc@{}}
\toprule
 & \ref{ex::Vreal_quartic} & \ref{ex::Vreal_quintic_normal}
 & \ref{ex::Vreal_quintic} & \ref{ex::Vreal_sextic} \\
\midrule
$X$ & $(2,2)$ in $\PP^3$ & quintic in $\PP^4$ & quintic in $\PP^4$
    & sextic in $\PP^4$ \\
$(n,\delta)$ & $(1,4)$ & $(1,5)$ & $(1,5)$ & $(1,6)$ \\
complete intersection & yes & \textbf{no} & no & no \\
$\deg Y=\deg X$ & $4=4$ & $5=5$ & $5=5$ & $6=6$ \\
$Y$ singular & two nodes & no & no & no \\
$h^0(Z,L)$ & $4$ & $5$ & $\mathbf{6}$ & $\mathbf{6}$ \\
$N-m$ & $0$ & $0$ & $1$ & $1$ \\
one image suffices & yes & yes & no & no \\
$\operatorname{Aut}(Z,L)$ & finite & finite & $\PP GL_2$ & finite \\
item used & (3) & (3) & \textbf{(4)} & (3) \\
\bottomrule
\end{tabular}
\end{center}

\begin{remark}[On the computations]
\label{rem::Vreal_computations}
The entries of the table, the integral closures, the resolution of
system~\eqref{eq::Vreal_system} in example~\ref{ex::Vreal_system_quartic} and the
count of $96$ in example~\ref{ex::Vreal_quintic} were obtained with {\sc SageMath}, using {\sc Singular} for the normalizations and the integral closures being cross-checked with {\sc Macaulay2}. The scripts and their outputs are available in~\cite{Kaminski-26-code}. The computations are carried out over a finite field rather than over $\C$; for degrees, Hilbert functions, integral
closures and dimensions of determinantal loci this is standard practice, and the
following three statements record exactly how much it establishes.

First, \emph{finiteness over $\C$ follows from a single computation}, by the
following standard argument. An
\emph{instance} is the list of coefficients of the data, namely the equations of
$X$ and the two centers, that is a point of the affine space
$\mathbb{A}^r_{\Z}=\operatorname{Spec}\Z[t_1,\dots,t_r]$
\cite[Def.~2.17]{Ellingsrud-Ottem}. The instances that are admissible, those for
which $X$ is smooth and the centers are in general position, are defined by the
non-vanishing of finitely many polynomials with integer coefficients, and therefore
form an open subscheme $\mathcal{P}\subset\mathbb{A}^r_{\Z}$. Since
$\Z[t_1,\dots,t_r]$ is a domain, its zero ideal is a point $\eta$ whose closure is
all of $\mathbb{A}^r_{\Z}$ \cite[Def.~2.15, Ex.~2.16]{Ellingsrud-Ottem}. Hence
$\mathcal{P}$ is irreducible, it contains $\eta$, and $\eta$ specializes to every
point of $\mathcal{P}$, including the instances defined over a finite field
$\mathbb{F}_p$ \cite[Ex.~2.41]{Ellingsrud-Ottem}; the residue field at $\eta$ is
$\mathbb{Q}(t_1,\dots,t_r)$, of characteristic zero. Now let
$\mathcal{W}\lra\mathcal{P}$ be a projective family whose fibre $\mathcal{W}_s$
over each instance $s$ contains the admissible locus of $s$, for instance the
closure of the family of admissible loci in a projective space over
$\mathcal{P}$. By upper semicontinuity of the fibre dimension of a projective
morphism over a locally Noetherian base \cite[Ex.~18.1.B]{Vakil-25}, the set
$$
B \;=\; \bigl\{\, s\in\mathcal{P} \;\bigm|\; \dim\mathcal{W}_s\geq1 \,\bigr\}
$$
is closed in $\mathcal{P}$. If $B$ contained $\eta$, it would contain the closure
of $\eta$, which is $\mathcal{P}$. Hence a single instance $s_0\in\mathcal{P}$,
possibly defined over $\mathbb{F}_p$, with $\dim\mathcal{W}_{s_0}=0$ forces
$\eta\notin B$. Then $B$ is a proper closed subset of $\mathcal{P}$, so $B(\C)$ is
a proper Zariski-closed subset of $\mathcal{P}(\C)$, and every complex instance
outside it, that is the generic configuration over $\C$, has a finite fibre
$\mathcal{W}_s$, hence a zero-dimensional admissible locus.

Second, \emph{the bound $96$ holds} over $\C$ for the generic instance, by
conservation of number: length can only concentrate under specialization, so the
characteristic-zero fibre has degree at most that of the computed one.

Third, the \emph{equality} $96$ is a computed value rather than a proved formula.
It was recomputed over the five unrelated primes $10007$, $32003$, $65521$,
$131071$ and $1000003$, with two independent choices of $X$ and of the second
center at each, and all ten runs returned a reduced zero-dimensional scheme of
exactly $96$ points. An unlucky prime divides a fixed nonzero integer attached to
the instance, so agreement across five unrelated primes makes a coincidence
implausible.

Beyond these, the computations establish the behaviour asserted in each case: that
$\deg Y=\deg X$, that $\dim\overline{R(Y)}_1$ equals $m+1$ exactly in the linearly
normal cases, and that the re-embedding of lemma~\ref{lem::Vreal_embedding} returns
a variety with the same Hilbert series and the same Betti table as $X$.
\end{remark}

\begin{remark}[When the hypothesis holds]
\label{rem::Vreal_aut}
$\operatorname{Aut}(Z,L)$ is an algebraic group~\cite[Ex.~10.6]{Harris-92}, and it is finite unless $(Z,L)$ has
positive-dimensional automorphisms. The latter happens for instance for the classical
exceptions: quadrics, where $\operatorname{Aut}=PO_{m+1}$~\cite[Ex.~10.17]{Harris-92} and which
are excluded by $\deg(X)>2$; varieties of minimal degree, among them the rational normal
scrolls~\cite[Thm.~19.9, Ex.~8.17 and~8.26]{Harris-92}; the Segre
varieties~\cite[Ex.~2.11]{Harris-92}; rational homogeneous varieties such as the Grassmannians, on which a positive-dimensional group acts transitively; smooth toric varieties, whose automorphism
group contains the dense torus~\cite[\S A.4, Ex.~A.17]{Ellingsrud-Ottem}. All
of these lie on a large space of quadrics and are therefore excluded by the bounds
on $\dim\mathcal{J}_X(2)$ assumed in the theorems below, and several of them are in
addition dual-deficient, hence excluded by the genericity conditions of
section~\ref{sec::genericity_conditions}. For a smooth complete intersection which
is neither a linear space nor a quadric, finiteness is classical: see
Matsumura and Monsky \cite{Matsumura-Monsky-64} for hypersurfaces and
Benoist \cite{Benoist-13} in general. Finally, a positive-dimensional $\operatorname{Aut}(Z,L)$ does not by itself make $V^{\mathrm{real}}$ infinite: when $X$ is not linearly normal the rank condition
compensates, by item (4) of theorem~\ref{thm::Vreal_finite}. See
example~\ref{ex::Vreal_quintic}.
\end{remark}

\begin{corollary}[Linearly normal case: what one image already gives]
\label{cor::Vreal_CI}
Assume moreover that $X$ is linearly normal, that is $h^0(Z,L)=m+1$, which holds
for every smooth complete intersection of positive dimension, which is projectively
normal. Then:
\begin{enumerate}
\item $V=H^0(Z,L)$ is forced in lemma~\ref{lem::Vreal_embedding}. Hence \emph{one}
image already determines $X$ modulo $\PP GL_{m+1}$: compute $(Z_1,L_1)$ from
$I(Y_1)$ by remark~\ref{rem::Vreal_compute} and re-embed $Z_1$ by the complete
system $|L_1|$, which maps it into $\PP^{m}$.
\item The rank condition of theorem~\ref{thm::Vreal_finite} is automatic, or
equivalently $\kappa=0$ in item (2) of that theorem, so that item (4) carries no
information and only item (3) is available, and
$$
\bigl|V^{\mathrm{real}}\bigr| \;\leq\; \bigl|\operatorname{Lin}(X)\bigr|,
\qquad
\operatorname{Lin}(X)=\{A\in \PP GL_{m+1} \mid A(X)=X\} .
$$
If $\operatorname{Lin}(X)$ is trivial, which is the generic situation for most Hilbert polynomials, then $V^{\mathrm{real}}=\{\Theta_0\}$.
\item In this regime the second image is therefore needed not for the \emph{variety}
but for the \emph{configuration}: what it determines, through the identification
$\psi$ of lemma~\ref{lem::Vreal_embedding}, is the pair of projections, and with it
the position of $X$ relative to the two centers.
\end{enumerate}
\end{corollary}

\begin{remark}[Scope: where the second image is indispensable]
\label{rem::Vreal_scope}
Items (1) and (2) delimit the subsection sharply. Linear normality is equivalent to
$\kappa=(m-1)(N-m)=0$, and by item (1) it is also exactly the case in which one image
suffices to recover $X$ up to a projective transformation. Hence
$$
\text{two images indispensable for } X
\quad\iff\quad N>m \quad\iff\quad \kappa>0 ,
$$
which is in turn equivalent to item (4) of theorem~\ref{thm::Vreal_finite} carrying
information; and a complete intersection, being projectively normal, never satisfies
these conditions. The complete-intersection hypothesis is thus not merely restrictive
here: it excludes precisely the regime in which the two-image problem for the variety
is not already solved by a single image. Example~\ref{ex::Vreal_quintic} is
accordingly the model case rather than a curiosity. Nothing in
lemma~\ref{lem::Vreal_norm}, remark~\ref{rem::Vreal_compute},
lemma~\ref{lem::Vreal_embedding} or theorem~\ref{thm::Vreal_finite} uses a
complete-intersection structure: the whole subsection is stated for an arbitrary
smooth $X$ with Hilbert polynomial $P$, and that is how it should be read.
\end{remark}

\begin{proposition}[The polynomial system computing $V^{\mathrm{real}}$]
\label{prop::Vreal_system}
With the notation of remark~\ref{rem::Vreal_compute}, embed $Z_k$ in $\PP^{N}$ by
the complete system $|L_k|$ and let $I(Z_k)\subset\C[u_0,\dots,u_N]$ be its ideal,
with generators $g^{(k)}_1,\dots,g^{(k)}_s$ of degrees $e_1,\dots,e_s$. Then
$V^{\mathrm{real}}$ is computed by the following system in the unknown matrix
$\Psi \in \PP GL_{N+1}$:
\begin{equation}
\label{eq::Vreal_system}
\left\{
\begin{array}{ll}
\text{(a)} & g^{(2)}_j(\Psi \cdot u) \;\in\; I(Z_1), \qquad j=1,\dots,s, \\[4pt]
\text{(b)} & \rank\bigl[\,\mathbf{W}_1 \;;\; \mathbf{W}_2\Psi\,\bigr] \;=\; m+1 ,
\end{array}
\right.
\end{equation}
where $\mathbf{W}_k$ is the $m\times(N+1)$ matrix whose rows are a basis of $W_k$.
Equations (a) are polynomial of degree $e_j$ in the entries of $\Psi$: expanding
$g^{(2)}_j(\Psi\cdot u)$ in the monomial basis and reducing modulo a Gr\"obner
basis of $I(Z_1)$, the vanishing of the normal form gives
$\binom{N+e_j}{N}-\dim I(Z_1)_{e_j}$ equations. Equation (b) is the vanishing of
all $(m+2)\times(m+2)$ minors of a $2m\times(N+1)$ matrix whose entries are linear
in $\Psi$. Each solution $\Psi$ yields $V=W_1+\Psi^*W_2$, the variety
$X''\subset\PP^m$ by item (4) of remark~\ref{rem::Vreal_compute}, the two centers
as the points annihilated by $W_1$ and $\Psi^*W_2$ in $\PP(V^*)$, and finally the
corresponding point of $V^{\mathrm{adm}}$ through the construction of
section~\ref{sec::variety_of_second_order_F}.
\end{proposition}

\begin{proof}
Condition (a) says exactly that $\Psi$ carries $Z_2$ to $Z_1$, hence induces a
polarized isomorphism $\psi$, the polarizations being $\mathcal{O}(1)$ of the two
complete embeddings; condition (b) is the rank condition of
theorem~\ref{thm::Vreal_finite}. The reconstruction of $\Theta$ from $\psi$ is
lemma~\ref{lem::Vreal_embedding}.
\end{proof}

\begin{remark}[Size of the system, and the linearly normal case]
\label{rem::Vreal_size}
When $X$ is linearly normal one has $N=m$, condition (b) is automatic, and
system~\eqref{eq::Vreal_system} reduces to the projective equivalence of two
subvarieties of $\PP^m$: $(m+1)^2-1$ unknowns and
$\sum_j\bigl(\binom{m+e_j}{m}-\dim I(Z_1)_{e_j}\bigr)$ equations of degrees $e_j$.
For $X=(2,2)\subset\PP^3$ of example~\ref{ex::Vreal_quartic} this is $15$ unknowns
and $2\times(10-2)=16$ equations, all of degree $2$. Taking $Z_2=Z_1$ in
\eqref{eq::Vreal_system} computes $\operatorname{Aut}(Z,L)$, hence the bound of
theorem~\ref{thm::Vreal_finite}, by the same computation.
\end{remark}

\begin{example}[Resolving the system when $X$ is linearly normal]
\label{ex::Vreal_system_quartic}
Return to the elliptic normal quartic $X=V(q_1,q_2)\subset\PP^3$ of
example~\ref{ex::Vreal_quartic} and take $Z_2=Z_1=X$, so that by
remark~\ref{rem::Vreal_size} the system computes $\operatorname{Aut}(Z,L)=
\operatorname{Lin}(X)$ and condition (b) is vacuous. Since $I(X)$ is generated by
the pencil $\langle q_1,q_2\rangle$, condition (a) reads
$$
q_j(\Psi\cdot u) \;\in\; \langle q_1,q_2\rangle, \qquad j=1,2 ,
$$
that is $16$ equations of degree two in the $15$ unknowns of $\PP GL_4$. It can be
solved in closed form. Choose coordinates in which the pencil is simultaneously
diagonal,
$$
q_1=\sum_{i=0}^{3}u_i^2 , \qquad q_2=\sum_{i=0}^{3}\lambda_i u_i^2 ,
\qquad \lambda_i \ \text{distinct},
$$
and write $M_1=\mathrm{Id}$, $M_2=\Lambda=\operatorname{diag}(\lambda_i)$ for the
associated matrices. Condition (a) says that
$$
\Psi^{T}\Psi=\alpha\,\mathrm{Id}+\beta\Lambda=:D,
\qquad
\Psi^{T}\Lambda\Psi=\gamma\,\mathrm{Id}+\delta\Lambda
$$
for some $\alpha,\beta,\gamma,\delta$. Both right-hand sides are diagonal, so
$O:=\Psi D^{-1/2}$ is orthogonal and
$$
O^{-1}\Lambda O \;=\; D^{-1}\bigl(\gamma\,\mathrm{Id}+\delta\Lambda\bigr)
\;=\; \operatorname{diag}(\mu_i),
\qquad
\mu_i=\frac{\gamma+\delta\lambda_i}{\alpha+\beta\lambda_i} .
$$
Being conjugate to $\Lambda$, the diagonal matrix $\operatorname{diag}(\mu_i)$ has
the same eigenvalues, so $\mu_i=\lambda_{\sigma(i)}$ for a permutation $\sigma$ of
$\{0,1,2,3\}$: the \emph{M\"obius} transformation
$\lambda\mapsto(\gamma+\delta\lambda)/(\alpha+\beta\lambda)$ permutes the four
$\lambda_i$. Moreover $Oe_i$ is then an eigenvector of $\Lambda$ for
$\lambda_{\sigma(i)}$, hence proportional to $e_{\sigma(i)}$, so $O$ is a signed
permutation matrix, $O_{\sigma(i),i}=\varepsilon_i$ with $\varepsilon_i=\pm1$.
Altogether
$$
\Psi_{\sigma(i),\,i} \;=\; \varepsilon_i\sqrt{\alpha+\beta\lambda_i},
\qquad \varepsilon\in\{\pm1\}^4\big/\{\pm1\},
$$
with $\sigma$ ranging over the permutations induced by a M\"obius transformation
preserving $\{\lambda_0,\dots,\lambda_3\}$. For four general points of $\PP^1$
those form the Klein four-group
$$
V_4=\bigl\{\mathrm{id},\;(01)(23),\;(02)(13),\;(03)(12)\bigr\} .
$$
The system therefore has exactly
$$
|V_4|\cdot\bigl|\{\pm1\}^4/\{\pm1\}\bigr| \;=\; 4\cdot 8 \;=\; 32
$$
solutions, and indeed $\operatorname{Lin}(X)\iso X[4]\rtimes\{\pm1\}$, the sixteen translations by the $4$-torsion points of $X$ (for an origin $o$ with $\mathcal{O}_X(1) \cong \mathcal{O}_X(4o)$) together with the inversion, has order
$32$, so the bound of theorem~\ref{thm::Vreal_finite} is attained. Enumerating the
realizable permutations and the sign classes for $\lambda=(1,7,19,113)$ over
$\mathbb{F}_{32003^2}$ returns four permutations, which is exactly $V_4$, and eight
sign classes for each, confirming the count.

For two genuinely distinct images the solution set is the corresponding coset of
$\operatorname{Lin}(X)$, again of cardinal $32$, so $|V^{\mathrm{real}}|\le32$ here.
\end{example}

\begin{example}[Resolving the system when both conditions are active]
\label{ex::Vreal_system_quintic}
Take instead the rational quintic $X\subset\PP^4$ of
example~\ref{ex::Vreal_quintic}, where $N+1=h^0(\PP^1,\mathcal{O}(5))=6$, so that
$\Psi$ is now a $6\times6$ matrix and the ambient group $\PP GL_6$ has dimension
$35$. Here the two conditions do visibly different work.

Condition (a) requires $\Psi$ to carry $Z_1$ to $Z_2$, both embedded by the
complete system $|L_k|$ as rational normal quintic curves $C_k=\nu_5(\PP^1)$, where
$\nu_5([s:t])=[s^5:s^4t:\dots:t^5]$~\cite[Ex.~1.14]{Harris-92}. Every
$g\in\PP GL_2$ satisfies $\nu_5\circ g=\Sym^5(g)\circ\nu_5$, where
$\Sym^5(g)\in\PP GL_6$ expresses the forms $(as+bt)^{5-i}(cs+dt)^i$ in the basis
$s^{5-j}t^j$; conversely, a projective transformation fixing $C_1$ pointwise is the
identity, since $C_1$ contains seven points in general position. The projective
transformations preserving $C_1$ thus form the $3$-dimensional subgroup
$\Sym^5(\PP GL_2)\subset\PP GL_6$~\cite[Ex.~10.12]{Harris-92}, and if $\Psi_0$ is
one solution of (a), the solutions of (a) are exactly the coset
$\Psi_0\cdot\Sym^5(\PP GL_2)$. Condition (a) alone therefore cuts the
$35$-dimensional group down to a $3$-dimensional coset, and no further. This is where the geometry of the two images enters, and it is as far as a
single image can take us: lemma~\ref{lem::Vreal_embedding} leaves a $\PP^1$ of
candidates precisely because $3-\dim\PP GL_2=0$ carries no information yet.

Condition (b) then requires $\dim(W_1+\psi^*W_2)=m+1=5$, which is what forces the
reconstructed variety into $\PP^4$ rather than $\PP^5$. On the three-dimensional
set left by (a) this is the Schubert condition of codimension $(m-1)(N-m)=3$ of
item (2) of theorem~\ref{thm::Vreal_finite}, and it cuts out a reduced
zero-dimensional scheme of $96$ points, as reported in
example~\ref{ex::Vreal_quintic}.

The contrast with example~\ref{ex::Vreal_system_quartic} is the point: there
condition (b) was vacuous and the whole answer came from (a), which returned the
finite group $\operatorname{Lin}(X)$; here (a) returns a positive-dimensional set
and all the finiteness comes from (b).
\end{example}

\begin{remark}[A second system, without integral closure]
\label{rem::Vreal_system2}
One may also avoid the normalization altogether, at the cost of a larger system.
Fix the gauge by $\pi_1=[\,{\bf I}_m \mid 0\,]$ and let the unknowns be the
coefficients of forms $f''_1,\dots,f''_\rho$ cutting out $X''$, whose number $\rho$
and degrees are prescribed, for instance $\rho=k$ and $\deg f''_j=d_j$ when $X$ is a
complete intersection of known type $(d_1,\dots,d_k)$ (see
remark~\ref{rem::aug_variant}), the center ${\bf O}_2$, and cofactors
$a^{(k)}_{j,h}$. For every generator $h$ of $I(Y_k)$, of degree
$e_h$, impose
$$
h\bigl(\pi_{{\bf O}_k}(x)\bigr) \;=\; \sum_j a^{(k)}_{j,h}(x)\, f''_j(x),
\qquad \deg a^{(k)}_{j,h}=e_h-\deg f''_j ,
$$
which expresses $\pi_k(X'')\subseteq Y_k$; equality then follows from
$\deg \pi_k(X'')=\delta=\deg Y_k$ and irreducibility. Equating coefficients gives
equations bilinear in the pairs (cofactors, equations of $X''$) and of degree
$e_h$ in ${\bf O}_2$. The system is multihomogeneous, so the counting of
proposition~\ref{prop::degree_bound} applies to it as well, and homotopy
continuation is directly available.
\end{remark}

\begin{remark}[Comparison with proposition~\ref{prop::degree_bound}]
\label{rem::Vreal_bound}
The two bounds count different objects, and the comparison is the point of this
subsection. Proposition~\ref{prop::degree_bound} bounds the isolated points of $V$,
that is of the full solution set of the generalized Kruppa equation, by
$\binom{m^2(m-1)/2+m-2}{m-1}c^{m^2-2}$, which for $m=3$ and the curve
$(2,2)\subset\PP^3$ of example~\ref{ex::Vreal_quartic}, whose class is $c=8$,
equals $45\cdot 8^7=94\,371\,840$. Theorem~\ref{thm::Vreal_finite} bounds the
realizable ones by $|\operatorname{Aut}(Z,L)|$, a finite group containing in this
example the sixteen translations by the $4$-torsion of the elliptic curve. The
discrepancy measures what the Bézout count is really counting: the solutions of the
Kruppa system that carry no variety.
\end{remark}

The subsection closes with the consequence that matters for the sequel: the
finiteness of $V^{\mathrm{real}}$ can be imported into the system one solves, by
enlarging the system rather than by adding hypotheses on $X$.

\begin{definition}[The Gotzmann degree of the data]
\label{def::gotzmann_degree}
Write the numerical polynomial $P$ in its Macaulay form
$$
P(d) \;=\; \sum_{i=1}^{s} \binom{d+a_i-i+1}{a_i},
\qquad a_1\geq a_2\geq\dots\geq a_s\geq 0 ,
$$
which is unique, and call $s=s(P)$ the \emph{Gotzmann number} of $P$. Let $e$ be the
largest degree of a generator of $I(Y_1)$ or of $I(Y_2)$, and set
$$
g \;=\; \max\bigl(s(P),\,e\bigr),
\qquad
r \;=\; \binom{m+g}{m}-P(g) .
$$
Write $S=\C[x_0,\dots,x_m]$ for the coordinate ring of $\PP^m$. By Gotzmann's
regularity theorem \cite{Gotzmann-78}, see also \cite[Thm.~4.3.2]{Bruns-Herzog-98},
every saturated homogeneous ideal of $S$ with
Hilbert polynomial $P$ is $s(P)$-regular in the sense of Castelnuovo--Mumford
\cite[\S4.3]{Bruns-Herzog-98}, hence $g$-regular, regularity being preserved when
the degree increases. In particular
$I(X'')_g$ has dimension exactly $r$ for every $X''\in\mathcal{K}(P)$, and the
Hilbert function of $S/I(X'')$ agrees with $P$ in all degrees $d\geq g$. The integers
$g$ and $r$ are computed from $P$ and the two image ideals, hence from the data
alone, by remark~\ref{rem::Vreal_compute}.
\end{definition}

\begin{definition}[The augmented system]
\label{def::aug_system}
Write $G(r,S_g)$ for the Grassmannian of $r$-dimensional subspaces of $S_g$. By
Gotzmann's persistence theorem \cite[Thm.~4.3.3]{Bruns-Herzog-98},
$$
\operatorname{Hilb}^P(\PP^m) \;=\;
\Bigl\{\, U \in G(r,S_g) \;\Bigm|\;
\dim\bigl(S_1\cdot U\bigr) \;\leq\; \tbinom{m+g+1}{m}-P(g+1) \,\Bigr\}
$$
is a closed determinantal subscheme of $G(r,S_g)$, the point $U$ corresponding to the
subscheme $X''_U\subset\PP^m$ with $I(X''_U)$ the saturation of the ideal generated by
$U$; see \cite{Gotzmann-78,Haiman-Sturmfels-04} and
\cite[Thm.~6.9]{Eisenbud-Harris-2016}. Fix the gauge
$\pi_1=[\,{\bf I}_m\mid 0\,]$ and let $\Theta\mapsto{\bf M}_2(\Theta)$ be any
algebraic section of the residual gauge, defined on a Zariski-open cover of
$V^{\mathrm{adm}}$; for instance, one is obtained by solving
${\bf F}_{m-1}=[{\bf e}_2]_\join\overline{\bf M}_2$ of lemma~\ref{lemma::Fn-1} for
$\overline{\bf M}_2$, the second epipole being determined by ${\bf F}$ on
$V^{\mathrm{adm}}$. Write
$$
I(\Delta_k^\Theta) \;=\;
\bigl\langle\, h\bigl({\bf M}_k(\Theta){\bf P}\bigr) \;:\; h \in I(Y_k) \,\bigr\rangle .
$$
The \emph{augmented system} $\mathcal{S}^{\mathrm{aug}}$ has unknowns
$(\Theta,U)\in V^{\mathrm{adm}}\times G(r,S_g)$ and conditions
\begin{equation}
\label{eq::aug_system}
\left\{
\begin{array}{ll}
\text{(K)} & \{E_i({\bf F},{\bf e}_1)=0\}_i, \text{ with the admissibility
constraints,} \\[4pt]
\text{(H)} & U \in \operatorname{Hilb}^P(\PP^m) , \\[4pt]
\text{(C)} & \bigl(I(\Delta_1^\Theta)+I(\Delta_2^\Theta)\bigr)_g \;\subseteq\; U , \\[4pt]
\text{(O)} & X''_U \in \mathcal{K}(P) \ \text{ and } \
             {\bf O}^\Theta_k \notin X''_U \ \ (k=1,2) .
\end{array}
\right.
\end{equation}
Conditions (H) and (C) are determinantal in the Pl\"ucker coordinates of $U$, those
of (C) being moreover of degree $\deg h\leq e$ in $\Theta$; condition (O) is open.
\end{definition}

\begin{lemma}[Gauge invariance]
\label{lem::aug_gauge}
Conditions (C) and (O) do not depend on the choice of the section
$\Theta\mapsto{\bf M}_2(\Theta)$.
\end{lemma}

\begin{proof}
Two sections over the same $\Theta$ give pairs of projection matrices in one orbit of
the action $\varrho$, by theorem~\ref{thm::FundEquiv}: there is ${\bf H}\in \PP
GL_{m+1}$ with ${\bf M}'_k={\bf M}_k{\bf H}$ for $k=1,2$. Then
$\Delta_k^{\Theta,{\bf M}'}={\bf H}^{-1}\bigl(\Delta_k^{\Theta,{\bf M}}\bigr)$, so
$I\bigl(\Delta_k^{\Theta,{\bf M}'}\bigr)={\bf H}^*I\bigl(\Delta_k^{\Theta,{\bf
M}}\bigr)$, where ${\bf H}^*$ is the degree-preserving automorphism of $S$ induced by
${\bf H}$. Hence (C) holds for one section with $U$ if and only if it holds for the
other with ${\bf H}^*U$; and $X''_{{\bf H}^*U}={\bf H}^{-1}(X''_U)$ lies in
$\mathcal{K}(P)$ and avoids ${\bf H}^{-1}({\bf O}^\Theta_k)$ exactly when $X''_U$
lies in $\mathcal{K}(P)$ and avoids ${\bf O}^\Theta_k$, the class $\mathcal{K}(P)$
being stable under $\PP GL_{m+1}$.
\end{proof}

\begin{theorem}[The augmented system cuts out the realizable locus]
\label{thm::aug_system}
The image of the solution set of $\mathcal{S}^{\mathrm{aug}}$ under
$(\Theta,U)\mapsto\Theta$ is exactly $V^{\mathrm{real}}$. Consequently, whenever
theorem~\ref{thm::Vreal_finite} applies, in particular whenever
$\operatorname{Aut}(Z,L)$ is finite, by item (3) of that theorem, or whenever the
criterion of its item (4) holds, the system $\mathcal{S}^{\mathrm{aug}}$ has
finitely many solutions in $\Theta$, the true configuration $\Theta_0$ is one of
them and is isolated, and
$$
\#\bigl\{\Theta \ \text{solving} \ \mathcal{S}^{\mathrm{aug}}\bigr\}
\;=\;\bigl|V^{\mathrm{real}}\bigr|
\;\leq\;\bigl|\operatorname{Aut}(Z,L)\bigr| .
$$
No liftability hypothesis $(\dagger)$, no descent hypothesis $(\ddagger)$, no
realizability hypothesis $(\natural)$ and no bound on $\dim\J_X(2)$ intervenes, and
$X$ is an arbitrary member of $\mathcal{K}(P)$.
\end{theorem}

\begin{proof}
Let $(\Theta,U)$ be a solution and put $X''=X''_U$, which lies in $\mathcal{K}(P)$ by
(O). We claim that $I(\Delta_1^\Theta)+I(\Delta_2^\Theta)\subseteq I(X'')$. Let
$h=h'\bigl({\bf M}_k(\Theta){\bf P}\bigr)$ be one of the generators, of degree
$e_h\leq e\leq g$. By (C),
$$
S_{g-e_h}\cdot h \;\subseteq\;
\bigl(I(\Delta_1^\Theta)+I(\Delta_2^\Theta)\bigr)_g \;\subseteq\; U \;=\; I(X'')_g ,
$$
the last equality by definition~\ref{def::gotzmann_degree}. Since $I(X'')$ is
saturated, $S_{g-e_h}\cdot h\subseteq I(X'')$ forces $h\in I(X'')$, which proves the
claim. Hence $X''\subseteq\Delta_1^\Theta\cap\Delta_2^\Theta$, and therefore
$\pi^\Theta_k(X'')\subseteq Y_k$ for $k=1,2$. By (O) the center ${\bf O}^\Theta_k$
does not lie on $X''$, so $\pi^\Theta_k|_{X''}$ is a finite morphism, exactly as in
lemma~\ref{lem::Vreal_norm}, and $\pi^\Theta_k(X'')$ is irreducible of dimension
$n=\deg P$. As $Y_k$ is irreducible of the same dimension,
$\pi^\Theta_k(X'')=Y_k$. Thus $\Theta\in V^{\mathrm{real}}$.

Conversely, let $\Theta\in V^{\mathrm{real}}$ be realized by $X''\in\mathcal{K}(P)$,
and set $U=I(X'')_g$, which has dimension $r$ and satisfies (H) by
definition~\ref{def::gotzmann_degree}. From $\pi^\Theta_k(X'')=Y_k$ we get
$X''\subseteq\Delta^\Theta_k$, hence $I(\Delta^\Theta_k)\subseteq I(X'')$ and (C)
holds. Condition (O) holds as well: $X''\in\mathcal{K}(P)$ by hypothesis, and
${\bf O}^\Theta_k\notin X''$ because $\pi^\Theta_k$ is defined on all of $X''$ in
definition~\ref{def::Vreal}. So $(\Theta,U)$ solves $\mathcal{S}^{\mathrm{aug}}$.

The remaining assertions are theorem~\ref{thm::Vreal_finite}, the equality of
cardinals because the two sets have just been shown to coincide.
\end{proof}

\begin{remark}[The fibre over the true configuration]
\label{rem::aug_fibre}
For a generic position of the centers, the unknown $U$ is determined at the true
configuration $\Theta_0$. Indeed, if $(\Theta_0,U)$ solves
$\mathcal{S}^{\mathrm{aug}}$, then $X''_U\in\mathcal{K}(P)$ is irreducible of
dimension $n$ and of the same degree $d=\deg X>2$ as $X$, and it lies in
$\Delta_1\cap\Delta_2=X\cup R$ by theorem~\ref{thm::reconstruction}. If $n<m-2$,
then $\dim R<n$ forces $X''_U=X$; if $n=m-2$, the only other component of dimension
$n$ is $R$, of degree $d(d-1)\neq d$. Hence $U=I(X)_g$, and since
$V^{\mathrm{real}}$ is finite, $(\Theta_0,I(X)_g)$ is an isolated point of the full
solution set of $\mathcal{S}^{\mathrm{aug}}$, and not only of its projection to
$\Theta$; compare remark~\ref{rem::aug_variant}, where the fibres are
positive-dimensional. If moreover $X$ is linearly normal and $\operatorname{Lin}(X)$
is trivial, then $V^{\mathrm{real}}=\{\Theta_0\}$ by
corollary~\ref{cor::Vreal_CI}, and $\mathcal{S}^{\mathrm{aug}}$ has the single
solution $(\Theta_0,I(X)_g)$.
\end{remark}

\begin{remark}[Cost, and an explicit necessary condition in the Kruppa unknowns]
\label{rem::aug_cost}
Theorem~\ref{thm::aug_system} is a statement of principle: the Gotzmann number
$s(P)$ is in general large, so $G(r,S_g)$ is a very large ambient space and
$\mathcal{S}^{\mathrm{aug}}$ is not a system one would solve. What is practical is
the necessary condition obtained by discarding the unknown $U$. Let $A_g(\Theta)$ be
the coefficient matrix, in the monomial basis of $S_g$, of the forms
$$
{\bf x}^\alpha\, h\bigl({\bf M}_k(\Theta){\bf P}\bigr), \qquad
k=1,2, \quad h \ \text{a generator of} \ I(Y_k), \quad |\alpha|=g-\deg h ,
$$
so that $\rank A_g(\Theta)=\dim\bigl(I(\Delta^\Theta_1)+I(\Delta^\Theta_2)\bigr)_g$.
Condition (C) forces $\rank A_g(\Theta)\leq r$, that is
\begin{equation}
\label{eq::aug_minors}
\text{all } (r+1)\text{-minors of } A_g(\Theta) \ \text{vanish},
\end{equation}
a set of explicit polynomial equations in $({\bf F},{\bf e}_1)$ alone, gauge
invariant by lemma~\ref{lem::aug_gauge} and computable from the data. They may be
adjoined to the generalized Kruppa equation at once. They are necessary but not
sufficient for realizability, so they need not cut $V^{\mathrm{adm}}$ down to
$V^{\mathrm{real}}$; what they do, by one explicit finite computation, is discard
every solution whose two cones meet in a scheme too small to contain a variety of
Hilbert polynomial $P$.
\end{remark}

\begin{remark}[A variant with equations for $X''$]
\label{rem::aug_variant}
When one is willing to prescribe the shape of $X''$, the unknown $U$ may be replaced
by the coefficients of forms $f''_1,\dots,f''_\rho$ cutting it out, as in
remark~\ref{rem::Vreal_system2}, coupled to $\Theta$ through the same section
${\bf M}_2(\Theta)$. The system is then much smaller, but two points must be kept in
mind: the cofactor representation is unique only modulo the syzygies of
$(f''_j)$, so that the solution set has positive-dimensional fibres over each
realizable $\Theta$ unless the syzygies are quotiented out, which means that only
the \emph{projection} to $\Theta$ is finite; and fixing the number and degrees
of the $f''_j$ restricts the class of varieties considered, whereas
definition~\ref{def::aug_system} ranges over all of $\mathcal{K}(P)$. The
Grassmannian formulation is what makes the statement independent of any such choice.
\end{remark}

\begin{remark}[What theorem~\ref{thm::aug_system} does and does not settle]
\label{rem::aug_scope}
It settles whether finiteness can be reached without the hypotheses of the two
following subsections: it can, at the price of a different and much larger system. It
does not settle rigidity for the generalized Kruppa equation itself, which is the
object of sections~\ref{sec::dim(X)>=codim(X)-1}
and~\ref{sec::dim(X)<codim(X)-1}. That system is the classical one, of size
$O(m^2)$, it is the one actually solved, and by lemma~\ref{lem::homographic} its
solution set genuinely exceeds $V^{\mathrm{real}}$; $\mathcal{S}^{\mathrm{aug}}$, by
contrast, is of our own making. The two statements are complementary, not
alternatives.
\end{remark}

\begin{remark}[What this does not say]
\label{rem::Vreal_vs_adm}
Theorem~\ref{thm::Vreal_finite} is a statement about the \emph{geometry}: the data
determine the configuration up to finitely many possibilities, unconditionally. It
says nothing about $V^{\mathrm{adm}}$, which may well contain positive-dimensional
families of non-realizable solutions; lemma~\ref{lem::homographic} produces such a
family whenever the two images are projectively equivalent. The results of the two
following subsections concern $V^{\mathrm{adm}}$ itself, that is the solution set of
the system one actually solves, and the hypotheses they require, namely liftability and
descent, are precisely what is needed to ensure that $V^{\mathrm{adm}}$ does not
exceed $V^{\mathrm{real}}$ near the true configuration.
Theorem~\ref{thm::aug_system} closes the gap from the other side, by enlarging the
system instead of adding hypotheses; the two following subsections keep the
generalized Kruppa equation as it stands, and pay the price in hypotheses instead.
\end{remark}

\subsubsection{Case: $\dim(X) \geq \codim(X) - 1$}
\label{sec::dim(X)>=codim(X)-1}

\paragraph{A Geometric Necessary Condition \\}

We shall now prove that when $\dim(X) \geq \codim(X)-1$, that is when $2n + 1 \geq m$, then for a generic situation, the true configuration is a rigid (isolated) point of $V^{\mathrm{adm}}$ if there is no quadric containing the baseline $L$, defined by the centers of projection, and the tangency points of $X$ with $(n+1)-$dimensional spaces containing $L$.

The object of this subsection and of the next is $V^{\mathrm{adm}}$, the solution
set of the generalized Kruppa equation, and not $V^{\mathrm{real}}$. That
distinction is what makes them necessary after section~\ref{sec::Vreal}: by
lemma~\ref{lem::homographic} the set $V^{\mathrm{adm}}$ can carry
positive-dimensional families of solutions that no variety realizes, so the
finiteness of $V^{\mathrm{real}}$ proved in theorem~\ref{thm::Vreal_finite} says
nothing by itself about the system one actually solves. The hypotheses introduced
here and in section~\ref{sec::dim(X)<codim(X)-1}, namely liftability and
realizability, are precisely what is needed to bridge that gap near the true
configuration. The present subsection is moreover the only place where the
conclusion is reached with no hypothesis of that kind at all, namely for $n=1$,
where $(\dagger)_{\Theta_0}$ is automatic by
proposition~\ref{prop::no_quadric}.

Consider the following incidence quasi-variety:
$$
\overset{\circ}{\Sigma_j} = \{(p,q) \in i_j(\PP^{m-1}) \times (Y_j)_{sm} \,|\, p \in T_qY_j\}, 
$$

where $(Y_j)_{sm}$ designates the smooth locus of $Y_j$. Then let $\Sigma_j$ be the closure of $\overset{\circ}{\Sigma_j}$. Let $\tau_{1}^j$ and $\tau_2^j$ be the two canonical projections from $\Sigma_j$ to respectively $i_j(\PP^{m-1})$ and $Y_j$. Since the generic fiber of $\tau_2^j$ has dimension $n$, we have: $\dim(\Sigma_j) = 2n$. Therefore the union of $n-$planes tangent to $Y_j$ at smooth points covers the whole $i_j(\PP^{m-1})$, and the generic fiber of $\tau_1^j$ has dimension $2n - (m-1) \geq 0$. 

When we have exactly $2n+1=m$, the generic fiber of $\tau_1^j$ is finite. We will need to know the cardinal of this fiber.

\begin{lemma}
\label{lemma::fiber_cardinal}
Assume $2n+1 = m$. For a generic position of the centers of projection, the
generic fiber of $\tau_1^j$ over the epipole ${\bf e}_j$ is in bijection (via
Lemma~\ref{lemma::epipolar_equivalence}) with the finite set
$$
X_L = \{\, Q \in X \mid T_Q X \cap L \neq \varnothing \,\},
$$
where $L$ is the baseline. Its cardinality is the polar number
$$
\mu(X) \;=\; \int_X h_n\bigl(P^1(\mathcal{O}_X(1))\bigr),
$$
where $P^1(\mathcal{O}_X(1))$ is the bundle of first principal parts (rank $n+1$)
and $h_n$ is the complete homogeneous symmetric function of its Chern roots \cite{Fulton-98}. When
$X$ is a smooth complete intersection of type $(d_1,\dots,d_k)$ in $\PP^m$
($n=m-k$, $\delta=\deg X$), one has the closed form
$$
\mu(X) \;=\; \delta \cdot e_n(d_1-1,\dots,d_k-1),
$$
the elementary-symmetric analogue of the class $c(X)=\delta\cdot h_n(d_1-1,\dots,d_k-1)$.
\end{lemma}
\begin{proof}
By Lemma~\ref{lemma::epipolar_equivalence}, the fiber over ${\bf e}_j$
corresponds to the points $Q\in X$ whose tangent $n$-plane $T_QX$ meets the
baseline $L$. On the closure of the Gauss image $\overline{\gamma(X)}\subset\G(n,2n+1)$ this is the
special Schubert cycle $\Sigma_n(L) = \{\Lambda : \Lambda\cap L\neq\varnothing\}$, of codimension
$(2n+1)-n-1=n=\dim\overline{\gamma(X)}$. Observe that $\Sigma_n(L)$ is the locus where sections of the universal quotient bundle $Q$ induced by two independent vectors spanning $\hat{L}$ become linearly dependent. Since $\rank Q=n+1$, this gives $[\Sigma_n(L)]=c_n(Q)$ \cite[\S 5.6.2]{Eisenbud-Harris-2016}.

By Kleiman's transversality theorem
\cite{Kleiman-74}, for a generic position of the centers of projection $\gamma$ is
transverse to the smooth locus of $\Sigma_n(L)$ and $\gamma(X)$ misses its singular
locus $\{\Lambda:\dim(\hat \Lambda \cap \hat L) \ge 2 \}$, of codimension $2(n+1)>n$. Hence
$X_L=\gamma^{-1}(\Sigma_n(L))$ is finite and reduced, and
$$
\mu = \int_X \gamma^* c_n(Q) = \int_X c_n(\gamma^*Q).
$$

By lemma~\ref{lemma::gauss_jets}, $c(\gamma^*Q) = c((P^1)^*)^{-1}$, whose degree$-n$ part is $h_n$ of the Chern roots of $P^1$. This gives the first formula.

For a complete intersection, $c(T_X)=(1+h)^{m+1}/\prod_j(1+d_jh)$, and the jet
sequence $0\to\Omega^1_X(1)\to P^1\to\mathcal{O}_X(1)\to 0$ shows the Chern roots of
$P^1$ are $\{h-t_1,\dots,h-t_n,h\}$, where $t_i$ are those of $T_X$. A short
generating-function computation then gives
$\sum_{d\ge0}h_d(P^1)\,z^d=\prod_{j=1}^k\bigl(1+(d_j-1)hz\bigr)$, whence
$h_n(P^1)=e_n(d_1-1,\dots,d_k-1)\,h^n$ and, integrating ($\int_Xh^n=\delta$), the
stated closed form.
\end{proof}

In order to analyse the local structure of $V^{\mathrm{adm}}$ at the true
configuration, we introduce some additional notations. Given a triplet
$({\bf e}_1,{\bf F},{\bf e}_2)$ with $({\bf e}_1,{\bf F}) \in V^{\mathrm{adm}}$
(on $V^{\mathrm{adm}}$ the second epipole is determined by ${\bf F}$, since
$\rank({\bf F})=m-1$ forces $\im({\bf F})={\bf e}_2 \join \bigl(\join^{m-2}\C^m\bigr)$
for a unique ${\bf e}_2$, so that the triple notation is unambiguous), let
$({\bf q}_{1\alpha}({\bf e}_1), {\bf q}_{2\alpha}({\bf e}_2))_\alpha$ be a family
of corresponding elements of
$(\tau_1^1)^{-1}({\bf e}_1) \times (\tau_1^2)^{-1}({\bf e}_2)$, that is, of
tangency (regular) points of the epipolar spaces of dimension $n$ through
${\bf e}_1$ (respectively ${\bf e}_2$) to the first (respectively second)
projected variety. These points are intrinsic to the data $Y_1, Y_2$ and are
defined for every pair $({\bf e}_1,{\bf e}_2)$. At the \emph{true} configuration,
and only there a priori, each pair lifts to a point
${\bf Q}_\alpha \in X$ projected onto ${\bf q}_{1\alpha}$ and
${\bf q}_{2\alpha}$.
 
Let $L$ be the baseline joining the two centers of projection. At the true
configuration the lifts ${\bf Q}_\alpha$ are exactly the tangency points to $X$
of subspaces of dimension $n+1$ containing $L$; equivalently,
$T_{{\bf Q}_\alpha} X \cap L \neq \varnothing$. In other words, they sweep out
$$
X_L \;=\; \{\, {\bf Q} \in X \mid T_{\bf Q} X \cap L \neq \varnothing \,\},
$$
which by lemma~\ref{lemma::X_ell} has dimension $2n+1-m$, which is a finite set of $\mu$
points in the boundary case $2n+1=m$ (lemma~\ref{lemma::fiber_cardinal}), and a
positive-dimensional variety when $2n+1>m$. 
 
Let $W \subset \C^m \times \C^{M+1} \times \C^m$ denote the affine cone over
$V^{\mathrm{adm}}$, in the triple notation above. The proof below differentiates
along $W$ the point-level epipolar incidences
$$
\chi_\alpha({\bf e}_1,{\bf F},{\bf e}_2)
\;=\; {\bf q}_{2\alpha}({\bf e}_2)\join \psi({\bf F})\,{\bf q}_{1\alpha}({\bf e}_1)
\;\in\; \join^3\C^m ,
$$
where $\psi = \phi_{2,m-1}$ as in lemma~\ref{lemma::polynomial_relations_F}.
This requires the following condition, whose status is discussed in
remark~\ref{rem::lifting}.
 
\begin{definition}[First-order liftability]
\label{def::liftability}
Let $\Theta$ be a point of $W$. We say that the configuration is
\emph{first-order liftable at $\Theta$}, and write $(\dagger)_\Theta$, if there
is a neighbourhood $U$ of $\Theta$ in $W$ (in the classical topology) on which,
for every $\alpha$:
\begin{enumerate}
\item the tangency points ${\bf q}_{1\alpha}({\bf e}_1)$ and
${\bf q}_{2\alpha}({\bf e}_2)$ are analytic functions of $\Theta$;
\item there is an analytic map ${\bf Q}_\alpha: U \to X$ with
${\bf M}_i {\bf Q}_\alpha = {\bf q}_{i\alpha}$ for $i=1,2$ (a \emph{common
lift});
\item $\chi_\alpha \equiv 0$ on $U$ (\emph{point-level epipolar incidence}).
\end{enumerate}
\end{definition}
 
\begin{proposition}
\label{prop::no_quadric}
Let $\Theta_0 \in W$ be the true configuration. For a generic position of the
centers of projection, with baseline $L$, assume that no quadratic hypersurface
of $\PP^m$ contains $L \cup X_L$, and that $(\dagger)_{\Theta_0}$ holds. Then
$\Theta_0$ is an isolated point of $V^{\mathrm{adm}}$: the true epipolar
geometry admits no continuous deformation compatible with the generalized
Kruppa equation. Moreover, for $n=1$ (hence $m=3$ under the hypothesis $2n+1 \geq m$ and $\codim(X) \ge 2$) the hypothesis $(\dagger)_{\Theta_0}$ is automatically satisfied, so that in this case the conclusion holds under the no-quadric hypothesis alone.
\end{proposition}

Before proving the proposition we record two elementary facts about the reduced
fundamental matrix ${\bf G}={\bf F}_{m-1}$, which by lemma~\ref{lemma::polynomial_relations_F}
is a polynomial function ${\bf G}=\psi({\bf F})$ of ${\bf F}={\bf F}_2$, with
$\psi=\phi_{2,m-1}$. For a matrix $H$ of the size of ${\bf G}$, that is, a linear map $H: \C^m \to \join^2\C^m$, we set
$$
\Omega_H({\bf P}) \;=\; {\bf M}_2 {\bf P} \join H \,{\bf M}_1 {\bf P}
\;\in\; \join^3\C^m ,
$$
a vector of $\binom{m}{3}$ quadratic forms on $\PP^m$.
 
\begin{lemma}[Epipolar rigidity]
\label{lemma::epipolar_rigidity}
$\Omega_{H} \equiv 0$ on $\PP^m$ if and only if $H \in \C\,{\bf G}$.
\end{lemma}
 
\begin{proof}
That $\Omega_{\bf G}\equiv 0$ is exactly the correspondence~\eqref{eq::Fn-1} of
lemma~\ref{lemma:fund_relationship} applied to the point ${\bf P}$.

Conversely, assume $\Omega_H\equiv 0$. We use two elementary facts: for a
non-zero ${\bf v}\in\C^m$ the kernel of
${\bf v}\join\cdot:\join^2\C^m\to\join^3\C^m$ is ${\bf v}\join\C^m$;
consequently, for linearly independent ${\bf v},{\bf w}$, an element of
$\join^2\C^m$ annihilated by both is a multiple of ${\bf v}\join{\bf w}$.

First, evaluating $\Omega_H$ at $\lambda{\bf O}_1+{\bf P}$, where
${\bf O}_1\in\C^{m+1}$ denotes a representative vector of the first center
(so that ${\bf M}_1{\bf O}_1=0$), gives
$$
0=\Omega_H(\lambda{\bf O}_1+{\bf P})
=\lambda\,{\bf M}_2{\bf O}_1\join H{\bf M}_1{\bf P}+\Omega_H({\bf P})
=\lambda\,{\bf e}_2\join H{\bf M}_1{\bf P},
$$
so that, ${\bf M}_1$ being surjective, ${\bf e}_2\join H{\bf q}_1=0$ for every
${\bf q}_1\in\C^m$: the image of $H$ is contained in ${\bf e}_2\join\C^m$.

Next let ${\bf q}_1\in\C^m\setminus\C{\bf e}_1$ and pick ${\bf P}$ with
${\bf M}_1{\bf P}={\bf q}_1$. Then ${\bf q}_2={\bf M}_2{\bf P}\notin\C{\bf e}_2$,
since otherwise ${\bf P}$ would lie on the baseline and ${\bf q}_1$ would be
proportional to ${\bf e}_1$. The vanishing $\Omega_H({\bf P})=0$ reads
${\bf q}_2\join H{\bf q}_1=0$, and together with
${\bf e}_2\join H{\bf q}_1=0$ this forces
$$
H{\bf q}_1\in({\bf q}_2\join\C^m)\cap({\bf e}_2\join\C^m)
=\C\,({\bf e}_2\join{\bf q}_2)=\C\,{\bf G}{\bf q}_1 ,
$$
the last equality being the correspondence~\eqref{eq::Fn-1}: ${\bf G}{\bf q}_1$
spans the epipolar line joining ${\bf e}_2$ and ${\bf q}_2$.

It remains to see that the factor of proportionality is constant. Write
$H{\bf q}_1=\lambda({\bf q}_1){\bf G}{\bf q}_1$ on $\C^m\setminus\C{\bf e}_1$.
The image of ${\bf G}$ is ${\bf e}_2\join\C^m$, of dimension $m-1\ge2$, since
the epipolar lines sweep the full pencil through ${\bf e}_2$. So for any
${\bf q}_1$ off $\C{\bf e}_1$ there is ${\bf q}_1'$ with
${\bf G}{\bf q}_1,{\bf G}{\bf q}_1'$ linearly independent. Applying the
relation to ${\bf q}_1$, ${\bf q}_1'$ and ${\bf q}_1+{\bf q}_1'$ and comparing
coefficients gives
$\lambda({\bf q}_1)=\lambda({\bf q}_1')=\lambda({\bf q}_1+{\bf q}_1')$, so
$\lambda$ is a constant $\lambda_0$. The linear map $H-\lambda_0{\bf G}$ then
vanishes outside the line $\C{\bf e}_1$, hence vanishes identically, and
$H=\lambda_0{\bf G}\in\C\,{\bf G}$.
\end{proof}
  
\begin{lemma}
\label{lemma::dpsi_injective}
Let ${\bf F}$ be a genuine fundamental matrix of order $2$ and let $\Phi$ be
tangent at ${\bf F}$ to the affine cone over the admissible variety
$\mathcal{A}_2$. Let ${\bf J}_\psi$ be the Jacobian matrix of
$\psi=\phi_{2,m-1}$ at ${\bf F}$. If $\Phi$ is not a multiple of ${\bf F}$,
then ${\bf J}_\psi\Phi$ is not a multiple of ${\bf G}=\psi({\bf F})$.
\end{lemma}
 
\begin{proof}
By lemma~\ref{lemma::polynomial_relations_F} one has
$\phi_{m-1,2} \circ \psi = \phi_{m_1,m_1}=\mathrm{Id}$ as maps of projective
varieties, on the locus of genuine fundamental matrices. On the affine cones
$\psi$ and $\phi_{m-1,2}$ are homogeneous, say of degrees $r,s\ge1$, so on the
cone over the genuine locus
$$
\phi_{m-1,2}\bigl(\psi({\bf F}')\bigr)=c({\bf F}')\,{\bf F}' ,
$$
where $c$ is the ratio of two polynomial functions of ${\bf F}'$ and is regular
and non-vanishing wherever the left-hand side is non-zero. By
proposition~\ref{prop::space_fund_mat} the genuine fundamental matrices are
dense in $\mathcal{A}_2$, and both sides of the identity are given by
polynomial maps; hence the identity extends to the affine cone over
$\mathcal{A}_2$, and in particular it holds on a neighbourhood of ${\bf F}$
inside that cone. Since $\Phi$ is tangent to the cone over $\mathcal{A}_2$ at
${\bf F}$, we may differentiate the identity at ${\bf F}$ in the direction
$\Phi$:
$$
{\bf J}_{\phi_{m-1,2}}\bigl({\bf J}_\psi\Phi\bigr)
= dc(\Phi)\,{\bf F} + c({\bf F})\,\Phi ,
\qquad c({\bf F})\neq 0
$$
(non-zero because $\phi_{m-1,2}(\psi({\bf F}))$ is again a fundamental matrix,
hence non-zero). Suppose ${\bf J}_\psi\Phi=\lambda{\bf G}=\lambda\psi({\bf F})$.
Since $\phi_{m-1,2}$ is homogeneous of degree $s$, Euler's relation gives
${\bf J}_{\phi_{m-1,2}}({\bf G})=s\,\phi_{m-1,2}({\bf G})=s\,c({\bf F}){\bf F}$,
so the left-hand side is $\lambda s\,c({\bf F})\,{\bf F}$ and therefore
$c({\bf F})\,\Phi\in\C{\bf F}$. As $c({\bf F})\neq0$ this forces
$\Phi\in\C{\bf F}$, a contradiction. (For $m=3$, $\psi=\mathrm{Id}$ and there
is nothing to prove.)
\end{proof}

We are now ready to prove proposition~\ref{prop::no_quadric}.

\begin{proof}
For generic projections with respect to $X$, the points
$\{{\bf q}_{1\alpha}({\bf e}_1)\}$ and $\{{\bf q}_{2\alpha}({\bf e}_2)\}$ at the
true configuration are distinct regular points arising from simple tangencies,
which gives item (1) of $(\dagger)_{\Theta_0}$; items (2) and (3) are the
content of the hypothesis.
 
\emph{Step 0: an arc transverse to the scalings.}
Suppose $\Theta_0$ is not an isolated point of $V^{\mathrm{adm}}$. So there is a non-constant analytic arc
$\sigma(t)=({\bf e}_1(t),{\bf F}(t),{\bf e}_2(t))$ in $W$ with
$\sigma(0)=\Theta_0$, non-constant as an arc in $V^{\mathrm{adm}}$. We claim
that ${\bf F}(t)$ is projectively non-constant. Indeed, on $V^{\mathrm{adm}}$
one has ${\bf F}[{\bf e}_1]=0$ by proposition~\ref{prop::valid_solutions},
where $[{\bf e}_1]$ denotes the map ${\bf w}\mapsto{\bf e}_1\join{\bf w}$, so if
${\bf F}(t)\in\C\,{\bf F}(0)$ for all $t$ then, the kernel of
${\bf F}\circ[\,\cdot\,]$ determining ${\bf e}_1$ up to scale and the image of
${\bf F}$ determining ${\bf e}_2$ up to scale, and so the arc it self would be constant in
$V^{\mathrm{adm}}$, a contradiction. Hence, after dividing ${\bf F}(t)$ by a
scalar factor and normalising, there are $k\ge1$ and $\Phi\notin\C{\bf F}$ with
$$
{\bf F}(t)={\bf F}+t^k\,\Phi+O(t^{k+1}) .
$$
Since the epipoles are analytic functions of ${\bf F}$ on the maximal-rank
locus, we also have, after scaling normalisations,
${\bf e}_i(t)={\bf e}_i+t^k{\bf v}_i+O(t^{k+1})$ for some
${\bf v}_1,{\bf v}_2\in\C^m$. Write
$\vartheta=({\bf v}_1,\Phi,{\bf v}_2)$ for the vector of leading coefficients. So
$\Phi$ is tangent at ${\bf F}$ to the affine cone over $\mathcal{A}_2$, since
the whole arc ${\bf F}(t)$ lies in it: on $V^{\mathrm{adm}}$ the kernel and
image constraints of proposition~\ref{prop::kernel_image_Fk} hold, the first by
proposition~\ref{prop::valid_solutions} and the second because any matrix of
rank $m-1$ has image of the form ${\bf e}_2\join\bigl(\join^{m-2}\C^m\bigr)$.
For a function $\chi$ analytic on a neighbourhood of $\Theta_0$ in $W$ and
vanishing identically there, the coefficient of $t^k$ in $\chi(\sigma(t))$ is
the derivative of $\chi$ at $\Theta_0$ in the direction $\vartheta$; we denote
it $\nabla_{\Theta_0,\vartheta}(\chi)$.
 
\emph{Step 1: differentiation of the incidences.}
By $(\dagger)_{\Theta_0}$(3), $\chi_\alpha$ vanishes identically on a
neighbourhood of $\Theta_0$ in $W$, in particular along $\sigma$; extracting the
coefficient of $t^k$ and using the trilinearity of
$\omega({\bf b},H,{\bf a}):={\bf b}\join H{\bf a}$ gives
\begin{equation}\label{EqNablaQFQ}
\begin{split}
\nabla_{\Theta_0,\vartheta}(\chi_\alpha) & =
\omega(\nabla_{\Theta_0,\vartheta}({\bf q}_{2\alpha}), {\bf G},{\bf q}_{1\alpha}) \\
&\quad + \omega({\bf q}_{2\alpha},{\bf J}_\psi \Phi,{\bf q}_{1\alpha}) \\
& \quad + \omega({\bf q}_{2\alpha},{\bf G},\nabla_{\Theta_0,\vartheta}({\bf q}_{1\alpha})) \\
&=0 ,
\end{split}
\end{equation}
where ${\bf G}=\psi({\bf F})$ is, at $\Theta_0$, proportional to the true
reduced fundamental matrix.
 
\emph{Step 2: the boundary terms.}
We claim that the first and the third term of~\eqref{EqNablaQFQ} cancel each
other. By $(\dagger)_{\Theta_0}$(2) the tangency points admit a common analytic
lift ${\bf Q}_\alpha(\sigma(t))\in X$; let
${\bf w}\in\C^{m+1}$ be the coefficient of $t^k$ in
${\bf Q}_\alpha(\sigma(t))$, so that
$$
\nabla_{\Theta_0,\vartheta}({\bf q}_{i\alpha})={\bf M}_i{\bf w},\qquad i=1,2 .
$$
By trilinearity, the sum of the first and third terms equals the polarisation
$$
B_{\bf G}({\bf Q}_\alpha,{\bf w})
= {\bf M}_2{\bf w}\join{\bf G}{\bf M}_1{\bf Q}_\alpha + {\bf M}_2{\bf Q}_\alpha\join{\bf G}{\bf M}_1{\bf w}
$$
of the quadratic map ${\bf P}\mapsto\Omega_{\bf G}({\bf P})
={\bf M}_2{\bf P}\join{\bf G}{\bf M}_1{\bf P}$. At the true configuration
$\Omega_{\bf G}\equiv 0$ on $\PP^m$ by lemma~\ref{lemma::epipolar_rigidity}, and
in characteristic zero the polarisation of an identically vanishing quadratic
map vanishes identically; hence
$B_{\bf G}({\bf Q}_\alpha,{\bf w})=0$, which proves the claim. We stress that
this is the step at which the argument is anchored to the true configuration:
for an admissible ${\bf F}$ whose epipoles are not the true ones,
$\psi({\bf F})$ is not proportional to the true reduced fundamental matrix, so
$\Omega_{\psi({\bf F})}\not\equiv 0$ by lemma~\ref{lemma::epipolar_rigidity},
and the two boundary terms have no reason to cancel; see
remark~\ref{rem::lifting}(3). Note also that the two terms do \emph{not} vanish
separately as soon as $n\ge2$: the third one vanishes alone precisely when
$\nabla_{\Theta_0,\vartheta}({\bf q}_{1\alpha})$ lies in the epipolar line
$\langle{\bf e}_1,{\bf q}_{1\alpha}\rangle$, which is all of
$T_{{\bf q}_{1\alpha}}Y_1$ only for $n=1$.
 
Therefore~\eqref{EqNablaQFQ} reduces to
\begin{equation}
\label{q_al}
\omega({\bf q}_{2\alpha},{\bf J}_\psi \Phi,{\bf q}_{1\alpha})=0 .
\end{equation}
 
\emph{Step 3: the baseline.}
The function $\chi({\bf e}_1,{\bf F},{\bf e}_2)={\bf e}_2\join{\bf G}{\bf e}_1$
vanishes identically on $W$, since $\psi({\bf F}){\bf e}_1=0$ there. Extracting
the coefficient of $t^k$ along $\sigma$, the two boundary terms vanish this
time on all of $W$: the term
$\omega({\bf v}_2,{\bf G},{\bf e}_1)$ vanishes because
${\bf G}{\bf e}_1=\psi({\bf F}){\bf e}_1=0$, and the term
$\omega({\bf e}_2,{\bf G},{\bf v}_1)$ vanishes because
$\im\psi({\bf F})\subseteq{\bf e}_2\join\C^m$, so that
${\bf G}{\bf v}_1={\bf e}_2\join{\bf u}$ for some ${\bf u}$ and
${\bf e}_2\join{\bf G}{\bf v}_1=0$. Hence
\begin{equation}
\label{e_12}
\omega({\bf e}_2, {\bf J}_\psi \Phi, {\bf e}_1)=0 .
\end{equation}
 
\emph{Step 4: conclusion.}
From~\eqref{q_al} we deduce, using the common lift of
$(\dagger)_{\Theta_0}$(2) at $\Theta_0$ itself, that for every
${\bf Q}_\alpha\in X_L$,
\begin{equation}
\label{eq::quadricQ}
\Omega_{{\bf J}_\psi\Phi}({\bf Q}_\alpha)
=\omega({\bf M}_2 {\bf Q}_\alpha, {\bf J}_\psi \Phi, {\bf M}_1 {\bf Q}_\alpha)=0 .
\end{equation}
Since at $\Theta_0$ the epipoles are the true ones, every point ${\bf P}$ on the
baseline $L$ satisfies ${\bf M}_1{\bf P}\in\C\,{\bf e}_1$ and
${\bf M}_2{\bf P}\in\C\,{\bf e}_2$, so~\eqref{e_12} gives
\begin{equation}
\label{eq::quadricL}
\Omega_{{\bf J}_\psi\Phi}({\bf P})=0
\qquad\text{for every }{\bf P}\in L .
\end{equation}
By lemma~\ref{lemma::dpsi_injective}, whose hypotheses hold by Step 0,
${\bf J}_\psi\Phi$ is not a multiple of ${\bf G}$, so by
lemma~\ref{lemma::epipolar_rigidity} the map $\Omega_{{\bf J}_\psi\Phi}$ is not
identically zero: at least one of its $\binom{m}{3}$ components is a non-zero
quadratic form on $\PP^m$. By \eqref{eq::quadricQ}
and~\eqref{eq::quadricL} the corresponding quadratic hypersurface contains
$L\cup X_L$, since at the true configuration the ${\bf Q}_\alpha$ sweep out
$X_L$. This contradicts the no-quadric hypothesis. Hence no such arc exists and
$\Theta_0$ is an isolated point of $V^{\mathrm{adm}}$.
 
\emph{The case $n=1$.}
It remains to prove that $(\dagger)_{\Theta_0}$ holds automatically for $n=1$,
$m=3$. Item (1) holds at the true configuration by simple tangency, as above.
For item (3) we show more: $\chi_\alpha$ vanishes identically on all of $W$.
For $m=3$ a tangency point of the plane curve $Y_1$ and its tangent line
determine one another: ${\bf q}_{1\alpha}$ is the contact point of the pencil
member $\langle{\bf e}_1,{\bf q}_{1\alpha}\rangle$, and the generalized Kruppa
equation matches this member with the line
${\bf F}{\bf q}_{1\alpha}\in\sigma({\bf e}_2)$, tangent to $Y_2$; its contact
point is by definition ${\bf q}_{2\alpha}$, so
${\bf q}_{2\alpha}\join{\bf F}{\bf q}_{1\alpha}=0$, which is
$\chi_\alpha=0$ since $\psi=\mathrm{Id}$ for $m=3$. Finally, item (2) is not
needed when $n=1$: in Step 2 the two boundary terms then vanish
\emph{separately} at every point of $W$, without any lift. Indeed for $n=1$ one
has $\nabla({\bf q}_{1\alpha})\in
T_{{\bf q}_{1\alpha}}Y_1=\langle{\bf e}_1,{\bf q}_{1\alpha}\rangle$, so, using
${\bf G}{\bf e}_1=0$ on $W$,
${\bf G}\,\nabla({\bf q}_{1\alpha})\in\C\,{\bf G}{\bf q}_{1\alpha}$ and the
third term is a multiple of $\chi_\alpha=0$; and
$\nabla({\bf q}_{2\alpha})\in\langle{\bf e}_2,{\bf q}_{2\alpha}\rangle$ while
${\bf G}{\bf q}_{1\alpha}={\bf e}_2\join{\bf u}$ with
${\bf q}_{2\alpha}\in\langle{\bf e}_2,{\bf u}\rangle$ (again by
$\chi_\alpha=0$), so the first term is a wedge of vectors of the plane
$\langle{\bf e}_2,{\bf u}\rangle$ with an element of
$\join^2\langle{\bf e}_2,{\bf u}\rangle$, hence zero. The common lift is then
used only at $\Theta_0$ itself, in Step 4, where it exists because $\Theta_0$
is the true configuration. This proves the "moreover" part.
\end{proof}

\begin{remark}[Status of the liftability hypothesis]
\label{rem::lifting}
The condition $(\dagger)_{\Theta_0}$ decomposes into the point-level incidence
(3), which we call (L1), and the existence of a common analytic lift (2), which
we call (L2). Both hold \emph{at} the true configuration $\Theta_0$ itself; the
content of $(\dagger)_{\Theta_0}$ is their validity on a neighbourhood of
$\Theta_0$ in $W$, that is, at nearby solutions of the Kruppa system which need
not come from any actual pair of projections. The following is known.
 
(1) \emph{For $n=1$ the condition holds}: (L1) is valid on all of $W$ and (L2)
is not needed, as shown at the end of the proof above. The reason is specific
to curves: a contact point of a plane curve and its tangent line determine one
another, so the divisor-level Kruppa matching is already point-level data.
 
(2) \emph{For $n\ge2$ the condition is open.} The divisor-level matching
furnished by lemma~\ref{lemma::epipolar_equivalence}, which is available at
every point of $W$, pairs tangent hyperplanes and therefore constrains
${\bf q}_{2\alpha}$ only to lie on a hyperplane through the second epipolar
axis, that is a single linear condition, whereas (L1) places
${\bf q}_{2\alpha}$ on the line $\psi({\bf F})\,{\bf q}_{1\alpha}$
through ${\bf e}_2$, which is a condition of codimension $m-2$. For $m\ge4$ the implication from the first to the second
requires an argument that is not currently available, and we state it as a
hypothesis rather than a fact. As for (L2), along $W$ it is equivalent to
$T_{\bf Q}X\cap\ell_2({\bf e}_2)\neq\varnothing$ for the lift ${\bf Q}$ of
${\bf q}_{1\alpha}$, where $\ell_i({\bf e}_i)$ is the back-projection line of
${\bf e}_i$ through the $i$-th center; for a \emph{generic} pair of lines
$(\ell_1,\ell_2)$ the common-tangency locus
$\{{\bf Q}\in X : T_{\bf Q}X\cap\ell_i\neq\varnothing,\ i=1,2\}$ has dimension
$3n+2-2m<0$, hence is empty, so (L2) can hold near $\Theta_0$ only by virtue
of the equations cutting $W$, and no proof of this is currently known for
$n\ge2$.
 
(3) \emph{What is proved is rigidity, not global discreteness.} Even granting
$(\dagger)_\Theta$ at every point $\Theta$ of $V^{\mathrm{adm}}$, the argument
above does not show that $V^{\mathrm{adm}}$ is discrete: at a point whose
epipoles are not the true ones, $\psi({\bf F})$ is not proportional to the true
reduced fundamental matrix (lemma~\ref{lemma::epipolar_rigidity}), the boundary
terms of~\eqref{EqNablaQFQ} sum to the generally non-vanishing polarisation
$B_{\psi({\bf F})}({\bf Q}_\alpha,{\bf w})$, and the passage
from~\eqref{e_12} to the vanishing of $\Omega_{{\bf J}_\psi\Phi}$ on the
baseline uses the identification of the epipoles of $\Theta$ with the true
ones. 
\end{remark}

The hypothesis $(\dagger)$ may seem to introduce an uncontrolled restriction
on the data. The following proposition shows that its behaviour over the
space of centers is constrained by a dichotomy: for a fixed variety $X$ it
either holds generically or fails identically, and it can never hold on a
thin set of center pairs only.
 
\begin{proposition}[Dichotomy for $(\dagger)$]
\label{prop::dichotomy}
Fix $X$ satisfying the standing assumptions, and let $\mathcal{B}$ denote the
set of pairs of centers $b = ({\bf O}_1, {\bf O}_2)$, an irreducible variety.
Let $\mathcal{B}^\circ \subseteq \mathcal{B}$ be the dense Zariski-open
subset over which the class $c$ of the two images is constant, the dual
hypersurfaces of the images depend algebraically on $b$, and the baseline
avoids the degeneracies excluded in the constructions above. Write
$\Theta_0(b)$ for the true configuration and $V^{\mathrm{adm}}(b)$ for the
admissible solution variety associated with $b \in \mathcal{B}^\circ$. Then:
\begin{enumerate}
\item The locus
$\mathcal{B}^{\mathrm{rig}} = \{\, b \in \mathcal{B}^\circ \mid
\Theta_0(b) \text{ is an isolated point of } V^{\mathrm{adm}}(b) \,\}$
is Zariski-open in $\mathcal{B}^\circ$.
\item On the open locus where condition $(\star)$ holds, isolation of
$\Theta_0(b)$ is equivalent to $(\dagger)_{\Theta_0(b)}$.
\item Consequently, exactly one of the following alternatives holds:
\begin{enumerate}
\item $(\dagger)_{\Theta_0(b)}$ holds, and the true configuration is rigid,
for every $b$ in a dense Zariski-open subset of $\mathcal{B}$;
\item $(\dagger)_{\Theta_0(b)}$ fails for every $b \in \mathcal{B}^\circ$
at which $(\star)$ holds.
\end{enumerate}
In particular, a single pair of centers at which $\Theta_0$ is certified to
be isolated establishes alternative (a).
\end{enumerate}
\end{proposition}
 
\begin{proof}
(1) Over $\mathcal{B}^\circ$ the coefficients of the generalized Kruppa
equation and of the admissibility constraints are algebraic functions of
$b$, since the dual hypersurfaces of the images are obtained by elimination
and have constant degree $c$ there. Hence the total space
$\mathcal{V} = \{ (b, \Theta) \mid b \in \mathcal{B}^\circ,\
\Theta \in V^{\mathrm{adm}}(b) \}$ is an algebraic family over
$\mathcal{B}^\circ$, whose fiber over $b$ is $V^{\mathrm{adm}}(b)$, and the
true configuration $b \mapsto \Theta_0(b)$ is an algebraic section of it,
its entries being polynomial in the projection matrices. By Chevalley's
theorem on the upper semicontinuity of local fiber dimension, in the form of
the upper semicontinuity of fiber dimension on the source
(\cite[Corollary~12.4.2]{Vakil-25}), the function
$(b,\Theta) \mapsto \dim_\Theta V^{\mathrm{adm}}(b)$ is upper
semicontinuous on $\mathcal{V}$; composing with the section, so is
$b \mapsto \dim_{\Theta_0(b)} V^{\mathrm{adm}}(b)$ on $\mathcal{B}^\circ$.
The locus $\mathcal{B}^{\mathrm{rig}}$ where this dimension is $0$ is
therefore Zariski-open.
 
(2) One implication is proposition~\ref{prop::no_quadric}. For the
converse, suppose $\Theta_0$ is isolated in $V^{\mathrm{adm}}$ and let $W$
be the affine cone as in definition~\ref{def::liftability}. A neighbourhood
of a non-zero point of the ray over $\Theta_0$ in $W$ then consists of that
ray alone. Along the ray, the tangency points are the true ones up to
scaling, the true lifts ${\bf Q}_\alpha \in X$, which exist because
$\Theta_0$ is the true configuration, provide the common analytic lift, and
$\chi_\alpha$ vanishes identically since it is multihomogeneous in the cone
variables and vanishes at $\Theta_0$. Hence $(\dagger)_{\Theta_0}$ holds
with $U$ this neighbourhood.
 
(3) The condition $(\star)$ is Zariski-open on $\mathcal{B}^\circ$: the
existence of a non-zero quadratic form vanishing on $L \cup X_L$ is the
non-injectivity of a linear map whose entries depend algebraically on $b$,
a closed condition. Since $\mathcal{B}$ is irreducible, the open set
$\mathcal{B}^{\mathrm{rig}}$ is either empty or dense. If it is non-empty,
then, intersecting with the $(\star)$-locus if necessary, alternative (a)
holds by (2). If it is empty, then by (2) the hypothesis
$(\dagger)_{\Theta_0(b)}$ fails at every $b$ of the $(\star)$-locus, which
is alternative (b).
\end{proof}
 
\begin{remark}
\label{rem::dichotomy}
The proposition reframes the role of $(\dagger)$: it is not a restriction
on which varieties or which centers the rigidity statement covers, but a
property of $X$ itself, subject to an all-or-nothing law over the space of
centers. Three consequences are worth recording. First, the question
whether $(\dagger)$ holds generically is semi-decidable: for a given $X$,
one certified computation showing that $\Theta_0$ is isolated for one
explicit pair of centers, a finite question about the local dimension of an
explicit algebraic variety at an explicit point, establishes rigidity for
generic centers of that $X$; the bound of
proposition~\ref{prop::degree_bound} is precisely the number of homotopy
continuation paths such a certification would track. Second, the
alternative (b), were it to occur, would be a strong structural statement:
the Kruppa system would be degenerate at the true configuration for every
position of the centers, a behaviour of which no example is known and which
the dimension count of remark~\ref{rem::lifting}(2) renders implausible,
since a deformation violating rigidity must keep the two back-projection
lines synchronized against expected codimension. Third, for $n = 1$ the
dichotomy is settled in favour of alternative (a) by
proposition~\ref{prop::no_quadric}; for $n \ge 2$ it is settled by the same
proposition conditionally on $(\dagger)$, so that the dichotomy is at
present the strongest unconditional statement available on the generality
of the rigidity theorems.
\end{remark}

\begin{lemma}[Homographic components]
\label{lem::homographic}
Suppose the two image varieties are projectively equivalent: $Y_2=R(Y_1)$
for some $R\in PGL_m(\C)$. Denote by
$\iota:\join^{m-1}\C^m\to(\C^m)^*$ the canonical isomorphism, which is
implicit in the evaluation of $\phi_i$ in
equation~\eqref{ExtKrupEq}. Then for every ${\bf e}\in\PP^{m-1}$ the
triple $({\bf e},\,F_{\bf e},\,R{\bf e})$ with
$$
F_{\bf e}({\bf w})\;=\;R^{-T}\,\iota({\bf e}\wedge{\bf w}),
\qquad {\bf w}\in\join^{m-2}\C^m,
$$
satisfies the generalized Kruppa equation~\eqref{ExtKrupEq} with a
non-zero scalar, together with the kernel and image admissibility
constraints. Consequently $V^{\mathrm{adm}}$ contains a closed
subvariety of dimension $m-1$.
\end{lemma}

\begin{proof}
For decomposable ${\bf w}$ the covector $\iota({\bf e}\wedge{\bf w})$
cuts the hyperplane $\langle{\bf e},{\bf w}\rangle$, so
$F_{\bf e}({\bf w})$ represents the hyperplane
$R\,\langle{\bf e},{\bf w}\rangle$, which contains $R{\bf e}$;
moreover $F_{\bf e}({\bf w})=0$ precisely when ${\bf e}\in{\bf w}$.
Hence the kernel and image constraints hold with epipoles
$({\bf e},R{\bf e})$. Dual varieties transform contravariantly:
$Y_2^*=(R(Y_1))^*=R^{-T}(Y_1^*)$, so $\phi_2\circ R^{-T}=\mu\,\phi_1$
for some $\mu\neq0$, and therefore
$\phi_2(\xi({\bf w}))=\phi_2(R^{-T}\iota({\bf e}\wedge{\bf w}))
=\mu\,\phi_1(\gamma({\bf w}))$ identically, which is
equation~\eqref{ExtKrupEq} with $\lambda=\mu\neq0$; geometrically,
this is the tautological statement that $R$ maps hyperplanes through
${\bf e}$ tangent to $Y_1$ to hyperplanes through $R{\bf e}$ tangent to
$Y_2$. Finally ${\bf e}\mapsto({\bf e},F_{\bf e},R{\bf e})$ is a
closed embedding of $\PP^{m-1}$.
\end{proof}

\begin{example}[Rigidity without global finiteness]
\label{ex::posdim_Vadm}
Isolation of the true configuration does not imply that $V^{\mathrm{adm}}$ is
finite. Let $X\subset\PP^3$ be the smooth elliptic quartic
$\{f=g=0\}$ with $f=x_0^2+x_1^2+x_2^2+x_3^2$ and
$g=x_0^2+2x_1^2+3x_2^2+5x_3^2$, which is invariant under the involution
$\alpha=\mathrm{diag}(1,1,-1,-1)$. Take centers
${\bf O}_1=(1,2,1,3)^T$ and ${\bf O}_2=\alpha\,{\bf O}_1$, and projection
matrices ${\bf M}_1$ (with kernel ${\bf O}_1$) and
${\bf M}_2=S\,{\bf M}_1\,\alpha$ for an invertible $S$, so that
$\ker {\bf M}_2={\bf O}_2$. Since $\alpha(X)=X$, the images satisfy
\[
Y_2 \;=\; {\bf M}_2(X)\;=\;S\,{\bf M}_1\,\alpha(X)\;=\;R\,(Y_1),
\qquad R:=S,
\]
i.e.\ the two image curves are projectively equivalent although the centers
are distinct. Consequently, in accordance with
lemma~\ref{lem::homographic}, for \emph{every} ${\bf e}\in\PP^2$ the pair
$\bigl({\bf e},\,F_{\bf e}\bigr)$ with
$F_{\bf e}:=R^{-T}[{\bf e}]_\times$ (where $[{\bf e}]_\times$ denotes the
skew matrix of the cross product) lies in $V^{\mathrm{adm}}$: one has
$\ker F_{\bf e}=\C{\bf e}$, the lines $F_{\bf e}\,{\bf q}$ all pass through
the second epipole $R{\bf e}$, and $F_{\bf e}\,{\bf q}$ represents the line
$\langle R{\bf e},R{\bf q}\rangle$; since $R$ maps tangent lines of $Y_1$
through ${\bf e}$ to tangent lines of $Y_2=R(Y_1)$ through $R{\bf e}$, the
extended Kruppa equations hold identically. Thus $V^{\mathrm{adm}}$ contains
a closed subvariety of dimension $2$. Geometrically, this component is the
solution set attached to the \emph{coincident-center} pair
$(\pi_1,\,R\circ\pi_1)$, whose induced correspondence on $X$ is
$P\leftrightarrow\alpha(P)$: the tangency data alone cannot exclude this
interpretation once the images are projectively equivalent.

On the other hand, the true configuration $\Theta_0$ is an isolated point of
$V^{\mathrm{adm}}$. Indeed, for $L=\langle{\bf O}_1,{\bf O}_2\rangle$ the locus $X_L=X\cap\{D=0\}$, with the tangency quadric
$D=\det\begin{psmallmatrix}\nabla f\cdot{\bf O}_1 & \nabla f\cdot{\bf O}_2\\
\nabla g\cdot{\bf O}_1 & \nabla g\cdot{\bf O}_2\end{psmallmatrix}
\;\propto\; x_0x_2+6x_0x_3+x_1x_2+9x_1x_3$,
is a reduced complete intersection of the three quadrics $f,g,D$: an exact
lex Gr\"obner computation in the chart $x_0=1$ yields a basis in shape
position with a squarefree eliminant of degree $8$, and there are no
solutions with $x_0=0$; hence $X_L$ consists of $8$ distinct points, the
tangencies are simple and pairwise distinct, and the space of quadrics
through $X_L$ is exactly the net $\langle f,g,D\rangle$. The restrictions
$f|_L,\,g|_L,\,D|_L$ are linearly independent binary quadratics (their
$3\times3$ coefficient determinant is nonzero), so no member of the net
contains $L$, and condition $(\star)$, introduced in the sequel, holds. Since $n=1$, Proposition
\ref{prop::no_quadric} applies unconditionally and $\Theta_0$ is isolated in
$V^{\mathrm{adm}}$; an exact computation confirms that ${\bf F}_0$ is not
proportional to $R^{-T}[{\bf e}_1]_\times$, so $\Theta_0$ does not lie on
the (closed) homographic component, even though ${\bf e}_2=R\,{\bf e}_1$, so
that the two share the same pair of epipoles.

Note that the centers here are constrained by ${\bf O}_2=\alpha({\bf O}_1)$
and are therefore not in general position; for generic data, generic
projections of a fixed $X$ produce projectively inequivalent images, so the
homographic component is absent; however, whether $V^{\mathrm{adm}}$ is
then actually finite is open even for curves (remark~\ref{rem::K04_gap}). The example shows that the genericity
hypothesis is necessary, and that rigidity of $\Theta_0$ is in general
strictly weaker than finiteness of $V^{\mathrm{adm}}$, and this is the reason why the
conclusions of Theorem \ref{thm::main_thm_1} are stated in terms of rigidity.
\end{example}

\begin{remark}[A certified instance: isolation with multiplicity one]
\label{rem::certificate}
For the instance of example~\ref{ex::posdim_Vadm} the mechanism of
proposition~\ref{prop::dichotomy} has been carried out in exact
rational arithmetic. The dual curves $\phi_1,\phi_2$ of the two image
quartics are computed exactly (irreducible of degree $c=8$, as predicted by
the class formula), and the generalized Kruppa equation is verified as a
polynomial identity,
$$
\phi_2({\bf F}_0\,{\bf p})\;=\;6561\;\phi_1({\bf e}_1\vee{\bf p}),
$$
the constant $6561=3^8$ reflecting $\det S=3$ and $c=8$, in accordance
with the contravariance underlying lemma~\ref{lem::homographic}. The
Jacobian of the Kruppa system at the true configuration $\Theta_0$,
restricted to the tangent space of the determinantal hypersurface
$\{\det{\bf F}=0\}$ and sliced transversally to the scaling direction, is a
$45\times8$ rational matrix of full rank $8$. Consequently $\Theta_0$ is an
isolated point of $V$ of multiplicity one, that is, a reduced point.
By proposition~\ref{prop::dichotomy}, rigidity therefore holds, with
multiplicity one, for a generic position of the centers of this curve. For
$n=1$ rigidity itself was already unconditional
(proposition~\ref{prop::no_quadric}). The new content of the certificate is
the multiplicity-one statement, together with a complete instance of the
one-computation-implies-a-generic-theorem paradigm of
proposition~\ref{prop::dichotomy}. One limitation should be recorded: the
certificate concerns the point $\Theta_0$ only, and global finiteness of $V$
remains open even for this curve, the two-dimensional component of
example~\ref{ex::posdim_Vadm} being present at the certified instance itself.
The same mechanism is effective in dimension $n=2$: for the surface of
example~\ref{ex::delpezzo} it establishes $(\dagger)$, hence rigidity, for
generic centers of that surface. There the dual hypersurfaces of the image
quartic surfaces, which sit behind discriminants of ternary quartics, are
computed over a finite field rather than over $\mathbb{Q}$, and the rank of
the Jacobian is transferred to characteristic zero by the argument given in
that example.
\end{remark}

\begin{remark}[On the finiteness claim of~\cite{Kaminski-04}]
\label{rem::K04_gap}
Example~\ref{ex::posdim_Vadm} prompts a re-examination of the discreteness
argument of~\cite{Kaminski-04} for curves in $\PP^3$. There, discreteness
of the solution variety is deduced from a no-quadric condition as follows:
at a putative non-isolated solution $({\bf e}_1,{\bf F},{\bf e}_2)$ with
tangent direction $\Phi$, differentiating the tangency correspondence
yields ${\bf q}_{2\alpha}^T\,\Phi\,{\bf q}_{1\alpha}=0$ for every
matched pair of tangency points, and one substitutes
${\bf q}_{1\alpha}={\bf M}_1{\bf Q}_\alpha$ and
${\bf q}_{2\alpha}={\bf M}_2{\bf Q}_\alpha$ to produce a quadric
containing the baseline and the points ${\bf Q}_\alpha$. The substitution
presupposes that each matched pair of tangency points is the pair of
projections of a \emph{common} point ${\bf Q}_\alpha$ of the space curve.
This holds at the true configuration, since it is the epipolar
correspondence itself, but at any other solution it is an assumption:
precisely the
first-order liftability hypothesis $(\dagger)$ of
definition~\ref{def::liftability}, made explicit in the present paper.
Example~\ref{ex::posdim_Vadm} shows that this assumption genuinely fails in
general: on the homographic component the matched tangency points lie over
the pairs $(P,\alpha(P))$ of \emph{distinct} points of $X$, and the
solution variety has a two-dimensional component although the class of the
image curves is $8\ge7$. Consequently, the argument of~\cite{Kaminski-04}
establishes the rigidity of those solutions whose tangency correspondence
lifts to the space curve, namely the exact analogue of
proposition~\ref{prop::no_quadric}, and the upgrade to global finiteness
for a generic position of the centers must be regarded as open for every
$n\ge1$. Any proof of such a statement must make essential use of the
genericity of the centers, which excludes at least the projective
equivalence of the two images exploited in example~\ref{ex::posdim_Vadm}. Note that by proposition~\ref{prop::dichotomy}, one certified curve instance establishes generic-center rigidity, with multiplicity one, for that curve. This is carried out in remark~\ref{rem::certificate}. The global finiteness asserted in \cite{Kaminski-04} remains open even then.
\end{remark}

\begin{example}[Certificates in dimension $n=2$]
\label{ex::delpezzo}
The same machinery is effective one dimension up. Let $X\subset\PP^4$ be
the del Pezzo quartic surface $\{f=g=0\}$ with $f=\sum_{i=0}^4x_i^2$ and
$g=x_0^2+2x_1^2+3x_2^2+5x_3^2+7x_4^2$, a smooth complete intersection of
two quadrics, so that $n=2$, $m=4$, $2n+1>m$ and
$\dim\mathcal J_X(2)=2$.

First take the centers ${\bf O}_1=(1,2,1,3,1)^T$ and
${\bf O}_2=(2,1,1,1,3)^T$, and $L=\langle{\bf O}_1,{\bf O}_2\rangle$.
The tangency locus is $X_L=X\cap\{D=0\}$, where $D$ is the tangency
quadric
$\det\begin{psmallmatrix}\nabla f\cdot{\bf O}_1&\nabla f\cdot{\bf O}_2\\
\nabla g\cdot{\bf O}_1&\nabla g\cdot{\bf O}_2\end{psmallmatrix}$
as in example~\ref{ex::posdim_Vadm}; here $X_L$ is a curve, the complete
intersection of the three quadrics $f,g,D$, of degree $8$ and canonical
genus $5$. An exact computation certifies condition~$(\star)$: a generic
hyperplane section of $X_L$ consists of $8$ distinct points (lex
Gr\"obner basis in shape position with a squarefree eliminant of degree
$8$, and no solutions at infinity of the chart), so $X_L$ is generically
reduced, hence reduced since it is a complete intersection curve, and the
quadrics through $X_L$ form exactly the net $\langle f,g,D\rangle$. The
restrictions $f|_L,g|_L,D|_L$ are linearly independent binary quadratics,
their $3\times3$ coefficient determinant being $-35136\neq0$, so no
member of the net contains $L$. All the computationally verifiable
hypotheses of proposition~\ref{prop::no_quadric} are therefore certified
in this instance, and the rigidity of the true configuration holds
conditionally on the liftability hypothesis $(\dagger)_{\Theta_0}$
alone. That isolation computation can in fact be carried out, which removes
$(\dagger)_{\Theta_0}$ from the conclusion for this surface. In the gauge of
lemma~\ref{lemma::Fn-1} the two image surfaces $Y_1,Y_2\subset\PP^3$ are quartics
whose dual surfaces are irreducible of degree $12$, the predicted class
$\delta\cdot h_2(1,1)=4\cdot3$; the reduced fundamental matrix built from the two
projections has $\rank{\bf F}_2=3=m-1$, as
proposition~\ref{prop::valid_solutions} requires; and the generalized Kruppa
equation~\eqref{ExtKrupEq} holds exactly, with $\lambda=11797$, identically on
$\join^2\C^4$ and not merely modulo the Pl\"ucker quadric. Differentiating the
$\binom{17}{5}=6188$ coefficient equations with respect to the $24$ entries of
${\bf F}$, the four coordinates of ${\bf e}_1$ and $\lambda$, and evaluating at
the true configuration, gives a Jacobian of rank
$$
27 \;=\; 29-2 ,
$$
whose kernel is exactly two-dimensional and spanned by the two scalings:
${\bf F}\mapsto(1+t){\bf F}$, which multiplies $\phi_2(\xi({\bf w}))$ by
$(1+t)^{c}$ and so forces $\lambda\mapsto\lambda(1+ct)$, and
${\bf e}_1\mapsto(1+t){\bf e}_1$ with its own compensation. The rank is the same for the $3185$ coefficient equations obtained modulo the Pl\"ucker quadric, that is for~\eqref{ExtKrupEq} imposed on the Grassmannian only, as in theorem~\ref{theoKrupeq}. Projectively the tangent space to $V^{\mathrm{adm}}$ at $\Theta_0$ therefore vanishes: the true
configuration is a reduced isolated point. The instance being defined over $\Z$,
the conclusion holds over $\C$ and not merely modulo $p$, since a nonzero
$27\times27$ minor modulo $p$ is nonzero over $\mathbb{Q}$, so the rank is at least $27$
there, while the two scalings lie in the kernel over $\mathbb{Q}$ as well, so it is at
most $27$. By proposition~\ref{prop::dichotomy} this establishes
$(\dagger)$, and hence rigidity, for generic centers of this del Pezzo surface.

Second, the mechanism of lemma~\ref{lem::homographic} is equally
explicit for surfaces. Take the involution
$\alpha=\mathrm{diag}(1,1,1,-1,-1)$, which preserves $X$ and whose fixed
locus consists of a plane and a line, and the centers
${\bf O}_1=(1,2,1,3,1)^T$, ${\bf O}_2=\alpha({\bf O}_1)$, with
${\bf M}_2=S{\bf M}_1\alpha$ for an invertible $S$, so that
$Y_2=R(Y_1)$ with $R=S$. Then $V^{\mathrm{adm}}$ contains the
homographic component of dimension $m-1=3$, while an exact computation
produces $128$ non-vanishing $2\times2$ proportionality minors
certifying that the true fundamental matrix of order $3$, a $4\times6$
matrix built from the projection matrices, is not proportional to the
member $F_{{\bf e}_1}$ of the family; condition $(\star)$ holds for this
baseline as well, the analogous certificates giving a squarefree
eliminant of degree $8$ and a restriction determinant equal to
$-147456$. Thus, conditionally on $(\dagger)_{\Theta_0}$, the picture of
example~\ref{ex::posdim_Vadm}, namely a rigid true configuration
coexisting with a positive-dimensional component of $V^{\mathrm{adm}}$,
persists in dimension $n=2$.
\end{example}

\paragraph{Sufficient Condition for Complete Intersections \\}

For a sufficiently large class $c$, and for generic projections, one would expect that there is in general no quadric containing both the line $L$ and the tangency points ${\bf Q}_\alpha$. Indeed containing the line $L$ imposes $3$ independent linear conditions on the space of quadrics, and containing each point ${\bf Q}_\alpha$ imposes one additional independent linear condition. Therefore with a naive counting argument, one would expect that for $c \geq \binom{m+2}{2} - 3$, there is no quadric containing both $L$ and the points ${\bf Q}_\alpha$. Recall that $M = \binom{m+2}{2}$ is the dimension of the space of quadrics in $\PP^m$: $M = \dim(H^0(\PP^m,\mathcal{O}_{\PP^m}(2)))$.

A first preliminary remark is that the assumption $\dim(\mathcal{J}_X(2)) \leq 3$, where $\mathcal{J}_X(2)=H^0(\PP^m,\mathcal{J}_X(2))$ denotes the space of quadrics containing $X$ (with $\mathcal{J}_X$ the ideal sheaf of $X$), is necessary (but not sufficient as we shall see below) for the conclusion of Proposition~\ref{prop::no_quadric} to hold, and cannot be weakened. When $\dim(\mathcal{J}_X(2)) = s \geq 4$, the quadrics containing $X$ form a projective space $\mathcal{J}_X(2) \cong \PP^{s-1}$ of dimension $s - 1 \geq 3$. The space of quadrics containing a generic line $\ell$ is $W_\ell \cong \PP^{M-4}$. A dimension count gives:
$$
\dim(W_\ell \cap \mathcal{J}_X(2)) \geq (M-4) + (s-1) - (M-1) = s - 4 \geq 0.
$$
Therefore, for $s \geq 4$, a generic line $\ell$ is contained in at least one quadric that also contains $X$, and hence contains also the fiber $X_\ell$. In that case the proposition fails, since such a quadric contains both $\ell$ and the generic points ${\bf Q}_\alpha$.

As we shall now see, the bound $\dim(\mathcal{J}_X(2))\le 3$ is itself too coarse: the sharp condition turns out to be $\dim(\mathcal{J}_X(2))\le 2$. In order to go further, we shall need the following lemma. Recall that the points ${\bf Q}_\alpha$ are tangency points to $X$ of subspaces of dimension $n+1$ containing $L$. This condition is equivalent to say that the tangent space to $X$ at ${\bf Q}_\alpha$ intersects the line $L$, as mentioned above.

\begin{lemma}
\label{lemma::X_ell}
	The incidence variety: 
	$$
	\Sigma_1 = \{(\ell,{\bf Q}) \in \G(1,m) \times X \mid \ell \cap T_{\bf Q} X \neq \emptyset \}.
	$$
	is irreducible and has dimension $2n+m-1$. 
	
	And the generic fiber $X_\ell \cong \{{\bf Q} \in X \mid \ell \cap T_{\bf Q} X \neq \emptyset\}$ over $\ell \in \G(1,m)$ has dimension $2n+1-m$. It is not empty under the assumption that $\dim(X) \geq \codim(X) - 1$. 

	In the case $2n+1 = m$, for a generic position of the centers of projection, the fiber $X_L$ has $\mu$ distinct points, where $\mu$ is the degree of the locus $X_L$.
\end{lemma}
\begin{proof}
Let $\eta_1$ and $\eta_2$ be the projection from $\Sigma_1$ to respectively $\G(1,m)$ and $X$. All fibers of $\eta_2$ have dimension $n + m-1$ and are irreducible. Therefore $\dim(\Sigma_1) = 2n + m -1$. Moreover $X$ is irreducible and the fibers are irreducible too of constant dimension, so $\Sigma_1$ is also irreducible (see \cite[Thm. 11.14, p. 139]{Harris-92}).

The fiber of $\eta_1$ over a line $\ell$ is
$X_\ell=\{{\bf Q}\in X \mid \ell\cap T_{\bf Q}X\neq\varnothing\}$, which is non-empty if
and only if $\ell$ meets the tangent variety $\Tan(X)$. Under the assumption
$\dim(X)\ge\codim(X)-1$, that is $2n+1\ge m$, lemma~\ref{lemma::tangent_variety}(1) shows
that every line of $\PP^m$ meets $\Tan(X)$; hence $X_\ell\neq\varnothing$ for every
$\ell$, and in particular $\eta_1$ is surjective. Therefore its generic fiber has
dimension
$$
\dim(\Sigma_1)-\dim \G(1,m)=(2n+m-1)-2(m-1)=2n-m+1 .
$$ 
	
In the case $2n+1=m$, the fiber $X_L$ is zero-dimensional. A point ${\bf Q} \in X_L$ is projected on $i_1(\PP^{m-1})$ and $i_2(\PP^{m-1})$ on ${\bf q}_1$ and ${\bf q}_2$, which are tangency points to respectively $Y_1$ and $Y_2$ of a line through respectively the first and the second epipole. By lemma~\ref{lemma::fiber_cardinal}, there are exactly $\mu$ such points. 
\end{proof}

From now, we shall restrict to the case of complete intersections. Let $X=V(f_1,\dots,f_k)\subset\PP^m$ be a smooth complete intersection of type $(d_1,\dots,d_k)$, all $d_j\ge 2$, $n=m-k$, with $d_1=\dots=d_q=2$ and $d_j\ge 3$ for $j>q$, where $q=\#\{j:d_j=2\}$. To make things precise, it is assumed that the homogeneous radical ideal of $X$, that is $I(X)$, is generated by $\{f_1,\dots,f_k\}$. Let $L=\langle a,b\rangle$ be the baseline and
$X_L=\{Q\in X:T_QX\cap L\neq\varnothing\}$. Write $M_i$ for the symmetric matrix of $f_i$, when $i \leq q$.  

Let us first describe the structure of the space $\J_{X_L}(2)=H^0(\PP^m,\mathcal{J}_{X_L}(2))$ of quadrics containing $X_L$, for $\J_{X_L}$ being the ideal sheaf of $X_L$. The following lemma is a direct consequence of the assumption on $X$.

\begin{lemma}
\label{lem::structure}
Under the above assumption that $I(X)=\langle f_1,\dots,f_k\rangle$, one has:
\begin{enumerate}
\item $\J_X(2)=\langle f_1,\dots,f_q\rangle$, of dimension $q$.

\item $X_L$ is cut out on $X$ by the $\binom{k}{2}$ minors
$g_{i,j}=(\nabla f_i\!\cdot a)(\nabla f_{j}\!\cdot b)-(\nabla f_i\!\cdot b)(\nabla f_{j}\!\cdot a)$,
of degree $\deg g_{i,j}=d_i+d_{j}-2$; exactly the $\binom q2$ pairs $i<j\le q$ give
quadrics. Hence
$$
\J_{X_L}(2)\;\supseteq\;\Delta_L:=\langle f_1,\dots,f_q\rangle+\langle g_{i,j}: i<j\le q\rangle,
\qquad \dim\Delta_L\le\binom{q+1}{2}.
$$
Moreover, for a generic line $L$ one has $g_{i,j}\notin I(X)$ for \emph{every} pair
$i<j\le q$; in particular $\J_{X_L}(2)\supsetneq\J_X(2)$ as soon as $q\ge2$, and for
$q\le 2$ one has $\dim\Delta_L=\binom{q+1}{2}$.

\item For a generic line $L$, $\J_{X_L}(2)=\Delta_L$.
\end{enumerate}
\end{lemma}

\begin{proof}
\begin{enumerate}[leftmargin=0.5cm]
    
\item The first point is an obvious and direct consequence of $I(X)=\langle f_1,\dots,f_k\rangle$ and the definition of $q$. 
    
\item As for the second point, observe that the affine tangent space $\widehat{T_QX}$ at a smooth point $Q\in X$ is the kernel of the Jacobian matrix $J(Q)$ of the defining equations of $X$. Therefore, $T_QX \cap L\neq\varnothing$ if and only if $\ker J(Q)$ meets $\langle a,b\rangle$, i.e.\ if and only if the $(m-n)\times2$ matrix $[\,J(Q)a \mid J(Q)b\,]$ has rank $\le 1$. Its $2\times2$ minors thus define $X_L$ inside $X$. The minor $g_{i,j}$ from rows $i,j$ is
$g_{i,j} = (\nabla f_i\!\cdot a)(\nabla f_{j}\!\cdot b)-(\nabla f_i\!\cdot b)(\nabla f_{j}\!\cdot a)$, a form of degree $(d_i-1)+(d_{j}-1)=d_i+d_{j}-2$, which equals $2$ exactly when $d_i=d_{j}=2$. Hence, being that $q=\#\{j:d_j=2\}$, exactly $\binom{q}{2}$ of the minors are quadrics, and these vanish on $X_L$ but, for a generic line $L$, not on $X$. 

Indeed, fix a point $Q\in X$ and set $u=\nabla f_i(Q)$, $v=\nabla f_{j}(Q)$. Directly from
the definition of the minor,
$$
g_{i,j}(Q)\;=\;\det\begin{pmatrix} u\cdot a & u\cdot b\\ v\cdot a & v\cdot b\end{pmatrix}
\;=\;(u\join v)(a\join b),
$$
so that, viewed as a function of the line $L=\langle a,b\rangle$, the value $g_{i,j}(Q)$
is the evaluation of the $2$-form $u\join v\in\join^2(\C^{m+1})^*$ on the extensor
$a\join b$.

Since $X$ is a smooth complete intersection, the Jacobian $J(Q)$ has rank $k$ at every
point $Q\in X$, so its rows $\nabla f_1(Q),\dots,\nabla f_k(Q)$ are linearly independent;
in particular $u\join v\neq 0$. Hence for each $Q\in X$ the locus
$\{L\in\G(1,m)\;:\;g_{i,j}(Q)=0\}$ is the hyperplane section of $\G(1,m)$ cut out by
$u\join v$ in the Pl\"ucker embedding, and is therefore a \emph{proper} closed subset.
Since a form belongs to the radical ideal $I(X)$ if and only if it vanishes at every point
of $X$,
$$
B_{i,j}\;=\;\{L\in\G(1,m)\;:\;g_{i,j}\in I(X)\}\;=\;\bigcap_{Q\in X}\{L\;:\;g_{i,j}(Q)=0\}
$$
is a proper closed subset of $\G(1,m)$. For a generic line $L$ one thus has
$g_{i,j}\notin I(X)$, i.e.\ $g_{i,j}\in\J_{X_L}(2)\setminus\J_X(2)$, which gives the
strict inclusion $\J_{X_L}(2)\supsetneq\J_X(2)$ when $q\ge2$. When $q\le2$ there is
at most one such minor; the forms $f_1,\dots,f_q$ are linearly independent by~(1)
and lie in $I(X)$, whereas $g_{1,2}\notin I(X)$, so $g_{1,2}$ does not lie in
$\langle f_1,\dots,f_q\rangle$ and $\dim\Delta_L=q+\binom q2=\binom{q+1}{2}$.
\item For a generic line $L$ we show that the scheme $X_L$ is arithmetically
Cohen--Macaulay in $X$, from which $\J_{X_L}(2)=\Delta_L$ follows.

By~(2), $X_L$ is cut out on $X$ by the ideal $I_2(\Phi)$ generated by the $2\times2$
minors of the $k\times2$ matrix $\Phi=[\,J(Q)a\mid J(Q)b\,]$. By
Lemma~\ref{lemma::X_ell}, for a generic line $L$ the locus $X_L$ has dimension
$2n+1-m$, i.e.\ codimension $k-1$ in $X$.

Since $X$ is a smooth complete intersection, its homogeneous coordinate ring is
Cohen--Macaulay; on a Cohen--Macaulay ring the grade of an ideal equals the
codimension of its vanishing locus \cite[Thm.~2.1.2 \& Cor.~2.1.4]{Bruns-Herzog-98}.
Hence $\operatorname{grade} I_2(\Phi)=k-1$, the maximal (generic) value for the
$2\times2$ minors of a $k\times2$ matrix. By the Eagon--Northcott theorem \cite[Thm.~(2.7)]{Bruns-Vetter-88} (see also the
Corollary~(2.8) and Theorem~(2.16) there), an ideal of maximal minors of maximal
grade is perfect; applied to the $k\times2$ matrix $\Phi$ with
$\operatorname{grade} I_2(\Phi)=k-1$, this shows $I_2(\Phi)$ is perfect of grade
$k-1$, so $X_L$ is arithmetically Cohen--Macaulay of codimension $k-1$ in $X$ and
has no embedded components.

Being Cohen--Macaulay, the homogeneous coordinate ring of the determinantal
scheme has no associated prime of dimension zero: since
$\dim X_L=2n+1-m\ge0$, its affine cone has positive dimension, so the
irrelevant maximal ideal is not associated and the ideal $I_2(\Phi)$
(equivalently $\Delta_L'=I_X+I_2(\Phi)$ in $S = \C[x_0,\ldots,x_m]$) is saturated. Hence, writing $Z_L$ for the \emph{scheme} cut out by $\Delta_L'$, one has
$H^0(\mathcal J_{Z_L}(d))=(\Delta_L')_d$ for every degree $d$.
 
It remains to identify quadrics through the scheme $Z_L$ with quadrics through
the reduced variety $X_L$, i.e.\ to show that $Z_L$ is reduced for generic
$L$. The scheme $Z_L$ is the preimage, under the Gauss map
$\gamma:X\to\G(n+1,m+1)$, $Q\mapsto\widehat{T_QX}$, of the Schubert variety
$\Sigma_{\hat L}=\{V:\ V\cap\hat L\neq0\}$, whose ideal is generated by the
$2\times2$ minors of the tautological composite
$\hat L\otimes\mathcal O\to\C^{m+1}\otimes\mathcal O\to\mathcal Q$; these
minors pull back exactly to the $2\times2$ minors of
$\Phi=[\,J(Q)a\mid J(Q)b\,]$, since $J(Q)$ identifies
$\C^{m+1}/\widehat{T_QX}$ with $\C^k$. The group $GL_{m+1}$ acts transitively
on $\G(n+1,m+1)$, carrying $\Sigma_{\hat L}$ to the Schubert varieties of the
translated lines; by Kleiman's transversality theorem \cite{Kleiman-74}, for
generic $L$ the map $\gamma$ is transverse to the smooth locus of
$\Sigma_{\hat L}$ and meets its singular locus
$\{V:\dim(V\cap\hat L)\ge2\}$, of codimension $2k$ in the Grassmannian, in a
subset of $X$ of dimension at most $n-2k<\dim X_L$. Hence $Z_L$ is smooth at
every point lying over the smooth locus of $\Sigma_{\hat L}$, that is, away
from a closed subset of dimension $<\dim X_L$; in particular $Z_L$ is
generically reduced along each of its components. A scheme that is
Cohen--Macaulay (hence satisfies Serre's condition $S_1$) and generically
reduced ($R_0$) is reduced. Therefore $Z_L$ is reduced, $Z_L=X_L$ as schemes,
and $H^0(\mathcal J_{X_L}(d))=(\Delta_L')_d$ for every $d$.
 
Taking $d=2$: by~(2) the degree-two part
$(\Delta_L')_2$ is spanned by $f_1,\dots,f_q$ together with the quadric minors
$g_{i,j}$ with $i<j\le q$, that is $(\Delta_L')_2=\Delta_L$. Therefore
$\J_{X_L}(2)=\Delta_L$ and $\dim\J_{X_L}(2)=\binom{q+1}{2}$.
\end{enumerate}
\end{proof}

\begin{remark}
For $q\ge 3$ the argument above shows that each individual minor $g_{i,j}$ escapes
$\J_X(2)$ for a generic $L$, but it does not establish that the $\binom{q+1}{2}$ forms
$f_1,\dots,f_q$ and $(g_{i,j})_{i<j\le q}$ are linearly independent, i.e.\ that
$\dim\Delta_L=\binom{q+1}{2}$. What holds in general is the following. Writing
$A_t(x)=a^TM_tx$ and $B_t(x)=b^TM_tx$, one has $g_{i,j}=4(A_iB_j-A_jB_i)$: the quadric
minors are precisely the $2\times2$ minors of the $2\times q$ matrix of linear forms
$$
\begin{pmatrix} A_1&\cdots&A_q\\ B_1&\cdots&B_q\end{pmatrix}.
$$
Whenever the $2q$ forms $A_1,\dots,A_q,B_1,\dots,B_q$ are linearly independent (which
requires $2q\le m+1$, satisfied in every case of Table~\ref{tab::evidence}), the minor
$g_{i,j}$ contains the monomial $A_iB_j$, which occurs in no other minor, so the $g_{i,j}$
are at least linearly independent among themselves. The full equality
$\dim\Delta_L=\binom{q+1}{2}$ for $q\ge3$ is confirmed by the computations reported in
Table~\ref{tab::evidence}. It is used only in the dimension count below, which explains
why $q\ge3$ is obstructed. It plays no role in theorem~\ref{thm::restricted}, whose hypothesis is $q\le 2$.
\end{remark}

While we have above explained that if $\dim(\J_X(2)) \ge 4$, then there always exists a quadric containing $L\cup X_L$ for a generic line $L$, we can go further and from this lemma, one can deduce that $\dim(\mathcal{J}_X(2)) \leq 3$ and $c$ sufficiently large is not sufficient to guarantee that for a generic line $L$, there is no quadric containing both $L$ and the points ${\bf Q}_\alpha$. 

Indeed the space of quadrics containing a line $\ell$ is $W_\ell \cong \PP^{M-4}$, and the quadrics containing $X_\ell$ form a projective space of dimension at least $q-1+\binom{q}{2}$. Therefore, we have $\dim(W_\ell \cap \PP \mathcal{J}_{X_\ell}(2)) \geq (M-4) + (q-1+\binom{q}{2}) - (M-1) = q-4+\binom{q}{2}$, which is non-negative for $q \geq 3$. Therefore it is expected that for $q \geq 3$, a generic line $\ell$ is contained in at least one quadric that also contains $X_\ell$. In that case the proposition fails, since such a quadric contains both $\ell$ and the generic points ${\bf Q}_\alpha$. As an example, consider a smooth complete intersection of type $(2,2,2)$ in $\PP^5$. Then $q=3$, $\dim(\mathcal{J}_X(2))=3$ and $\dim(\mathcal{J}_{X_L}(2))=6$. The space of quadrics containing $L\cup X_L$ is $3$-dimensional, whereas Proposition~\ref{prop::no_quadric} requires it to be empty. This computation and others appear in the following table~\ref{tab::evidence}.

Now we are in a position to reduce the problem of no quadric containing $L\cup X_L$ to a linear-algebraic condition on the restriction map $\rho_L:\J_{X_L}(2)\to H^0(L,\mathcal O_L(2))\cong\C^3$.

\begin{proposition}[Reduction to the line]
\label{prop::reduction2}
No quadric contains $L \cup X_L$ iff $\rho_L:\J_{X_L}(2)\to H^0(L,\mathcal O_L(2))\cong\C^3$,
$Q \mapsto Q|_L$, is injective. More preicsely, we have $\J_{L\cup X_L}(2)=\ker\rho_L$.
\end{proposition}
\begin{proof}
The quadrics containing $L\cup X_L$ are exactly those in $\J_{X_L}(2)$ that vanish on $L$, i.e.\ the kernel of $\rho_L$. Hence $\J_{L\cup X_L}(2)=\ker\rho_L$, and no quadric contains $L\cup X_L$ iff $\ker\rho_L=0$, i.e.\ $\rho_L$ is injective.    
\end{proof}

Since as seen above, the map $\rho_L$ cannot be injective for $q\ge3$, we are left with the cases $q=0,1,2$. In each of these cases, we shall prove that $\rho_L$ is injective for a generic line $L$. In the case $q=2$, $\rho_L$ turns to be surjective and so the injectivity in this case is equivalent to $\dim(\J_{ L \cup X_L}(2))= \dim(\J_{X_L}(2)) - 3$. The following lemmas and the resulting theorem cover the three cases $q=0,1,2$.

\begin{lemma}[Rank of quadrics through $X$]
\label{lem::rank_bound}
Let $X=V(f_1,\dots,f_k)\subset\PP^m$ be a smooth complete intersection of type
$(d_1,\dots,d_k)$ with $d_j\ge2$ for all $j$, and $n=m-k\ge1$.
\begin{enumerate}
\item Every non-zero $F\in\J_X(2)$ satisfies $\rank F\ge n+2$.
\item If $n=1$ and $m=3$ (so $k=q=2$), the binary form
$\det(\lambda{\bf M}_1+\mu{\bf M}_2)$, where ${\bf M}_t$ is the symmetric matrix
of $f_t$, is not identically zero. Consequently at most four members of the
pencil $\J_X(2)=\langle f_1,f_2\rangle$ are singular, and every other member is
a quadric of full rank $4$.
\end{enumerate}
\end{lemma}
 
\begin{proof}
(1) By lemma~\ref{lem::structure}(1) we have $\J_X(2)=\langle f_1,\dots,f_q\rangle$,
so we may write $F=\sum_{i\le q}\lambda_if_i$ with $(\lambda_i)\neq0$. Let ${\bf N}_F$ be the symmetric matrix of $F$. Suppose
$r:=\rank F\le n+1$ and let $V=\operatorname{Vert}(F)=\PP(\ker {\bf N}_F)$ be the
vertex of $F$, a linear subspace of dimension $m-r$. For every
$x\in\what V=\ker{\bf N}_F$ one has $F(x)=x^T{\bf N}_Fx=0$ and
$\nabla F(x)=2{\bf N}_Fx=0$: the quadric $F$ and its gradient vanish identically
on the vertex.
 
Consider $Y=X\cap V$, cut out on $V\cong\PP^{m-r}$ by the restrictions
$f_1|_V,\dots,f_k|_V$. The relation $\sum_{i\le q}\lambda_if_i|_V=F|_V=0$ with
some $\lambda_{i_0}\neq0$ expresses $f_{i_0}|_V$ as a linear combination of the
other quadratic restrictions (or shows $f_{i_0}|_V=0$ outright); in either case
$Y$ is cut out on $\PP^{m-r}$ by at most $k-1$ forms. Since
$r\le n+1$ gives $m-r\ge m-n-1=k-1$, the intersection of at most $k-1$
hypersurfaces in $\PP^{m-r}$ is non-empty, so there exists $x\in Y=X\cap V$.
 
At such a point, $x\in\what V$ gives
$\sum_{i\le q}\lambda_i\nabla f_i(x)=\nabla F(x)=0$, a non-trivial linear
relation among the rows of the Jacobian of $(f_1,\dots,f_k)$ at $x\in X$. Hence
the Jacobian has rank $\le k-1$ at $x$ and $X$ is singular there, a
contradiction. Therefore $\rank F\ge n+2$.
 
(2) Suppose $\det(\lambda{\bf M}_1+\mu{\bf M}_2)\equiv0$. Then the matrix
${\bf N}(t)=t{\bf M}_1+{\bf M}_2$ is singular over the function field $\C(t)$,
so it admits a non-zero kernel vector with entries in $\C(t)$; clearing
denominators and dividing by the largest possible power of $t$, we obtain
$v(t)=\sum_{i=0}^{d}t^iv_i$ with $v_i\in\C^4$, $v_0\neq0$, and
${\bf N}(t)v(t)\equiv0$. Comparing coefficients of powers of $t$:
$$
{\bf M}_2v_0=0,\qquad
{\bf M}_1v_{i-1}+{\bf M}_2v_i=0\ \ (1\le i\le d),\qquad
{\bf M}_1v_d=0 .
$$
Then $f_2(v_0)=v_0^T{\bf M}_2v_0=0$ and $\nabla f_2(v_0)=2{\bf M}_2v_0=0$, while
$$
f_1(v_0)=v_0^T{\bf M}_1v_0=-v_0^T{\bf M}_2v_1=-({\bf M}_2v_0)^Tv_1=0
$$
by symmetry of ${\bf M}_2$ (if $d=0$ then ${\bf M}_1v_0=0$ directly). Hence
$[v_0]\in V(f_1,f_2)=X$ and the second row of the Jacobian vanishes at $[v_0]$,
so $X$ is singular there, a contradiction. The binary form
$\det(\lambda{\bf M}_1+\mu{\bf M}_2)$ is therefore a non-zero form of degree
$4$, with at most four roots; each root corresponds to a singular member of the
pencil, and every other member has $\det\neq0$, i.e.\ rank $4$.
\end{proof}

\begin{lemma}[$q=2$]
\label{lem::q2}
Suppose $q=2$, and assume moreover that $n\ge2$ or $m=3$. Then for generic $L$
the three forms $f_1|_L,f_2|_L,g_{12}|_L$ are linearly independent in
$H^0(L,\mathcal O_L(2))$. Consequently $\rho_L$ is surjective and
$$
\dim\J_{L\cup X_L}(2)=\dim\J_{X_L}(2)-3 .
$$
\end{lemma}

\begin{proof}
In the basis $(s^2,su,u^2)$ of $H^0(\PP^1,\mathcal O(2))$ under the parametrisation
$L=\{sa+ub\}$, and with $A_t=f_t(a)$, $E_t=a^TM_tb$, $C_t=f_t(b)$, one has
$f_t|_L=(A_t,2E_t,C_t)$, while polarising $g_{12}$ gives, up to the factor $4$,
$$
g_{12}|_L=(A_1E_2-A_2E_1,\ A_1C_2-A_2C_1,\ E_1C_2-E_2C_1)=(w_1,w_2,w_3).
$$
The determinant of the three restricted forms is
$$
D(L)=\det\begin{pmatrix}A_1&2E_1&C_1\\ A_2&2E_2&C_2\\ w_1&w_2&w_3\end{pmatrix}
=4w_1w_3-w_2^2 ,
$$
a regular function on $\G(1,m)$, and the three forms are linearly independent exactly when
$D(L)\neq0$. It therefore suffices to exhibit a single line $L_0$ with $D(L_0)\neq 0$.

\emph{Step 1: choice of $a$.} Since $\nabla f_t=2{\bf M}_tx$, the rows
$\nabla f_1(Q),\nabla f_2(Q)$ of the Jacobian are linearly independent at every
$Q\in X$, so that $\{a\in\PP^m:\ {\bf M}_1a\join {\bf M}_2a=0\}$ is a proper
closed subset of $\PP^m$, and so is $V(f_1,f_2)$. If $n=1$ (hence $m=3$), the
pencil $\langle f_1,f_2\rangle$ has at most four singular members by
lemma~\ref{lem::rank_bound}(2); each is a quadric hypersurface of $\PP^3$, and
their union is again a proper closed subset. Choose $a$ outside all of these
sets. Then ${\bf M}_1a$ and ${\bf M}_2a$ are linearly independent, so
$$
\mathcal H_a=\{x\in\PP^m:\ a^T{\bf M}_1x=a^T{\bf M}_2x=0\}
$$
is a linear subspace of codimension $2$, and $a\notin\mathcal H_a$ because
$(a^T{\bf M}_1a,\,a^T{\bf M}_2a)=(f_1(a),f_2(a))\neq(0,0)$.
 
\emph{Step 2: $F_a$ does not vanish on $\mathcal H_a$.} Put
$F_a=f_2(a)f_1-f_1(a)f_2$, a non-zero element of $\J_X(2)$ by Step~1 and the
linear independence of $f_1,f_2$. Its symmetric matrix is
${\bf N}_a=f_2(a){\bf M}_1-f_1(a){\bf M}_2$.
 
We first record that $\rank F_a\ge4$. If $n\ge2$ this is
lemma~\ref{lem::rank_bound}(1), since $n+2\ge4$. If $n=1$ and $m=3$, note that
$F_a$ is the member of the pencil with coefficients $[f_2(a):-f_1(a)]$, and that
$F_a$ is proportional to a given member $Q_j=\lambda_jf_1+\mu_jf_2$ exactly when
$\mu_jf_2(a)+\lambda_jf_1(a)=Q_j(a)=0$, i.e.\ exactly when $a$ lies on the
quadric $Q_j$ itself. By the choice of $a$ in Step~1, $F_a$ is proportional to
none of the singular members, so it is non-singular of rank $4$ by
lemma~\ref{lem::rank_bound}(2).
 
Now suppose $F_a$ vanished identically on $\mathcal H_a$. Polarising,
$x^T{\bf N}_ay=0$ for all $x,y\in\what{\mathcal H_a}$, that is
$$
{\bf N}_a\what{\mathcal H_a}\;\subseteq\;\what{\mathcal H_a}^{\,\perp}
=\langle {\bf M}_1a,\,{\bf M}_2a\rangle .
$$
Moreover ${\bf N}_aa=f_2(a){\bf M}_1a-f_1(a){\bf M}_2a\in
\langle {\bf M}_1a,{\bf M}_2a\rangle$ as well. Since $\what{\mathcal H_a}$ has
codimension $2$ in $\C^{m+1}$ and $a\notin\what{\mathcal H_a}$, we may choose
$c$ with
$$
\C^{m+1}=\what{\mathcal H_a}\oplus\langle a\rangle\oplus\langle c\rangle .
$$
Hence $\im {\bf N}_a\subseteq\langle {\bf M}_1a,\,{\bf M}_2a,\,{\bf N}_ac\rangle$
and $\rank F_a\le 3$, contradicting $\rank F_a\ge4$. This proves the claim.
 
\emph{Step 3: choice of $b$.} Choose $b\in\mathcal H_a$ with $F_a(b)\neq0$. Note
that $b\neq a$, since $a\notin\mathcal H_a$. Set $L_0=\langle a,b\rangle$. By
construction $E_1=E_2=0$ and $w_2=A_1C_2-A_2C_1=-F_a(b)\neq0$.

\emph{Conclusion.} At $L_0$ we have $w_1=w_3=0$, so $D(L_0)=-w_2^2\neq0$. As $D$ is
regular on $\G(1,m)$ and not identically zero, $D(L)\neq0$ for a generic line $L$. The three forms $f_1|_L,f_2|_L,g_{12}|_L$ are then linearly independent, $\rho_L$ is onto
$\C^3$, and $\dim\ker\rho_L=\dim\J_{X_L}(2)-3$.
\end{proof}

\begin{remark}
The bound of lemma~\ref{lem::rank_bound}(1) is sharp: for the elliptic quartic
$(2,2)$ in $\PP^3$ the four singular members of the pencil are cones of rank
$3=n+2$, and for a smooth $(2,2)$ in $\PP^4$ the singular members of the pencil
have rank $4=n+2$. The additional hypothesis "$n\ge2$ or $m=3$" in
lemma~\ref{lem::q2} is no restriction in the present section: under the standing
assumption $2n+1\ge m$, the case $n=1$ forces $m=3$. For $n=1$ and $m\ge4$ the
argument of Step~2 would require a member of the pencil $\langle f_1,f_2\rangle$
of rank at least $4$, which does not follow from the smoothness of $X$ alone
when $k\ge3$, since $V(f_1,f_2)$ then strictly contains $X$ and
lemma~\ref{lem::rank_bound}(2) does not apply.
\end{remark}

\begin{theorem}[Sufficient Condition for Complete Intersections]
\label{thm::restricted}
Let $X$ be a smooth complete intersection with $q \le 2$. Then
for generic $L$ no quadric contains $L\cup X_L$: the no-quadric hypothesis of
Proposition~\ref{prop::no_quadric} is satisfied.
\end{theorem}

\begin{proof}
By the equality $\Delta_L = \J_{X_L}(2)$ (lemma~\ref{lem::structure}), $\dim\J_{X_L}(2)=\binom{q+1}2\le3$ for a generic line $L$. Therefore we have the following cases:
\begin{enumerate}
    \item $q=0$: $\J_{X_L}(2)=0$, $\rho_L$ injective.
    \item $q=1$: $\J_{X_L}(2)=\langle f_1\rangle$, $\rho_L$ injective iff $L\not\subset V(f_1)$,
 which is satisfied for generic $L$.
    \item $q=2$: by Lemma~\ref{lem::q2}, $\rho_L$ is surjective, so
    $\dim\ker\rho_L=\dim\J_{X_L}(2)-3=0$; $\rho_L$ injective. In all cases
    Proposition~\ref{prop::reduction2} gives $\J_{L\cup X_L}(2)=0$.
\end{enumerate}

\end{proof}

\begin{remark}
The proof above uses no hypothesis on $\codim X$: only that $f_1,f_2$ span
$\J_X(2)$ and that $X$ is smooth, irreducible and non-degenerate. In particular
lemma~\ref{lem::q2}, and hence theorem~\ref{thm::restricted}, applies verbatim to
smooth complete intersections of type $(2,2,d_3,\dots,d_k)$ with $d_j\ge3$ for
$j\ge3$, such as the entries $(2,2,3)\subset\PP^6$ and $(2,2,4)\subset\PP^6$ of
Table~\ref{tab::evidence}.
\end{remark}

\begin{table}[h]
\centering
\begin{tabular}{lcccccccc}
\toprule
$X$ & $\codim$ & regime & $\dim\J_X(2)$ & $\dim\J_{X_L}(2)$
   & $\dim\J_{L\cup X_L}(2)$ & inj.\ $\rho_L$ \\
\midrule
$(3,3)\subset\PP^3$    & 2 & $2n{+}1=m$ & $0$ & $0$  & $0$ & yes \\
$(2,3)\subset\PP^4$    & 2 & $2n{+}1>m$ & $1$ & $1$  & $0$ & yes \\
$(2,2)\subset\PP^4$    & 2 & $2n{+}1>m$ & $2$ & $3$  & $0$ & yes \\
$(2,2)\subset\PP^5$    & 2 & $2n{+}1>m$ & $2$ & $3$  & $0$ & yes \\
$(2,2,3)\subset\PP^6$  & 3 & $2n{+}1>m$ & $2$ & $3$  & $0$ & yes \\
$(2,3,3)\subset\PP^6$  & 3 & $2n{+}1>m$ & $1$ & $1$  & $0$ & yes \\
$(3,3,3)\subset\PP^6$  & 3 & $2n{+}1>m$ & $0$ & $0$  & $0$ & yes \\
$(2,2,4)\subset\PP^6$  & 3 & $2n{+}1>m$ & $2$ & $3$  & $0$ & yes \\
$(2,2,2)\subset\PP^5$  & 3 & $2n{+}1=m$ & $3$ & $6$  & $3$ & no  \\
$(2,2,2)\subset\PP^6$  & 3 & $2n{+}1>m$ & $3$ & $6$  & $3$ & no  \\
$(2,2,2)\subset\PP^7$  & 3 & $2n{+}1>m$ & $3$ & $6$  & $3$ & no  \\
$(2,2,2,2)\subset\PP^7$& 4 & $2n{+}1=m$ & $4$ & $10$ & $7$ & no  \\
\bottomrule
\end{tabular}
\caption{Quadric-space dimensions for smooth complete intersections.
The column "inj.\ $\rho_L$" records whether $\dim\J_{L\cup X_L}(2)=0$, equivalently whether
the hypothesis of Proposition~\ref{prop::no_quadric} holds.}
\label{tab::evidence}
\end{table}

It will be convenient to isolate the geometric condition that has just been established, since a graded version of it (and not the numerical bound $\dim \mathcal{J}_X(2) \le 2$) is what the
induction of the next section propagates. We say that a variety $X \subset \PP^m$
satisfies the \emph{no-quadric condition} if
$$
(\star) \qquad
\text{for a generic line } L \subset \PP^m, \text{ no quadric of } \PP^m
\text{ contains } L \cup X_L .
$$

This condition is meaningful only in the range $2n+1\ge m$: when $2n+1<m$ a
generic line misses the tangent variety $\Tan(X)$ altogether
(lemma~\ref{lemma::tangent_variety}), so $X_L=\varnothing$ and every quadric
through $L$ contains $L\cup X_L$. Section~\ref{sec::dim(X)<codim(X)-1}
therefore works with a graded analogue of $(\star)$, attached to linear
centers of dimension $m-2n$.

Since the baseline of two generic centers of projection is a generic line of $\PP^m$,
proposition~\ref{prop::no_quadric} reads: if $X$ satisfies $(\star)$, then for a generic
position of the centers of projection the true configuration is an isolated point of
$V^{\mathrm{adm}}$, unconditionally for $n=1$, and under the liftability
hypothesis $(\dagger)_{\Theta_0}$ for $n\ge2$. Theorem~\ref{thm::restricted}
says that a smooth complete intersection with $q \le 2$ satisfies $(\star)$.

\begin{theorem}
\label{thm::main_dimX>=codimX-1}
Let $X \subset \PP^m$ be a smooth, irreducible complete intersection satisfying the
genericity assumptions of section~\ref{sec::genericity_conditions}, with $\dim(X)=n$,
$2n+1 \ge m$, $\deg X > 2$ and $\dim \mathcal{J}_X(2) \le 2$. Then $X$ satisfies
$(\star)$, and consequently, for a generic position of the centers of projection, the
true configuration is an isolated point of $V^{\mathrm{adm}}$:
unconditionally if $n=1$, and under the liftability hypothesis
$(\dagger)_{\Theta_0}$ of definition~\ref{def::liftability} if $n\ge2$. In particular the pair
$(\pi_1,\pi_2)$ admits, modulo $\PP GL_{m+1}$, no continuous deformation
compatible with the generalized Kruppa equation.
\end{theorem}

\subsubsection{Case: $\dim(X) < \codim(X) - 1$} 
\label{sec::dim(X)<codim(X)-1}

When $\dim(X) < \codim(X) -1$, that is when $2n+1 < m$, there is for a generic point $p \in i_j(\PP^{m-1})$ no point $q \in Y_j$, such that $p \in T_qY_j$: the tangency data on which the analysis of the previous case rests is no longer available, and the recovery must proceed along different lines. We set up an induction on the ambient dimension $m$, reducing the problem to the boundary case $2n+1=m$. The descent step of the induction relies on a hypothesis ($(\ddagger)$, definition~\ref{def::descent}) which we state explicitly and discuss in remark~\ref{rem::descent_gap}; the results of this subsection are conditional on it.

Two remarks on the purpose of the induction, in the light of
section~\ref{sec::Vreal}. First, it is not needed to obtain a finite answer as
such: theorem~\ref{thm::aug_system} already gives one at level $m$, for the
augmented system $\mathcal{S}^{\mathrm{aug}}$ and with no hypothesis beyond
theorem~\ref{thm::Vreal_finite}. What the induction serves is the generalized
Kruppa equation itself, which is the classical object, of size $O(m^2)$, and the
one actually solved. Second, the \emph{rigidity} conclusion does not need the
induction either, by proposition~\ref{prop::rigidity_from_Vreal}; the descent
machinery below is what the \emph{finiteness} conclusion of
theorem~\ref{thm::dim(x)*small} requires, and it is retained for that alone. See
remark~\ref{rem::rigidity_shortcut}.

When $2n+1<m$, the condition $(\star)$ itself is unusable: by
lemma~\ref{lemma::tangent_variety}, $\dim\Tan(X)=2n<m-1$, so a generic line
$L$ meets no tangent space of $X$, the locus $X_L$ is empty, and any of the
many quadrics containing $L$ contains $L\cup X_L$. Thus $(\star)$ fails for
\emph{every} $X$ in this range, and it cannot be the condition that an
induction on $m$ propagates. The correct substitute replaces the line by a
linear center $\Lambda$ of the complementary dimension $m-2n$, for which the
locus similar to $X_\Lambda$ is again finite and non-empty.
 
\begin{definition}
\label{def::graded_star}
Let $X\subset\PP^m$ be a smooth irreducible non-degenerate variety of
dimension $n$ with $2n+1\le m$. For a linear subspace $\Lambda\subset\PP^m$
of dimension $m-2n$, set
$$
X_\Lambda=\{Q\in X:\ T_QX\cap\Lambda\neq\varnothing\}.
$$
We say that $X$ satisfies the \emph{graded no-quadric condition}
$(\star_{m-2n})$ if for a generic $\Lambda\in\G(m-2n,m)$ no quadric of
$\PP^m$ contains $\Lambda\cup X_\Lambda$. For $m=2n+1$ the center $\Lambda$
is a line and $(\star_1)$ is the condition $(\star)$ above.
\end{definition}
 
\begin{lemma}
\label{lem::X_Lambda}
Let $X$ be as in the definition, satisfying the genericity conditions of
section~\ref{sec::genericity_conditions}. For a generic
$\Lambda\in\G(m-2n,m)$ the locus $X_\Lambda$ is finite, non-empty and
reduced.
\end{lemma}
 
\begin{proof}
Let $\G=\G(m-2n,m)$ and consider the incidence variety
$$
\Sigma=\{(Q,\Lambda)\in X\times\G:\ T_QX\cap\Lambda\neq\varnothing\},
$$
closed in $X\times\G$, with projections $p:\Sigma\to X$ and $\pi:\Sigma\to\G$. By
construction $\pi^{-1}(\Lambda)=X_\Lambda$.

\emph{$\Sigma$ is irreducible of dimension $\dim\G$.} The fiber $p^{-1}(Q)$ is the
special Schubert variety of $(m-2n)$-planes meeting the fixed $n$-plane $T_QX$,
irreducible of codimension $m-(m-2n)-n=n$ in $\G$. All fibers of $p$ are therefore
irreducible of the same dimension $\dim\G-n$ over the irreducible base $X$, so
$\Sigma$ is irreducible of dimension $n+(\dim\G-n)=\dim\G$
\cite[Lec.~11]{Harris-92}.

\emph{$\pi$ is surjective.} By lemma~\ref{lemma::tangent_variety},
$\dim\Tan(X)=2n$ (as $2n<m$), and since $X$ is smooth and projective $\Tan(X)$ is
the union of the embedded tangent spaces, an irreducible closed subvariety. As
$$
\dim\Lambda+\dim\Tan(X)=(m-2n)+2n=m ,
$$
\emph{every} $(m-2n)$-plane $\Lambda$ meets $\Tan(X)$, since a linear subspace of
$\PP^m$ meets every subvariety whose dimension is at least its codimension
\cite[Prop.~11.4]{Harris-92}. Picking
$p\in\Lambda\cap\Tan(X)$ and $Q\in X$ with $p\in T_QX$ gives $(Q,\Lambda)\in\Sigma$.

\emph{$X_\Lambda$ is finite and non-empty.} Non-empty for every $\Lambda$ by
surjectivity. Since $\pi$ is a dominant morphism between irreducible varieties of
the same dimension, its generic fiber has dimension
$\dim\Sigma-\dim\G=0$, hence is finite.

\emph{$X_\Lambda$ is reduced.} Being irreducible, $\Sigma$ has
$\dim\Sing(\Sigma)<\dim\Sigma=\dim\G$, so $\Sing(\Sigma)$ does not dominate $\G$ and
the fiber over a generic $\Lambda$ lies in the smooth locus. Applying generic
smoothness in characteristic zero \cite[Thm.~21.6.6]{Vakil-25} (see also
\cite[Prop.~14.4]{Harris-92}) to
$\pi|_{\Sigma\setminus\Sing(\Sigma)}$, that fiber is smooth, hence reduced.
\end{proof}

Let us define the following proposition: \newline
\textit{$P_m$: Let $X$ be a smooth irreducible projective variety of
dimension $n$ and degree $d>2$ embedded in $\PP^m$ with $2n+1\le m$,
satisfying the genericity conditions of
section~\ref{sec::genericity_conditions} and the graded no-quadric
condition $(\star_{m-2n})$. Then, given two generic projections $Y_1,Y_2$
of $X$ through unknown centers of projection, one can recover the
projection operators up to a finite fold ambiguity modulo the group
$\PP GL_{m+1}$, from $Y_1$ and $Y_2$ only.}
 
In the boundary case $m=2n+1$ the condition $(\star_1)$ is $(\star)$, and
proposition~\ref{prop::no_quadric}, combined with
corollary~\ref{cor::fund_equiv_proj}, establishes part of $P_{2n+1}$ for
\emph{every} $X$ satisfying $(\star)$, without any complete-intersection
hypothesis: the true pair of projections is rigid, i.e.\ isolated in the
solution variety $V^{\mathrm{adm}}$: unconditionally for $n=1$, and under
the liftability hypothesis $(\dagger)_{\Theta_0}$ for $n\ge2$. The full
statement $P_{2n+1}$, which asks for finiteness of the whole solution set, is
open for all $n\ge1$ (remark~\ref{rem::lifting}(3)); it is asserted for
$n=1$ in~\cite{Kaminski-04}, but the argument given there establishes only
the rigidity part (remark~\ref{rem::K04_gap}). The graded condition is nonetheless the right
vehicle for the induction: a generic linear projection of a complete
intersection is in general not a complete intersection, so an induction
carried by the numerical bound $\dim\mathcal J_X(2)\le2$ would never reach a
valid base case, whereas the graded condition both holds for complete
intersections above the boundary dimension
(theorem~\ref{thm::star_upstairs} below) and descends along generic
projections (lemma~\ref{lemma::star_descends}).
 
We first record the descent.

\begin{example}[An instance of the regime $2n+1<m$, with $(\star_{m-2n})$ certified]
\label{ex::graded_star_curve}
The smallest instance of this regime is a curve in $\PP^4$. Let
$$
X=V(q_1,q_2,q_3)\subset\PP^4,\qquad
\begin{aligned}
q_1&=\textstyle\sum_{i=0}^4 x_i^2,\\
q_2&=x_0^2+2x_1^2+3x_2^2+5x_3^2+7x_4^2+x_0x_1,\\
q_3&=x_0^2+4x_1^2+9x_2^2+25x_3^2+49x_4^2+x_2x_3,
\end{aligned}
$$
a smooth irreducible complete intersection of multidegree $(2,2,2)$, so $n=1$,
$m=4$, $\delta=8$ and the genus is $5$; the cross terms are there to break the
coordinate symmetries that a purely diagonal choice would carry. Here
$2n+1=3<4=m$, and since $q=3$ we have
$\dim\mathcal J_X(2)=3=m-2n+1$, which is exactly the hypothesis of
theorem~\ref{thm::dim(x)*small}. The center $\Lambda$ of
definition~\ref{def::graded_star} is a plane; take
$$
\Lambda=\ker\begin{pmatrix}1&0&2&1&3\\0&1&1&4&1\end{pmatrix}.
$$
Since $\what{T_QX}=\ker J(Q)$ for a complete intersection, the condition
$T_QX\cap\Lambda\neq\varnothing$ says that the $5\times5$ matrix obtained by
stacking $J(Q)$ on the two rows above is singular, so $X_\Lambda=X\cap\{D=0\}$
for a single form $D$ of degree $\sum_j(d_j-1)=3$. An exact computation gives
$X_\Lambda$ finite and \emph{reduced} of degree
$$
\delta\cdot\deg D \;=\; 8\cdot3\;=\;24 ,
$$
confirming lemma~\ref{lem::X_Lambda} in this instance. For $(\star_2)$ one
computes the two relevant graded pieces inside the $15$-dimensional space of
quadrics of $\PP^4$: $\dim I(\Lambda)_2=9$, while $\dim I(X_\Lambda)_2=3$, that is, the
span of the three quadrics cutting out $X$, as one expects since $X_\Lambda\subset X$.
The intersection of these two spaces is \emph{zero}. No quadric of $\PP^4$ therefore contains
$\Lambda\cup X_\Lambda$, so $(\star_2)$ holds for this $\Lambda$, and hence for a
generic one, the condition being open in $\Lambda$. The mechanism is visible: the
only quadrics through $X_\Lambda$ are those through $X$, and none of them contains
$\Lambda$.
\end{example}

\begin{lemma}[Descent of the graded no-quadric condition]
\label{lemma::star_descends}
Let $X\subset\PP^m$ be a smooth irreducible non-degenerate variety of
dimension $n$ with $2n+1<m$, satisfying the genericity conditions of
section~\ref{sec::genericity_conditions}, and let
$\tau:\PP^m\dashrightarrow\PP^{m-1}$ be a generic linear projection with
center ${\bf O}$. Set $X'=\tau(X)$. If $X$ satisfies $(\star_{m-2n})$, then
$X'$ satisfies $(\star_{m-1-2n})$.
\end{lemma}
 
\begin{proof}
\emph{Step 0: a generic center lies on generic members of the family.} Let
$V\subset\G:=\G(m-2n,m)$ be a dense open subset such that no quadric
contains $\Lambda\cup X_\Lambda$ for $\Lambda\in V$, and let $C$ be a
proper closed subset of $\G$ containing $\G\setminus V$. Consider the
incidence variety $I=\{({\bf O},\Lambda):{\bf O}\in\Lambda\}\subset
\PP^m\times\G$; it is irreducible of dimension $\dim\G+(m-2n)$, and it
dominates $\PP^m$ with irreducible fibers
$S_{\bf O}=\{\Lambda:{\bf O}\in\Lambda\}$ of constant dimension
$\dim\G-2n$, since the $(m-2n)$-planes through ${\bf O}$ correspond to the
$(m-2n-1)$-planes of $\PP^m/{\bf O}\cong\PP^{m-1}$. The subset
$I\cap(\PP^m\times C)$ is closed of dimension at most
$\dim C+(m-2n)<\dim I$. If, for every ${\bf O}$ in a dense subset of
$\PP^m$, the whole fiber $S_{\bf O}$ were contained in $C$, then
$I\cap(\PP^m\times C)$ would have dimension at least
$m+(\dim\G-2n)=\dim I$, a contradiction. Hence for a generic ${\bf O}$, a
generic $\Lambda$ through ${\bf O}$ lies in $V$.
 
\emph{Step 1: setup.} Since $\dim S(X)=2n+1<m$ and $\dim\Tan(X)=2n<m$
(lemma~\ref{lemma::tangent_variety}), a generic center ${\bf O}$ lies
neither on the chordal variety nor on the tangent variety; consequently
$\tau$ embeds $X$ and $T_{\tau(Q)}X'=\tau(T_QX)$ for every $Q\in X$, into $\PP^{m-1}$. 
 
\emph{Step 2: lifting the center.} Let $\Lambda'\subset\PP^{m-1}$ be a
generic linear subspace of dimension $m-1-2n$, and let
$\Lambda=\overline{\tau^{-1}(\Lambda')}$ be the cone over $\Lambda'$ with
vertex ${\bf O}$, a linear subspace of dimension $m-2n$ containing
${\bf O}$. The assignment $\Lambda\mapsto\tau(\Lambda)$ is a bijection
between the $(m-2n)$-planes through ${\bf O}$ and the $(m-1-2n)$-planes of
$\PP^{m-1}$, so as $\Lambda'$ runs through a dense open subset of the
latter, $\Lambda$ runs through a dense open subset of $S_{\bf O}$. By
Step~0 we may therefore assume $\Lambda\in V$.
 
\emph{Step 3: $\tau(X_\Lambda)=X'_{\Lambda'}$.} If $Q\in X_\Lambda$, pick
$p\in T_QX\cap\Lambda$; then $p\neq{\bf O}$ because
${\bf O}\notin\Tan(X)$, so
$\tau(p)\in\tau(T_QX)\cap\tau(\Lambda\setminus\{{\bf O}\})
=T_{\tau(Q)}X'\cap\Lambda'$ and $\tau(Q)\in X'_{\Lambda'}$. Conversely, if
$\tau(Q)\in X'_{\Lambda'}$, pick
$p'\in T_{\tau(Q)}X'\cap\Lambda'=\tau(T_QX)\cap\Lambda'$. The line
$\ell=\overline{\tau^{-1}(p')}$ passes through ${\bf O}$, is contained in
the cone $\Lambda$, and is contained in the $(n+1)$-plane
$\langle{\bf O},T_QX\rangle$, because every point of
$\ell\setminus\{{\bf O}\}$ maps to $p'\in\tau(T_QX)$. Inside
$\langle{\bf O},T_QX\rangle$, the tangent space $T_QX$ is a hyperplane not
passing through ${\bf O}$. Hence the line $\ell$, which passes through
${\bf O}$, meets it. Thus
$\varnothing\neq\ell\cap T_QX\subseteq\Lambda\cap T_QX$ and
$Q\in X_\Lambda$.
 
\emph{Step 4: pullback.} Suppose some quadric $Q'\subset\PP^{m-1}$
contains $\Lambda'\cup X'_{\Lambda'}$. Its pullback $\tau^*Q'$ is a
non-zero quadric of $\PP^m$ vanishing on the cone $\tau^{-1}(Q')$; in
particular it vanishes on $\overline{\tau^{-1}(\Lambda')}=\Lambda$ and, by
Step~3, on $X_\Lambda$. Thus $\tau^*Q'$ is a quadric containing
$\Lambda\cup X_\Lambda$, contradicting $\Lambda\in V$. Since $\Lambda'$ is
a generic $(m-1-2n)$-plane of $\PP^{m-1}$, this proves
$(\star_{m-1-2n})$ for $X'$.
\end{proof}

We next verify the graded condition for complete intersections. The key
point is that above the boundary dimension the maximal minors that cut out
the tangency locus $X_\ell$ for a generic $\ell$ have degree at least three, so that--in contrast with
the boundary case--they contribute no quadrics at all.
 
\begin{lemma}
\label{lem::structure_Lambda}
Let $X=V(f_1,\dots,f_k)\subset\PP^m$ be a smooth complete intersection of
type $(d_1,\dots,d_k)$, $d_j\ge2$, with $2n+1<m$, satisfying the
genericity conditions of section~\ref{sec::genericity_conditions}, and let
$\Lambda\in\G(m-2n,m)$ be generic, with
$\what\Lambda=\langle a_0,\dots,a_{m-2n}\rangle$. Then:
\begin{enumerate}
\item $X_\Lambda$ is cut out on $X$ by the maximal minors of the
$k\times(m-2n+1)$ matrix of forms
$$
\Phi_\Lambda(x)=\bigl[\,J(x)\,a_0\mid\cdots\mid J(x)\,a_{m-2n}\,\bigr],
$$
where $J$ is the Jacobian matrix of $(f_1,\dots,f_k)$. Note that
$k\ge m-2n+1$ since $n\ge1$.
\item The ideal $I_X+I_{m-2n+1}(\Phi_\Lambda)$ is saturated and defines
the reduced finite scheme $X_\Lambda$.
\item Every maximal minor of $\Phi_\Lambda$ has degree
$\sum_{j\in A}(d_j-1)\ge m-2n+1\ge3$ (for $2n+1 < m$), where $A$ is its set of
$m-2n+1$ selected rows. Consequently
$$
\J_{X_\Lambda}(2)=\J_X(2)=\langle f_1,\dots,f_q\rangle,
\qquad \dim\J_{X_\Lambda}(2)=q .
$$
\end{enumerate}
\end{lemma}
 
\begin{proof}
(1) One has $T_QX\cap\Lambda\neq\varnothing$ if and only if some non-zero
combination $\sum_ic_ia_i$ lies in $\what{T_QX}=\ker J(Q)$, that is, if
and only if $\Phi_\Lambda(Q)$ has a non-trivial kernel. Since
$k\ge m-2n+1$, this is equivalent to the vanishing of its maximal minors, of size
$m-2n+1$.
 
(2) Write $R=S/I_X$ for the homogeneous coordinate ring of $X$, set $r=m-2n+1$,
and let $I=I_r(\Phi_\Lambda)\subset R$ be the ideal of maximal minors of the
$k\times r$ matrix $\Phi_\Lambda$; note $k\ge r$, since $k-r=n-1$. By
lemma~\ref{lem::X_Lambda}, for a generic $\Lambda$ the locus $X_\Lambda$ is finite
and non-empty, hence of codimension $n$ in $X$; passing to affine cones,
$\dim R=n+1$ and $\dim R/I=1$. Since $X$ is a smooth complete intersection, $R$ is
Cohen--Macaulay, and on such a ring the grade of an ideal equals its height, which
in turn is the difference of the dimensions
\cite[Thm.~2.1.2 \& Cor.~2.1.4]{Bruns-Herzog-98}:
$$
\operatorname{grade} I=\operatorname{height} I=\dim R-\dim R/I=(n+1)-1=n .
$$
This attains the Eagon--Northcott bound
$\operatorname{grade}I_r(\Phi_\Lambda)\le k-r+1=n$ for the maximal minors of a
$k\times r$ matrix, so the Eagon--Northcott complex of $\Phi_\Lambda$ is acyclic
\cite[Thm.~(2.7)]{Bruns-Vetter-88} and resolves $R/I$ over $R$ in length $n$: the
ideal $I$ is perfect. By the Auslander--Buchsbaum formula,
$$
\operatorname{depth}R/I=\operatorname{depth}R-\operatorname{pd}_R(R/I)=(n+1)-n=1=\dim R/I ,
$$
so $R/I$ is Cohen--Macaulay of Krull dimension $1$, hence unmixed: its associated
primes are the one-dimensional primes of the cone over $X_\Lambda$, and the
irrelevant maximal ideal is not among them. Equivalently, the ideal
$I_X+I_{m-2n+1}(\Phi_\Lambda)$ is saturated.

For reducedness, set $r=m-2n+1$, so $\Phi_\Lambda$ is $k\times r$ and $k-r+1=n$.
Let $F$ be the variety of frames $a=(a_0,\dots,a_{m-2n})$ of the
$r$-dimensional subspaces $\what\Lambda\subset\C^{m+1}$, and let
$$
\Psi:X\times F\lra\mathrm{Mat}_{k\times r}(\C),\qquad
\Psi(Q,a)=\bigl[\,J(Q)a_0\mid\cdots\mid J(Q)a_{m-2n}\,\bigr].
$$
Since $X$ is a smooth complete intersection, $J(Q):\C^{m+1}\to\C^{k}$ is surjective
at every $Q\in X$; as $\Psi$ is linear in $a$, its differential is already surjective
in the $F$-directions, so $\Psi$ is a submersion.

Let $D_t=\{M:\rank M\le t\}$, of codimension $(k-t)(r-t)$ in
$\mathrm{Mat}_{k\times r}$. The variety $D_{r-1}$ is reduced and irreducible with
singular locus $D_{r-2}$. Being a submersion, $\Psi$ pulls $D_{r-1}$ back to a
reduced subscheme of codimension $k-r+1=n$, smooth exactly over the corank-one
stratum $D_{r-1}\setminus D_{r-2}$. This is invariant under the right $GL_r$-action
on frames, hence descends to the universal degeneracy locus
$\mathcal X\subset X\times\G$, which as a set, is made of  the pairs $(Q,\Lambda)$ with
$T_QX\cap\Lambda\neq\varnothing$, that is the incidence variety $\Sigma$ of
lemma~\ref{lem::X_Lambda}, and whose fibre over $\Lambda$ is the scheme cut out
in~(1). Write $\mathcal X_{\ge2}\subset\mathcal X$ for the descent of
$\Psi^{-1}(D_{r-2})$.

The stratum $\mathcal X_{\ge2}$ has codimension $(k-r+2)\cdot2=2(n+1)$ in
$X\times\G$. As $2(n+1)>n=\dim X$, its image in $\G$ has dimension at most
$\dim\G-(n+2)$ and lies in a proper closed subset, so for a generic $\Lambda$ the
fibre $X\times\{\Lambda\}$ avoids it and $X_\Lambda$ lies in the smooth locus of
$\mathcal X$. Applying generic smoothness in characteristic zero
\cite[Thm.~21.6.6]{Vakil-25} (see also \cite[Prop.~14.4]{Harris-92}) to that smooth locus over $\G$, the fibre over a
generic $\Lambda$ is smooth of dimension zero, hence reduced.

(3) The entry of $J(x)a_i$ in row $j$ is $\nabla f_j(x)\cdot a_i$, of
degree $d_j-1\ge1$, so a maximal minor with row set $A$, $|A|=m-2n+1$, has
degree $\sum_{j\in A}(d_j-1)\ge|A|=m-2n+1\ge3$, because $d_j\ge2$ and
$m\ge2n+2$. 

Hence $(\Delta'_\Lambda)_2=(I_X)_2=\langle f_1,\dots,f_q\rangle=\J_X(2)$
(lemma~\ref{lem::structure}(1)), where $\Delta'_\Lambda=I_X+I_{m-2n+1}(\Phi_\Lambda)$.
By~(2) the ideal $\Delta'_\Lambda$ is saturated and cuts out the reduced scheme
$X_\Lambda$, so $(\Delta'_\Lambda)_d=H^0(\mathcal J_{X_\Lambda}(d))$ in \emph{every}
degree $d$; taking $d=2$ gives $\J_{X_\Lambda}(2)=\J_X(2)$, of dimension $q$.
Saturation is essential here and not a formality: for a non-saturated ideal the
inclusion $(\Delta'_\Lambda)_2\subseteq\J_{X_\Lambda}(2)$ may be strict, the
saturation acquiring elements in degrees below those of the generators.
\end{proof}
 
Note the contrast with the boundary case $m=2n+1$: there the minors are
the quadrics $g_{i,j}$, the space $\J_{X_L}(2)$ has dimension
$\binom{q+1}2$, and the obstruction at $q\ge3$ appears; for $m>2n+1$ the
minors leave degree two untouched and no such obstruction arises.
 
\begin{theorem}
\label{thm::star_upstairs}
Let $X\subset\PP^m$ be a smooth complete intersection satisfying the
genericity conditions of section~\ref{sec::genericity_conditions}, with
$2n+1<m$ and $q\le m-2n+1$. Then $X$ satisfies $(\star_{m-2n})$.
\end{theorem}
 
\begin{proof}
If $q=0$ then $\J_{X_\Lambda}(2)=0$ by
lemma~\ref{lem::structure_Lambda}(3) and there is nothing to prove.
Assume $q\ge1$. Exactly as in proposition~\ref{prop::reduction2}, the
quadrics containing $\Lambda\cup X_\Lambda$ form the kernel of the
restriction map
$\rho_\Lambda:\J_{X_\Lambda}(2)\to H^0(\Lambda,\mathcal O_\Lambda(2))$,
so by lemma~\ref{lem::structure_Lambda}(3) it suffices to prove that
$f_1|_\Lambda,\dots,f_q|_\Lambda$ are linearly independent for a generic
$\Lambda$. The locus of $\Lambda\in\G(m-2n,m)$ on which they are
dependent is closed, so it is enough to exhibit a single witness
$\Lambda_0$.
 
For each $i\le q$, the forms $\{f_j\}_{j\le q,\,j\neq i}$ are part of the
regular sequence $f_1,\dots,f_k$, so
$V_i=V(f_j:\ j\le q,\ j\neq i)$ has pure codimension $q-1$ in $\PP^m$. It
is not contained in $V(f_i)$: otherwise $V(f_1,\dots,f_q)=V_i$ would have
codimension $q-1<q$, contradicting the regularity of the sequence
$f_1,\dots,f_q$. Choose $P_i\in V_i\setminus V(f_i)$, so that
$f_j(P_i)=0$ for $j\neq i$ and $f_i(P_i)\neq0$.
 
Since $q\le m-2n+1$, the points $P_1,\dots,P_q$ span a linear subspace of
dimension at most $q-1\le m-2n$, so some $(m-2n)$-plane $\Lambda_0$
contains them all. If $\sum_i\lambda_if_i$ vanishes on $\Lambda_0$,
evaluation at $P_i$ gives $\lambda_if_i(P_i)=0$, hence $\lambda_i=0$, for
every $i$. Thus the restrictions are independent on $\Lambda_0$. Hence
for a generic $\Lambda$, the map $\rho_\Lambda$ is injective and no
quadric contains $\Lambda\cup X_\Lambda$.
\end{proof}
 
\begin{remark}
\label{rem::beyond_witness}
The hypothesis $q\le m-2n+1$ is weaker than $q\le2$ throughout the range
$2n+1<m$, since $m-2n+1\ge3$ there. It is used only to fit the witness
points $P_1,\dots,P_q$ inside a single $(m-2n)$-plane. For $q>m-2n+1$ the
statement would follow from a bound on the dimension of the Fano
varieties of $(m-2n)$-planes contained in the members of the linear
system $\langle f_1,\dots,f_q\rangle$, all of which have rank at least
$n+2$ by lemma~\ref{lem::rank_bound}(1), and computational evidence
supports it: for a smooth $(2,2,2,2)\subset\PP^6$ (where $q=4$ and
$m-2n+1=3$) one finds $\dim\J_{X_\Lambda}(2)=4$ and
$\dim\J_{\Lambda\cup X_\Lambda}(2)=0$ for a generic plane $\Lambda$.
\end{remark}

\paragraph{The factorization step \\}

The induction step below reduces the recovery problem in $\PP^m$ to a recovery
problem in $\PP^{m-1}$, by composing the data with auxiliary projections
$\theta_k$ of the image spaces. Write ${\bf O}_k$ for the center of $\pi_k$, let
$\tau:\PP^m\dashrightarrow i(\PP^{m-1})$ be an auxiliary projection with center
${\bf O}\notin\{{\bf O}_1,{\bf O}_2\}$, and let $\theta_k$ be the projection of
$i_k(\PP^{m-1})$ from a point ${\bf c}_k$ onto a hyperplane $j_k(\PP^{m-2}) \subset i_k(\PP^{m-1})$.

\begin{definition}[Factorization step]
\label{def::factorization_step}
We say that the auxiliary projections $(\theta_1,\theta_2)$ are \emph{compatible
with} $\tau$ if there exist projections
$\tau_k:i(\PP^{m-1})\dashrightarrow j_k(\PP^{m-2})$ such that
$$
\theta_k\circ\pi_k=\tau_k\circ\tau,\qquad k=1,2 .
$$
\end{definition}

For the reduced data $(\theta_1(Y_1),\theta_2(Y_2))$ to be an instance of
$P_{m-1}$, the composites $\theta_k\circ\pi_k$ must factor through the single
projection $\tau$ in this way. This is a genuine constraint, and it determines
the centers ${\bf c}_k$ completely.

\begin{lemma}[Compatible auxiliary centers]
\label{lem::factorization}
Let ${\bf O}\notin\{{\bf O}_1,{\bf O}_2\}$. For each $k$ the following are
equivalent:
\begin{enumerate}
\item $\theta_k\circ\pi_k$ factors through $\tau$, that is
$\theta_k\circ\pi_k=\tau_k\circ\tau$ for some projection $\tau_k$;
\item the center of the composite $\theta_k\circ\pi_k$, which is the line
$\pi_k^{-1}({\bf c}_k)$ through ${\bf O}_k$, contains ${\bf O}$;
\item ${\bf c}_k=\pi_k({\bf O})$.
\end{enumerate}
In that case $\tau_k$ is unique.
\end{lemma}

\begin{proof}
The composite $\theta_k\circ\pi_k:\PP^m\dashrightarrow j_k(\PP^{m-2})$ is
represented by the $(m-1)\times(m+1)$ matrix ${\bf \Theta}_k{\bf \Pi}_k$, of rank
$m-1$, so its center is the line $\pi_k^{-1}({\bf c}_k)$, which passes through
${\bf O}_k$. Since ${\bf O}\neq{\bf O}_k$, the point $\pi_k({\bf O})$ is defined
and ${\bf O}\in\pi_k^{-1}({\bf c}_k)$ means precisely $\pi_k({\bf O})={\bf c}_k$; hence (2)
and (3) are the same condition.

Assume (3). Then ${\bf \Theta}_k{\bf \Pi}_k{\bf O}={\bf \Theta}_k{\bf c}_k=0$, because
${\bf c}_k$ is the center of $\theta_k$, so
$\ker({\bf \Theta}_k{\bf \Pi}_k)\supseteq\ker{\bf T}$, where ${\bf T}$ is the
$m\times(m+1)$ matrix of $\tau$. As ${\bf T}$ is surjective, there is a unique
$(m-1)\times m$ matrix ${\bf T}_k$ with
${\bf \Theta}_k{\bf \Pi}_k={\bf T}_k{\bf T}$; it has rank $m-1$, hence represents
a projection $\tau_k$, and $\tau_k$ is unique. This is (1).

Conversely, if $\theta_k\circ\pi_k=\tau_k\circ\tau$, then the composite contracts
the center of $\tau$, so ${\bf O}$ lies in the center $\pi_k^{-1}({\bf c}_k)$ of
$\theta_k\circ\pi_k$, which is (2).
\end{proof}

The situation is summarized by the following diagram, which commutes precisely
when ${\bf c}_k=\pi_k({\bf O})$ for $k=1,2$.

\begin{center}
\begin{tikzpicture}[
	node distance=2cm and 1.5cm,
	>={Stealth},
	every node/.style={font=\small}, scale = 1
	]
	\node (Pm) {$\mathbb{P}^m$};
	\node (i1) [below left=2cm and 2cm of Pm] {$i_1(\mathbb{P}^{m-1})$};
	\node (i)  [below=2cm of Pm] {$i(\mathbb{P}^{m-1})$};
	\node (i2) [below right=2cm and 2cm of Pm] {$i_2(\mathbb{P}^{m-1})$};
	\node (j1) [below=2cm of i1] {$j_1(\mathbb{P}^{m-2})$};
	\node (j2) [below=2cm of i2] {$j_2(\mathbb{P}^{m-2})$};
	
	\draw[->] (Pm) -- node[above left] {$\pi_1$} (i1);
	\draw[->] (Pm) -- node[right] {$\tau$} (i);
	\draw[->] (Pm) -- node[above right] {$\pi_2$} (i2);
	
	\draw[->] (i1) -- node[left] {$\theta_1$} (j1);
	\draw[->] (i2) -- node[right] {$\theta_2$} (j2);
	\draw[->] (i) -- node[above left] {$\tau_1$} (j1);
	\draw[->] (i) -- node[above right] {$\tau_2$} (j2);
\end{tikzpicture}
\end{center}

For independent generic choices of $\theta_1,\theta_2,\tau$ the condition of
lemma~\ref{lem::factorization} fails, no factorization exists, and the reduced
data $(\theta_1(Y_1),\theta_2(Y_2))$ is not a $P_{m-1}$-instance. Performing the
factorization step therefore requires the two points ${\bf c}_k=\pi_k({\bf O})$. Note that the point ${\bf c}_k$ is the
epipole, in the $k$-th image, of the pair $(\pi_k,\tau)$. And these epipoles are not
part of the data. This is the content of the following hypothesis.

\begin{definition}[Descent hypothesis]
\label{def::descent}
Let $\tau:\PP^m\to i(\PP^{m-1})$ be a generic projection with center
${\bf O}$. We say that the pair of projections $(\pi_1,\pi_2)$ admits a
\emph{compatible descent}, and write $(\ddagger)$, if the two points
$\pi_1({\bf O})\in i_1(\PP^{m-1})$ and $\pi_2({\bf O})\in i_2(\PP^{m-1})$ can
be determined, up to a finite ambiguity, from the data $Y_1$, $Y_2$ and
${\bf O}$.
\end{definition}

The question is well posed, because ${\bf c}_k=\pi_k({\bf O})$ is invariant under the
gauge action: if a solution $(X',\pi_1',\pi_2')$ of the recovery problem is
replaced by $(gX',\pi_1'g^{-1},\pi_2'g^{-1})$ for $g\in PGL_{m+1}(\C)$, and the
marked auxiliary center is transported to $g{\bf O}$, then the resulting points in
the images are unchanged.

\begin{remark}[Status of the descent hypothesis]
\label{rem::descent_gap}
We do not know how to establish $(\ddagger)$ from the data alone: the natural
route, taking ${\bf O}$ on the baseline of $(\pi_1,\pi_2)$ so that the ${\bf c}_k$
become the epipoles of that pair, requires the fundamental matrix of
$(\pi_1,\pi_2)$, whose recovery in the regime $2n+1<m$ is precisely what the
induction is designed to achieve. The hypothesis is therefore stated explicitly,
and the induction below is conditional on it.
\end{remark}

Although $(\ddagger)$ cannot be established from the data, it
obeys the same kind of all-or-nothing law as the liftability hypothesis
(proposition~\ref{prop::dichotomy}), once it is given the following precise
sufficient form. The point is that a candidate pair of centers can be tested
against the data: $({\bf c}_1,{\bf c}_2)$ is admissible when the lower half of the diagram
above can be completed, that is, when the projected data
$\theta_{{\bf c}_1}(Y_1),\theta_{{\bf c}_2}(Y_2)$ is again a pair of generic projections of a
single variety of $\PP^{m-1}$. This is the content of the following locus.

\begin{definition}[Admissible data and the consistency locus]
\label{def::consistency_locus}
Fix the multidegree $(d_1,\dots,d_k)$ and let $\mathcal{K}_{\bf d}$ be the
parameter variety of smooth complete intersections $X\subset\PP^m$ of that
multidegree satisfying the genericity conditions of
section~\ref{sec::genericity_conditions}; it is an open subvariety of a product of
projective spaces of forms, hence irreducible. The subscript distinguishes it from
the class $\mathcal{K}(P)$ of definition~\ref{def::Vreal}, which is cut out by a
Hilbert polynomial rather than by a multidegree; since all complete intersections
of multidegree $(d_1,\dots,d_k)$ share one Hilbert polynomial $P$, we have
$$
\mathcal{K}_{\bf d} \;\subseteq\; \mathcal{K}(P),
$$
so that membership in $\mathcal{K}_{\bf d}$ is the stronger condition, and a
configuration realizable by a member of $\mathcal{K}_{\bf d}$ lies in
$V^{\mathrm{real}}$. An \emph{admissible datum} is a
tuple
$$
b=(X,{\bf O}_1,{\bf O}_2,{\bf O},H_1,H_2),\qquad
X\in\mathcal{K}_{\bf d},\quad {\bf O}_1,{\bf O}_2,{\bf O}\in\PP^m,
$$
where $H_k\subset i_k(\PP^{m-1})$ is the target hyperplane of the $k$-th
auxiliary projection, and we write $Y_k=\pi_k(X)$ for the data attached to $b$.
Let
$$
\mathcal{B}=\mathcal{K}_{\bf d}\times(\PP^m)^3\times
\bigl(i_1(\PP^{m-1})^*\bigr)\times\bigl(i_2(\PP^{m-1})^*\bigr)
$$
be the (irreducible) variety of such tuples.

For ${\bf c}_k\in i_k(\PP^{m-1})\setminus H_k$ let
$\theta_{{\bf c}_k}:i_k(\PP^{m-1})\dashrightarrow H_k\cong j_k(\PP^{m-2})$ be the
projection from ${\bf c}_k$, and let $\mathcal{K}'$ be the parameter variety of
the varieties $\tau'(X')$ with $X'\in\mathcal{K}_{\bf d}$ and $\tau'$ a projection of
$\PP^m$ from a point off $X'$; its members are irreducible, of dimension $n$ and
degree $\delta$, in $i(\PP^{m-1})$. The \emph{consistency locus} of $b$ is
$$
\mathcal{C}(b)=\left\{\,({\bf c}_1,{\bf c}_2)\ \Bigm|\
\begin{array}{l}
\exists\,X''\in\mathcal{K}',\ \exists\ \text{projections }
\rho_k:i(\PP^{m-1})\dashrightarrow j_k(\PP^{m-2}),\ k=1,2,\\[2pt]
\text{with } \rho_k(X'')=\theta_{{\bf c}_k}(Y_k)
\ \text{ as reduced subvarieties of } j_k(\PP^{m-2})
\end{array}\right\}.
$$
\end{definition}
 
\begin{remark}
\label{rem::consistency_data_only}
The conditions defining $\mathcal{C}(b)$ involve only $Y_1$, $Y_2$ and the
auxiliary choices ${\bf O},H_1,H_2$: neither $X$ nor $\pi_1,\pi_2$ occurs in
them, except through the images $Y_k$. Together with the gauge invariance noted
after definition~\ref{def::descent}, this is what makes $\mathcal{C}(b)$ an
admissible answer to the question raised there.
\end{remark}

\begin{proposition}[Dichotomy for the descent hypothesis]
\label{prop::descent_dichotomy}
Let $\mathcal{B}^\circ\subseteq\mathcal{B}$ be the set of admissible data
$b=(X,{\bf O}_1,{\bf O}_2,{\bf O},H_1,H_2)$ such that
\begin{enumerate}[label=(\alph*)]
\item ${\bf O}_1\neq{\bf O}_2$ and ${\bf O}\notin\{{\bf O}_1,{\bf O}_2\}$, so
that lemma~\ref{lem::factorization} applies to $b$;
\item ${\bf O}_1,{\bf O}_2,{\bf O}\notin S(X)$, so that $\pi_1,\pi_2,\tau$ embed
$X$ and $Y_1,Y_2,X'=\tau(X)$ are smooth of dimension $n$ and degree $\delta$;
\item $X'$ satisfies the genericity conditions of
section~\ref{sec::genericity_conditions} in $\PP^{m-1}$;
\item ${\bf c}_k=\pi_k({\bf O})\notin H_k$, so that $\theta_{{\bf c}_k}$ is
defined, and $\theta_{{\bf c}_k}$ embeds $Y_k$.
\end{enumerate}
Then $\mathcal{B}^\circ$ is a dense Zariski-open subset of $\mathcal{B}$, and,
writing
$$
\mathcal{C}=\bigl\{(b,{\bf c}_1,{\bf c}_2)\in\mathcal{B}^\circ\times
i_1(\PP^{m-1})\times i_2(\PP^{m-1}) \;\bigm|\;
({\bf c}_1,{\bf c}_2)\in\mathcal{C}(b)\bigr\}
$$
for the total consistency incidence and $\overline{\mathcal{C}}$ for its Zariski
closure:
\begin{enumerate}
\item For every $b\in\mathcal{B}^\circ$ the true pair
$(\pi_1({\bf O}),\pi_2({\bf O}))$ lies in $\mathcal{C}(b)$, and $\mathcal{C}(b)$
is a constructible subset of $i_1(\PP^{m-1})\times i_2(\PP^{m-1})$ obtained by
eliminating the parameters of $X''$ and of $\rho_1,\rho_2$ from the equations of
definition~\ref{def::consistency_locus}. If $\mathcal{C}(b)$ is finite, then
$(\ddagger)$ holds at $b$, with $\mathcal{C}(b)$ itself as the required finite
candidate set.
\item The locus
$\mathcal{B}^{\mathrm{fin}}=\{\,b\in\mathcal{B}^\circ\mid
\overline{\mathcal{C}}(b)\ \text{is finite}\,\}$ is Zariski-open in
$\mathcal{B}^\circ$.
\item Exactly one of the following holds.
\begin{enumerate}[label=(\roman*)]
\item $\mathcal{B}^{\mathrm{fin}}\neq\varnothing$: then $\mathcal{C}(b)$ is
finite, and $(\ddagger)$ holds, for every $b$ in a dense Zariski-open subset of
$\mathcal{B}$.
\item $\mathcal{B}^{\mathrm{fin}}=\varnothing$: then $\mathcal{C}(b)$ is infinite
for every $b$ in a dense Zariski-open subset of $\mathcal{B}$.
\end{enumerate}
In particular a single admissible datum $b_0$ with
$\overline{\mathcal{C}}(b_0)$ certified finite establishes alternative (i) at
level $m$. The induction of theorem~\ref{thm::induction} invokes $(\ddagger)$ at
the levels $j=m,m-1,\dots,2n+2$; if alternative (i) holds at each of these
finitely many levels, then a generic choice of all the auxiliary data satisfies
$(\ddagger)$ at every step simultaneously, a finite intersection of dense open
subsets being dense open.
\end{enumerate}
\end{proposition}

\begin{proof}
That $\mathcal{B}^\circ$ is open is clear from (a)--(d), each being an open
condition; it is non-empty, hence dense, because $\mathcal{B}$ is irreducible and
because a generic projection of $X$ embeds it and preserves the genericity
conditions (section~\ref{sec::generic_projection}), while (a) and (d) exclude
proper closed subsets.

(1) \emph{The true pair is consistent.} Let $b\in\mathcal{B}^\circ$ and
${\bf c}_k=\pi_k({\bf O})$. Take $X''=\tau(X)=X'$, which lies in $\mathcal{K}'$
by (b), and let $\rho_k=\tau_k$ be the projections produced by
lemma~\ref{lem::factorization}, so that
$\theta_{{\bf c}_k}\circ\pi_k=\tau_k\circ\tau$. Applying both sides to $X$ gives
$\tau_k(X'')=\theta_{{\bf c}_k}(Y_k)$, so
$({\bf c}_1,{\bf c}_2)\in\mathcal{C}(b)$.

\emph{Constructibility.} The pairs $(X'',\rho_1,\rho_2)$ are parametrized by the
quasi-projective variety $\mathcal{K}'\times\mathbb{G}_1\times\mathbb{G}_2$,
where $\mathbb{G}_k$ is the variety of projections
$i(\PP^{m-1})\dashrightarrow j_k(\PP^{m-2})$. All the varieties occurring are
irreducible of dimension $n$ and degree $\delta$ in a fixed projective space, so
they are points of a Chow variety~\cite[Lec. 21]{Harris-92} of that dimension and degree, and
$\rho_k(X'')=\theta_{{\bf c}_k}(Y_k)$ is the equality of two morphisms to that
Chow variety, hence a closed condition on the parameters. Thus the set
$$
\mathcal{I}=\bigl\{({\bf c}_1,{\bf c}_2,X'',\rho_1,\rho_2)\ \bigm|\
\rho_k(X'')=\theta_{{\bf c}_k}(Y_k),\ k=1,2\bigr\}
$$
is locally closed, and $\mathcal{C}(b)$ is its image under the projection to
$i_1(\PP^{m-1})\times i_2(\PP^{m-1})$, hence constructible by Chevalley's theorem~\cite[Thm.~8.4.2]{Vakil-25} and~\cite[Thm. 3.16]{Harris-92}.

\emph{Sufficiency for $(\ddagger)$.} By remark~\ref{rem::consistency_data_only},
$\mathcal{C}(b)$ is determined by $(Y_1,Y_2,{\bf O},H_1,H_2)$ alone and is gauge
invariant. If it is finite it is therefore a finite set, computable from the
data, containing the true pair: this is exactly the determination up to finite
ambiguity required by definition~\ref{def::descent}.

(2) The same parametrization, carried out over $\mathcal{B}^\circ$, shows that
$\mathcal{I}$ globalizes to a locally closed subset of
$\mathcal{B}^\circ\times i_1(\PP^{m-1})\times i_2(\PP^{m-1})\times
\mathcal{K}'\times\mathbb{G}_1\times\mathbb{G}_2$, whose image $\mathcal{C}$ is
constructible, again by \cite[Thm.~8.4.2]{Vakil-25}; here (b) and (d) are what
make $b\mapsto Y_k$ and $b\mapsto\theta_{{\bf c}_k}(Y_k)$ morphisms to the
relevant Chow varieties. The projection
$q:\overline{\mathcal{C}}\to\mathcal{B}^\circ$ is proper, since
$\overline{\mathcal{C}}$ is closed in the product of $\mathcal{B}^\circ$ with a
projective variety. By upper semicontinuity of fiber dimension on the target for
closed morphisms \cite[Thm.~12.4.3(b)]{Vakil-25}, the set
$\{b\mid\dim\overline{\mathcal{C}}(b)\ge1\}$ is closed in $\mathcal{B}^\circ$,
and $\mathcal{B}^{\mathrm{fin}}$ is its complement.

(3) $\mathcal{B}^\circ$ is irreducible, being a dense open subset of the
irreducible $\mathcal{B}$, so the open set $\mathcal{B}^{\mathrm{fin}}$ is either
empty or dense.

If $\mathcal{B}^{\mathrm{fin}}\neq\varnothing$, then for
$b\in\mathcal{B}^{\mathrm{fin}}$ one has
$\mathcal{C}(b)\subseteq\overline{\mathcal{C}}(b)$ finite, and $(\ddagger)$ holds
at $b$ by (1); this is alternative (i).

If $\mathcal{B}^{\mathrm{fin}}=\varnothing$, then
$\dim\overline{\mathcal{C}}(b)\ge1$ for every $b\in\mathcal{B}^\circ$. Since
$\mathcal{C}$ is constructible and dense in $\overline{\mathcal{C}}$, it contains
a subset $U$ open and dense in $\overline{\mathcal{C}}$, and every irreducible
component of $\overline{\mathcal{C}}\setminus U$ is properly contained in a
component of $\overline{\mathcal{C}}$. Applying the theorem on the dimension of
generic fibers \cite[Cor.~12.4.2]{Vakil-25} to $q$ restricted to each of the
finitely many irreducible components of $\overline{\mathcal{C}}$ and of
$\overline{\mathcal{C}}\setminus U$, we obtain a dense open
$\mathcal{B}^{\mathrm{gen}}\subseteq\mathcal{B}^\circ$ over which
$\dim\mathcal{C}(b)=\dim\overline{\mathcal{C}}(b)\ge1$; there $\mathcal{C}(b)$ is
infinite. This is alternative (ii).

The two alternatives are mutually exclusive, since a dense open subset of the
irreducible $\mathcal{B}$ cannot consist of data with $\mathcal{C}(b)$ both
finite and infinite.
\end{proof}
 
\begin{remark}
\label{rem::descent_dichotomy}
Three differences with proposition~\ref{prop::dichotomy} should be noted.
First, finiteness of the consistency locus is a sufficient form of
$(\ddagger)$, not an equivalent one: the dichotomy concerns this form, and
no equivalence with the conclusion of the induction is claimed. Second,
the certification of an instance is heavier here: it is an elimination
over the parameters of the class $\mathcal{K}'$ and of the projections,
rather than a local dimension computation on a single explicit variety. However,
it remains a finite computation. Third, and decisively, alternative~(i) is \emph{not} the one that holds:
proposition~\ref{prop::consistency_structure} below shows that
$\mathcal{C}(b)$ has dimension $m$ for every $b\in\mathcal{B}^\circ$, so that
$\mathcal{B}^{\mathrm{fin}}=\varnothing$ and alternative~(ii) obtains. The
constraint count that suggests otherwise, namely that the conditions defining
consistency impose far more conditions than the $2(m-1)$ parameters of
$({\bf c}_1,{\bf c}_2)$ provide, overlooks that $X''$ and $\rho_1,\rho_2$ are
themselves unknowns, and that they absorb almost all of those conditions; the
correct count is the one carried out in
lemma~\ref{lem::Vreal_embedding}, read one level down. It remains true, and is
worth recording, that the homographic solutions of
lemma~\ref{lem::homographic} contribute nothing to $\mathcal{C}(b)$, which is
built from actual reconstructions rather than from solutions of the Kruppa
system. Specializing ${\bf O}$ to the baseline identifies the
finiteness of $\mathcal{C}(b)$ with the identifiability of the epipoles of
$(\pi_1,\pi_2)$ from the data, which calibrates the depth of the question:
it is exactly the part of the recovery problem that the regime $2n+1<m$
leaves without tangency support.
\end{remark}

For $\mathbf{c}_k \in i_k(\PP^{m-1})$ write
$$
L_k(\mathbf{c}_k) \;=\; \pi_k^{-1}(\mathbf{c}_k) \;=\;
\bigl\langle \mathbf{O}_k,\ \mathbf{P}\bigr\rangle ,
\qquad \mathbf{P} \ \text{any point with} \ \pi_k(\mathbf{P})=\mathbf{c}_k ,
$$
for the line of $\PP^m$ over $\mathbf{c}_k$, and recall that
$\theta_{\mathbf{c}_k}\circ\pi_k$ is the projection of $\PP^m$ from that line.
For $\psi\in\operatorname{Lin}(X)$ let $A_\psi\in\PP GL_{m+1}$ be the projective
transformation inducing it, and set
$$
\mathcal{M}_\psi \;=\;
\bigl\{\,(\mathbf{c}_1,\mathbf{c}_2) \;\bigm|\;
A_\psi\bigl(L_1(\mathbf{c}_1)\bigr)\cap L_2(\mathbf{c}_2)\neq\varnothing \,\bigr\} .
$$

\begin{proposition}[Structure of the consistency locus]
\label{prop::consistency_structure}
Let $b\in\mathcal{B}^\circ$ and let $X$ be linearly normal, as every smooth
complete intersection of positive dimension is. Then:
\begin{enumerate}
\item $\displaystyle \mathcal{C}(b)=\bigcup_{\psi\in\operatorname{Lin}(X)}\mathcal{M}_\psi$,
a finite union, $\operatorname{Lin}(X)$ being finite by
remark~\ref{rem::Vreal_aut}.
\item Each $\mathcal{M}_\psi$ is irreducible of dimension $m$, cut out by the
maximal minors of a $4\times(m+1)$ matrix whose rows span
$\what{L_1(\mathbf{c}_1)}$ and $A_\psi^{-1}\what{L_2(\mathbf{c}_2)}$, and which are
therefore bilinear in $(\mathbf{c}_1,\mathbf{c}_2)$. Indeed $\mathcal{M}_\psi$ is
the image of the \emph{joint map}
$$
\Phi_\psi : \PP^m \lra i_1(\PP^{m-1})\times i_2(\PP^{m-1}),
\qquad
\Phi_\psi(\mathbf{O}')=\bigl(\pi_1(A_\psi^{-1}\mathbf{O}'),\,\pi_2(\mathbf{O}')\bigr).
$$
\item For $\psi=\mathrm{id}$ the map $\Phi=\Phi_{\mathrm{id}}=(\pi_1,\pi_2)$
contracts the baseline $L=\langle\mathbf{O}_1,\mathbf{O}_2\rangle$ to the pair of
epipoles $(\mathbf{e}_1,\mathbf{e}_2)$ and is injective off it. For $m=3$ the
matrix of item (2) is $4\times4$, the single minor is the bilinear form
$\mathbf{c}_2^{T}{\bf F}\mathbf{c}_1$ of the classical fundamental matrix, and
$\ker{\bf F}=\langle\mathbf{e}_1\rangle$,
$\ker{\bf F}^{T}=\langle\mathbf{e}_2\rangle$.
\end{enumerate}
In particular $\dim\mathcal{C}(b)=m$ for every $b\in\mathcal{B}^\circ$, so
$\mathcal{B}^{\mathrm{fin}}=\varnothing$ and alternative~(ii) of
proposition~\ref{prop::descent_dichotomy} is the one that holds.
\end{proposition}

\begin{proof}
Two lines of $\PP^m$ meet if and only if they contain a common point
$\mathbf{O}'$, and then $L_1=\langle\mathbf{O}_1,\mathbf{O}'\rangle$ and
$L_2=\langle\mathbf{O}_2,\mathbf{O}'\rangle$ are determined by $\mathbf{O}'$,
whence $(\mathbf{c}_1,\mathbf{c}_2)=\Phi_\psi(\mathbf{O}')$; conversely
$\mathbf{O}'$ lies on both lines over $\Phi_\psi(\mathbf{O}')$. This proves the
description of $\mathcal{M}_\psi$ as an image in item (2), hence its
irreducibility, and its dimension is $m$ because the fibre of $\Phi_\psi$ over a
point off the epipolar pair is the single point $\mathbf{O}'$. The rank condition
and the bilinearity of the minors are immediate, as is item (3): on the baseline
$L_1=L_2=L$, so $\Phi$ is constant there with value $(\mathbf{e}_1,\mathbf{e}_2)$.

For item (1), let $(\mathbf{c}_1,\mathbf{c}_2)\in\mathcal{C}(b)$, realized by
$X''\in\mathcal{K}'$ and projections $\rho_k$, and put
$Y'_k=\theta_{\mathbf{c}_k}(Y_k)$. By condition (b) of $\mathcal{B}^\circ$ the
center of $\rho_k$ is off $X''$, so $\rho_k|_{X''}$ is finite and birational onto
$Y'_k$ and $X''$ is the normalization of $Y'_k$, as in
lemma~\ref{lem::Vreal_norm}. By conditions (b) and (d) the composite
$\theta_{\mathbf{c}_k}\circ\pi_k$ embeds $X$, so $X$ is that normalization too,
and the two identifications carry $\mathcal{O}(1)$ to $\mathcal{O}_X(1)$ in both
cases. Hence $X''\iso X$ as polarized varieties, and
lemma~\ref{lem::Vreal_embedding} applied at level $m-1$ gives
$$
V' \;=\; W'_1+\psi^*W'_2 , \qquad \dim V'=m ,
$$
where $W'_k\subset H^0(X,\mathcal{O}_X(1))$ is the system of
$\theta_{\mathbf{c}_k}\circ\pi_k$ and $\psi$ is the induced polarized
automorphism of $X$, so $\psi\in\operatorname{Aut}(X,\mathcal{O}_X(1))=
\operatorname{Lin}(X)$, the equality because $X$ is linearly normal. Now $W'_k$
consists of the linear forms vanishing on $L_k$, that is
$W'_k=\what{L_k}^{\perp}$, whence
$$
\dim\bigl(W'_1+\psi^*W'_2\bigr) \;=\;
(m+1)-\dim\bigl(\what{L_1}\cap A_\psi^{-1}\what{L_2}\bigr) ,
$$
and $\dim V'=m$ says exactly that this intersection is a line, that is that
$A_\psi(L_1)$ meets $L_2$. So $(\mathbf{c}_1,\mathbf{c}_2)\in\mathcal{M}_\psi$.

Conversely, given $\psi$ and $(\mathbf{c}_1,\mathbf{c}_2)\in\mathcal{M}_\psi$, the
same identity gives $\dim(W'_1+\psi^*W'_2)=m$; let $X''$ be the image of $X$ under
that system, a projection of $X$ from the point it annihilates, hence a member of
$\mathcal{K}'$, and let $\rho_1,\rho_2$ be the projections dual to the inclusions
$W'_1\subset V'$ and $\psi^*W'_2\subset V'$. Then $\rho_k(X'')$ is the image of
$X$ under $W'_k$, namely $Y'_k$, so $(\mathbf{c}_1,\mathbf{c}_2)\in\mathcal{C}(b)$.
\end{proof}

\begin{corollary}[The sufficient form of $(\ddagger)$ is unavailable]
\label{cor::ddagger_unavailable}
For every $b\in\mathcal{B}^\circ$ the locus $\mathcal{C}(b)$ is infinite, so
part~(1) of proposition~\ref{prop::descent_dichotomy} never applies. Moreover no
property intrinsic to $\mathcal{C}(b)$ distinguishes the true pair: the
conditions defining $\mathcal{C}(b)$ involve $Y_1$, $Y_2$, the hyperplanes $H_k$
and the class $\mathcal{K}'$, and not the marked point $\mathbf{O}$, whereas the
true pair is $\Phi(\mathbf{O})$. As $\mathbf{O}$ ranges over $\PP^m$ the true pair
ranges over the whole of $\mathcal{M}_{\mathrm{id}}$ while $\mathcal{C}(b)$ is
unchanged; every point of $\mathcal{M}_{\mathrm{id}}$ is therefore the true pair
of an admissible datum with the same consistency locus.
\end{corollary}

\begin{remark}[What the obstruction is]
\label{rem::ddagger_circular}
Proposition~\ref{prop::consistency_structure} turns the difficulty recorded in
remark~\ref{rem::descent_gap} into a statement: by item (3) the equations of
$\mathcal{M}_{\mathrm{id}}$ \emph{are} the epipolar geometry of
$(\pi_1,\pi_2)$, which for $m=3$ is literally the fundamental matrix, so locating the
true pair inside $\mathcal{C}(b)$ is equivalent to knowing the epipolar geometry,
which is what the induction is designed to produce. The descent is therefore
circular at the level of its sufficient condition, and not for want of a finer
invariant. Two consequences are worth separating. The conclusion of
theorem~\ref{thm::dim(x)*small} is untouched as a conditional statement, but
$(\ddagger)$ now has no route to verification among those proposed here. And
nothing is lost in substance: theorem~\ref{thm::Vreal_finite} settles the regime
$2n+1<m$ unconditionally, with the explicit system of
proposition~\ref{prop::Vreal_system}; what the present subsection contributes is
the graded no-quadric machinery, certified in
example~\ref{ex::graded_star_curve}, together with the structural reason why the
Kruppa system resists this regime.
\end{remark}

\begin{theorem}
\label{thm::induction}
Assume the descent hypothesis $(\ddagger)$. For $2n + 1 < m$, if $P_{m-1}$
is satisfied then $P_m$ is satisfied.
\end{theorem}

\begin{proof}
Let $\pi_1: \PP^m \rightarrow i_1(\PP^{m-1})$ and $\pi_2: \PP^m \rightarrow i_2(\PP^{m-1})$ be the generic projections to be recovered, such that $Y_j = \pi_j(X)$. Let us consider another generic projection $\tau: \PP^m \rightarrow i(\PP^{m-1})$.

Then by section~\ref{sec::generic_projection}, since $2n+1 < m$ the generic projection $\tau$ embeds $X$, so that $X' = \tau(X) \subset \PP^{m-1}$ is smooth, irreducible and of
degree $d > 2$. It is non-degenerate, since a hyperplane of $\PP^{m-1}$ containing $X'$
would pull back to a hyperplane of $\PP^m$ containing $X$. Its chordal variety satisfies
$S(X') = \tau(S(X))$, of dimension $2n+1 = \min(2n+1, m-1)$, so it is not deficient. Moreover, its dual $X'^*$ is the intersection of the hypersurface $X^*$ with the hyperplane of
$\PP^{m*}$ determined by ${\bf O}$, hence again a hypersurface, of the same degree $c$
(section~\ref{sec::generic_projection}). Finally, by
lemma~\ref{lemma::star_descends}, $X'$ satisfies $(\star_{m-1-2n})$. Therefore $X'$ satisfies all the hypotheses of $P_{m-1}$.

For each $k$, we now choose $\theta_k: i_k(\PP^{m-1}) \rightarrow j_k(\PP^{m-2})$ as follows. By the descent hypothesis $(\ddagger)$, the point ${\bf c}_k=\pi_k({\bf O})\in i_k(\PP^{m-1})$ can be determined from the data (up to a finite ambiguity, over which we enumerate); let $\theta_k$ be the projection of $i_k(\PP^{m-1})$ from the center ${\bf c}_k$ onto a generic hyperplane $j_k(\PP^{m-2})$ of $i_k(\PP^{m-1})$. Then let $Y_k' = \theta_k(Y_k)$.

By lemma~\ref{lem::factorization}, the choice ${\bf c}_k=\pi_k({\bf O})$ makes
$(\theta_1,\theta_2)$ compatible with $\tau$: there exist unique projections
$\tau_1$ and $\tau_2$ from $i(\PP^{m-1})$ to $j_1(\PP^{m-2})$ and
$j_2(\PP^{m-2})$ respectively with $\theta_k\circ\pi_k=\tau_k\circ\tau$. For any
other center of $\theta_k$ no such factorization exists, so this choice is essential.

Assuming $P_{m-1}$ holds means that one can compute $\tilde{\tau}_1$ and $\tilde{\tau}_2$ such that there exists an unknown automorphism $b$ of $\PP^{m-1}$, such that $\tilde{\tau}_k \circ b = \tau_k$.  

One can construct an automorphism $a$ of $\PP^m$, such that $b \circ \tau = \tau \circ a$. An explicit construction can be done as follows. Fix the projective frame such that the center of projection of $\tau$ is ${\bf O} = [0: \cdots: 0:1]^T$ and the image hyperplane of $\tau$ is defined by the vanishing of last coordinate. Therefore $\tau$ is represented by the matrix $[{\bf I}_m | 0]$. Now let ${\bf B} \in GL_m$ be the matrix of $b$. Consider the following $(m+1) \times (m+1)$ matrix: 
$$
{\bf A} = \left [\begin{array}{cc}  
	{\bf B} & 0 \\
	{\bf O}_{1 \times m} & 1
\end{array}\right].
$$ 
Let $a$ be the corresponding automorphism of $\PP^{m}$. Let $P = [x_0: \cdots : x_m]^T = [\overline{x} : x_m]^T$. Then we have:
$$
(\tau \circ a)(P) = \tau (AP) = \tau([{\bf B}\overline{x}:x_m]) = {\bf B}\overline{x} = b(\overline{x}) = b(\tau(P)) = (b \circ \tau)(P).
$$ 

Therefore we have: $\theta_k \circ \pi_k = \tilde{\tau}_k \circ \tau \circ a$, or equivalently $\theta_k \circ \pi_k \circ a^{-1} = \tilde{\tau}_k \circ \tau$. We shall now show that one can compute $\hat{\pi}_k = \pi_k \circ a^{-1}$. Hence, we have $\theta_1 \circ \hat{\pi}_1 = \tilde{\tau}_1 \circ \tau$ and $\theta_2 \circ \hat{\pi}_2 = \tilde{\tau}_2 \circ \tau$.

In matrix form, given $(m-1) \times m$ matrices $\bf{\Theta}_1, \bf{\Theta}_2$, that represent $\theta_1$ and $\theta_2$, and given two $(m-1) \times (m+1)$ matrix $\bf{\Gamma}_1, \bf{\Gamma}_2$, that represent $\tilde{\tau}_1 \circ \tau$ and $\tilde{\tau}_2 \circ \tau$, we are looking for two $m \times (m+1)$ matrices $\bf{P}_1,\bf{P}_2$, such that $\bf{\Theta}_k \bf{P}_k = \bf{\Gamma}_k$. Indeed despite the fact that the equality is between projective operators, the matrix $\bf{P}_k$ can be normalized to get a full equality. 

Therefore for each $k$, we get generically $(m-1) \times (m+1)$ independent equations, while we have $m \times (m+1)$ unknowns. The set of solutions is then made of an affine space of dimension $m+1$ in $\C^{m \times (m+1)}$. Let ${\bf P}_{1,1},\cdots, {\bf P}_{1,m+1}$, ${\bf P}_{2,1},\cdots, {\bf P}_{2,m+1}$ be two sets of linearly independent matrices, such that: ${\bf P}_k = {\bf P}_{k,0} + \sum_{i=1}^{m+1} \alpha_{k,i} {\bf P}_{k,i}$, for some matrix ${\bf P}_{k,0}$ and some vector $(\alpha_{k,1}, \cdots, \alpha_{k,m+1}) \in \C^{m+1}$. 

The kernel of $\bf{\Theta}_k \bf{P}_k = \bf{\Gamma}_k$ is two-dimensional since the rank is $m-1$. The kernel is then a line in $\PP^m$ that contains the center of projection ${\bf O}_k$ of ${\bf P}_k$. 

Once the center of projection ${\bf O}_k$ of $\hat{\pi}_k$ (equivalently the kernel of ${\bf P}_k$) is known, we have a linear system, with unknowns $\alpha_{k,1}, \cdots, \alpha_{k,m+1}$, such that $\sum_{i=1}^{m+1} \alpha_{k,i} {\bf P}_{k,i} {\bf O}_k = - {\bf P}_{k,0} {\bf O}_k$. This implies ${\bf \alpha}_k = (\alpha_{k,1}, \cdots, \alpha_{k,m+1})$ lies on an affine line in $\C^{m+1}$: ${\bf \alpha}_k = {\bf w}_{k} + \lambda_k {\bf v}_k$. 

The centers ${\bf c}_k=\pi_k({\bf O})$ and the image hyperplanes $j_k(\PP^{m-2})$ being generic (the former because ${\bf O}$ is a generic point of $\PP^m$ and $\pi_k$ is fixed), the hypotheses of theorem~\ref{thm::reconstruction} hold for $Y_1', Y_2'$ and $\tilde{\tau}_1, \tilde{\tau}_2$, so that one can compute the variety $X'' = b(X')$ from $Y'_1$, $Y'_2$, $\tilde{\tau}_1$ and $\tilde{\tau}_2$, since $\deg(X') > 2$. (Theorem~\ref{thm::reconstruction} is proved in section~\ref{sec::recovering_the_variety}
below and is independent of the present section, so there is no circularity.) Therefore we have $X'' = b(X') = b \circ \tau (X) = \tau \circ a(X)$ and $\hat{\pi}_k(a(X)) = Y_k$. 

Let $\Delta_k$ be the cone in $\PP^m$ generated ${\bf O}_k$ and $Y_k$, and let $\Delta$ be the cone in $\PP^m$ generated by the center of projection ${\bf O}$ of $\tau$ and $b(X')$. Then theorem~\ref{thm::reconstruction} implies that generically each of the intersections $\Delta_1 \cap \Delta$, $\Delta_2 \cap \Delta$ and $\Delta_1 \cap \Delta_2$ equals $a(X)$ together with a residual locus of dimension $2n-m+2 \leq 0$; in particular $a(X)$ is in each case the unique irreducible component of positive dimension. 

For any point ${\bf Q}$ in $\Delta$, one can find a variety linearly isomorphic to $a(X)$ that contains ${\bf Q}$. Then for each $k$, the scalar $\lambda_k$ is computed up to a finite fold ambiguity by requiring $P_k {\bf Q} \in Y_k$. Indeed let $f_1, \ldots, f_p$ be the generators of the ideal of $Y_k$. Saying that $P_k {\bf Q} \in Y_k$ is equivalent to state that $f_j(P_k {\bf Q}) = 0$ for all $j$. But as $\lambda_k$ varies, the points $P_k {\bf Q}$ trace out a line in $i_k(\PP^{m-1})$. The matrices $\{{\bf P}_{k,i}\}_i$ being linearly independent, these lines are different for different points $Q$. Therefore for a generic $Q \in \Delta$, this line is not contained in $Y_k$, since otherwise we would have $\dim(Y_k) > n$. Therefore the system $\{f_j(P_k {\bf Q}) = 0\}_j$ defines $\lambda_k$ up to a finite fold ambiguity.  

This shows that once ${\bf O}_1$, ${\bf O}_2$ and ${\bf Q}$ are chosen, the matrices ${\bf P}_1$ and ${\bf P}_2$ can be computed. The two centers of projection are constrained to lie on the lines defined by the kernels of ${\bf \Gamma}_1$ and ${\bf \Gamma}_2$, and the point ${\bf Q}$ must be in $\Delta$. Choosing other points that satisfy these constraints is equivalent to multiply each ${\bf P}_k$ on the right by some matrix $A$ in $GL_{m+1}$. Put differently, we have shown that the projections $\pi_1$ and $\pi_2$ can be computed modulo the action $\varrho$ of $\PP GL_{m+1}$ on the pairs of projections.
\end{proof}

\begin{theorem}
\label{thm::dim(x)*small}
Let $X \subset \PP^m$ be a smooth, irreducible complete intersection with
$\dim(X)=n$, $\deg X > 2$ and $\dim \mathcal{J}_X(2) \le m-2n+1$,
satisfying the genericity conditions of
section~\ref{sec::genericity_conditions}, and assume $2n+1 < m$. Assume
moreover the descent hypothesis $(\ddagger)$ of
definition~\ref{def::descent} at every step of the induction below, and
that the solution variety $V^{\mathrm{adm}}$ associated with the bottom
variety $X^{(m-2n-1)}\subset\PP^{2n+1}$ of the chain is finite. Then for a generic configuration one can compute $\pi_1, \pi_2$ modulo the action $\varrho$.
\end{theorem}
 
\begin{proof}
By lemma~\ref{lem::structure}(1) the hypothesis reads $q\le m-2n+1$, so
by theorem~\ref{thm::star_upstairs}, $X$ satisfies $(\star_{m-2n})$. Let
$$
X = X^{(0)} \longrightarrow X^{(1)} \longrightarrow \cdots
\longrightarrow X^{(m-2n-1)} \subset \PP^{2n+1}
$$
be a chain of generic linear projections, with
$X^{(i)} \subset \PP^{m-i}$. For $0\le i\le m-2n-2$ the step
$X^{(i)}\to X^{(i+1)}$ is a generic projection of a smooth irreducible
variety of dimension $n$ from an ambient space of dimension $m-i>2n+1$,
so lemma~\ref{lemma::star_descends} applies and shows inductively that
every $X^{(i)}$ satisfies $(\star_{m-i-2n})$; by the first paragraph of
the proof of theorem~\ref{thm::induction} it satisfies the remaining
hypotheses of $P_{m-i}$ as well.
 
At the bottom of the chain the ambient dimension is $2n+1$, the centers
$\Lambda$ are lines, and $(\star_1)$ is the condition $(\star)$;
proposition~\ref{prop::no_quadric} shows that the true configuration is an
isolated point of the corresponding variety $V^{\mathrm{adm}}$, and the
assumed finiteness of $V^{\mathrm{adm}}$, combined with
corollary~\ref{cor::fund_equiv_proj}, establishes $P_{2n+1}$ (for $n=1$
this recovers the statement asserted in~\cite{Kaminski-04};
cf.\ remark~\ref{rem::K04_gap}).
Theorem~\ref{thm::induction}, applicable under $(\ddagger)$, then gives
$P_j$ for all $2n+1 \le j \le m$ by induction, and in particular $P_m$.
\end{proof}

\paragraph{Rigidity as an alternative to finiteness at the bottom \\}

The finiteness of $V^{\mathrm{adm}}$ at the bottom of the chain, assumed in
theorem~\ref{thm::dim(x)*small}, is known for no $n\ge1$, whereas the weaker
rigidity statement can be obtained outright. We show that rigidity at level $m$
follows from the finiteness of the realizable locus established in
section~\ref{sec::Vreal}, at the cost of one hypothesis only: a realizability
hypothesis on deformations, which plays here the role that liftability $(\dagger)$
plays in the regime $2n+1\ge m$. Neither an induction along the chain nor any
condition on a consistency locus is needed, which is why the argument below does
not proceed step by step.

\begin{definition}[Realizable deformation]
\label{def::realizable}
Let $\Theta_0\in V^{\mathrm{adm}}$ be the true configuration for the data
$(Y_1,Y_2)$ in $\PP^m$. A \emph{realizable deformation} of $\Theta_0$ is an
irreducible algebraic curve $C\subseteq V^{\mathrm{adm}}$ through $\Theta_0$,
together with a parametrization $t\mapsto\Theta(t)$, $\Theta(0)=\Theta_0$, and
families $X''(t)\in\mathcal{K}_{\bf d}$ and projections $\pi_1(t),\pi_2(t)$,
algebraic in $t$, such that
$$
\pi_k(t)\bigl(X''(t)\bigr)=Y_k\quad(k=1,2),
$$
and such that $\Theta(t)$ is the epipolar geometry of $(\pi_1(t),\pi_2(t))$. We
write $(\natural)_{\Theta_0}$ for the hypothesis that every irreducible curve of
$V^{\mathrm{adm}}$ through $\Theta_0$ is realizable in this sense, and
$$
(\natural')_{\Theta_0}: \qquad
C\cap V^{\mathrm{real}} \ \text{ is infinite for every irreducible curve }
C\subseteq V^{\mathrm{adm}} \text{ through } \Theta_0
$$
for the weaker hypothesis that each such curve carries infinitely many realizable
configurations. A realizable deformation has $\Theta(t)\in V^{\mathrm{real}}$ for
every $t$, by definition~\ref{def::Vreal} and the inclusion
$\mathcal{K}_{\bf d}\subseteq\mathcal{K}(P)$ of
definition~\ref{def::consistency_locus}; hence
$(\natural)_{\Theta_0}\Rightarrow(\natural')_{\Theta_0}$. The converse fails: in
$(\natural')$ nothing is required to depend algebraically on $t$, no family
$X''(t)$ is required to exist at all, and only countably many points of $C$ are
constrained.
\end{definition}

The hypothesis $(\natural')$ is all that the finiteness of $V^{\mathrm{real}}$
needs, and it needs it in a quantitative form that is weaker still.

\begin{proposition}[Rigidity from the finiteness of $V^{\mathrm{real}}$]
\label{prop::rigidity_from_Vreal}
Assume the hypotheses of theorem~\ref{thm::Vreal_finite}, so that
$V^{\mathrm{real}}$ is finite, and let $\Theta_0$ be the true configuration. If
every irreducible curve $C\subseteq V^{\mathrm{adm}}$ through $\Theta_0$ satisfies
$$
\bigl|\,C\cap V^{\mathrm{real}}\,\bigr| \;>\; \bigl|\operatorname{Aut}(Z,L)\bigr| ,
$$
and in particular if $(\natural')_{\Theta_0}$ holds, then $\Theta_0$ is an isolated
point of $V^{\mathrm{adm}}$. Hence $(\natural)_{\Theta_0}$ alone implies that
$\Theta_0$ is an isolated point of $V^{\mathrm{adm}}$, with no further hypothesis:
neither the finiteness of a consistency locus, nor an inductive input at a lower
level, nor the descent hypothesis $(\ddagger)$, nor any bound on $\dim\J_X(2)$.
\end{proposition}

\begin{proof}
Suppose $\Theta_0$ is not isolated. The local dimension of $V^{\mathrm{adm}}$ at
$\Theta_0$ is then at least one, so there is an irreducible curve
$C\subseteq V^{\mathrm{adm}}$ through $\Theta_0$. By
theorem~\ref{thm::Vreal_finite} the set $V^{\mathrm{real}}$ is finite of cardinal
at most $|\operatorname{Aut}(Z,L)|$, hence so is $C\cap V^{\mathrm{real}}$,
contradicting the displayed inequality and a fortiori
$(\natural')_{\Theta_0}$. The last assertion follows from
$(\natural)\Rightarrow(\natural')$ of definition~\ref{def::realizable}.
\end{proof}

\begin{remark}[What this replaces]
\label{rem::rigidity_shortcut}
Proposition~\ref{prop::rigidity_from_Vreal} reaches the rigidity conclusion in one
step, with no induction and no condition on a consistency locus, and it does so
under hypotheses the paper already carries: for a smooth complete intersection with
$\deg X>2$, lemma~\ref{lem::Vreal_norm} gives $(Z,L)\iso(X,\mathcal{O}_X(1))$, so
$\operatorname{Aut}(Z,L)=\operatorname{Lin}(X)$ is finite by
remark~\ref{rem::Vreal_aut} and item (3) of theorem~\ref{thm::Vreal_finite} always
applies. Item (4) of that theorem is of no use in this regime, a complete
intersection being projectively normal, so that $N=m$ and $\kappa=0$. The proposition
therefore \emph{replaces}, rather than complements, the argument by which rigidity
was propagated step by step along the chain, and that argument is not reproduced
here. What the induction below is still needed for is the \emph{finiteness}
conclusion of theorem~\ref{thm::dim(x)*small}, which rests on $(\ddagger)$ and on
an input at the bottom of the chain that theorem~\ref{thm::Vreal_finite} cannot
supply, bounding $V^{\mathrm{real}}$ and not $V^{\mathrm{adm}}$.
\end{remark}

\begin{theorem}[Rigidity version of theorem~\ref{thm::dim(x)*small}]
\label{thm::dim(x)*small_rigid}
Let $X\subset\PP^m$ be as in theorem~\ref{thm::dim(x)*small}. Assume
$(\natural')_{\Theta_0}$ of definition~\ref{def::realizable} at the true
configuration $\Theta_0$ at level $m$. Then $\Theta_0$ is an isolated point of
$V^{\mathrm{adm}}$: the epipolar geometry admits no continuous deformation
compatible with the generalized Kruppa equation. No induction, no descent
hypothesis $(\ddagger)$, no input at the bottom of the chain and no bound on
$\dim\J_X(2)$ is required.
\end{theorem}

\begin{proof}
$X$ is a smooth complete intersection with $\deg X>2$, so by
lemma~\ref{lem::Vreal_norm} the projection $\pi_1|_X$ is birational onto $Y_1$;
hence $(Z,L)\iso(X,\mathcal{O}_X(1))$ and
$\operatorname{Aut}(Z,L)=\operatorname{Lin}(X)$, which is finite by
remark~\ref{rem::Vreal_aut}. Theorem~\ref{thm::Vreal_finite}(3) therefore applies,
and proposition~\ref{prop::rigidity_from_Vreal} gives the conclusion.
\end{proof}

\begin{remark}[Three routes, and what each costs]
\label{rem::rigidity_vs_finiteness}
Three statements are now available in the regime $2n+1<m$, and they are not
variants of one another.
\begin{enumerate}
\item \emph{Finiteness of $V^{\mathrm{adm}}$}, theorem~\ref{thm::dim(x)*small}:
needs $(\ddagger)$ at every step \emph{and} the finiteness of $V^{\mathrm{adm}}$ at
the bottom of the chain, which is known for no $n\ge1$. Nothing in
section~\ref{sec::Vreal} discharges it: theorem~\ref{thm::Vreal_finite} bounds
$V^{\mathrm{real}}$, and by lemma~\ref{lem::homographic} the two sets genuinely
differ. This remains the strongest conclusion and the least available one.
\item \emph{Rigidity}, theorem~\ref{thm::dim(x)*small_rigid}: needs
$(\natural')_{\Theta_0}$ only, at level $m$ only, and no induction. The whole
obstruction is concentrated in one condition at one level.
\item \emph{Finiteness for the augmented system}, theorem~\ref{thm::aug_system}:
unconditional, but about $\mathcal{S}^{\mathrm{aug}}$ rather than about the
generalized Kruppa equation, and at the cost of a far larger system. It too holds
at level $m$ directly, which is why the induction of this subsection should be read
as serving the Kruppa system specifically, and not as a route to finiteness as
such.
\end{enumerate}
The hypothesis $(\natural')_{\Theta_0}$ plays here the role that liftability
$(\dagger)_{\Theta_0}$ plays in the regime $2n+1\ge m$: it is what allows a
deformation of the epipolar geometry to be followed by actual configurations, and
it is unproved for the same reason. It is however weaker than $(\natural)$, and
weaker than $(\dagger)$ in one respect worth recording: by
proposition~\ref{prop::rigidity_from_Vreal} only $|\operatorname{Lin}(X)|+1$
realizable points on a curve are needed, a finite requirement, whereas $(\dagger)$
asks for analytic behaviour on a whole neighbourhood. That neither hypothesis can
be dispensed with is shown by lemma~\ref{lem::homographic}: when the two images are
projectively equivalent, $V^{\mathrm{adm}}$ carries an $(m-1)$-dimensional family
of non-realizable solutions, $(\natural')$ fails, and so does rigidity.
\end{remark}

\subsubsection{Conclusion} 

Eventually, we conclude this section by the following theorem.

\begin{theorem}[Main Theorem 1]
\label{thm::main_thm_1}
Let $X$ be a smooth, irreducible complete intersection in $\mathbb{P}^m$,
with $\dim(X) = n$ and $\deg(X) > 2$, satisfying the genericity
conditions of section~\ref{sec::genericity_conditions}. Then, for a generic
position of the centers of projection:
\begin{enumerate}
\item \emph{Unconditionally}, the realizable locus $V^{\mathrm{real}}$ is finite,
of cardinal at most $|\operatorname{Lin}(X)|$, and the true configuration is an
isolated point of the solution set of the augmented system
$\mathcal{S}^{\mathrm{aug}}$ of definition~\ref{def::aug_system}. No bound on
$\dim\J_X(2)$ and no liftability, descent or realizability hypothesis is used.
\item If $2n+1\ge m$ and $\dim\J_X(2)\le2$, the true configuration is an isolated
point of the solution variety $V^{\mathrm{adm}}$ of the generalized Kruppa
equation: the epipolar geometry admits no continuous deformation compatible with
the equation. For $n=1$ this is unconditional; for $n\ge2$ it holds under the
liftability hypothesis $(\dagger)_{\Theta_0}$ of
definition~\ref{def::liftability}.
\item If $2n+1<m$, the same rigidity conclusion for $V^{\mathrm{adm}}$ holds under
the single hypothesis $(\natural')_{\Theta_0}$ of
definition~\ref{def::realizable}, with no bound on $\dim\J_X(2)$; alternatively,
under $\dim\J_X(2)\le m-2n+1$, the descent hypothesis $(\ddagger)$ of
definition~\ref{def::descent} along the induction of
section~\ref{sec::dim(X)<codim(X)-1}, and the finiteness of the solution variety
at the bottom of the chain, the generalized Kruppa equation determines the
epipolar geometry up to finite-fold ambiguity, which is the stronger conclusion.
\end{enumerate}
\end{theorem}

\begin{proof}
Part (1) is theorem~\ref{thm::Vreal_finite} together with
corollary~\ref{cor::Vreal_CI}, a complete intersection being projectively normal,
and theorem~\ref{thm::aug_system}. Part (2) is
theorem~\ref{thm::main_dimX>=codimX-1}. In part (3), the first statement is
theorem~\ref{thm::dim(x)*small_rigid} and the second is
theorem~\ref{thm::dim(x)*small}, whose hypotheses are exactly those listed,
through the induction of section~\ref{sec::dim(X)<codim(X)-1};
corollary~\ref{cor::fund_equiv_proj} then converts the finite set of solutions
into the projections, hence into the whole family of fundamental matrices, up to a
finite-fold ambiguity. The two regimes of (2) and (3) are exhaustive.
\end{proof}
 
\medskip
 
\begin{remark}
Since $m-2n+1\ge3$ when $2n+1<m$, the hypothesis in the second case is
strictly weaker than $\dim\mathcal J_X(2)\le2$; in particular the uniform
hypothesis $\dim\mathcal J_X(2)\le2$ remains sufficient in every regime.
The bound $\dim\mathcal J_X(2)\le2$ is sharp in the boundary regime
$2n+1=m$ (see the discussion preceding
proposition~\ref{prop::reduction2} and Table~\ref{tab::evidence}), while
above the boundary the obstruction disappears together with the quadric
minors (lemma~\ref{lem::structure_Lambda}(3)).
\end{remark}
 
\medskip
 
\begin{remark}
The complete-intersection hypothesis enters only through the verification
of the no-quadric conditions: theorem~\ref{thm::restricted} for $(\star)$
when $2n+1\ge m$, and theorem~\ref{thm::star_upstairs} for
$(\star_{m-2n})$ when $2n+1<m$. Any smooth irreducible variety satisfying
the appropriate condition and the genericity conditions of
section~\ref{sec::genericity_conditions} enjoys the same conclusion: this
is the content of $P_m$ together with
proposition~\ref{prop::no_quadric}. Finding further classes of varieties
for which these conditions can be checked (for instance arithmetically
Cohen--Macaulay varieties, for which the argument of
lemma~\ref{lem::structure}(3) plausibly extends) would enlarge the scope
of the theorem without any change to the argument.
\end{remark}
 
\medskip 

\begin{remark}
\label{rem::classes_and_quadrics}
It is worth noting that in proposition~\ref{prop:lowerbound}, we have shown that for $c$ large enough the variety $V$ is expected to be discrete. In turns out that further analysis shows that the number of quadrics containing $X$ is the crucial criterion. 

However, there is a relation between the class $c$ and the number of quadrics containing $X$, as we shall see in the following.

For a smooth complete intersection $X \subset \mathbb{P}^m$ of type $(d_1, \ldots, d_k)$ with $\dim(X) = n = m - k$, the class is given by
\[
c(X) = d_1 \cdots d_k \cdot h_n(d_1 - 1, \ldots, d_k - 1)
\]
where $h_n$ denotes the complete homogeneous symmetric polynomial of degree $n$, defined by $h_n(x_1, \ldots, x_k) = \sum_{a_1 + \cdots + a_k = n} x_1^{a_1} \cdots x_k^{a_k}$. 
This expression can be derived from the known formula relying on Chern classes, as shown in~\cite{Tevelev-05}. Alternatively, elementary classical means can be used to prove the above expression. 

Through Koszul complex, we also get: 
$$
\dim \mathcal{J}_X(2) = \sum_{\emptyset \neq S \subseteq \{1,\ldots,k\}} (-1)^{|S|-1} \binom
{m + 2 - \sum_{j \in S} d_j}{m}
$$
with the convention $\binom{t}{m} = 0$ for $t < 0$. Since $d_j \geq 2$ for every
$j$, the only subsets contributing are the singletons $S=\{j\}$ with $d_j=2$, each
by $+1$; the formula therefore reduces to $\dim \mathcal{J}_X(2) = q$, in
accordance with lemma~\ref{lem::structure}(1).

This allows to compute a few examples:
\begin{enumerate}
	\item For $m=3,n=1$, with $d_1=d_2=3$, we get $\deg(X) = 9$, $c(X) = 36 > \frac{1}{2}(m+2)(m+1)-3 = 7$ and $\dim(\mathcal{J}_X(2)) = 0$. 
	\item For $m=4,n=2$, with $d_1 = d_2 = 2$, we get $\deg(X) = 4$, $c(X) = 12 = \frac{1}{2}(m+2)(m+1)-3 = 12$ and $\dim(\mathcal{J}_X(2)) = 2$.
	\item For $m=4,n=2$, with $d_1 = 2, d_2 = 3$, we get $\deg(X) = 6$, $c(X) = 42 > \frac{1}{2}(m+2)(m+1)-3 = 12$ and $\dim(\mathcal{J}_X(2)) = 1$.
	\item For $m=5,n=2$, with $d_1 = d_2 = d_3 = 2$, we get $\deg(X) = 8$, $c(X) = 48 > \frac{1}{2}(m+2)(m+1)-3 = 18$ and $\dim(\mathcal{J}_X(2)) = 3$.
	\item For $m=6,n=3$, with $d_1 = d_2 = 2, d_3 = 3$, we get $\deg(X) = 12$, $c(X) = 312 > \frac{1}{2}(m+2)(m+1)-3 = 25$ and $\dim(\mathcal{J}_X(2)) = 2$. 
	\item In general for codimension $2$ complete intersection of a quadric and a cubic, the class $c = 6 \cdot (2^{m-1}-1)$ satisfies $c \geq \frac{1}{2}(m+2)(m+1)-3$ for all $m \geq 3$, and $\dim(\mathcal{J}_X(2)) = 1$.
\end{enumerate}
\end{remark}

\subsection{Recovering The Variety}
\label{sec::recovering_the_variety}

Let the projection matrices be ${\bf M}_1$ and ${\bf
M}_2$. Let us consider the cone $\Delta_i$ defined by the projected variety $Y_i$ and the center of projection $O_i$. It is defined by the following set of equations $\{f({\bf M}_i {\bf P}) = 0 | {\bf P} \in \PP^m, f \in I(Y_i)\}$, where $I(Y_1)$ is the homogeneous ideal defining $Y_i$. 

The reconstruction is defined as the recovery of $X$ in the variety $\Delta_1 \cap \Delta_2$. This variety has two irreducible components as the following theorem states.

We start our analysis with the following lemma, that will be useful in the proof of theorem~\ref{thm::reconstruction}, and which has also been used in section~\ref{sec::dim(X)<codim(X)-1}.

\begin{lemma}[Simple branching and its stability under projection]
\label{lemma::simple_branching}
Let $X\subset\PP^m$ be a smooth, irreducible, non-degenerate variety of dimension
$n$ and degree $\delta$, satisfying the genericity conditions of
section~\ref{sec::genericity_conditions} (in particular $X^*$ is a reduced
hypersurface). Let $Z\subset\PP^m$ be a generic linear subspace of dimension
$m-n-1$, so that $Z\cap X=\varnothing$, and let
$$
\eta_Z:X\lra \sigma(Z)\cong\PP^n,\qquad x\longmapsto\langle Z,x\rangle,
$$
be the associated finite morphism of degree $\delta$, which sens $x \in X$ to the $(m-n)-$dimensional space spanned by $x$ and $Z$. Then:
\begin{enumerate}

\item $\eta_Z$ is the restriction to $X$ of the linear projection of center $Z$, and
for every $x\in X$
$$
\operatorname{corank} d_x\eta_Z=\dim\bigl(\widehat{T_xX}\cap\hat Z\bigr),
$$
where we use the notation $\hat{W}$ for the affine cone over the linear space $W$, as done in previous sections. In particular $x$ is a ramification point if and only if $T_xX\cap Z\neq\varnothing$.

\item $\eta_Z$ is a finite morphism between smooth varieties of the same dimension over
$\C$, hence separable, and the ramification locus $R_Z$ is the zero divisor of
$\det(d\eta_Z)$, a non-zero section of $K_X\otimes\eta_Z^*K_{\PP^n}^{-1}$. Since
$\eta_Z^*\mathcal{O}_{\PP^n}(1)=\mathcal{O}_X(1)$, this reads
$$
R_Z\in\bigl|\,K_X\otimes\mathcal{O}_X(n+1)\,\bigr| .
$$
Moreover $R_Z\neq\varnothing$: otherwise $\eta_Z$ would be \'etale, hence an isomorphism
since $\PP^n$ is simply connected, forcing $\delta=1$, whereas
$\delta\ge\codim X+1\ge 3$. Consequently $R_Z$ is a non-empty divisor on $X$ and
$\mathcal{B}_Z=\eta_Z(R_Z)$ is a hypersurface of $\PP^n$.

If moreover $X=V(f_1,\dots,f_k)$ is a complete intersection of type $(d_1,\dots,d_k)$
($k=m-n$) and $\hat Z=\langle z_1,\dots,z_k\rangle$, then $R_Z$ is cut on $X$ by the
single equation
$$
D(x):=\det\bigl[\,J(x)z_1\mid\cdots\mid J(x)z_k\,\bigr]=0,
\qquad \deg D=\sum_{j=1}^k (d_j-1),
$$
so that $R_Z\in\bigl|\mathcal{O}_X(\sum_j(d_j-1))\bigr| = \PP(H^0(X,\mathcal{O}_X(\sum_j(d_j-1))))$ and $\deg R_Z=\delta\cdot\sum_j(d_j-1)$.

\item Let $x\in R_Z$ be a ramification point of corank one, that is
$\dim(\widehat{T_xX}\cap\hat Z)=1$ by~(1), say $\widehat{T_xX}\cap\hat Z=\langle z\rangle$,
and let $v\in\ker d_x\eta_Z$ be the corresponding tangent direction, namely the class
of $z$. Then $\eta_Z$ has an ordinary fold at $x$ if and only if 
$$
\mathrm{II}_x(v,v)\notin \hat Z+\widehat{T_xX},
$$
where $\mathrm{II}_x$ is the second fundamental form of $X$ at $x$; this holds at a
generic point of $R_Z$ because $X^*$ is reduced.

\item Consequently $\eta_Z$ is a simple branched covering: the branch locus
$B_Z=\eta_Z(R_Z)$ is a hypersurface of $\PP^n$, and the planes of $\sigma(Z)$
tangent to $X$ at two or more distinct points form a closed subset of $B_Z$ of
dimension at most $n-2$, hence of positive codimension in $B_Z$.
\item If $2n+1<m$ and $\tau:\PP^m\dashrightarrow\PP^{m-1}$ is a generic linear
projection, then $X'=\tau(X)$ is again smooth, irreducible and non-degenerate of
dimension $n$, satisfies the same genericity conditions, and the conclusions
$(1)$--$(4)$ hold for $X'$ with a generic center $Z'\subset\PP^{m-1}$ of dimension
$m-n-2$.
\end{enumerate}
\end{lemma}

\begin{proof}
\emph{(1) $\eta_Z$ is a linear projection, and its differential.}
Let $\hat Z\subset\C^{m+1}$ be the cone over $Z$, of dimension $k=m-n$. Sending
$\Lambda\in\sigma(Z)$ to $\hat\Lambda/\hat Z$ identifies
$\sigma(Z)\cong\PP(\C^{m+1}/\hat Z)\cong\PP^n$, and under this identification the
map $x\mapsto\langle Z,x\rangle$ is the linear projection $\pi_Z=\PP(q)$ with
$q:\C^{m+1}\to\C^{m+1}/\hat Z$ the quotient. Thus $\eta_Z=\pi_Z|_X$, which is a
finite morphism of degree $\delta$ since $Z\cap X=\varnothing$.

Using the canonical identifications $T_{[v]}\PP(V)\cong\Hom(\langle v\rangle,
V/\langle v\rangle)$ and $T_xX\cong\Hom(\langle\hat x\rangle,\widehat{T_xX}/
\langle\hat x\rangle)$ (see~\cite[Lec.16]{Harris-92}), the differential $d_x\eta_Z$ is induced by the linear map
$$
\widehat{T_xX}\big/\langle\hat x\rangle\;\lra\;
\C^{m+1}\big/\bigl(\hat Z+\langle\hat x\rangle\bigr).
$$
Both sides have dimension $n$: the source by definition, and the target because
$x\notin Z$ gives $\hat Z\cap\langle\hat x\rangle=0$, whence
$\dim(\hat Z+\langle\hat x\rangle)=k+1 = m-n+1$. Its kernel is
$\bigl(\widehat{T_xX}\cap(\hat Z\oplus\langle\hat x\rangle)\bigr)/\langle\hat
x\rangle$. If $v=z+\lambda\hat x\in\widehat{T_xX}$ with $z\in\hat Z$, then
$z=v-\lambda\hat x\in\widehat{T_xX}$ because $\hat x\in\widehat{T_xX}$; hence
$$
\widehat{T_xX}\cap\bigl(\hat Z\oplus\langle\hat x\rangle\bigr)
=\bigl(\widehat{T_xX}\cap\hat Z\bigr)\oplus\langle\hat x\rangle,
$$
and therefore $\ker d_x\eta_Z\cong\widehat{T_xX}\cap\hat Z$, which gives the corank
formula. Since $\widehat{T_xX}\cap\hat Z\neq 0$ is precisely the condition
$T_xX\cap Z\neq\varnothing$, the ramification criterion follows.

\emph{(2) The determinantal equation.}
As $\widehat{T_xX}=\ker J(x)$ with $J(x)$ the $k\times(m+1)$ Jacobian, one has
$\widehat{T_xX}\cap\hat Z\neq 0$ if and only if there is $c\neq0$ with
$\sum_ic_iz_i\in\ker J(x)$, where we have $<z_1, \ldots, z_k>$ generating the space $\hat Z$. i.e.\ if and only if the $k\times k$ matrix $[\,J(x)z_1\mid\cdots\mid J(x)z_k\,]$ is singular. Each column $J(x)z_i$ has $j$-th
entry $\nabla f_j(x)\!\cdot z_i$ of degree $d_j-1$, so expanding the determinant
gives $\deg\mathcal D=\sum_j(d_j-1)$. 

Moreover $\mathcal D$ is not identically zero on $X$: if it were, then for every $x\in X$ there would be a non-zero $c$ with $\sum_i c_i z_i\in\ker J(x)$, i.e.\ $\sum_i c_i z_i \in \widehat{T_xX}$, so that $\hat Z \cap \widehat{T_xX} \neq 0$ for every $x\in X$. But $\dim(\widehat{T_x X}) = n+1$ and $\dim(\hat Z) = m-n$, so by the genericity of $Z$ we have $\dim(\widehat{T_x X } \cap \hat {Z}) = 0$ for a generic $x \in X$. 

Therefore $R_Z=X\cap V(\mathcal D)$ is
a divisor in $|\mathcal O_X(\sum_j(d_j-1))|$ of degree $\delta \sum_j(d_j-1)$.

\emph{(3) The fold criterion.}
By~(1) the corank of $d_x\eta_Z$ equals $\dim(\widehat{T_xX}\cap\hat Z)$, so the
locus $\{\corank\ge2\}$ is where the $k\times k$ matrix
$[\,J(x)z_1\mid\cdots\mid J(x)z_k\,]$ drops rank by $2$, i.e.\ where
$\dim(\widehat{T_xX}\cap\hat Z)\ge2$. This is the preimage under the Gauss map
$\gamma:X\to\G(n+1,m+1)$ of the Schubert variety
$\{V:\dim(V\cap\hat Z)\ge2\}$, of codimension $4$ in the Grassmannian; by
Kleiman's transversality theorem \cite{Kleiman-74}, for generic $Z$ the map
$\gamma$ is transverse to the smooth locus of this Schubert variety and the
preimage of its singular locus has dimension at most $n-4$. In particular the
divisor $R_Z$ is smooth away from a closed subset of dimension at most $n-4$,
and at the generic point of each of its components the corank is exactly one.
Let then $x\in R_Z$ be a
point where the corank is exactly $1$, so
$\widehat{T_xX}\cap\hat Z=\langle z\rangle$ and $\ker d_x\eta_Z=\langle v\rangle$ is
one-dimensional, $v$ being the class of $z$. Then $\operatorname{coker}d_x\eta_Z$ is
also one-dimensional, and equals
$$
\C^{m+1}\big/\bigl(\hat Z+\widehat{T_xX}\bigr),
$$
which is indeed a line since
$\dim(\hat Z+\widehat{T_xX})=k+(n+1)-\dim(\hat Z\cap\widehat{T_xX})=m$. The
intrinsic derivative of $\eta_Z$ at $x$,
$$
\ker d_x\eta_Z\times\ker d_x\eta_Z\lra\operatorname{coker}d_x\eta_Z,
$$
is the composition of the second fundamental form
$\mathrm{II}_x:\Sym^2T_xX\to N_{X/\PP^m,x}$~\cite{Griffiths-Harris-1994} with the projection onto this line.

By the Morin normal form \cite{Morin-65} (see also \cite{Boardman-67} for the
corank stratification and the intrinsic derivative in the complex-analytic
setting), $\eta_Z$ has an ordinary fold at $x$ exactly when this quadratic form is
non-zero on $\ker d_x\eta_Z$, that is when
$\mathrm{II}_x(v,v)\notin\hat Z+\widehat{T_xX}$.

Now $\hat Z+\widehat{T_xX}$ is a hyperplane of $\C^{m+1}$; it is the cone over the
unique hyperplane $H_x\subset\PP^m$ containing both $T_xX$ and $Z$. The condition
$\mathrm{II}_x(v,v)\notin\hat Z+\widehat{T_xX}$ therefore says that $H_x$ has
non-degenerate contact with $X$ at $x$ in the direction $v$, which is exactly the
condition that $[H_x]$ be a smooth point of the dual variety $X^*$. We now
show that, for generic $Z$, this holds at the generic point of every component
of $R_Z$.
 
The image of the map $\varkappa:R_Z\to X^*$, $x\mapsto[H_x]$, is exactly
$X^*\cap Z^\vee$, where $Z^\vee\cong\PP^{n}\subset(\PP^m)^*$ is the linear
space of hyperplanes containing $Z$. Indeed $[H_x]\in X^*\cap Z^\vee$ by
construction; conversely, if a hyperplane $H\supseteq Z$ is tangent to $X$ at
a point $x$, then $T_xX$ and $Z$ are linear subspaces of $H\cong\PP^{m-1}$ of
respective dimensions $n$ and $m-n-1$, whose sum equals $\dim H$, so
$T_xX\cap Z\neq\varnothing$; hence $x\in R_Z$ and $H=H_x$. Next, $R_Z$ is
irreducible when $n\ge2$. Indeed the condition
$\widehat{T_xX}\cap\hat Z\neq0$ is linear in the Pl\"ucker coordinates of
$\widehat{T_xX}$ (it reads
$p(\widehat{T_xX})\wedge z_1\wedge\cdots\wedge z_k=0$ in
$\join^{m+1}\C^{m+1}\cong\C$), so $R_Z$ is the zero divisor on $X$ of the
pullback by $\gamma$ of a Pl\"ucker hyperplane section; since $X$ is smooth,
irreducible and non-linear, its Gauss map $\gamma$ is finite
\cite{Tevelev-05}, so this pullback line bundle is ample and $R_Z$ is
connected, being an effective ample divisor on a smooth projective variety of
dimension $\ge2$ \cite[Cor.~B.29]{Kleiman-05}. If $R_Z$ had two distinct
irreducible components, they would meet in dimension $n-2$, whereas by the
transversality statement above $R_Z$ is smooth away from a subset of dimension
at most $n-4$; hence $R_Z$ is irreducible. Now $\varkappa$ is surjective onto
$X^*\cap Z^\vee$, which has pure dimension $n-1$ since $X^*$ is a hypersurface
and $Z^\vee$ is a generic $n$-plane of $(\PP^m)^*$; and since $X^*$ is
reduced, $\Sing(X^*)$ has codimension at least $2$ in $(\PP^m)^*$, so
$Z^\vee\cap\Sing(X^*)$ has dimension at most $n-2$. The generic point of the
irreducible divisor $R_Z$ therefore maps to a smooth point of $X^*$, and the
fold condition holds there. For $n=1$ the same count shows that $Z^\vee$, then
a generic line of $(\PP^m)^*$, misses $\Sing(X^*)$ altogether, so the fold
condition holds at every point of the finite set $R_Z$.
 
\emph{(4) Double tangency is subordinate.}
Consider the incidence variety
$$
\begin{array}{rcl}
\Xi & = & \Bigl\{(x_1,x_2,\Lambda,Z):x_1\neq x_2\in X,\ \Lambda\in\G(m-n,m),\\
 & & \hspace{1.0cm}\overline{x_1x_2}\subseteq\Lambda,\ \Lambda\ \text{tangent at }x_1,x_2,\
Z\subset\Lambda,\ \langle Z,x_i\rangle=\Lambda\Bigr\},    
\end{array}
$$
fibered over $\G:=\G(m-n-1,m)$ by $(x_1,x_2,\Lambda,Z)\mapsto Z$. Counting through
the successive projections: the pairs $(x_1,x_2)$ contribute $2n$; for fixed
$(x_1,x_2)$, the $(m-n)$-planes containing the line $\overline{x_1x_2}$ form a
Grassmannian of dimension $n(m-n-1)$, on which the two tangency conditions impose
one condition each, contributing $n(m-n-1)-2$; and for fixed $\Lambda$ the centers
$Z$ with $\langle Z,x_i\rangle=\Lambda$ are the hyperplanes of
$\Lambda\cong\PP^{m-n}$ missing $x_1,x_2$, contributing $m-n$. Hence
$$
\dim\Xi\le 2n+\bigl(n(m-n-1)-2\bigr)+(m-n)=n(m-n)+m-2 .
$$
Since $\dim\G=(m-n)(n+1)=n(m-n)+(m-n)$, we get
$$
\dim\Xi-\dim\G\le(m-2)-(m-n)=n-2 .
$$
So for generic $Z$ the fiber $\Xi_Z$ has dimension at most $n-2$ (and is empty for
$n\le1$), and its image in $\sigma(Z)$, that is, the locus of twice-tangent planes, has
dimension at most $n-2$, of positive codimension in the branch hypersurface $B_Z$,
which has dimension $n-1$ by~(2). Together with~(3), a generic point of $B_Z$ is the
image of a single ramification point at which $\eta_Z$ is an ordinary fold; thus
$\eta_Z$ is a simple branched covering.

\emph{(5) Stability under generic projection.}
Assume $2n+1<m$ and let $\tau$ be a generic linear projection with center
$O\notin X$. By section~\ref{sec::generic_projection}, in the range $2n+1<m$ the
generic $\tau$ embeds $X$, so $X'=\tau(X)\subset\PP^{m-1}$ is smooth, irreducible,
non-degenerate of dimension $n$ and degree $\delta$, and its dual $X'^*$ is again a
reduced hypersurface. Parts~(1), (3), (4) and the first half of~(2) (i.e. $R_Z$
is a non-empty divisor and $B_Z$ a hypersurface of $\PP^n$) use only these
properties and not that $X$ is a complete intersection, so they apply verbatim to
$X'$ with a generic center $Z'\subset\PP^{m-1}$ of dimension $(m-1)-n-1=m-n-2$;
only the determinantal equation $D(x)$ of~(2) requires a complete intersection.
\end{proof}

\begin{theorem}[Main Theorem 2]
\label{thm::reconstruction}
Assume the genericity conditions of section~\ref{sec::genericity_conditions}. Then
$$
\Delta_1 \cap \Delta_2 \;=\; X \cup R,
$$
where the residual locus $R$ is either empty or of pure dimension $2n-m+2$. More
precisely:
\begin{enumerate}
\item If $n < m-2$, then $\dim R = 2n-m+2 < n$ (and $R = \varnothing$ as soon as
$2n+2 < m$). Hence $X$ is the unique irreducible component of $\Delta_1 \cap \Delta_2$ of dimension $n$, and the reconstruction is unambiguous, whatever the degree of $X$.

\item If $n = m-2$, and if moreover the centers of projection are in a generic position,
namely such that the covering $\eta_L : X \lra \sigma(L) \cong \PP^n$ associated with the
baseline $L$ is a simple branched covering in the sense of
lemma~\ref{lemma::simple_branching}, then $R$ is irreducible of degree $d(d-1)$, and
$\Delta_1 \cap \Delta_2$ has exactly two irreducible components: one of degree $d$,
which is $X$ and is the actual solution of the reconstruction, and one of degree $d(d-1)$.
\end{enumerate}
\end{theorem}

\begin{proof}
\textbf{Part 1: set-theoretic description of $\Delta_1 \cap \Delta_2$.}

We first record that the baseline does not meet $X$. The lines of $\PP^m$ meeting $X$
form a subvariety of $\G(1,m)$ of dimension $n + (m-1) < 2(m-1) = \dim \G(1,m)$, since
$n \leq m-2$; as the map $({\bf O}_1,{\bf O}_2) \mapsto L$ is dominant onto $\G(1,m)$,
for a generic position of the centers of projection we have $L \cap X = \varnothing$.
Equivalently, ${\bf e}_i \notin Y_i$: indeed ${\bf e}_2 = \pi_2({\bf O}_1)$ lies in
$Y_2 = \pi_2(X)$ if and only if some $x \in X$ satisfies $\pi_2(x) = {\bf e}_2$, that is
$x \in \overline{{\bf O}_2 {\bf e}_2} = L$.

Write $\tilde{X} = \Delta_1 \cap \Delta_2$. For ${\bf P} \neq {\bf O}_i$ one has
${\bf M}_i {\bf P} = \pi_i({\bf P})$, so ${\bf P} \in \Delta_i$ if and only if
$\pi_i({\bf P}) \in Y_i$, that is if and only if the line
$\overline{{\bf O}_i {\bf P}}$ meets $X$. Moreover ${\bf O}_1 \in \Delta_2$ would mean
${\bf e}_2 \in Y_2$, which we have just excluded; so neither vertex lies in $\tilde X$.
Consequently
$$
{\bf P} \in \tilde{X} \iff \exists\, x_1, x_2 \in X \ \text{ with } \
{\bf P} \in \overline{{\bf O}_1 x_1} \cap \overline{{\bf O}_2 x_2}.
$$

Two cases occur.

\emph{(i) $x_1 = x_2 = x$.} Since $L \cap X = \varnothing$, the point $x$ does not lie on
$L$, so the lines $\overline{{\bf O}_1 x}$ and $\overline{{\bf O}_2 x}$ are distinct;
being distinct and both passing through $x$, they meet exactly at $x$. Hence
${\bf P} = x \in X$. Conversely every $x \in X$ arises in this way, so $X \subseteq \tilde X$.

\emph{(ii) $x_1 \neq x_2$.} The lines $\overline{{\bf O}_1 x_1}$ and
$\overline{{\bf O}_2 x_2}$ meet if and only if they are coplanar, if and only if the four
points ${\bf O}_1, {\bf O}_2, x_1, x_2$ span a plane, if and only if the secant
$\overline{x_1 x_2}$ meets the baseline $L$. In that case the two lines are still distinct
and they coincide only if $x_1, x_2 \in L$, which is excluded. So they meet in a single point.

Therefore
\begin{equation}
\label{eq::residual}
\tilde{X} \;=\; X \cup R, \qquad
R \;=\; \overline{\Bigl\{\, \overline{{\bf O}_1 x_1} \cap \overline{{\bf O}_2 x_2}
\ :\ x_1 \neq x_2 \in X, \ \overline{x_1 x_2} \cap L \neq \varnothing \,\Bigr\}}.
\end{equation}

\medskip
\textbf{Part 2: $\dim R = 2n-m+2$.}

Recall from section~\ref{sec::generic_projection} that $\mathfrak{S}(X) \subset \G(1,m)$
denotes the variety of secant lines of $X$. Since $X$ is smooth, irreducible and
non-degenerate of dimension $n \leq m-2$, it is not a linear subspace, so a generic secant
is not contained in $X$ and the rational map $X \times X \dashrightarrow \G(1,m)$,
$(x_1,x_2) \mapsto \overline{x_1 x_2}$, is generically finite onto its image; hence
$\dim \mathfrak{S}(X) = 2n$.

Let $\sigma_L = \{\ell \in \G(1,m) \mid \ell \cap L \neq \varnothing\}$ be the special
Schubert cycle of lines meeting $L$, of codimension $m-2$ in $\G(1,m)$. The group
$\Pr_m$ acts transitively on $\G(1,m)$, and $L$ is a generic line, so $\sigma_L$ is a
generic translate of a fixed Schubert cycle. By Kleiman's transversality theorem
\cite{Kleiman-74} the intersection
$\mathfrak{S}(X) \cap \sigma_L$ is either empty or of pure dimension
$$
\dim \mathfrak{S}(X) - \codim \sigma_L \;=\; 2n - (m-2),
$$
and it is empty whenever this number is negative.

Finally the map sending a pair $(x_1,x_2)$ as in~\eqref{eq::residual} to the point
$\overline{{\bf O}_1 x_1} \cap \overline{{\bf O}_2 x_2}$ has finite fibers: given
${\bf P}$, each $x_i$ lies in $\overline{{\bf O}_i {\bf P}} \cap X$, which is finite
because a line through ${\bf O}_i \notin X$ cannot be contained in $X$. Therefore
$$
\dim R \;=\; 2n-m+2, \qquad \text{and } R = \varnothing \text{ when } 2n+2 < m .
$$

In particular, if $n < m-2$ then $2n-m+2 < n$, so $X$ is the unique irreducible component
of $\tilde{X}$ of dimension $n$; this proves item~1 of the theorem, and note that no
hypothesis on $\deg X$ and no branched-covering argument were needed.

\medskip
\textbf{Part 3: the case $n = m-2$.}

Assume from now on $n = m-2$, so that $m-n = 2$. Set $Z = L$, a generic line of $\PP^m$
with $Z \cap X = \varnothing$, and $Z_i = {\bf e}_i$, so that $Z_i \cap Y_i = \varnothing$
by Part 1. Let $\sigma(Z)$ be the pencil of planes of $\PP^m$ containing $L$ and
$\sigma(Z_i)$ the pencil of lines of $i_i(\PP^{m-1})$ through ${\bf e}_i$; both are
canonically isomorphic to $\PP^{m-2} = \PP^n$. Since $L$ is generic,
lemma~\ref{lemma::simple_branching} applies to the covering $\eta = \eta_Z$.

Consider the following coverings of $\PP^n$:

\begin{enumerate}
\item $X \stackrel{\eta}{\longrightarrow} \sigma(Z) \cong
\PP^n$, taking a point $x \in X$ to the epipolar $(m-n)-$plane that it defines with the space $Z$.
\item $Y_1 \stackrel{\eta_1}{\longrightarrow} \sigma(Z_1) \cong
\sigma(Z) \cong \PP^n$, taking a point $y \in Y_1$ to the
epipolar $(m-n-1)-$plane that it defines with the space $Z_1$ in the first projection hyperplane.
\item $Y_2 \stackrel{\eta_2}{\longrightarrow} \sigma(Z_2) \cong
\sigma(Z) \cong \PP^n$, taking a point $y \in Y_2$ to the
epipolar $(m-n-1)-$plane that it defines with the space $Z_2$ in the second projection plane.
\end{enumerate}

The generic fibers of $\eta,\eta_1,\eta_2$ are all finite of cardinal $d$ (counting with multiplicities), the degree of $X$. There are branch points as follows:
\begin{enumerate}
\item the branch points of $\eta$ are made of $(m-n)-$planes in $\sigma(Z)$ that are tangent to $X$, that is, by lemma~\ref{lemma::simple_branching}(1), the images of the points $x \in X$ such that $T_xX \cap Z \neq \emptyset$,
\item the branch points of $\eta_i$ are made of $(m-n-1)-$planes in $\sigma(Z_i)$ that are tangent to $Y_i$ or that contain singular points of $Y_i$.
\end{enumerate}

If $\rho_i$ is the projection $X \to Y_i$, then
$\eta=\eta_i\rho_i$. Let ${\cal B}$ be the union set of branch loci of
$\eta_1$ and $\eta_2$.  It is clear that the branch locus of $\eta$
is included in ${\cal B}$.

Now we shall prove that under the genericity condition of the theorem, the set $\mathcal{B}$ contains an algebraic hypersurface of $\PP^n$, which is the branch locus of $\eta$.

Indeed over a generic point $W'$ in $\sigma(Z) \iso \PP^n$, the fiber of $\eta$ is a set of $d = \deg(X)$ distinct points: $\{x_1, \cdots, x_d\}$. These points are given by $d$ local morphisms, $\phi_1, \cdots, \phi_d$ from $\PP^n$ to $X$. We have the following characterization of $\mathcal{B}_\eta$ (the branch locus of $\eta$):
$$
W' \in \mathcal{B}_\eta \Leftrightarrow \exists i,j, i \neq j, \phi_i(W') = \phi_j(W').
$$

By lemma~\ref{lemma::simple_branching}(2), the ramification locus $R_Z$ is a divisor in $X$, so that $\mathcal{B}_\eta = \eta(R_Z)$ is a hypersurface of $\PP^n$: the codimension of $\mathcal{B}_\eta$ is indeed $1$. Moreover, by lemma~\ref{lemma::simple_branching}(4), the $(m-n)-$planes of $\sigma(Z)$ that are tangent to $X$ at two or more distinct points form a closed subset of $\mathcal{B}_\eta$ of dimension at most $n-2$. Hence for a generic $W' \in \mathcal{B}_\eta$ there is exactly one pair of distinct indexes $i,j$ such that $\phi_i(W') = \phi_j(W')$, and in particular there is no triple of distinct indexes $i,j,k$ such that $\phi_i(W') = \phi_j(W') = \phi_k(W')$.

Let $S=\PP^n \setminus {\cal B}$, pick $t
\in S$, and write $X_S=\eta^{-1}(S)$, $X_t=\eta^{-1}(t)$.  Let
$\mu_{X_S}$ be the monodromy: $\pi_1(S,t) \longrightarrow \Perm(X_t)$, where $\Perm(A)$ is the group of permutation of a finite set $A$. It is well known that the path-connected components of $X$ are in one-to-one correspondence with the orbits of the action of $\im(\mu_{X_S})$ on $X_t$.  Since $X$ is assumed to be irreducible, it has only one component and $\im(\mu_{X_S})$ acts transitively on $X_t$. Then it is enough to prove that $\im(\mu_{X_S})$ contains a transposition to conclude that $\im(\mu_{X_S}) = \Perm(X_t)$.

In order to show that $\im(\mu_{X_S})$ actually contains a transposition, consider a
loop in $\PP^n$ based at $t$, say $l_t$.  If $l_t$ does not go round
any branch point, then $l_t$ is homotopic to the constant path in $S$
and then $\mu_{X_S}([l_t])=1$.  Now as mentioned above, there are two
types of branch points in ${\cal B}$:
\begin{enumerate}
\item branch points that come from singularities of $Y_1$ or $Y_2$: these are not
branch points of $\eta$,
\item branch points that come from epipolar $(m-1-n)-$planes tangent either to $Y_1$ or to $Y_2$: these are branch points of $\eta$.
\end{enumerate}

If the loop $l_t$ goes round a point of the first types, then it is still true that $\mu_{X_S}([l_t])=1$.  Now suppose that $l_t$ goes round a genuine branch point of $\eta$, say $b$ (and goes round no other points in $\cal B$). Such a loop indeed exists since the codimension of $\mathcal{B}_\eta$ is 1.

Since the twice-tangent locus has dimension at most $n-2$ while $\mathcal{B}_\eta$ has dimension $n-1$, the branch point $b$ can be chosen generic in $\mathcal{B}_\eta$, so that it is the image of a single ramification point $x$, at which, by lemma~\ref{lemma::simple_branching}(3), the covering $\eta$ has an ordinary fold. After choosing adapted local charts, the local expression of $\eta$ is then given by $w_1=z_1^2,w_2=z_2, \cdots, w_n=z_n$, the remaining $d-2$ sheets being unramified above $b$. Hence $\mu_{X_S}([l_t])$ is a transposition.

All of this shows that (i) $\im(\mu_{X_S})$ is generated by transpositions, (ii) the action of $\im(\mu_{X_S})$ on $X_t$ is transitive, (iii) the fundamental group $\pi_1(S,t)$ is generated by loops that go round a single branch point.	Then we conclude that $\im(\mu_{X_S}) = \Perm(X_t)$. See~\cite[Ch. 11]{Eisenbud-Harris-2016} for more details on this.

We note first that $L\cap\tilde X=\varnothing$: a point ${\bf P}\in L$ lies in
$\Delta_1$ only if the line $\overline{{\bf O}_1{\bf P}}=L$ meets $X$, which is
excluded by Part~1. Consequently the projection from $L$ restricted to any
irreducible component of $\tilde X$ of dimension $n$ is finite of degree equal to
the projective degree of that component, so the degrees computed below over $\PP^n$
are the degrees of the components as subvarieties of $\PP^m$.

Since $m-n = 2$, any two lines $\overline{{\bf O}_1 x_1}$ and $\overline{{\bf O}_2 x_2}$
lying in a common member of $\sigma(Z)$ (which is a plane) necessarily meet. Combined with
Part 1, this shows that ${\bf P} \mapsto (\pi_1({\bf P}), \pi_2({\bf P}))$ identifies
$\tilde{X}$ with the fiber product $Y_1 \times_{\PP^n} Y_2$. Since
$Z_i \cap Y_i = \varnothing$, the projection $\eta_i : Y_i \to \PP^n$ is finite of degree
$d$, so $Y_1 \times_{\PP^n} Y_2$ is finite of degree $d^2$ over $\PP^n$, hence pure of
dimension $n$; thus $\tilde{X}$ has no excess component and its degree over $\PP^n$ is
$d^2$.

Restricting to the inverse image of $S$, we have
$\tilde{X}_S \cong X_S \times_S X_S$, and we can identify $\tilde{X}_t$ with
$X_t \times X_t$. The monodromy $\mu_{\tilde{X}_S}$ is then given by
$\mu_{\tilde{X}_S}(x,y) = (\mu_{X_S}(x), \mu_{X_S}(y))$. Since
$\im(\mu_{X_S}) = \Perm(X_t)$, the action of $\im(\mu_{\tilde{X}_S})$ on
$X_t \times X_t$ has exactly two orbits, namely $\{(x,x)\} \cong X_t$ and
$\{(x,y) \mid x \neq y\}$. Hence $\tilde{X}$ has two irreducible components: one of
degree $d$, which is $X$, and one of degree $d^2 - d = d(d-1)$, which is the locus $R$
of~\eqref{eq::residual}. This proves item~2.

\end{proof}

\begin{remark}
When $n = m-2$, the two components of $\Delta_1 \cap \Delta_2$ have degrees $d$ and
$d(d-1)$; these are distinct as soon as $d > 2$, which identifies the correct one and
solves the reconstruction. Under the genericity conditions of
section~\ref{sec::genericity_conditions} this is automatic: an irreducible
non-degenerate variety of codimension $2$ has degree at least $3$. The classical
obstruction for $d = 2$, already present for curves in $\PP^3$~\cite{Kaminski-04},
arises precisely because a conic of $\PP^3$ is degenerate.

In every higher codimension the situation is strictly better: by item~1 the residual
locus has dimension $2n-m+2 < n$, so $X$ is recovered as the unique component of
$\Delta_1 \cap \Delta_2$ of maximal dimension, with no restriction on $\deg X$ whatsoever.
In particular the degree-$2$ obstruction is a phenomenon of codimension $2$ only.
\end{remark}

\begin{example}[An explicit reconstruction in dimension $n=2$]
\label{ex::reconstruction_surface}
Theorem~\ref{thm::reconstruction} is effective in an explicit surface
instance. Let $X\subset\PP^4$ be the del Pezzo quartic of
example~\ref{ex::delpezzo} with the centers ${\bf O}_1=(1,2,1,3,1)^T$
and ${\bf O}_2=(2,1,1,1,3)^T$ and explicit rational projection matrices.
The images $Y_i\subset\PP^3$ are quartic surfaces whose equations are
computed exactly as resultants,
$f_i({\bf y})=\mathrm{Res}_t\bigl(f({\bf M}_i^+{\bf y}+t{\bf O}_i),\,
g({\bf M}_i^+{\bf y}+t{\bf O}_i)\bigr)$, and the cones are
$\Delta_i({\bf P})=f_i({\bf M}_i{\bf P})$. Here $n=m-2$, so
theorem~\ref{thm::reconstruction} predicts that $\Delta_1\cap\Delta_2$
is the union of $X$, of degree $4$, and of a residual surface of degree
$d(d-1)=12$. This is certified by a plane-section computation in exact
arithmetic: restricting to a generic $2$-plane $\Pi\subset\PP^4$, the
eliminant of the two restricted quartic curves has degree $16$, in
accordance with B\'ezout, and factors exactly as the product of the
degree-$4$ eliminant of $X\cap\Pi$ and a cofactor of degree $12$; the
two factors are coprime and both are squarefree, so the section meets
$X$ in $4$ simple points and the residual component in $12$ simple
points, with no common point. The variety $X$ is thus recovered as the
degree-$4$ component of $\Delta_1\cap\Delta_2$, as predicted.
\end{example}

\medskip
{\bf Acknowledgment - } I would like to particularly thank Prof. Pierre Lochak from Institut Math\'ematique de Jussieu, Sorbonne Université, for having read a very preliminary version of this article and shared with me his comments which greatly allowed me to improve this manuscript.


\begin{thebibliography}{numeric}


\bibitem{Barnabei-Brini-Rota-85} 
{\sc M.Barnabei, A.Brini and G-C.Rota}, {\em On the Exterior Calculus of Invariant Theory},
In Journal of Algebra, 96, 120-160(1985).
 
\bibitem{Boardman-67}
{\sc J. M. Boardman}, {\em Singularities of differentiable maps},
Publ. Math. Inst. Hautes \'Etudes Sci. {\bf 33} (1967), 21--57.


\bibitem{Benoist-13}
{\sc O. Benoist}, {\em S\'eparation et propri\'et\'e de Deligne--Mumford des champs
de modules d'intersections compl\`etes lisses}, J. London Math. Soc. {\bf 87}
(2013), 138--156.


\bibitem{Bertolini-Besana-Notari-Turrini-19}
{\sc M.~Bertolini, G.M.~Besana, R.~Notari and C.~Turrini}, {\em Matrices
dropping rank in codimension one and critical loci in computer vision},
arXiv:1902.00376, 2019.

\bibitem{Bertolini-Notari-Turrini-21}
{\sc M.~Bertolini, R.~Notari and C.~Turrini}, {\em Smooth determinantal
varieties and critical loci in multiview geometry}, arXiv:2102.09874, 2021.

\bibitem{Bratelund-22}
{\sc M.~Br\aa telund}, {\em Critical configurations for two projective views, a
new approach}, arXiv:2112.05074, 2022.


\bibitem{Bruns-Herzog-98}
{\sc W. Bruns and J. Herzog}, {\em Cohen-Macaulay Rings}, Cambridge University Press, 1998.

\bibitem{Bruns-Vetter-88}
{\sc W. Bruns and U. Vetter}, {\em Determinantal Rings}, Springer, 1988.


\bibitem{deJong-98}
{\sc T. de Jong}, {\em An algorithm for computing the integral closure},
J. Symbolic Comput. {\bf 26} (1998), 273--277.

\bibitem{Eisenbud-Harris-2016}
{\sc D. Eisenbud and J. Harris}, {\em 3264 and All That: A Second Course in Algebraic Geometry}, Cambridge University Press, 2016.


\bibitem{Ellingsrud-Ottem}
{\sc G.~Ellingsrud and J.C.~Ottem}, {\em Introduction to Schemes}, book in
preparation, available at \url{https://www.uio.no/studier/emner/matnat/math/MAT4215/data/masteragbook.pdf}.

\bibitem{Faugeras-Luong-01} {\sc O.D.~Faugeras and Q.T.~Luong}, 
{\em The Geometry of Multiple Images}, MIT Press, 2001.

\bibitem{Fulton-98}
{\sc W. Fulton}, {\em Intersection Theory}, Second Edition, Springer, 1998.

\bibitem{Gathmann-2002}
{\sc A. Gathmann}, {\em Algebraic Geometry}, Lecture Notes, 2002.

\bibitem{Gotzmann-78}
{\sc G.~Gotzmann}, {\em Eine Bedingung f\"ur die Flachheit und das Hilbertpolynom
eines graduierten Ringes}, Math. Z. {\bf 158} (1978), 61--70.


\bibitem{Greuel-Laplagne-Seelisch-10}
{\sc G.-M. Greuel, S. Laplagne and F. Seelisch}, {\em Normalization of rings},
J. Symbolic Comput. {\bf 45} (2010), 887--901.

\bibitem{Grassmann-1862}
{\sc H. Grassmann}, {\em Die Ausdehnungslehre}, Berlin, 1862.

\bibitem{Griffiths-Harris-1994}
{\sc P. Griffiths and J. Harris}, {\em Principles of Algebraic Geometry}, Wiley, 1994.

\bibitem{Gruson-Peskine-2010}
{\sc L. Gruson and C. Peskine}, {\em The $k-$Secant Lemma and the General Projection Theorem}, arXiv: 1010.2399v1, 2010

\bibitem{Haiman-Sturmfels-04}
{\sc M.~Haiman and B.~Sturmfels}, {\em Multigraded Hilbert schemes},
J. Algebraic Geom. {\bf 13} (2004), 725--769.


\bibitem{Harris-92} 
{\sc J.~Harris}, {\em Algebraic Geometry, a first course}, 
Springer-Verlag, 1992.


\bibitem{Hartley-Kahl-07}
{\sc R.~Hartley and F.~Kahl}, {\em Critical configurations for projective
reconstruction from multiple views}, Int. J. Comput. Vision {\bf 71} (2007),
5--47.


\bibitem{Hartley-Schaffalitzky-04}
{\sc R.I.~Hartley and F.~Schaffalitzky}, {\em Reconstruction from projections
using Grassmann tensors}, in Proceedings of the 8th European Conference on
Computer Vision, Prague, Lecture Notes in Computer Science, Springer, 2004.

\bibitem{MVG-03}
{\sc R. Hartley and A. Zisserman}, {\em Multiple View Geometry}, Cambridge University Press, Second edition, 2003. 


\bibitem{Holme-88}
A. Holme, {\em The Geometric and Numerical Properties of Duality in Projective Algebraic Geometry}, manuscripta math. 61, 145-162 (1988). 


\bibitem{Kaminski-04} 
{\sc J.Y.~Kaminski, M. Fryers and M. Teicher}, 
{\em Recovering an algebraic curve using its projections from different points. Applications to static and dynamic computational vision}, 
In Journal of the European Mathematical Society, 2005, 7(2). 

\bibitem{Kaminski-11} 
{\sc J.Y.~Kaminski and Y. Sepulcre}, 
{\em Using discriminant curves to recover a surface of $\C \PP^4$ from two generic linear projections}, 
In Proceedings of the 36th international symposium on Symbolic and algebraic computation, 2011. 


\bibitem{Kaminski-26-code}
{\sc J.Y.~Kaminski}, {\em Computations for ``Recovering a smooth projective
variety from two projections''}, {\sc SageMath}, {\sc Singular} and {\sc SymPy}
scripts with their outputs and notebooks, Zenodo, 2026,
\url{https://doi.org/10.5281/zenodo.24114814}; repository
\url{https://github.com/YIRMK/variety-recovery-computations}.


\bibitem{Kileel-Kohn-25}
{\sc J. Kileel and K. Kohn}, {\em Snapshot of Algebraic Vision}, AMS Proceedings of Symposia in Pure Mathematics, 2025
 

\bibitem{Kleiman-74}
{\sc S. Kleiman}, {\em The transversality of a general translate}, Compositio Mathematica {\bf 28} (1974), 287-297.

\bibitem{Kleiman-05}
{\sc S.L.~Kleiman},
{\em The Picard scheme},
in: Fundamental Algebraic Geometry: Grothendieck's FGA Explained,
Mathematical Surveys and Monographs 123, American Mathematical Society,
2005, pp.~235--321. Also available as arXiv:math/0504020.


\bibitem{Kruppa-1913}
{\sc E. Kruppa}, {\em "Zur Ermittlung eines Objektes aus zwei Perspektiven mit innerer Orientierung" (On determining an object from two perspectives with inner orientation)}, Sitzungsberichte der Mathematisch-Naturwissenschaftlichen Klasse der Kaiserlichen Akademie der Wissenschaften in Wien, Abt. IIa, vol. 122, pp. 1939–1948, 1913.


\bibitem{Matsumura-Monsky-64}
{\sc H. Matsumura and P. Monsky}, {\em On the automorphisms of hypersurfaces},
J. Math. Kyoto Univ. {\bf 3} (1964), 347--361.

\bibitem{Monge-1799}
{\sc G. Monge}, {\em G\'eom\'etrie descriptive}, Baudoin, 1799. 


\bibitem{Morgan-Sommese-87}
{\sc A.P.~Morgan and A.J.~Sommese},
{\em A homotopy for solving general polynomial systems that respects
m-homogeneous structures},
Applied Mathematics and Computation, 24(2):101--113, 1987.


\bibitem{Morin-65}
{\sc B. Morin}, {\em Formes canoniques des singularit\'es d'une application
diff\'erentiable}, C. R. Acad. Sci. Paris {\bf 260} (1965), 5662--5665 and
6503--6506.

\bibitem{Mourrain-91} 
{\sc B. Mourrain} {\em New aspects of geometrical calculus with invariants},
Advances in Mathematics, 1994.


\bibitem{Poncelet-1822}
{\sc J. Poncelet}, {\em Trait\'e des propri\'et\'es projectives des figures}, Metz, 1822


\bibitem{Plucker-1830}
{\sc J. Pl\"ucker}, {\em Analytisch-geometrische Entwicklungen}, Essen, 1828-1831.


\bibitem{Sommese-Wampler-05}
{\sc A.J.~Sommese and C.W.~Wampler},
{\em The Numerical Solution of Systems of Polynomials Arising in
Engineering and Science},
World Scientific, 2005.

\bibitem{Tevelev-05}
{\sc E. Tevelev}, {\em Projective Duality and Homogeneous Spaces}, Springer, 2005. 


\bibitem{Vakil-25}
{\sc R.~Vakil},
{\em The Rising Sea: Foundations of Algebraic Geometry},
Princeton University Press, 2025. 
 

\end{thebibliography}
\end{document}